\documentclass[12pt, chapterprefix=false]{report}% For LaTeX2e
\usepackage{suthesis-2e}
\usepackage{geometry}
\usepackage[noadjust]{cite}
\usepackage{caption}
\usepackage{subcaption}
\usepackage{hyperref}
\usepackage{url}
\usepackage[table]{xcolor}
\usepackage{amsfonts}
\usepackage{graphicx}
\usepackage{color}
\usepackage{placeins}
\usepackage{float}
\usepackage{slashbox}
\usepackage{tabularx,colortbl}
\usepackage{amsmath}
\usepackage{amssymb}
\usepackage{amsthm}
\usepackage{bbm}
\usepackage{epstopdf}
\usepackage{breqn}
\usepackage{graphicx}
\usepackage{color}
\usepackage{placeins}
\usepackage{float}
\usepackage{tabularx,colortbl}
\usepackage{amssymb}
\usepackage{amsbsy}
\usepackage{bm}
\usepackage{amsthm}
\usepackage{bbm}
\usepackage{epstopdf}
\usepackage{breqn}
\usepackage{cite}
\usepackage{framed}
\usepackage{url}
\usepackage{amsfonts}
\usepackage{mathrsfs}
\usepackage{mathtools}
\usepackage{algorithm}
\usepackage{algpseudocode}
\usepackage{pifont}
\usepackage{dblfloatfix}
\usepackage{arydshln}
\usepackage{subfiles}
\usepackage{setspace}

\usepackage{sectsty}

\usepackage{titlesec}
\titleformat{\chapter}[hang] 
{\normalfont\Large\bfseries}{\chaptertitlename\ \thechapter:}{1em}{} 
\chaptertitlefont{\Large}
\titlespacing*{\chapter}{0pt}{-50pt}{40pt}

\DeclareGraphicsExtensions{.pdf}
\DeclareMathOperator*{\argmax}{arg\,max}
\DeclareMathOperator*{\argmin}{arg\,min}
\DeclareMathOperator*{\subjectto}{subject\,to}

\DeclareMathOperator*{\minimize}{minimize}
\DeclareMathOperator*{\maximize}{maximize}
\DeclareMathOperator*{\hertz}{Hz}

\DeclareMathOperator*{\support}{supp}

\newtheorem{thm}{Theorem}
\newtheorem{lemma}{Lemma}
\newtheorem{corollary}{Corollary}
\newtheorem{prop}{Proposition}
\newtheorem{definition}{Definition}

\newcommand*{\QEDB}{\hfill\ensuremath{\blacksquare}}%
\definecolor{purple1}{RGB}{100, 85, 255}
\definecolor{orange1}{RGB}{233, 146, 0}

\dept{Electrical and Computer Engineering}
\begin{document}

\graphicspath{ {./figures/} }

\title{\textbf{Compressed Sensing Beyond the I.I.D. and Static Domains:} \\\textbf{ Theory, Algorithms and Applications}}
\author{Abbas Kazemipour}

\principaladviser{Min Wu}
\coprincipaladviser{Behtash Babadi}
\firstreader{Shaul Druckmann}
\secondreader{Radu Balan}
\thirdreader{Prakash Narayan} %if needed\
\fourthreader{Johnathan Fritz} %if needed\
\singlespacing
  \begin{singlespace}
 %Abstract Page

%\hbox{\ }
\pagenumbering{gobble}
\renewcommand{\baselinestretch}{1}
\small \normalsize

\begin{center}
\large{{ABSTRACT}}

\vspace{2em}

\end{center}
\hspace{-.15in}
\begin{tabular}{ll}
Title of dissertation:    & {\large  COMPRESSED SENSING BEYOND THE IID}\\
&                     {\large   AND STATIC DOMAINS: } \\
&                     {\large  THEORY, ALGORITHMS AND APPLICATIONS} \\
\ \\
&                          {\large  Abbas Kazemipour} \\
&                           {\large Doctor of Philosophy, 2017} \\
\ \\
Dissertation directed by: & {\large  Professors Min Wu and  Behtash Babadi} \\
&               {\large  Department of Electrical and Computer Engineering } \\
\end{tabular}

\vspace{1em}

\renewcommand{\baselinestretch}{2}
\large \normalsize
 \begin{singlespace}
Sparsity is a ubiquitous feature of many real world signals such as natural images and neural spiking activities. Conventional compressed sensing utilizes sparsity to recover low dimensional signal structures in high ambient dimensions using few measurements, where i.i.d measurements are at disposal. However real world scenarios typically exhibit non i.i.d and dynamic structures and are confined by physical constraints, preventing applicability of the theoretical guarantees of compressed sensing and limiting its applications. In this thesis we develop new theory, algorithms and applications for non i.i.d and dynamic compressed sensing by considering such constraints.

In the first part of this thesis we derive new optimal sampling-complexity tradeoffs for two commonly used processes used to model dependent temporal structures: the autoregressive processes and self-exciting generalized linear models. Our theoretical results successfully recovered the temporal dependencies in neural activities, financial data and traffic data.

Next, we develop a new framework for studying temporal dynamics by introducing compressible state-space models, which simultaneously utilize spatial and temporal sparsity. We develop a fast algorithm for optimal inference on such models and prove its optimal recovery guarantees. Our algorithm shows significant improvement in detecting sparse events in biological applications such as spindle detection and calcium deconvolution.

Finally, we develop a sparse Poisson image reconstruction technique and the first compressive two-photon microscope which uses lines of excitation across the sample at multiple angles. We recovered diffraction-limited images from relatively few incoherently multiplexed measurements, at a rate of 1.5 billion voxels per second.
\end{singlespace}
\end{singlespace}
  \begin{singlespace}
\beforepreface
\end{singlespace}
        \thispagestyle{empty}%
        \newpage
		\mbox{}
		\newpage
\setcounter{page}{2}
\doublespacing
\prefacesection{Preface}

Sparsity, coherence and dynamics are among the ubiquitous  features of many real world signals and systems. Examples include natural images, sounds and neural spiking activities. These signals exhibit sparsity, that is there exists a basis for which the effective dimension of the signal is much smaller than its ambient dimension. Moreover natural signals entail rich temporal dynamics. Theory of compressed sensing {\cite{donoho2006compressed, candes2006compressive, candes2006stable, candes2008introduction, needell2009cosamp, bruckstein2009sparse}} is concerned with reconstruction of such signals by utilizing their sparse structures, using very few measurements . Compressed sensing provides sharp trade-offs between the number of measurement, sparsity, and estimation accuracy when random i.i.d measurements are at disposal. However, in most practical applications of interest, the measured signals and the corresponding covariates are highly interdependent and follow specific temporal dynamics. Although, the theory of compressed sensing has not considered non-i.i.d and dynamic domains, the recovery algorithms suggested by it show remarkable performance once the structure of the underlying signal is taken into account.

On a high level, research in compressive sensing is conducted in three main branches: theory, algorithms and applications. In this thesis we revisit all three branches  for nonlinear models with interdependent covariates as well as dynamic compressive sensing with applications in neural signal processing, traffic modeling, financial data and media forensics. Three main proposed problems, our ongoing research and future work are included in details in the following chapters. The rest of this thesis is organized as follows: In Chapter \ref{chap:ar} we introduce the problem of stable estimation of high-order AR processes, generalizing  results from i.i.d compressive sensing to general class of stable AR processes with sub-Gaussian innovations. In doing so we will show that spectral properties of stationary processes determine the interdependence of the covariates and in general the sampling-complexity tradeoffs of AR processes.
In Chapter \ref{chap:hawkes} we introduce the problem of robust estimation of generalized linear models. We will prove theoretical guarantees of compressive sensing for such models, characterizing the sampling-complexity tradeoffs between the interdependence of the covariates and the number of measurement. We further corroborate our theoretical guarantees with simulated data as well as an application to retinal ganglion cell spiking activities. In Chapter \ref{chap:css} we provide theoretical guarantees of stable state estimation where the underlying dynamics follow a compressible state-space model. Moreover we show application of our results in denoising and spike deconvolution from calcium imaging recordings of neural spiking activities and show promising recovery results from compressed imaging data. We propose an application of this method in designing compressive calcium imaging devices. Finally, in Chapter \ref{chap:slapmi} we use ideas from projection microscopy to develop a two-photon imaging technique that scans lines of excitation across the sample at multiple angles, recovering high-resolution images from relatively few incoherently multiplexed measurements. By using a static image of the sample as a prior, we recover neural activity from megapixel fields of view at rates exceeding 1kHz.

\prefacesection{Acknowledgment}

I feel very fortunate to have the chance to learn from and collaborate with incredibly talented individuals in the past 5 years. First and foremost, I would like to thank my academic  advisor, Professor Min Wu, who supported and believed in me throughout my PhD. I feel extremely indebted to her, for her patience, critical thinking and giving me the opportunity to follow my research interests. Second, I would like to thank my co-advisor, Professor Behtash Babadi. His unique blend of creativity, generosity and energy makes him a role model for my future career. I would also like to thank my mentors at Janelia Research Campus. I would like to thank Professor Shaul Druckmann and Dr. Kaspar Podgorski for their incredible support of my work and great contributions to my thesis.

I would like to thank my committee members and group leaders at Janelia for their encouragements and my collaborators at University of Maryland and Janelia who provided critical feedback on my work. Many of the interesting research problems and ideas came from fruitful discussions with them. I would also like to thank Professor Ali Olfat at University of Tehran for his incredible support and believing in me so early in my career.

I would like to thank my friends at UMD and Janelia, who made the past 5 years a plausible experience for me. 

Last but not least, I would like to thank my family, especially my parents who always chose my happiness over their own convenience. Thank you for filling my life with love.

% The list is too long to fit in here, but thank you all.

%Getting through these 5 years wouldn't be as easy without the support of my friends at UMD and Janelia. I would like to thank Mohammad Rastegari who gave me advice on all the tough decisions of my life and I never regret listening to, also to my friends in Janelia who made the last year of my PhD one of the best of my life. The list is too long to fit in here, but thank you all. 

\prefacesection{Notations}
Throughout this thesis we use bold lower and upper case letters for denoting vectors and matrices, respectively. Parameter vectors are {denoted} by bold-face {Greek} letters. For example, $\boldsymbol{\theta}=[\theta_1,\theta_2,\cdots,\theta_p]'$ denotes a $p$-dimensional parameter vector, with $[\cdot]'$ denoting the transpose operator. For a vector $\boldsymbol{\theta}$, we define its decomposition into positive and negative parts given by:
\[
\boldsymbol{\theta} = \boldsymbol{\theta}^+ - \boldsymbol{\theta}^{-},
\]
where $\boldsymbol{\theta}^\pm = \max\{\pm\boldsymbol{\theta}, \mathbf{0}\}$.
It can be shown that
\[
\| \boldsymbol{\theta}^\pm\|_1= \mathbf{1}' \boldsymbol{\theta}^\pm = \frac{\|\boldsymbol{\theta}\|_1 \pm \mathbf{1}'\boldsymbol{\theta}}{2}
\]
are convex in $\boldsymbol{\theta}$. Similarly, for a summation 
\[
L = \sum_{i=1}^n l_i = L^+ -L^-,
\]
where $L^+ = \sum_{i=1} ^n\max\{l_i,0\}$, $L^- = -\sum_{i=1}^n \max\{-l_i,0\}$.
We will use the notation $\mathbf{x}_i^j$ to denote the vector $[x_i,\cdots,x_j]^T$ for any $i, j \in \mathbb{Z}$ with $i \le j$. We will denote the estimated values by $\widehat{(.)}$ and the biased estimates with the superscript $(.)^b$.  Throughout the proofs, $c_i$'s express absolute constants which may change from line to line where there is no ambiguity. By $c_\eta$ we mean an absolute constant which only depends on a positive constant $\eta$. We denote the support of a vector $\mathbf{x}_t \in \mathbb{R}^p$ by $\support (\mathbf{x}_t)$ and its $j$th element by $(\mathbf{x}_{t})_j$. 

Given a sparsity level $s$ and a vector $\mathbf{x}$, we denote the set of its $s$ largest magnitude entries by $S$, and its best $s$-term approximation error by $\sigma_s(\mathbf{x})= \|\mathbf{x}-\mathbf{x}_S\|_1$. When $\sigma_s(\mathbf{x}) \sim \mathcal{O}^{(1/2 - \xi)}$ for some $\xi \ge 0$, we refer to $\mathbf{x}$ as $(s,\xi)$--compressible.

For simplicity of notation, we define $\mathbf{x}_0$ to be the all-zero vector in $\mathbb{R}^p$. For a matrix $\mathbf{A}$, we denote restriction of  $\mathbf{A}$ to its first $n$ rows by $(\mathbf{A})_n$.

We use the convention $[T] = \{1,\cdots,T\}$ and $\mathbf{W}_{[T]} = \left[ \mathbf{w}_1,\cdots,\mathbf{w}_T \right] $, i.e. $\mathbf{w}_k$ represents the $k$th column of $\mathbf{W}_{[T]}$. $\odot$ and $\oslash$ denote elementwise multiplication and division respectively. Throughout the chapter we will use the terms innovations and spikes interchangeably. Unless otherwise stated, a function acts on a vector elementwise. For a matrix $\mathbf{A} =[a_{ij}] \in \mathbb{R}^{m \times n}$ its mixed $p,q$-norm is denoted by $\left\|\mathbf{A}\right\|_{p,q}$, i.e.
\[
\left\|\mathbf{A}\right\|_{p,q} = \left[  \sum_{i=1}^m \left( \sum_{j=1}^n \left|a_{ij}\right|^p \right)^{q/p} \right]^{1/q},
\]
and $\|\mathbf{x}\|_{\mathbf{\Sigma}}^2 = \mathbf{x}^T \mathbf{\Sigma}^{-1} \mathbf{x}$.

\afterpreface
\chapter{Introduction}
Data compression and its fundamental limits are one of the main questions in communication theory \cite{cover2012elements}. Classical communication theory treats data acquisition and data compression as two separate problems and has studied each of these {modules} extensively, for example  in information theory it is well known that the fundamental limit of data compression is given by the entropy of source. With the emergence of big data applications and the prohibitive  cost of sensing mechanisms for high resolution data, smarter sensing mechanisms seem to be inevitable. 

The theory of compressive sensing  \cite{foucart2013mathematical} takes its name from the premise that in many applications data acquisition and compression can be performed simultaneously. This is done by utilizing the sparse structure of many natural signals such as images, sound etc. For example, natural images are known to be sparse in Fourier bases. By taking advantage of the sparse structure one could go beyond the fundamental limits imposed by physical constraints and uncertainty principles. 

Research in compressive sensing can be categorized in three major branches: mathematical theory, algorithm design and applications.  In this thesis we focus on all three aspects. From a mathematical perspective, we use tools from empirical process theory and statistical signal processing in order to study sampling complexity tradeoffs of compressive sensing for dynamic and non i.i.d compressive sensing. This is motivated by the fact that interdependence and dynamic structures are ubiquitous features of natural signals. From the algorithmic perspective, we provide fast algorithms for signal recovery under these conditions, and finally we will apply our theory in several applications of interest.  

In order to facilitate reading of this of thesis we will provide a very brief introduction to compressive sensing and a few key results which will be used recurrently throughout the thesis. A more detailed treatment can be found in \cite{foucart2013mathematical}.

\section{Sparse Solutions of Underdetermined Systems}
A vector $\boldsymbol{\theta} \in \mathbb{R}^p$ is $s$-sparse if it has at most $s$ nonzero entries, i.e. if
\[
\|\boldsymbol{\theta}\|_0 := \text{card}\left(\support(\boldsymbol{\theta}) \right) \leq s.
\]
We define
\begin{equation}
\sigma_s(\boldsymbol{\theta}) := \|\boldsymbol{\theta}-\boldsymbol{\theta}_s\|_1
\end{equation}  
and \begin{equation}
\varsigma_s(\boldsymbol{\theta}) := \|\boldsymbol{\theta}-\boldsymbol{\theta}_s\|_2
\end{equation}
which are scalar functions of $\boldsymbol{\theta}$ and $s$, and capture the compressibility of the parameter vector $\boldsymbol{\theta}$ in the $\ell_1$ and $\ell_2$ sense, respectively. Note that by definition 
$\varsigma_s(\boldsymbol{\theta}) \leq \sigma_s(\boldsymbol{\theta})$. For a fixed $\xi \in (0,1)$, we say that $\boldsymbol{\theta}$ is \emph{$(s,\xi)$-compressible} if $\sigma_s(\boldsymbol{\theta}) = \mathcal{O}(s^{1-\frac{1}{\xi}})$ \cite{needell2009cosamp}. Note that when $\xi = 0$, the parameter vector $\boldsymbol{\theta}$ is exactly $s$-sparse.

Let
\begin{align}
\label{eq:comp1}
\mathbf{y} = \mathbf{A}\boldsymbol{\theta} + \mathbf{v},
\end{align}
where $\mathbf{A}\in \mathbf{R}^{n \times p}$ is a known measurement matrix and $\mathbf{y} \in \mathbb{R}^n $ consists of $n$ linear measurements of $\boldsymbol{\theta}$ and $\mathbf{v}$ is the bounded measurement noise satisfying
\[
\| \mathbf{v}\|_2 \leq \epsilon.
\]

Compressive sensing problem is concerned with solving (\ref{eq:comp1}) in the underdetermined setting, i.e. when $n < p$.  In this setup, the key assumption of $\boldsymbol{\theta}$ being (close to) an $s$-sparse vector can be used to recover $\boldsymbol{\theta}$. In the absence of observation noise $\boldsymbol{\theta}$ can be \textit{exactly} recovered from $\mathbf{y}$ if rank$(A) \geq 2s$, which requires at least $n \geq 2s$ measurements. Special designs of measurement matrices such as Vandermonde matrices or Fourier matrices have resulted in elegant reconstruction algorithms such as Prony's method \cite{parks1987digital}. 

In the presence of noise one would need higher number of measurements . Ideally one would like to solve the optimization of finding the \textit{sparsest} $\boldsymbol{\theta}$ which satisfies the bounded noise constraints, i.e.
\begin{align}
\label{eq:comp2}
\begin{tabular}{l}$\minimize_{\boldsymbol{\theta} \in \mathbb{R}^p} \quad \|\boldsymbol{\theta}\|_0$,\\
$\subjectto \quad \|\mathbf{Ax}-\mathbf{y}\|_2 \leq \epsilon.$
\end{tabular}
\end{align}

Unfortunately  ($\ref{eq:comp2}$) is an NP-hard problem and due to the combinatorial nature of $\ell_0$norm. The convex relaxation of (\ref{eq:comp2}) is therefore used most often, which is achieved by replacing the $\ell_0$ norm with an $\ell_1$ norm, i.e.

\begin{align}
\label{eq:comp3}
\begin{tabular}{l}$\minimize_{\boldsymbol{\theta} \in \mathbb{R}^p} \quad \|\boldsymbol{\theta}\|_1$,\\
$\subjectto \quad \|\mathbf{A} \boldsymbol{\theta}-\mathbf{y}\|_2 \leq \epsilon.$
\end{tabular}
\end{align}

\section{Restricted Isometry Property}
For the convex optimization problem (\ref{eq:comp3}) the following property is \textit{sufficient} for stable recovery of $\boldsymbol{\theta}$:
\begin{definition}[Restricted Isometry Property \cite{candes2008introduction}]
The matrix $\mathbf{A} \in \mathbb{R}^{n \times p}$ satisfies the restricted isometry property (RIP) \cite{candes2006compressive} of order $s$, if for all $s$-sparse $\boldsymbol{\theta}\in \mathbb{R}^p$, we have
\begin{equation}
\label{eq:tv_rip}
(1-\delta_s) \|\boldsymbol{\theta}\|_2^2 \leq \|\mathbf{A}\boldsymbol{\theta}\|_2^2 \leq (1+\delta_s)\|\boldsymbol{\theta}\|_2^2,
\end{equation}
where $\delta_s \in (0,1)$ is the smallest constant for which Eq. (\ref{eq:tv_rip}) holds. 
\end{definition}
Intuitively speaking, RIP requires that the linear measurements acts as an almost isometry on $s$-sparse vectors, resulting in \textit{invertibility} of the underdetermined system of equations (\ref{eq:comp1}). The following formalizes this idea:

\begin{thm}[Implications of the RIP \cite{RIP_orig} ]
\label{thm:thm1} 
Suppose that $\mathbf{A}$ satisfies the RIP of order $2S$ with $\delta_{2s} < {\sqrt{2}} -1$. Then any solution $\widehat{\boldsymbol{\theta}}$ to (\ref{eq:comp3}) satisfies
\begin{align}
\|\boldsymbol{\theta}-\widehat{\boldsymbol{\theta}}\|_1 \leq c_1 \sigma_s(\boldsymbol{\theta}) + c_2 \sqrt{s} \epsilon,
\end{align}
and
\begin{align}
\|\boldsymbol{\theta}-\widehat{\boldsymbol{\theta}}\|_2 \leq c_1' \frac{\sigma_s(\boldsymbol{\theta})}{\sqrt{s}} + c_2'  \epsilon.
\end{align}
\end{thm}

For an arbitrary value $\delta_s  \in (0,1)$ one can prove existence of a measurement matrix satisfying the RIP as long as $n > c_{\delta_s} s \log (p/s)$ namely:

\begin{thm}[RIP for Random Matrices \cite{foucart2013mathematical} ] 
Let $\mathbf{A}$ be an $n \times p$ subgaussian random matrix. Then there exits a constant $c>0$ depending only on subgaussian parameters, such that the restricted isometry constant of $\frac{1}{\sqrt{n} \mathbf{A}}$ satisfies $\delta_s \leq \delta$ with probability at least $1-\epsilon$ provided
\[
n \geq \frac{c}{\delta^2} \left( s \log\left( e p/s \right) + \log({2}/{\epsilon}) \right).
\]
\end{thm}
Setting $\epsilon = 2 \exp(-\delta^2 n/2c)$ yields the condition
\[
n \geq 2c/\delta^2 s \log \left( ep/s \right),
\]
which guarantees that $\delta_s \leq \delta$ with probability at least $1-2 \exp(-\delta^2n/2c)$.

We will next discuss the commonly used reconstruction algorithms for sparse recovery.

\section{Sparse Recovery Algorithms}
Popular sparse recovery algorithms in compressive sensing can be divided into three main categories: optimization methods, greedy methods, and thresholding-based methods \cite{foucart2013mathematical}. In this thesis we focus on developing algorithms for the optimization methods and greedy methods.

\subsection{Optimization Problems}
Three of the most popular optimization problems for sparse recovery are the quadratically constrained basis pursuit, lasso \cite{tibshirani1996regression} and basis pursuit denoising. The quadratically constrained basis pursuit algorithm solves the optimization problem
\begin{align}
\label{eq:qcbp}
\begin{tabular}{ll}$\minimize_{\boldsymbol{\theta} \in \mathbb{R}^p}$ & $\|\boldsymbol{\theta}\|_1$,\\
$\subjectto$ & $\|\mathbf{A}\boldsymbol{\theta}-\mathbf{y}\|_2 \leq \epsilon.$
\end{tabular}
\end{align}
Alternatively one can find the solution which is closest to the measurements while maintaining a controlled sparsity level. Such formulation is known as the lasso given by

\begin{align}
\label{eq:lasso_def}
\begin{tabular}{ll}$\minimize_{\boldsymbol{\theta} \in \mathbb{R}^p}$ & $\|\mathbf{A}\boldsymbol{\theta}-\mathbf{y}\|_2  $,\\
$ \subjectto$ & $ \|\boldsymbol{\theta}\|_1 \leq \tau.$
\end{tabular}
\end{align}

A closely related optimization problem via the lagrangian formulation of both problems is known as basis pursuit denoising, given by

\begin{align}
\label{eq:lasso_def}
\begin{tabular}{ll}$\minimize_{\boldsymbol{\theta} \in \mathbb{R}^p}$ & $\|\mathbf{A} \boldsymbol{\theta}-\mathbf{y}\|_2 + \lambda \|\boldsymbol{\theta}|\|_1.$
\end{tabular}
\end{align}

\subsection{Greedy Algorithms}

Although there exist fast solvers to convex problems of the type given by Eq. (\ref{eq:qcbp}), these algorithms are polynomial time in $n$ and $p$, and may not scale well with high-dimensional data.  This motivates us to consider greedy solutions for the estimation of sparse parameters. In particular, in this thesis we will consider generalizations of the Orthogonal Matching Pursuit (OMP) \cite{OMP} for general convex cost functions.  The main idea behind the OMP is in the greedy selection stage, where the absolute value of the gradient of the cost function at the current solution is considered as the selection metric. The OMP algorithm adds one index to the estimated support of $\boldsymbol{\theta}$ at each step. A flowchart of the algorithm is given in Table (\ref{tab:OMP_main}).

%\floatsetup[table]{capposition=top}
\begin{table}
\centering
\framebox{$\begin{array}{l}
\text{Input: } \mathbf{A}, \mathbf{y} \\
\text{Output: } \widehat{\boldsymbol{\theta}}_{\sf OMP}\\
\text{Initialization:}\Big\{\begin{array}{l}
\text{Start with the index set } S^{(0)}=\emptyset\\
\text{and the initial estimate }\widehat{\boldsymbol{\theta}}^{(0)}_{{\sf OMP}} = \mathbf{0}
\end{array}\\
\textbf{for } k=1,2,\cdots,s^\star\\
\text{  }\begin{array}{l}
j = \argmax \limits_i \left| \left(  \mathbf{A}' \left(\mathbf{y}-\mathbf{A}\widehat{\boldsymbol{\theta}}_{{\sf OMP}}^{(k-1)}\right) \right)_i\right|\\
S^{(k)}=S^{(k-1)}\cup \{j\}\\
\widehat{\boldsymbol{\theta}}_{{\sf OMP}}^{(k)} = \argmin \limits_{\support ({\boldsymbol{\theta}}) \subset S^{(k)}} \|\mathbf{y}-\mathbf{A}\boldsymbol{\theta}\|_2
\end{array}\\
\textbf{end }\\
\end{array}$
}
\caption{\small{Orthogonal Matching Pursuit (OMP)}}
\label{tab:OMP_main}
\end{table}

The choice of the maximum number of iterations $s^\star$ will be discussed in detail later. One can repeat the process until a stopping criterion is met. The advantage of the OMP algorithm is breaking the high ($p$)-dimensional optimization problem into several low ($\leq s^\star$)-dimensional problem which can usually be solved much faster.

\section{Theoretical Guarantees for  Convex Cost Functions}

In many applications of interest the measurements are not linear, or the noise is not additive or Gaussian. In these scenarios the objective function is not in the form of squared errors. We will now introduce the theoretical requirements which generalize RIP to general convex cost functions, and the corresponding theoretical guarantees. 

\subsection{Restricted Strong Convexity}
General convex cost functions are usually in the form of a loss function or a negative log-likelihood. Estimation problems for such cost functions are known as M-estimation problems. The general form of an M-estimation problem is given by
\begin{equation}
\label{eq:M_Est}
\widehat{\boldsymbol{\theta}}:=\argmin\limits_{\boldsymbol{\theta}\in \boldsymbol{\Theta}} \quad \mathfrak{L}(\boldsymbol{\theta}),
\end{equation}
and its $\ell_1$-regularized counterpart is given by
\begin{equation}
\label{eq:sp_est_pp_L}
\widehat{\boldsymbol{\theta}}_{{\sf sp}}:=\argmin\limits_{\boldsymbol{\theta}\in \boldsymbol{\Theta}} \quad \mathfrak{L}(\boldsymbol{\theta})+ \gamma_n\|\boldsymbol{\theta}\|_1.
\end{equation}
where, $\mathfrak{L}(\boldsymbol{\theta})$ is a cost function which is convex in $\boldsymbol{\theta}$, $\gamma_n > 0$ is a regularization parameter and $\boldsymbol{\Theta}$ is the convex set of admissible solutions.

The main sufficient condition for theoretical guarantees of general convex cost functions is the notion of Restricted Strong Convexity (RSC) \cite{Negahban}. By the convexity of the cost function, it is clear that a small change in $\boldsymbol{\theta}$ results in a small change in the cost function. However, the converse is not necessarily true. Intuitively speaking, the RSC condition guarantees that the converse holds: a small change in the cost implies a small change in the parameter vector, i.e., the cost function is not too \textit{flat} around the true parameter vector. A depiction of the RSC condition for $p=2$, adopted from \cite{Negahban}, is given in Figure \ref{fig:rsc}. In Figure \ref{fig:rsc}(a), the RSC does not hold since a change along $\theta_2$ does not change the log-likelihood, whereas the log-likelihood in Figure \ref{fig:rsc}(b) satisfies the RSC.

\begin{figure}[h]
\centering
\includegraphics[width=.8\columnwidth]{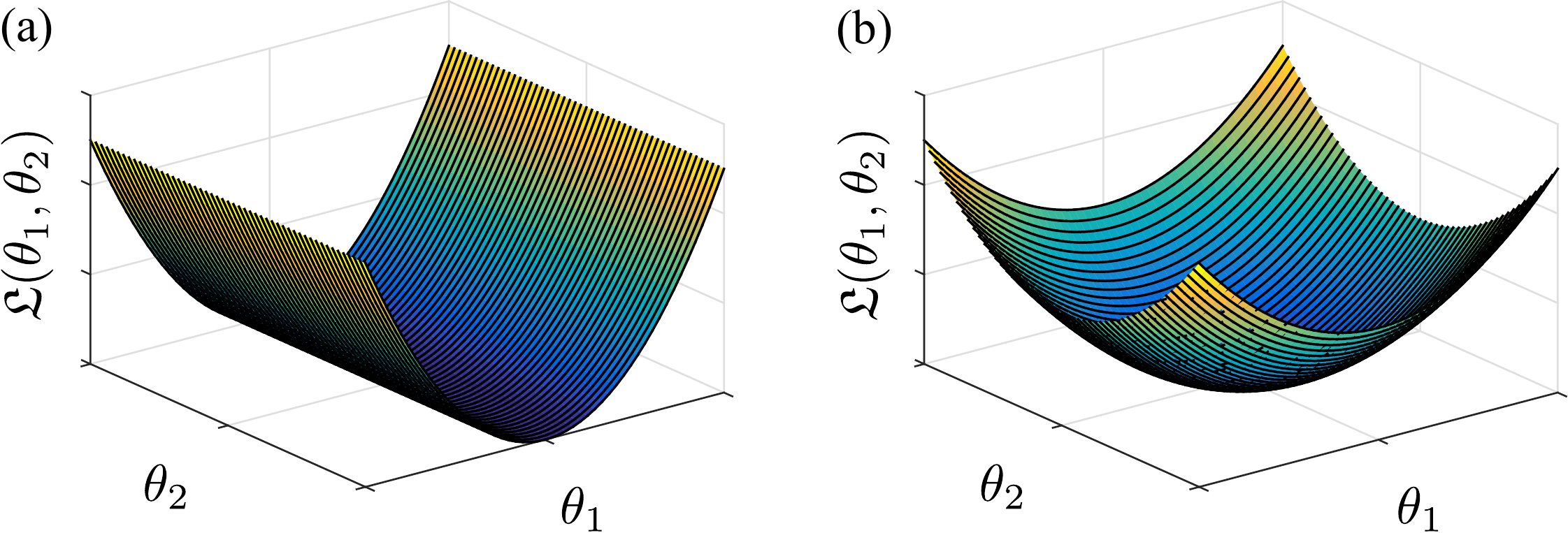}
\caption{\small{Illustration of RSC (a) RSC does not hold (b) RSC does hold.}}\label{fig:rsc}
\end{figure}
More formally, if the log-likelihood is twice differentiable at $\boldsymbol{\theta}$, the RSC is equivalent to existence of a lower quadratic bound on the negative log-likelihood:

\begin{equation}
\label{RSC}
\mathfrak{D_L}({\boldsymbol{\psi}},\boldsymbol{\theta}):= \mathfrak{L}(\boldsymbol{\theta}+{\boldsymbol{\psi}})-  \mathfrak{L}(\boldsymbol{\theta})-{\boldsymbol{\psi}}'\nabla\mathfrak{L}(\boldsymbol{\theta})\geq \kappa\|{\boldsymbol{\psi}}\|_2^2,
\end{equation}

for a positive constant $\kappa > 0$ and all { ${\boldsymbol{\psi}}\in\mathbb{R}^p$} in a carefully-chosen neighborhood of $\boldsymbol{\theta}$ depending on $s$ and $\xi$. Based on the results of \cite{Negahban}, when the RSC is satisfied, sufficient conditions akin to that in Theorem \ref{thm:thm1} can be obtained by estimating the {Euclidean} extent of the solution set around the true parameter vector. 
Here we  restate the main result of \cite{Negahban} concerning RSC and its implications in controlling the estimation error for general convex cost functions:
 
\begin{prop}[Implications of RSC (Theorem 1 of \cite{Negahban})]
\label{negahban_thm}
For a negative log-likelihood $\mathfrak{L}(\boldsymbol{\theta})$ which satisfies the RSC with parameter $\kappa$, every solution to the convex optimization problem (\ref{eq:sp_est_pp_L}) satisfies 
\begin{equation}
\label{error_bound}
\left\|\widehat{\boldsymbol{\theta}}_{\sf {sp}}-\boldsymbol{\theta}\right\|_2 \leq \frac{2\gamma_n \sqrt{s}}{\kappa}+\sqrt{\frac{2\gamma_n \sigma_s(\boldsymbol{\theta})}{\kappa}}
\end{equation}
with a choice of the regularization parameter

\begin{equation}
\label{reg}
\gamma_n \geq   2\left\|\nabla\mathfrak{L}(\boldsymbol{\theta})\right\|_\infty.
\end{equation}
\end{prop}

The first term in the bound (\ref{reg}) is increasing in $s$ and corresponds to the estimation error of the $s$ largest components of $\boldsymbol{\theta}$ in magnitude, whereas the second term is decreasing in $s$ and represents the cost of replacing $\boldsymbol{\theta}$ with its best $s$-sparse approximation. 

Similarly the counterpart of OMP for general cost functions was introduced in \cite{zhang_omp} and is summarized in Table \ref{tab:OMP_gen}. 

\begin{table}
\centering
\framebox{$\begin{array}{l}
\text{Input: } \mathfrak{L}(\boldsymbol{\theta}) , s^\star\\
\text{Output: } \widehat{\boldsymbol{\theta}}_{\sf OMP}^{(s^\star)}\\
\text{Initialization:}\Big\{\begin{array}{l}
\text{Start with the index set } S^{(0)}=\emptyset\\
\text{and the initial estimate }\widehat{\boldsymbol{\theta}}^{(0)}_{{\sf OMP}} = 0
\end{array}\\
\textbf{for } k=1,2,\cdots,s^\star\\
\text{  }\begin{array}{l}
j = \argmax \limits_i \left| \left( \nabla \mathfrak{L} \; \left(\widehat{\boldsymbol{\theta}}_{{\sf OMP}}^{(k-1)}\right) \right)_i\right|\\
S^{(k)}=S^{(k-1)}\cup \{j\}\\
\widehat{\boldsymbol{\theta}}_{{\sf OMP}}^{(k)} = \argmin \limits_{\support ({\boldsymbol{\theta}}) \subset S^{(k)}} \mathfrak{L}(\boldsymbol{\theta})
\end{array}\\
\textbf{end }\\
\end{array}$
}
\caption{\small{Generalized Orthogonal Matching Pursuit }}
\label{tab:OMP_gen}
\end{table}

The main theoretical result regarding the generalized OMP is given by the following Proposition stating that the greedy procedure is successful in obtaining a reasonable $s^\star$-sparse approximation, if the cost function satisfies the RSC:

\begin{prop} [Guarantees of OMP (Theorem 2.1 of \cite{zhang_omp})]
\label{prop:omp_gen}
Suppose that $\mathfrak{L}(\boldsymbol{\theta})$ satisfies RSC with a constant $\kappa > 0$. Let $s^\star$ be a constant such that
\begin{equation}
\label{sstar}
s^\star = \mathcal{O}(s\log s),
\end{equation}
Then, we have
\begin{equation*}
\left \|\widehat{\boldsymbol{\theta}}^{(s^\star)}_{{\sf OMP}}-\boldsymbol{\theta}_s \right \|_2 \leq \frac{\sqrt{6} \epsilon_{s^\star}}{\kappa},
\end{equation*}
where $\epsilon_{s^\star}$ satisfies
\begin{equation}
\label{eps_bound}
\epsilon_{s^\star} \leq \sqrt{s^\star+s} \|\nabla\mathfrak{L}(\boldsymbol{\theta}_s)\|_\infty .
\end{equation}
\end{prop}

\section{Roadmap of the Thesis}

As noted earlier most of the theoretical guarantees of compressive sensing provide sharp trade-offs between the number of measurement, sparsity, and estimation accuracy when random i.i.d measurements are at disposal. However, in most practical applications of interest, the measured signals and the corresponding covariates are highly interdependent and follow specific dynamics. In Chapters \ref{chap:ar} and \ref{chap:hawkes} we will generalize these theoretical guarantees to two large classes of stationary processes, namely the autoregressive (AR) processes and generalized linear models (GLM), where the covariates are highly nonlinear and non-i.i.d. Our theoretical results provide insight into the tradeoffs in sampling requirements and a measure of dependence in these models.

In Chapters \ref{chap:css} and \ref{chap:slapmi} we provide fast iterative algorithms for compressible state-space models as well as general convex optimization problems with positivity constraints. We will use ideas from projection microscopy to develop a two-photon imaging technique that scans lines of excitation across the sample at multiple angles, recovering high-resolution images from relatively few incoherently multiplexed measurements. By using a static image of the sample as a prior, we recover neural activity from megapixel fields of view at rates exceeding 1kHz.

\chapter{Sampling Requirements for Stable Autoregressive Estimation}
\chaptermark{Robust Estimation of AR Processes}
\label{chap:ar}
Autoregressive (AR) models are among the most fundamental tools in analyzing time series. {They have been useful for modeling signals in many applications including  financial time series analysis \cite{sang2015simultaneous} and traffic modeling \cite{farokhi2014vehicular,ahmed1982application,ahmed1979analysis,barcelo2010travel,clark2003traffic,robinson2003time}.} Due to their well-known approximation property, these models are commonly used to represent stationary processes in a parametric fashion and thereby preserve  the underlying structure of these processes \cite{akaike1969fitting}. In general, the ubiquitous long-range dependencies in real-world time series, such as financial data, results in AR model fits with large orders \cite{sang2015simultaneous}. 

In this chapter, we close the gap in theory of compressed sensing for non-i.i.d. data by providing theoretical guarantees on stable estimation of autoregressive, where the history of the process takes the role of the interdependent covariates. In doing so, we relax the assumptions of i.i.d. covariates and exact sparsity common in CS. Our results indicate that utilizing sparsity recovers important information about the intrinsic frequencies of such processes. 

We consider the problem of estimating the parameters of a linear univariate autoregressive model with sub-Gaussian innovations from a limited sequence of consecutive observations. Assuming that the parameters are compressible, we analyze the performance of the $\ell_1$-regularized least squares as well as a greedy estimator of the parameters and characterize the sampling trade-offs required for stable recovery in the non-asymptotic regime. In particular, we show that {for a fixed sparsity level,} stable recovery of AR parameters is possible when the number of samples scale \emph{sub-linearly} with the AR order. Our results improve over existing sampling complexity requirements in AR estimation using the LASSO, when the sparsity level scales faster than the square root of the model order. We further derive sufficient conditions on the sparsity level that guarantee the minimax optimality of the $\ell_1$-regularized least squares estimate. Applying these techniques to simulated data as well as real-world datasets from crude oil prices and traffic speed data confirm our predicted theoretical performance gains in terms of estimation accuracy and model selection.

\section{Introduction}
Autoregressive (AR) models are among the most fundamental tools in analyzing time series. {Applications include financial time series analysis \cite{sang2015simultaneous} and traffic modeling \cite{farokhi2014vehicular,ahmed1982application,ahmed1979analysis,barcelo2010travel,clark2003traffic,robinson2003time}.} Due to their well-known approximation property, these models are commonly used to represent stationary processes in a parametric fashion and thereby preserve  the underlying structure of these processes \cite{akaike1969fitting}. In order to leverage the approximation property of AR models, often times parameter sets of very large order are required \cite{poskitt2007autoregressive}. For instance, any autoregressive moving average (ARMA) process can be represented by an AR process of infinite order. Statistical {inference} using these models is usually performed through fitting a long-order AR model to the data, which can be viewed as a truncation of the infinite-order representation \cite{shibata1980asymptotically, galbraith1997some, galbraith2001autoregression, ing2005order}. In general, the ubiquitous long-range dependencies in real-world time series, such as financial data, results in AR model fits with large orders \cite{sang2015simultaneous}.

In various applications of interest, the AR parameters fit to the data exhibit sparsity, that is, only a small number of the parameters are non-zero. Examples include autoregressive communication channel models, quasi-oscillatory data tuned around specific frequencies and financial time series \cite{baddour2005autoregressive, mann1999oscillatory, robinson2003time}. The non-zero AR parameters in these models correspond to significant time lags at which the underlying dynamics operate. {Traditional AR order selection criteria such as the Final Prediction Error (FPE) \cite{akaike1973maximum}, Akaike Information Criterion (AIC) \cite{akaike1970statistical} and Bayesian Information Criterion (BIC) \cite{schwarz1978estimating}, are based on asymptotic lower bounds on the mean squared prediction error.  Although there exist several improvements over these traditional results aiming at exploiting sparsity \cite{ing2005order,shibata1980asymptotically,wang2007regression}, the resulting criteria pertain to the asymptotic regimes and their finite sample behavior is not well understood  \cite{goldenshluger2001nonasymptotic}. Non-asymptotic results for AR estimation, such as \cite{goldenshluger2001nonasymptotic,nardi2011autoregressive}, do not fully exploit the sparsity of the underlying parameters in favor of reducing the sample complexity. In particular, for an AR process of order $p$, sufficient sampling requirements of $n \sim \mathcal{O}(p^4) \gg p$ and $n \sim \mathcal{O}(p^5) \gg p$ are established in \cite{goldenshluger2001nonasymptotic} and \cite{nardi2011autoregressive}, respectively.}  

{A relatively recent line of research employs the theory of compressed sensing (CS) for studying non-asymptotic sampling-complexity trade-offs for regularized M-estimators.} In recent years, the CS theory has become the standard framework for measuring and estimating sparse statistical models \cite{donoho2006compressed, candes2006compressive, candes2008introduction}. The theoretical guarantees of CS imply that when the number of incoherent measurements are roughly proportional to the sparsity level, then stable recovery of the model parameters is possible. {A key underlying assumption in many existing theoretical analyses of linear models is the independence and identical distribution (i.i.d.) of the covariates' structure. The matrix of covariates is either formed by fully i.i.d. elements \cite{rudelson2008sparse,baraniuk2008simple}, is based on row-i.i.d. correlated designs \cite{zhao2006model,raskutti2010restricted}, is Toeplitz-i.i.d. \cite{Toeplitz}, {or circulant i.i.d. \cite{rauhut2012restricted}}, where the design is extrinsic, fixed in advance and is independent of the underlying sparse signal. The matrix of covariates formed from the observations of an AR process does not fit into any of these categories, as the intrinsic history of the process plays the role of the covariates. Hence the underlying interdependence in the model hinders a straightforward application of existing CS results to AR estimation.} {Recent non-asymptotic results on the estimation of multi-variate AR (MVAR) processes have been relatively successful in utilizing sparsity for such dependent structures. For Gaussian and low-rank MVAR models, respectively, sub-linear sampling requirements have been established in \cite{loh2012high,han2013transition} and \cite{negahban2011estimation}, using regularized LS estimators, under bounded operator norm assumptions on the transition matrix. These assumptions are shown to be restrictive for MVAR processes with lags larger than 1 \cite{wong2016regularized}. By relaxing these boundedness assumptions for Gaussian, sub-Gaussian and heavy-tailed MVAR processes, respectively, sampling requirements of $n \sim \mathcal{O}(s \log p)$ and $\mathcal{O}((s \log p)^2)$ have been established in \cite{basu2015regularized} and \cite{ wong2016regularized, wu2016performance}. However, the quadratic scaling requirement in the sparsity level for the case of sub-Gaussian and heavy-tailed innovations incurs a significant gap with respect to the optimal guarantees of CS (with linear scaling in sparsity), particularly when the sparsity level $s$ is allowed to scale with $p$.}

{In this chapter, we consider two of the widely-used estimators in CS, namely the $\ell_1$-regularized Least Squares (LS) or the LASSO and the Orthogonal Matching Pursuit (OMP) estimator, and extend the non-asymptotic recovery guarantees of the CS theory to the estimation of univariate AR processes with compressible parameters using these estimators. In particular, we improve the aforementioned gap between non-asymptotic sampling requirements for AR estimation and those promised by compressed sensing by providing sharper sampling-complexity trade-offs which improve over existing results when the sparsity grows faster than the square root of $p$. Our focus on the analysis of univariate AR processes is motivated by the application areas of interest in this chapter which correspond to one-dimensional time series. Existing results in the literature \cite{loh2012high,han2013transition,negahban2011estimation,basu2015regularized,wong2016regularized, wu2016performance}, however, consider the MVAR case and thus are broader in scope. We will therefore compare our results to the univariate specialization of the aforementioned results.} Our main contributions can be summarized as follows:

First, we establish that {for a univariate AR process with sub-Gaussian innovations} when the number of measurements scales \emph{sub-linearly} with the product of the ambient dimension $p$ and the sparsity level $s$, i.e., $n \sim \mathcal{O}(s  (p \log p)^{1/2}) \ll p$, then stable recovery of the underlying AR parameters is possible using the LASSO and the OMP estimators, even though the covariates are highly interdependent and solely based on the history of the process. In particular, when $s \propto p^{\frac{1}{2} + \delta}$ for some $\delta \ge 0$ and the LASSO is used, our results improve upon those of \cite{wong2016regularized, wu2016performance}, when specialized to the univariate AR case, by a factor of $p^{\delta} (\log p)^{{3}/{2}}$. For the special case of Gaussian AR processes, stronger results are available which require a scaling of $n \sim \mathcal{O}(s \log p)$ \cite{basu2015regularized}. Moreover, our results provide a theory-driven choice of the number of iterations for stable estimation using the OMP algorithm, which has a significantly lower computational complexity than the LASSO.

Second, in the course of our analysis, we establish the Restricted {Eigenvalue} (RE) condition \cite{bickel2009simultaneous} for $n \times p$ design matrices formed from a realization of an AR process in a Toeplitz fashion, when $n \sim \mathcal{O}(s  (p \log p)^{1/2}) \ll p$. To this end, we invoke appropriate concentration inequalities for sums of \emph{dependent} random variables in order to capture and control the high interdependence of the design matrix. {In the special case of a white noise sub-Gaussian process, i.e., a sub-Gaussian i.i.d. Toeplitz measurement matrix, we show that our result can be strengthened from $n \sim \mathcal{O}(s (p \log p)^{1/2})$ to $n \sim \mathcal{O}(s (\log p)^2)$, which improves by a factor of $s / \log p$ over the results of \cite{Toeplitz} requiring $n \sim \mathcal{O} (s^2 \log p)$.}

Third, we establish sufficient conditions on the sparsity level which result in the minimax optimality of the $\ell_1$-regularized LS estimator. Finally, we provide simulation results as well as application to oil price and traffic data which reveal that the sparse estimates significantly outperform traditional techniques such as the Yule-Walker based estimators \cite{stoica1997introduction}. We have employed statistical tests in time and frequency domains to compare the performance of these estimators.

The rest of the chapter is organized as follows. In Section \ref{sec:ar_notations}, we will introduce the notations and problem formulation. In Section \ref{sec:ar_theory}, we will describe several methods for the estimation of the parameters of an AR process, present the main theoretical results of this chapter on robust estimation of AR parameters, and establish the minimax optimality of the $\ell_1$-regularized LS estimator. Section \ref{sec:ar_sim} includes our simulation results on simulated data as well as the real-world financial and traffic data, followed by concluding remarks in Section \ref{sec:conc}.

\section{Problem Formulation}
\label{sec:ar_notations}

%{For a matrix $\mathbf{B}$ we denote its $j$-th column and row by $\mathbf{B}^{(j)}$ and $\mathbf{B}_{(j)}$ respectively. For a vector $\mathbf{b}$ we denote it's $j$-th element by $(\mathbf{b})_j$.}

Consider a univariate AR($p$) process defined by
\begin{equation}
\label{AR_def}
x_k = \theta_1 x_{k-1} + \theta_2 x_{k-2} + \cdots + \theta_p x_{k-p} + w_k =  \boldsymbol{\theta}' \mathbf{x}_{k-p}^{k-1} +w_k,
\end{equation}
where $\{w_k\}_{k=-\infty}^{\infty}$ is an i.i.d sub-Gaussian innovation sequence with zero mean and variance $\sigma^2_{\sf w}$. This process can be considered as the output of an LTI system with transfer function
\begin{equation}
\label{eq:ar_tf}
H(z) = \frac{\sigma^2_{\sf w}}{1-\sum_{\ell=1}^p \theta_\ell z^{-\ell}}.
\end{equation}

Throughout the chapter we will assume $\|\boldsymbol{\theta}\|_1 \leq 1-\eta <1$ to enforce the stability of the filter. We will refer to this assumption as \emph{the sufficient stability assumption}, since an AR process with poles within the unit circle does not necessarily satisfy $\| \boldsymbol{\theta} \|_1 < 1$. However, beyond second-order AR processes, it is not straightforward to state the stability of the process in terms of its parameters in a closed algebraic form, which in turn makes both the analysis and optimization procedures intractable.  {As we will show later, the only major requirement of our results is the boundedness of the spectral spread (also referred to as condition number) of the AR process. Although the sufficient stability condition is more restrictive, it will significantly simplify the spectral constants appearing in the analysis and clarifies the various trade-offs in the sampling bounds (See, for example, Corollary  \ref{cor:eig_conv}). }

The AR($p$) process given by $\{x_k\}_{k=-\infty}^\infty$ in (\ref{AR_def}) is stationary in the strict sense. Also by (\ref{eq:ar_tf}) the power spectral density of the process equals
\begin{equation}
\label{psd}
S(\omega)= \frac{\sigma_{\sf w}^2}{|1-\sum_{\ell=1}^p \theta_\ell e^{-j \ell \omega}|^2}.
\end{equation}

The sufficient stability assumption implies boundedness of the spectral spread of the process defined as 
\begin{equation*}
\label{def:ar_condition_nr}
\rho = \sup_{\omega} S(\omega) \Big \slash \inf_{\omega} S(\omega).
\end{equation*}
We will discuss how this assumption can be further relaxed in {Appendix \ref{app:ar_main}}. The spectral spread of stationary processes in general is a measure of how quickly the process reaches its ergodic state \cite{goldenshluger2001nonasymptotic}. An important property that we will use later in this chapter is that the spectral spread is an upper bound on the eigenvalue spread of the covariance matrix of the process of arbitrary size \cite{haykin2008adaptive}.

We will also assume that the parameter vector $\boldsymbol{\theta}$ is compressible (to be defined more precisely later), and  can be well approximated by an $s$-sparse vector where $s \ll p$. We observe $n$ consecutive snapshots of length $p$ (a total of $n+p-1$ samples) from this process given by $\{x_k\}_{k=-p+1}^n$ and aim to estimate $\boldsymbol{\theta}$ by exploiting its sparsity; to this end, we aim at addressing the following questions {in the non-asymptotic regime}:
\begin{itemize}
\item Are the conventional LASSO-type and greedy techniques suitable for estimating $\boldsymbol{\theta}$?
\item What are the sufficient conditions on $n$ in terms of $p$ and $s$, to guarantee stable recovery?
\item Given these sufficient conditions, how do these estimators perform compared to conventional AR estimation techniques?
\end{itemize}

Traditionally, the Yule-Walker (YW) equations or least squares formulations are used to fit AR models. Since these methods
do not utilize the sparse structure of the parameters, they usually require $n \gg p$ samples in order to achieve satisfactory performance. The YW equations can be expressed as
\begin{equation}
\label{yule}
\mathbf{R} \boldsymbol{\theta} = \mathbf{r}_{-p}^{-1}, \quad r_0 = \boldsymbol{\theta}' \mathbf{r}_{-p}^{-1} + \sigma^2_{\sf w},
\end{equation}
where $\mathbf{R}: = \mathbf{R}_{p\times p} = \mathbb{E}[\mathbf{x}_1^p \mathbf{x}_1^{p'}]$ is the $p \times p$ covariance matrix of the process and $r_k = \mathbb{E}[{x_ix_{i+k}}]$ is the autocorrelation of the process at lag $k$. The covariance matrix $R$ and autocorrelation vector $\mathbf{r}_{-p}^{-1}$ are typically replaced by their sample counterparts. Estimation of the AR($p$) parameters from the YW equations can be efficiently carried out using the Burg's method \cite{burg1967maximum}. Other estimation techniques include LS regression and maximum likelihood (ML) estimation. In this chapter, we will consider the Burg's method and LS solutions as comparison benchmarks. When $n$ is comparable to $p$, these two methods are known to exhibit substantial performance differences \cite{marple1987digital}.

%Usually two distinct variants of maximum likelihood are considered: in conditional ML estimation one the objective function corresponds to the conditional distribution of last $n$ samples given the initial $p$ values in the series (this is equivalent to the least squares regression problem); in the second, the likelihood function considered is that corresponding to the unconditional joint distribution of all the values in the observed series. Substantial differences in the results of these approaches can occur if the observed series is short, or if the process is close to non-stationarity. Writing the  In this paper we will focus on the conditional ML estimation method (ML from now on) and its regularized version. Yule-Walker based methods exhibit poor performance when $n$ is small or comparable to $p$. This is mainly due to the fact that solving (\ref{yule}) requires an inversion of $\widehat{R}_p$ which might not be numerically stable. As a result the estimates will be poor. It can be shown that a \textit{biased} estimate of $r_k$'s and $R$ resolves the invertibility problem but this will be at the cost of distorting the Yule-Walker equations.

When fitted to the real-world data, the parameter vector $\boldsymbol{\theta}$ usually exhibits a degree of sparsity. That is, only certain lags in the history have a significant contribution in determining the statistics of the process. These lags can be thought of as the intrinsic delays in the underlying dynamics. To be more precise, for a sparsity level $s < p$, we denote by $S \subset \{1,2,\cdots,p \}$ the support of the $s$ largest elements of $\boldsymbol{\theta}$ in absolute value, and by $\boldsymbol{\theta}_s$ the best $s$-term approximation to $\boldsymbol{\theta}$. We also define
\begin{equation}
\sigma_s(\boldsymbol{\theta}) := \|\boldsymbol{\theta}-\boldsymbol{\theta}_s\|_1~\text{and}~\varsigma_s(\boldsymbol{\theta}) := \|\boldsymbol{\theta}-\boldsymbol{\theta}_s\|_2,
\end{equation}
which capture the compressibility of the parameter vector $\boldsymbol{\theta}$ in the $\ell_1$ and $\ell_2$ sense, respectively. Note that by definition 
$\varsigma_s(\boldsymbol{\theta}) \leq \sigma_s(\boldsymbol{\theta})$. For a fixed $\xi \in (0,1)$, we say that $\boldsymbol{\theta}$ is \emph{$(s,\xi)$-compressible} if $\sigma_s(\boldsymbol{\theta}) = \mathcal{O}(s^{1-\frac{1}{\xi}})$ \cite{needell2009cosamp} and \emph{$(s,\xi,2)$-compressible} if $\varsigma_s(\boldsymbol{\theta}) = \mathcal{O}(s^{1-\frac{1}{\xi}})$. Note that \emph{$(s,\xi,2)$-compressibility} is a weaker condition than \emph{$(s,\xi)$-compressibility} and when $\xi = 0$, the parameter vector $\boldsymbol{\theta}$ is exactly $s$-sparse.

Finally, in this chapter, we are concerned with the compressed sensing regime where $n \ll p$, i.e., the observed data has a much smaller length than the ambient dimension of the parameter vector. The main estimation problem of this chapter can be summarized as follows: \emph{given observations $\mathbf{x}_{-p+1}^n$ from an AR process with sub-Gaussian innovations and bounded spectral spread, the goal is to estimate the unknown $p$-dimensional $(s,\xi,2)$-compressible AR parameters $\boldsymbol{\theta}$ in a stable fashion (where the estimation error is controlled) when $n \ll p$.} 

\section{Theoretical Results}\label{theoretical}
\label{sec:ar_theory}

In this section, we will describe the estimation procedures and present the main theoretical results of this chapter.

\subsection{$\ell_1$-regularized least squares estimation}
Given the sequence of observations $\mathbf{x}_{-p+1}^n$ and an estimate $\widehat{\boldsymbol{\theta}}$, the normalized estimation error can be expressed as:
\begin{equation}
\label{def:ar_L}
\mathfrak{L}\Big(\widehat{\boldsymbol{\theta}}\Big) :=  \frac{1}{n} \left \|\mathbf{x}_1^n-\mathbf{X}\widehat{\boldsymbol{\theta}}\right \|_2^2,
\end{equation}
\vspace{-.2cm}
where
\vspace{-.2cm}
\begin{equation}
\mathbf{X} = \left[ \begin{array}{cccc}
x_{n-1} & x_{n-2} & \cdots & x_{n-p}  \\
x_{n-2} & x_{n-3} & \cdots & x_{n-p-1} \\
\vdots & \vdots & \ddots & \vdots \\
x_{0} & x_{-1} & \cdots & x_{-p+1} \end{array} \right].
\end{equation}

Note that the matrix of covariates $\mathbf{X}$ is Toeplitz with highly interdependent elements. The LS solution is thus given by:

\begin{equation}
\label{eq:ar_LS}
\widehat{\boldsymbol{\theta}}_{{\sf LS}}=\argmin\limits_{\boldsymbol{\theta}\in \boldsymbol{\Theta} } \mathfrak{L}(\boldsymbol{\theta}),
\end{equation}
where 
\vspace{-.2cm}
\[\boldsymbol{\Theta} := \left \{ \boldsymbol{\theta} \in \mathbb{R}^p | \; \|\boldsymbol{\theta}\|_1 < 1- \eta \right \}\]
is the convex feasible region for which the stability of the process is guaranteed. {Note that the sufficient constraint of $\| \boldsymbol{\theta} \|_1 < 1- \eta$ is by no means necessary for stability. However, the set of all $\boldsymbol{\theta}$ resulting in stability is in general not convex. We have thus chosen to cast the LS estimator of Eq. (\ref{eq:ar_LS}) --as well as its $\ell_1$-regularized version that follows-- over a convex subset $\boldsymbol{\Theta}$, for which fast solvers exist. In addition, as we will show later, this assumption significantly clarifies the various constants appearing in our theoretical analysis. In practice, the Yule-Walker estimate is obtained without this constraint, and is guaranteed to result in a stable AR process. Similarly, for the LS estimate, this condition is relaxed by obtaining the unconstrained LS estimate and checking \emph{post hoc} for stability \cite{percival1993spectral}.} 

{Consistency of the LS estimator given by (\ref{eq:ar_LS}) was shown in \cite{sang2015simultaneous} when $n \rightarrow\infty$ for Gaussian innovations. In the case of Gaussian innovations the LS estimates correspond to conditional ML estimation and are asymptotically unbiased under mild conditions, and with $p$ fixed, the solution converges to the true parameter vector as $n \rightarrow \infty$. For fixed $p$, the estimation error is of the order $\mathcal{O}(\sqrt{p/n})$ in general \cite{Toeplitz}. However, when $p$ is allowed to scale with $n$, the convergence rate of the estimation error is not known in general.} 

%Note that the LS solution coincides with the ML estimates for gaussian innovations. and boundedness of the spectral spread of the process 
In the regime of interest in this thesis, where $n \ll p$, the LS estimator is ill-posed and is typically regularized with a smooth norm. In order to capture the compressibility of the parameters, we consider the $\ell_1$-regularized LS estimator:

\begin{equation}
\label{eq:ar_lasso}
\widehat{\boldsymbol{\theta}}_{{\ell_1}}:=\argmin\limits_{\boldsymbol{\theta}\in \boldsymbol{\Theta}} \quad \mathfrak{L}(\boldsymbol{\theta})+ \gamma_n\|\boldsymbol{\theta}\|_1,
\end{equation}
where $\gamma_n > 0$ is a regularization parameter. {This estimator, deemed as the Lagrangian form of the LASSO \cite{tibshirani1996regression}, has been comprehensively studied in the sparse recovery literature \cite{wainwright2009sharp,knight2000asymptotics,meinshausen2006high} as well as AR estimation \cite{knight2000asymptotics,zhao2006model,nardi2011autoregressive,wang2007regression}. A general asymptotic consistency result for LASSO-type estimators was established in \cite{knight2000asymptotics}. Asymptotic consistency of LASSO-type estimators for AR estimation was shown in \cite{zhao2006model,wang2007regression}. For sparse models, non-asymptotic analysis of the LASSO with covariate matrices from row-i.i.d. correlated design has been established in \cite{zhao2006model,wainwright2009sharp}.}

In many applications of interest, the data correlations are exponentially decaying and negligible beyond a certain lag, and hence for large enough $p$, autoregressive models fit the data very well in the prediction error sense. {An important question is thus how many measurements are required for estimation stability? In the \emph{overdetermined} regime of $n \gg p$, the non-asymptotic properties of LASSO for model selection of AR processes has been studied in \cite{nardi2011autoregressive}, where a sampling requirement of $ n \sim \mathcal{O}(p^5)$ is established. {Recovery guarantees for LASSO-type estimators of multivariate AR parameters in the \emph{compressive} regime of $n \ll p$ are studied in \cite{loh2012high,han2013transition,negahban2011estimation,wong2016regularized,basu2015regularized,wu2016performance}. In particular, sub-linear scaling of $n$ with respect to the ambient dimension is established in \cite{loh2012high,han2013transition} for Gaussian MVAR processes and in \cite{negahban2011estimation} for low-rank MVAR processes, respectively, under the assumption of bounded operator norm of the transition matrix. In \cite{basu2015regularized} and \cite{wong2016regularized,wu2016performance}, the latter assumption is relaxed for Gaussian, sub-Gaussian, and heavy-tailed MVAR processes, respectively.} These results have significant practical implications as they will reveal sufficient conditions on $n$ with respect to $p$ as well as a criterion to choose $\gamma_n$, which result in stable estimation of $\boldsymbol{\theta}$ from a considerably short sequence of observations. The latter is indeed the setting that we consider in this chapter, where the ambient dimension $p$ is fixed and the goal is to derive sufficient conditions on $n \ll p$ resulting in stable estimation.}

It is easy to verify that {the objective function and constraints in Eq.} (\ref{eq:ar_lasso}) are convex in $\boldsymbol{\theta}$ and hence $\widehat{\boldsymbol{\theta}}_{\ell_1}$ can be obtained using standard numerical solvers. Note that the solution to (\ref{eq:ar_lasso}) might not be unique. However, we will provide error bounds that hold for all possible solutions of (\ref{eq:ar_lasso}), with high probability. 

Recall that, the Yule-Walker solution is given by
\vspace{-0.5mm}
\begin{equation}
\label{eq:ar_yw}
\widehat{\boldsymbol{\theta}}_{{\sf yw}}:=\argmin\limits_{\boldsymbol{\theta}\in \boldsymbol{\Theta}} \quad \mathfrak{J}(\boldsymbol{\theta}) = \widehat{\mathbf{R}}^{-1}\widehat{\mathbf{r}}_{-p}^{-1},
\end{equation}
\vspace{-2.5mm}

\noindent where $\mathfrak{J}(\boldsymbol{\theta}):= \|\widehat{\mathbf{R}} \boldsymbol{\theta}- \widehat{\mathbf{r}}_{-p}^{-1}\|_{{2}}$. We further consider two other sparse estimators for $\boldsymbol{\theta}$ by penalizing the Yule-Walker equations. The $\ell_1$-regularized Yule-Walker estimator is defined as:
\begin{equation}
\label{eq:ar_ywl21}
\widehat{\boldsymbol{\theta}}_{{\sf yw,\ell_{2,1}}}:=\argmin\limits_{\boldsymbol{\theta}\in \boldsymbol{\Theta}} \quad \mathfrak{J}(\boldsymbol{\theta})+ \gamma_n\|\boldsymbol{\theta}\|_1,
\end{equation}
where $\gamma_n > 0$ is a regularization parameter.
Similarly, using the robust statistics instead of the Gaussian statistics, the estimation error can be re-defined as:
\begin{equation*}
\label{def:ar_J1}
\mathfrak{J}_1(\boldsymbol{\theta}):= \|\widehat{\mathbf{R}} \boldsymbol{\theta}- \widehat{\mathbf{r}}_{-p}^{-1}\|_{{1}},
\end{equation*}
we define the $\ell_1$-regularized estimates as
\begin{equation}
\label{eq:ar_ywl11}
\widehat{\boldsymbol{\theta}}_{{\sf yw,\ell_{1,1}}}:=\argmin\limits_{\boldsymbol{\theta}\in \boldsymbol{\Theta}} \quad \mathfrak{J}_1(\boldsymbol{\theta})+ \gamma_n\|\boldsymbol{\theta}\|_1.
\end{equation}

%The latter cost function corresponds to a robust estimator which shows show a slight advantage in favor of the  $\widehat{\boldsymbol{\theta}}_{{\sf yw,\ell_{2,1}}}$ estimator.

\vspace{-.3cm}
\subsection{Greedy estimation} \label{ar:subsec_greedy}
Although there exist fast solvers for the convex problems of the type given by (\ref{eq:ar_lasso}), (\ref{eq:ar_ywl21}) and (\ref{eq:ar_ywl11}), these algorithms are polynomial time in $n$ and $p$, and may not scale well with the dimension of data. This motivates us to consider greedy solutions for the estimation of $\boldsymbol{\theta}$. In particular, we will consider and study the performance of a generalized Orthogonal Matching Pursuit (OMP) algorithm \cite{OMP, zhang_omp}. A flowchart of this algorithm is given in Table \ref{tab:aromp} for completeness. At each iteration, a new component of $\boldsymbol{\theta}$ for which the gradient of the error metric $\mathfrak{f}(\boldsymbol{\theta})$ is the largest in absolute value is chosen and added to the current support. The algorithm proceeds for a total of $s^\star = \mathcal{O}(s\log s)$ steps, resulting in an estimate with $s^\star$ components. When the error metric $\mathfrak{L}(\boldsymbol{\theta})$ is chosen, the generalized OMP corresponds to the original OMP algorithm. For the choice of the YW error metric $\mathfrak{J}(\boldsymbol{\theta})$, we denote the resulting greedy algorithm by {\sf yw}OMP.
\begin{table}[h]
\centering
%\begin{framed}
\framebox{$\begin{array}{l}
\text{Input: } \mathfrak{f}(\boldsymbol{\theta}) , s^\star\\
\text{Output: } \widehat{\boldsymbol{\theta}}_{\sf OMP}=\widehat{\boldsymbol{\theta}}_{\sf OMP}^{(s^\star)}\\
\text{Initialization:}\Big\{\begin{array}{l}
\text{Start with the index set } S^{(0)}=\emptyset\\
\text{and the initial estimate }\widehat{\boldsymbol{\theta}}^{(0)}_{{\sf OMP}} = 0
\end{array}\\
\textbf{for } k=1,2,\cdots,s^\star\\
\text{  }\begin{array}{l}
j = \argmax \limits_i \left| \left( \nabla \mathfrak{f} \; \left(\widehat{\boldsymbol{\theta}}_{{\sf OMP}}^{(k-1)}\right) \right)_i\right|\\
S^{(k)}=S^{(k-1)}\cup \{j\}\\
\widehat{\boldsymbol{\theta}}_{{\sf OMP}}^{(k)} = \argmin \limits_{\support (\boldsymbol{\theta}) \subset S^{(k)}} \mathfrak{f}(\boldsymbol{\theta})
\end{array}\\
\textbf{end }\\
\end{array}$}
%\end{framed}
\caption{Generalized Orthogonal Matching Pursuit (OMP)}
\label{tab:aromp}
\end{table}

\subsection{Estimation performance guarantees} 

The main theoretical result regarding the estimation performance of the $\ell_1$-regularized LS estimator is given by the following theorem:

\begin{thm}
\label{thm:ar_1}
If $\sigma_s(\boldsymbol{\theta}) = \mathcal{O}(\sqrt{s})$, there exist positive constants ${d_0}, d_1,d_2,d_3$ and $d_4$ such that for $n > s \max \{d_0 (\log p)^2, d_1 (p \log p)^{1/2}\}$ and a choice of regularization parameter $\gamma_n = d_2 \sqrt{\frac{\log p}{n}}$, any solution $\widehat{\boldsymbol{\theta}}_{{\ell_1}}$ to (\ref{eq:ar_lasso}) satisfies the bound
\begin{equation}
\label{eq:ar1error}
\left \|\widehat{\boldsymbol{\theta}}_{{\ell_1}}-\boldsymbol{\theta}\right \|_2 \leq d_3 \sqrt{\frac{s \log p}{n}}+ \sqrt{d_3\sigma_s(\boldsymbol{\theta})}\sqrt[4]{\frac{\log p}{n}},
\end{equation}
with probability greater than $1-\mathcal{O}(\frac{1}{n^{d_4}})$. The constants {depend on the spectral spread of the process and are} explicitly given in the proof.
\end{thm}

Similarly, the following theorem characterizes the estimation performance bounds for the OMP algorithm:
\begin{thm}
\label{thm_OMP}
If $\boldsymbol{\theta}$ is $(s,\xi,2)$-compressible for some $\xi < 1/2$, there exist positive constants ${d'_0}, d'_1,d'_{2},d'_{3}$ and $d'_4$ such that for {{$n > s \log s \max \{d'_0 (\log p)^2, d'_1 {(p \log p)^{1/2}}\}$}}, the OMP estimate satisfies the bound
\begin{equation}
\label{eq:aromp}
\left \|\widehat{\boldsymbol{\theta}}_{{\sf OMP}}-\boldsymbol{\theta}\right\|_2 \leq d'_{2} \sqrt{\frac{s \log s \log p}{n}} + d'_{3} \frac{\log s}{s^{{\frac{1}{\xi}-2}}}
\end{equation}
after $s^\star ={ 4{\rho}s \log {20 \rho s}}$ iterations with probability greater than $1-\mathcal{O}\left(\frac{1}{n^{d'_4}}\right)$. The constants {depend on the spectral spread of the process and are} explicitly given in the proof.
\end{thm}

The results of Theorems \ref{thm:ar_1} and \ref{thm_OMP} suggest that under suitable compressibility assumptions on the AR parameters, one can estimate the parameters reliably using the $\ell_1$-regularized LS and OMP estimators with much fewer measurements compared to those required by the Yule-Walker/LS based methods. To illustrate the significance of these results further, several remarks are in order:

{\noindent  \textit{\textbf{Remark 1.}} The sufficient stability assumption of $\|\boldsymbol{\theta}\|_1 \leq 1- \eta <1$ is restrictive compared to the class of stable AR models. In general, the set of parameters $\boldsymbol{\theta}$ which admit a stable AR process is not necessarily convex. This condition ensures that the resulting estimates of (\ref{eq:ar_lasso})-(\ref{eq:ar_ywl11}) pertain to stable AR processes and at the same time can be obtained by convex optimization techniques, for which fast solvers exist. A common practice in AR estimation, however, is to solve for the unconstrained problem and check for the stability of the resulting AR process \emph{post hoc}. In our numerical studies in Section \ref{sec:ar_sim}, this procedure resulted in a stable AR process in all cases. Nevertheless, the stability guarantees of Theorems \ref{thm:ar_1} and \ref{thm_OMP} hold for the larger class of stable AR processes, even though they may not necessarily be obtained using convex optimization techniques. We further discuss this generalization in Appendix \ref{app:ar_main}.
}

\noindent  \textit{\textbf{Remark 2.}} {When $\boldsymbol{\theta} = \mathbf{0}$, i.e., the process is a sub-Gaussian white noise and hence the matrix $\mathbf{X}$ is i.i.d. Toeplitz with sub-Gaussian elements, the constants $d_1$ and $d'_1$ in Theorems 1 and 2 vanish, and the measurement requirements strengthen to $n > d_0 s (\log p)^2$  and $n > d'_0 s \log s (\log p)^2$, respectively. Comparing this sufficient condition with that of \cite{Toeplitz} given by $n \sim \mathcal{O}(s^2 \log p)$ reveals an improvement of order $s (\log p)^{-1}$ by our results.}

\noindent  \textit{\textbf{Remark 3.}} {When $\boldsymbol{\theta} \neq \mathbf{0}$, the dominant measurement requirements are $n > d_1 s { (p \log p)^{1/2}}$ and $n > d'_1 s \log s {(p \log p)^{1/2}}$.} Comparing the sufficient condition $n \sim \mathcal{O}(s {(p \log p)^{1/2}})$ of Theorem \ref{thm:ar_1} with those of \cite{donoho2006compressed, candes2006compressive, candes2008introduction,wainwright2009sharp} for linear models with i.i.d. measurement matrices or row-i.i.d. correlated designs \cite{zhao2006model,raskutti2010restricted} given by $n \sim \mathcal{O}(s \log p)$  a loss of order $\mathcal{O}({(p/ \log p)^{1/2}})$ is incurred, although all these conditions require $n \ll p$. However, the loss seems to be natural as it stems from {a major difference of our setting as compared to traditional CS:} each row of the measurement matrix $\mathbf{X}$ highly depends on the entire observation sequence $\mathbf{x}_1^n$, whereas in traditional CS, each row of the measurement matrix is only related to the corresponding measurement. Hence, the aforementioned loss can be viewed as the price of self-averaging of the process accounting for the low-dimensional nature of the covariate sample space and the high inter-dependence of the covariates to the observation sequence. Recent results on M-estimation of sparse {MVAR} processes with sub-Gaussian and {heavy-tailed} innovations \cite{wu2016performance,wong2016regularized} require $n \sim \mathcal{O}(s^2 (\log p)^2)$ {when specialized to the univariate case,} which compared to our results improve the loss of $\mathcal{O}({(p/\log p)^{1/2}})$ to $(\log p)^2$ with the additional cost of quadratic requirement in the sparsity $s$. {However, in the over-determined regime of $s \propto p^{\frac{1}{2} + \delta}$ for some $\delta \ge 0$, our results imply $n \sim \mathcal{O} (p^{1 + \delta} (\log p)^{1/2})$, providing a saving of order ${p^{\delta} (\log p)^{3/2}}$ over those of \cite{wu2016performance,wong2016regularized}.}

\noindent \textit{\textbf{Remark 4.}} It can be shown that the estimation error for the LS method {in general} scales as $\sqrt{p/n}$ \cite{Toeplitz} which is not desirable when $n \ll p$. Our result, however, guarantees a much smaller error rate of the order $\sqrt{s \log p/n}$. Also, the sufficiency conditions of Theorem \ref{thm_OMP} require high compressibility of the parameter vector $\boldsymbol{\theta}$  ($\xi < 1/2$), whereas Theorem \ref{thm:ar_1} does not impose any extra restrictions on $\xi \in (0,1)$. Intuitively speaking, these two comparisons reveal the trade-off between computational complexity and measurement/compressibility requirements for convex optimization vs. greedy techniques, which are well-known for linear models \cite{Bruckstein2009}.

\noindent  \textit{\textbf{Remark 5.}} The condition $\sigma_s(\boldsymbol{\theta})=\mathcal{O}(\sqrt{s})$ in Theorem \ref{thm:ar_1} is not restricting for the processes of interest in this chapter. This is due to the fact that the boundedness assumption on the spectral spread implies an exponential decay of the parameters (See Lemma 1 of \cite{goldenshluger2001nonasymptotic}). Finally, {the constants $d_1$, $d'_{1}$ are increasing with respect to the spectral spread of the process $\rho$. Intuitively speaking, the closer the roots of the filter given by (\ref{eq:ar_tf}) get to the unit circle (corresponding to larger $\rho$ and smaller $\eta$), the slower the convergence of the process will be to its ergodic state, and hence more measurements are required.} A similar dependence to the spectral spread has appeared in the results of \cite{goldenshluger2001nonasymptotic} for $\ell_2$-regularized least squares estimation of AR processes. 

%\noindent \textcolor{blue}{ \textit{\textbf{Remark 4.}} Theorems \ref{thm:ar_1} and \ref{thm_OMP} suggest an order selection criteria, so that for fixed $n$ and $s$, choosing $p \sim \mathcal{O}\left(\left(\frac{n}{s}\right)^{3/2}\right)$ guarantees the error bounds (\ref{eq:ar1error}) and (\ref{eq:aromp}). This choice provides a long enough model order $p$ to be able to capture long-range dependencies in the data, as compared to asymptotic order selection criteria such as the Akaike or Bayesian information criteria, or the non-asymptotic minimax criterion of \cite{goldenshluger2001nonasymptotic}. In Section \ref{sec:ar_sim}, we will use this choice of the model order in our numerical and real data analysis.}

\noindent  \textit{\textbf{Remark 6.}} The main ingredient in the proofs of Theorems \ref{thm:ar_1}  and \ref{thm_OMP} is to establish the restricted eigenvalue (RE) condition introduced in \cite{bickel2009simultaneous} for the covariates matrix $\mathbf{X}$. Establishing the RE condition for the covariates matrix $\mathbf{X}$ is a nontrivial problem due to the high interdependence of the matrix entries. We will indeed show that if the sufficient stability assumption holds, then with $n \sim \mathcal{O}\left (s \max \{d_0 (\log p)^2, d_1 {(p \log p)^{1/2}} \} \right)$ the sample covariance matrix is sharply concentrated around the true covariance matrix and hence the RE condition can be guaranteed. All constants appearing in Theorems \ref{thm:ar_1} and \ref{thm_OMP} are explicitly given in Appendix \ref{app:ar_main}. As a typical numerical example, for $\eta = 0.9$ and $\sigma_w^2 = 0.1$, the constants of Theorem \ref{thm:ar_1} can be chosen as $d_0 \approx 1000, {d_1 \approx 3 \times 10^{8}}, {d_2 \approx  0.15} , d_3 \approx 140$, and  $d_4 = 1$. The full proofs are given in Appendix \ref{app:ar_main}.

\subsection{Minimax optimality}
\label{sec:ar_minimax}
{In this section}, we establish the minimax optimality of the $\ell_1$-regularized LS estimator for AR processes with sparse parameters. To this end, we will focus on the class $\mathcal{H}$ of stationary processes which admit an AR($p$) representation with $s$-sparse parameter $\boldsymbol{\theta}$ such that $\|\boldsymbol{\theta}\|_1 \leq 1-\eta <1$. The theoretical results of this section are inspired by the results of \cite{goldenshluger2001nonasymptotic} on non-asymptotic order selection {via $\ell_2$-regularized LS estimation in the absence of sparsity}, and extend them by studying the $\ell_1$-regularized LS estimator of (\ref{eq:ar_lasso}). 

We define the maximal \textit{estimation} risk over $\mathcal{H}$ to be
\begin{equation}
\label{eq:ar_minimax_risk}
\mathcal{R}_{\sf est} (\widehat{\boldsymbol{\theta}}):= \sup_{\mathcal{H}} \left(\mathbb{E}\left[\|\widehat{\boldsymbol{\theta}}-\boldsymbol{\theta}\|_2^2\right]\right)^{1/2}.
\end{equation}
The minimax estimator is the one minimizing the maximal estimation risk, i.e.,
\begin{equation}
\label{eq:ar_minimax_est}
\widehat{\boldsymbol{\theta}}_{\sf minimax} := \argmin\limits_{\boldsymbol{\theta}\in \boldsymbol{\Theta}} \quad \mathcal{R}_{\sf est} (\widehat{\boldsymbol{\theta}}).
\end{equation}
Minimax estimator $\widehat{\boldsymbol{\theta}}_{\sf minimax}$, in general, cannot be constructed explicitly \cite{goldenshluger2001nonasymptotic}, and the common practice in non-parametric estimation is to construct an estimator $\widehat{\boldsymbol{\theta}}$ which is \emph{order optimal} as compared to the minimax estimator:
\begin{equation}
\label{eq:ar_minimax_optinorder}
\mathcal{R}_{\sf est} (\widehat{\boldsymbol{\theta}}) \leq L \mathcal{R}_{\sf est} (\widehat{\boldsymbol{\theta}}_{\sf minimax}).
\end{equation}
with $L \ge 1$ being a constant. One can also define the minimax \textit{prediction} risk by the maximal prediction error over all possible realizations of the process:
\begin{equation}
\label{eq:ar_minimax_pred_risk}
\mathcal{R}^2_{\sf pre} (\widehat{\boldsymbol{\theta}}) := \sup_{\mathcal{H}} \mathbb{E}\left[\left(x_{k}-\widehat{\boldsymbol{\theta}}'\mathbf{x}_{k-p}^{k-1}\right)^2 \right].
\end{equation}
In \cite{goldenshluger2001nonasymptotic}, it is shown that an $\ell_2$-regularized LS estimator with an {order} $p^\star = \mathcal{O} (\log n) $ is minimax optimal. This order pertains to the denoising regime where $n \gg p$. Hence, in order to capture long order lags of the process, one requires a sample size exponentially large in $p$, which may make the estimation problem computationally infeasible. For instance, consider a $2$-sparse parameter with only $\theta_1$ and $\theta_p$ being non-zero. Then, in order to achieve minimax optimality, $n \sim \mathcal{O}(2^p)$ measurements are required. In contrast, in the compressive regime where $s, n \ll p$, the goal, instead of selecting $p$,  is to find conditions on the sparsity level $s$, so that for a given $n$ and large enough $p$, the $\ell_1$-regularized estimator is minimax optimal without explicit knowledge of the value of $s$ (See for example, \cite{candes2006modern}).

In the following proposition, we establish the minimax optimality of the $\ell_1$-regularized estimator over the class of sparse AR processes with $\boldsymbol{\theta} \in \boldsymbol{\Theta}$:
\begin{prop}
\label{thm:ar_minimax}
Let $\mathbf{x}_1^n$ be samples of an AR process with $s$-sparse parameters satisfying $\|\boldsymbol{\theta}\|_1 \leq 1-\eta$ and $s  \le  \min \left \{\frac{1- \eta}{\sqrt{8\pi} \eta} \sqrt{\frac{n}{\log p}}\ , \frac{n}{d_1{(p \log p)^{1/2}}}, {\frac{n}{d_0 (\log p)^2}} \right\}$. Then, we have:
\begin{equation*}
\mathcal{R}_{\sf est} (\widehat{\boldsymbol{\theta}}_{\ell_1}) \leq L \mathcal{R}_{\sf est} (\widehat{\boldsymbol{\theta}}_{\sf minimax}).
\end{equation*}
%\leq t_2 \mathcal{R}_{\sf est}(\widehat{\boldsymbol{\theta}}_{\sf minimax})
where $L$ is a constant and is only a function of $\eta$ and $\sigma_{\sf w}^2$ and is explicitly given in the proof.
%where $d_5$, $t_1$ and $t_2$ are only functions of $\eta$ and $\sigma_{\sf w}^2$ (modulo logarithmic factors in $p^{\star}$ for $t_2$) and are explicitly given in the proof.
\end{prop}

\noindent  \textit{\textbf{Remark 5.}} Proposition \ref{thm:ar_minimax} implies that $\ell_1$-regularized LS is minimax optimal in estimating the $s$-sparse parameter vector $\boldsymbol{\theta}$, for small enough $s$. The proof of the Proposition \ref{thm:ar_minimax} is given in Appendix \ref{prf:minimax}. This result can be extended to compressible $\boldsymbol{\theta}$ in a natural way with a bit more work, but we only present the proof for the case of $s$-sparse $\boldsymbol{\theta}$ for brevity. We also state the following proposition on the prediction performance of the $\ell_1$-regularized LS estimator:
\begin{prop}
\label{thm:ar_minimax2}
Let $\mathbf{x}_{-p+1}^n$ be samples of an AR process with $s$-sparse parameters and Gaussian innovations, then there exists  a positive constant $d_5$ such that for large enough $n,p$ and $s$ satisfying $n>d_1 s(p \log p)^{1/2}$, we have:
\begin{equation}
\mathcal{R}^2_{\sf pre}(\widehat{\boldsymbol{\theta}}_{\ell_1})  \leq d_5 \frac{s \log p}{n}+\sigma^2_{\sf w}.
\end{equation}
%\begin{equation}
%\mathcal{R}_p(\widehat{\boldsymbol{\theta}}_{\ell_1})  \leq c_{\eta}'\left(\frac{1}{n^2}+(1-\eta)^{2p}s+\frac{s}{n}\right)+\sigma_w^2.
%\end{equation}
\end{prop}
It can be readily observed that for {$n \gg s \log p$} the prediction error variance is very close to the variance of the innovations. The proof is similar to Theorem 3 of \cite{goldenshluger2001nonasymptotic} and is skipped in this chapter for brevity.

%or equivalently if power spectral density is exactly zero over an interval or over more than one consecutive point.
%
%It can be shown that Mathematical arguments or simple reasoning will show that a function that is identically zero over some interval cannot be represented by a rational function of $e^{j\omega}$. Historically, this has been formulated as follows \cite{priestley1981spectral}: if
%\begin{equation}
%\int_{-\pi}^\pi \log S(\omega) d\omega > -\infty,
%\end{equation} then a unique one-sided $G(z)$ exists with the sequence $g_0, g_1, \cdots$ which has all zeros inside the unit circle. The integral of the logarithm will become $-\infty$ only if the power spectral density is exactly zero over an interval or over more than one consecutive point. The spectrum may touch zero at a single point, not at an interval. This is a mild requirement for spectra in practice. Therefore, almost all stationary stochastic processes can be modeled by a unique, stationary, and invertible ARMA process. In this paper we require a slightly stronger condition for the process to have a bounded spectral spread. 
\section{Application to Simulated and Real Data}
\label{sec:ar_sim}
In this section, we study and compare the performance of Yule-Walker based  estimation methods with those of the $\ell_1$-regularized and greedy estimators given in Section \ref{sec:ar_theory}. These methods are applied to simulated data as well as real data from crude oil price and traffic speed.

\subsection{Simulation studies}
In order to simulate an AR process, we filtered a Gaussian white noise process using an IIR filter with sparse parameters. Figure \ref{fig:sample_ar} shows a typical sample path of the simulated AR process used in our analysis. For the parameter vector  $\boldsymbol{\theta}$, we chose a length of $p=300$, and employed $n = 1500$ generated samples of the corresponding process for estimation. The parameter vector $\boldsymbol{\theta}$ is of sparsity level $s=3$ and $\eta = 1-\|\boldsymbol{\theta}\|_1=0.5$. A value of $\gamma_n = 0.1$ is used, which is slightly tuned around the theoretical estimate given by Theorem \ref{thm:ar_1}. The order of the process is assumed to be known. We compare the performance of seven estimators: 1) $\widehat{\boldsymbol{\theta}}_{\sf LS}$ using LS, 2) $\widehat{\boldsymbol{\theta}}_{\sf yw}$ using the Yule-Walker equations, 3) $\widehat{\boldsymbol{\theta}}_{\ell_1}$ from $\ell_1$-regularized LS, 4) $\widehat{\boldsymbol{\theta}}_{\sf OMP}$ using OMP,  5) $\widehat{\boldsymbol{\theta}}_{\sf yw, \ell_{2,1}}$ using Eq. (\ref{eq:ar_ywl21}), 6) $\widehat{\boldsymbol{\theta}}_{\sf yw, \ell_{1,1}}$ using Eq. (\ref{eq:ar_ywl11}), and 7) $\widehat{\boldsymbol{\theta}}_{\sf ywOMP}$ using the cost function $\mathfrak{J}(\boldsymbol{\theta})$ in the generalized OMP. {Note that for the LS and Yule-Walker estimates, we have relaxed the condition of $\| \boldsymbol{\theta} \|_1 < 1$, to be consistent with the common usage of these methods. The Yule-Walker estimate is guaranteed to result in a stable AR process, whereas the LS estimate is not \cite{percival1993spectral}.} Figure \ref{fig:ar_param} shows the estimated parameter vectors using these algorithms. It can be visually observed that $\ell_1$-regularized and greedy estimators (shown in {purple}) significantly outperform the Yule-Walker-based estimates (shown in {orange}).

\begin{figure}[H]
\begin{center}
\noindent
\includegraphics[width=.9\columnwidth]{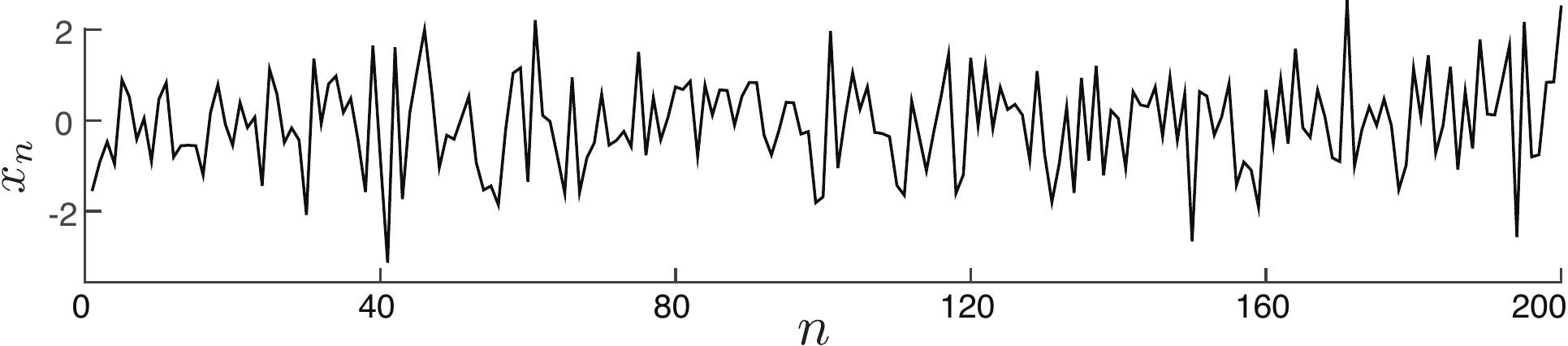}
\caption{Samples of the simulated AR process.}\label{fig:sample_ar}
\end{center}
\vspace{-5mm}
\end{figure}

\begin{figure}[H]
\begin{center}
\noindent
\includegraphics[width=.9\columnwidth]{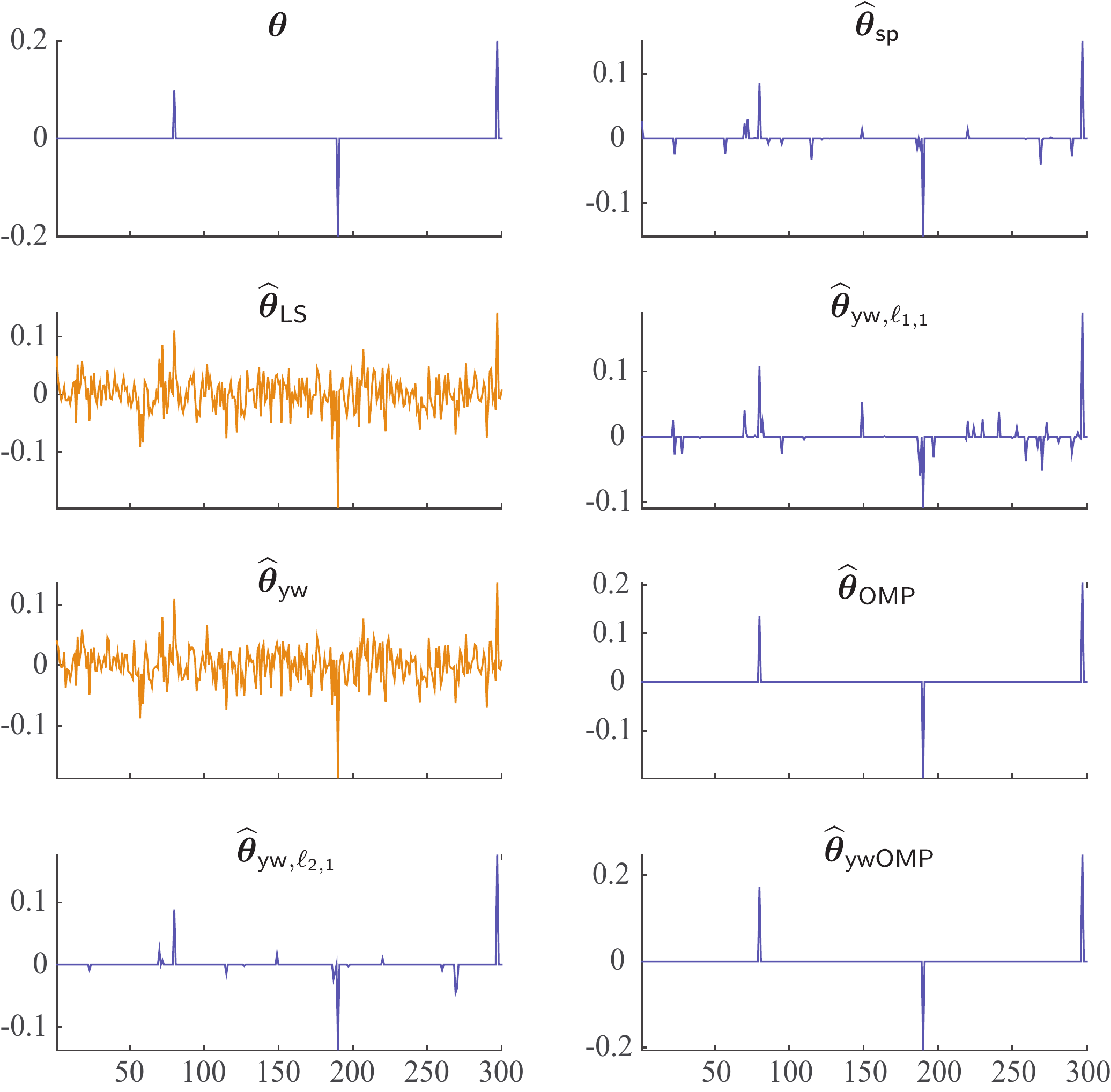}
\caption{Estimates of $\boldsymbol{\theta}$ for $n=1500$, $p =300$, and $s=3$ (These results are best viewed in the color version).}\label{fig:ar_param}
\end{center}
\vspace{-5mm}
\end{figure}

In order to quantify the latter observation precisely, we repeated the same experiment for $p=300, s=3$ and $10 \leq n \leq 10^5$. A comparison of the normalized MSE of the estimators vs. $n$ is shown in Figure \ref{fig:ar_mse}. As it can be inferred from Figure \ref{fig:ar_mse}, in the region where $n$ is comparable to or less than $p$ {(shaded in light {purple})}, the sparse estimators have a systematic performance gain over the Yule-Walker based estimates, with the $\ell_1$-regularized LS and ywOMP estimates outperforming the rest.

\begin{figure}[H]
\begin{center}
\noindent
\includegraphics[width=.75\columnwidth]{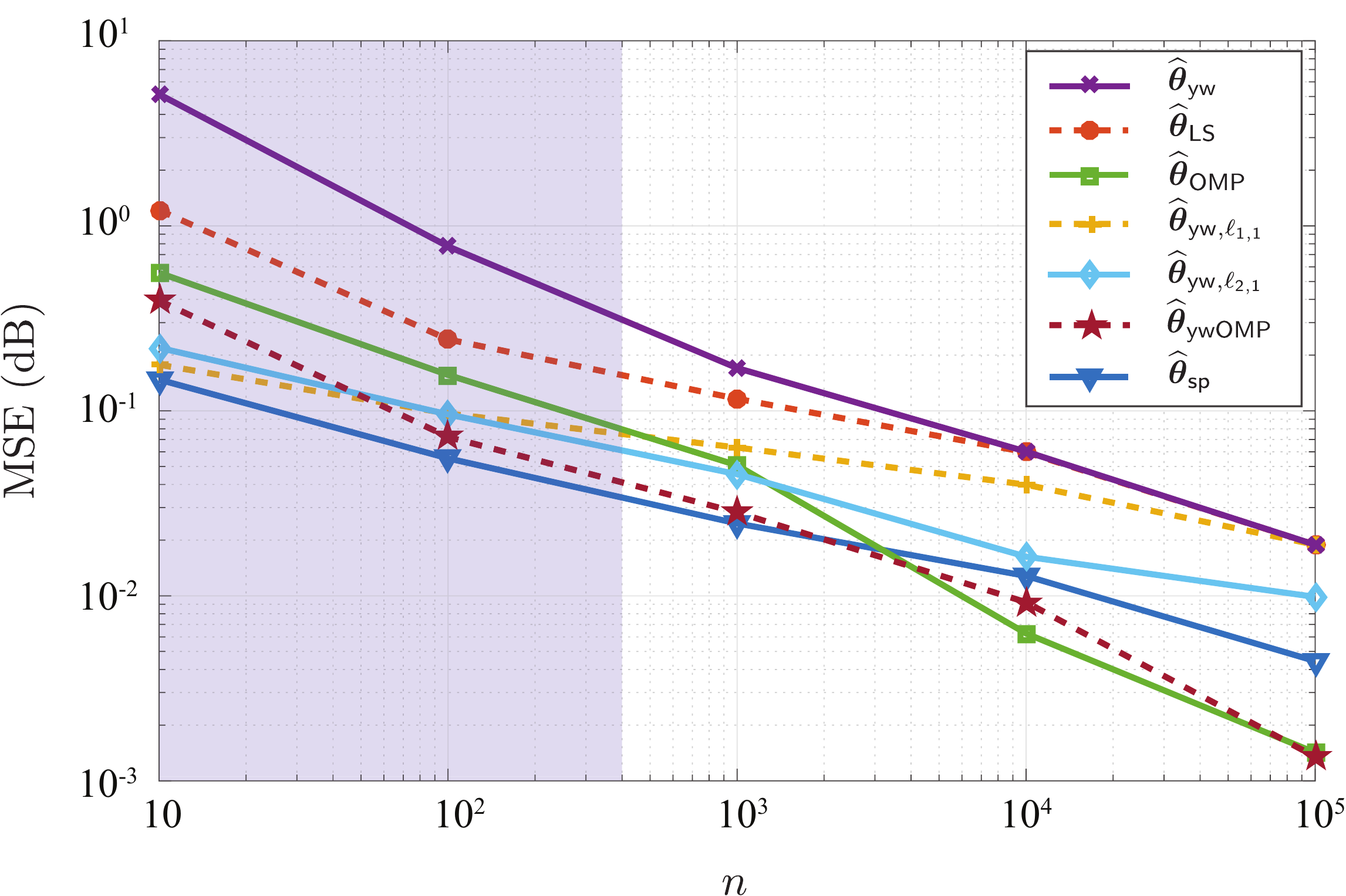}
\caption{MSE comparison of the estimators vs. the number of measurements $n$. {The shaded region corresponds to the compressive regime of $n < p$.}}\label{fig:ar_mse}
\end{center}
\vspace{-5mm}
\end{figure}

The MSE comparison in Figure \ref{fig:ar_mse} requires one to know the true parameters. In practice, the true parameters are not available for comparison purposes. In order to quantify the performance gain of these methods, we use statistical tests to assess the goodness-of-fit of the estimates. The common chi-square type statistical tests, such as the F-test, are useful when the hypothesized distribution to be tested against is discrete or categorical. For our problem setup with sub-Gaussian innovations, we will use a number of statistical tests appropriate for AR processes, namely, the Kolmogorov-Smirnov (KS) test, the Cram\'er-von Mises (CvM) criterion, {the spectral Cram\'er-von Mises (SCvM) test} and the Anderson-Darling (AD) \cite{d1986goodness,johansen1995likelihood,anderson1997goodness}. A summary of these tests is given in {Appendix \ref{app:ar_tests}}. Table \ref{tab:synthetic_table} summarizes the test statistics for different estimation methods. {Cells colored in orange (darker shade in grayscale) correspond to traditional AR estimation methods and those colored in blue (lighter shade in grayscale) correspond to the sparse estimator with the best performance among those considered in this work}. These tests are based on the known results on limiting distributions of error residuals. As noted from Table \ref{tab:synthetic_table}, our simulations suggest that the OMP estimate achieves the best test statistics for the CvM, AD and KS tests, whereas the $\ell_1$-regularized estimate achieves the best SCvM statistic.

% \textbf{For comparison purposes, after the model is fitted, differencing is converted and all forcast values are constructed for the oroginal series.}

\begin{table}[h!]
\centering
\caption{Goodness-of-fit tests for the simulated data}
\label{tab:synthetic_table}
\begin{tabular}{lllll}
\multicolumn{1}{l|}{\backslashbox{Estimate}{Test}} & CvM                  & AD                   & KS                  & SCvM \\ \cline{1-5} \hline
\multicolumn{1}{l|}{${\boldsymbol{\theta}}$}   &  0.31          &  1.54  &  0.031    &   0.009      \\
\multicolumn{1}{l|}{$\widehat{\boldsymbol{\theta}}_{\sf LS}$}  &   \cellcolor{orange!70}0.68         &   \cellcolor{orange!70}5.12         &  \cellcolor{orange!70}0.037      &  \cellcolor{orange!70}0.017   \\
\multicolumn{1}{l|}{$\widehat{\boldsymbol{\theta}}_{\sf yw}$}  & \cellcolor{orange!70}0.65 & \cellcolor{orange!70}4.87 &  \cellcolor{orange!70}0.034 &  \cellcolor{orange!70}0.025 \\
\multicolumn{1}{l|}{$\widehat{\boldsymbol{\theta}}_{\ell_1}$}  & 0.34 & 1.72 &  0.030  &    \cellcolor{blue!10}0.009        \\
\multicolumn{1}{l|}{$\widehat{\boldsymbol{\theta}}_{\sf OMP}$}  & \cellcolor{blue!10}0.29 &            \cellcolor{blue!10}1.45  &   \cellcolor{blue!10}0.028     &   0.009    \\
\multicolumn{1}{l|}{$\widehat{\boldsymbol{\theta}}_{{\sf yw},\ell_{2,1}}$}  &  0.35 & 1.80 &   0.032   & 0.009  \\
\multicolumn{1}{l|}{$\widehat{\boldsymbol{\theta}}_{{\sf yw},\ell_{1,1}}$}  &         0.42  &   2.33       &    0.040     &   0.008  \\
\multicolumn{1}{l|}{$\widehat{\boldsymbol{\theta}}_{\sf ywOMP}$}  &  0.29 &  1.46 &        0.030 &   0.009
\end{tabular}
\vspace{-2mm}
\end{table}

%\color{blue}
%\noindent \textbf{\textit{ 7:}} The asymptotic model order selection criteria of \cite{goldenshluger2001nonasymptotic} results in an underestimated order for the process, and therefore the small choice of mode order is not capable of capturing the sparse and long order structure of the parameters. Throughout this section, we have instead used the order selection described in Remark 4.
%

\subsection{Application to the analysis of {crude oil prices}}

In this and the following subsection, we consider applications with real-world data. As for the first application, we apply the {sparse} AR estimation techniques to analyze the crude oil price of the Cushing, OK WTI Spot Price FOB dataset \cite{cushing}. This dataset consists of 7429 daily values of oil prices in dollars per barrel. In order to avoid outliers, usually the dataset is filtered with a moving average filter of high order. We have skipped this procedure by visual inspection of the data and selecting $n=4000$ samples free of outliers. Such financial data sets are known for their non-stationarity and long order history dependence. In order to remove the deterministic trends in the data, one-step or two-step time differencing is typically used. We refer to \cite{robinson2003time} for a full discussion of this detrending method. We have used a first-order time differencing which resulted in a sufficient detrending of the data. Figure \ref{fig:sample_oil} shows the data used in our analysis. We have chosen $p=150$ by inspection. The histogram of first-order differences as well the estimates are shown in Figure \ref{fig:oil_estimate}.

\begin{figure}[H]
\begin{center}
\noindent
\includegraphics[width=.8\columnwidth]{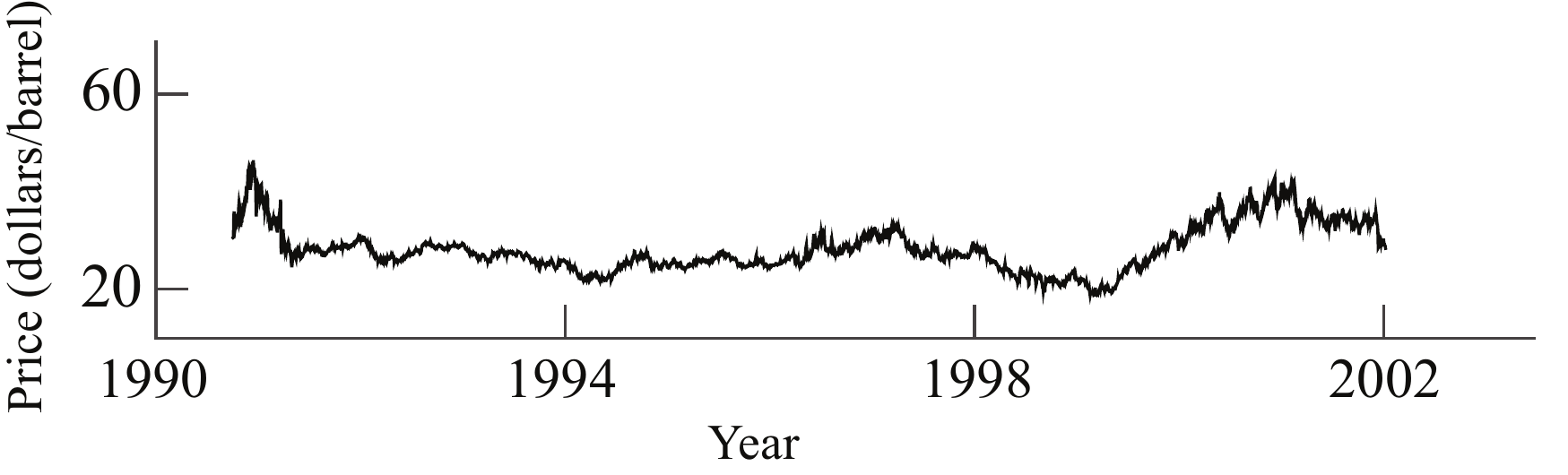}
\caption{{A sample segment of the Cushing, OK WTI Spot Price FOB data.}}\label{fig:sample_oil}
\end{center}
\vspace{-2mm}
\end{figure}

\begin{figure}[h!]
\begin{center}
\noindent
\includegraphics[width=.9\columnwidth]{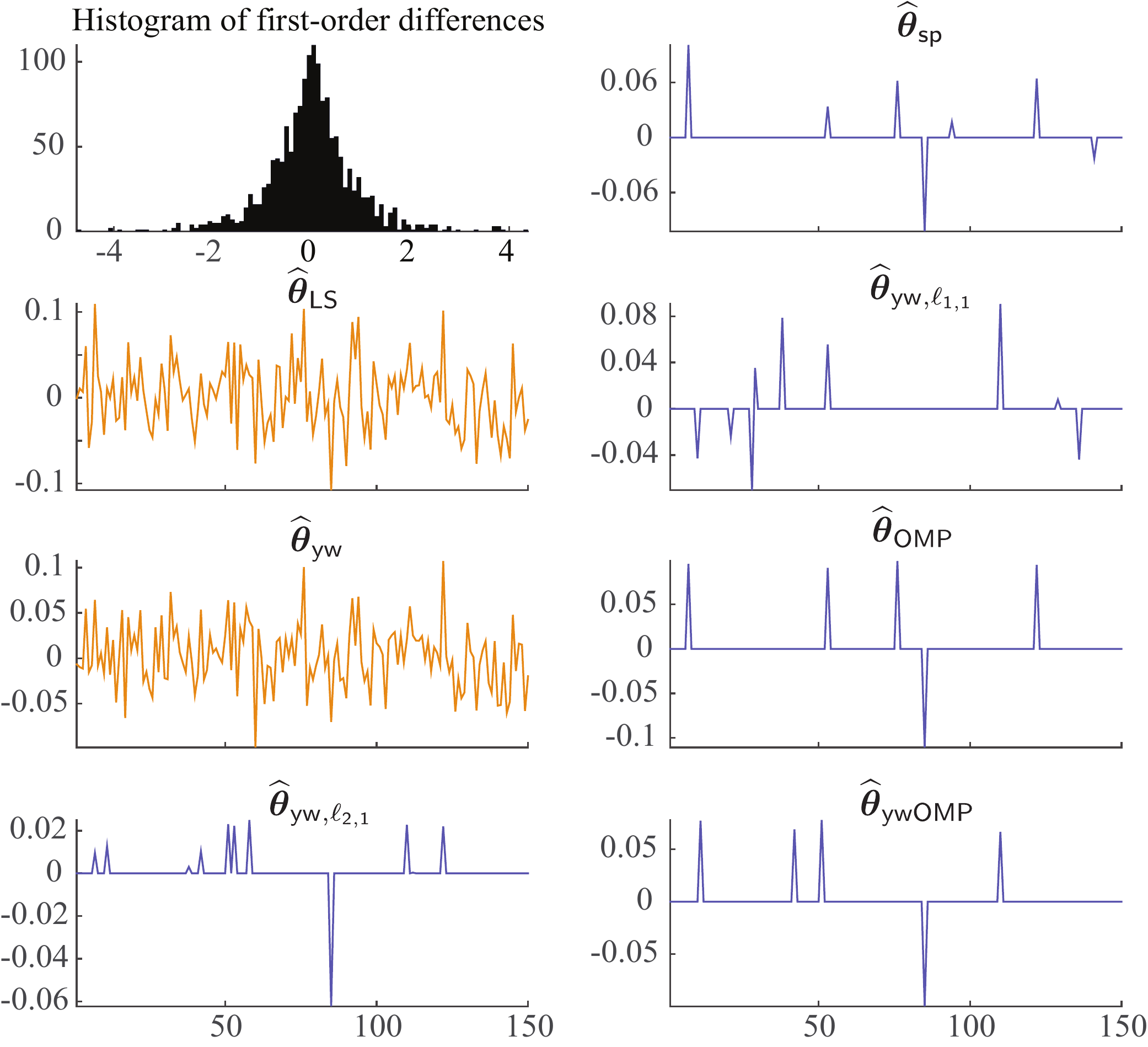}
\caption{Estimates of $\boldsymbol{\theta}$ for the second-order differences of the oil price data.}\label{fig:oil_estimate}
\end{center}
\vspace{-2mm}
\end{figure}

A visual inspection of the estimates in Figure \ref{fig:oil_estimate} shows that the $\ell_1$-regularized LS ($\widehat{\boldsymbol{\theta}}_{\ell_1}$) and OMP ($\widehat{\boldsymbol{\theta}}_{\sf OMP}$)  estimates consistently select specific time lags in the AR parameters, whereas the Yule-Walker and LS estimates seemingly overfit the data by populating the entire parameter space. In order to perform goodness-of-fit tests, we use an even/odd two-fold cross-validation. Table \ref{tab:oil_table} shows the corresponding test statistics, which reveal that indeed the $\ell_1$-regularized and OMP estimates outperform the traditional estimation techniques.

\begin{table}[h]
\centering
\caption{{Goodness-of-fit tests for the crude oil price data}}
\label{tab:oil_table}
\begin{tabular}{lllll}
\multicolumn{1}{l|}{\backslashbox{Estimate}{Test}} & CvM                  & AD                   & KS                  & SCvM \\ \cline{1-5} \hline
\multicolumn{1}{l|}{$\widehat{\boldsymbol{\theta}}_{\sf LS}$}  &   \cellcolor{orange!70}0.88         &   \cellcolor{orange!70}5.55        &  \cellcolor{orange!70}0.055      &  \cellcolor{orange!70}0.046   \\
\multicolumn{1}{l|}{$\widehat{\boldsymbol{\theta}}_{\sf yw}$}  & \cellcolor{orange!70}0.58 & \cellcolor{orange!70}3.60 &  \cellcolor{orange!70}0.043 &  \cellcolor{orange!70}0.037 \\
\multicolumn{1}{l|}{$\widehat{\boldsymbol{\theta}}_{\ell_1}$}  & 0.27 & 1.33 &  0.031  &    \cellcolor{blue!10}0.020       \\
\multicolumn{1}{l|}{$\widehat{\boldsymbol{\theta}}_{\sf OMP}$}  & \cellcolor{blue!10}0.22 &            \cellcolor{blue!10}1.18  &   \cellcolor{blue!10}0.025     &   0.022    \\
\multicolumn{1}{l|}{$\widehat{\boldsymbol{\theta}}_{{\sf yw},\ell_{2,1}}$}  &  0.28 & 1.40 &   0.027   & 0.021  \\
\multicolumn{1}{l|}{$\widehat{\boldsymbol{\theta}}_{{\sf yw},\ell_{1,1}}$}  &         0.24  &   1.26       &    0.027     &   0.022  \\
\multicolumn{1}{l|}{$\widehat{\boldsymbol{\theta}}_{\sf ywOMP}$}  &  0.23 &  \cellcolor{blue!10} 1.18 &        0.026 &   0.022
\end{tabular}
\vspace{-2mm}
\end{table}

\subsection{Application to the analysis of traffic data}
Our second real data application concerns traffic speed data. The data used in our simulations is the INRIX \textregistered\  speed data for I-495 Maryland inner loop freeway (clockwise) between US-1/Baltimore Ave/Exit 25 and Greenbelt Metro Dr/Exit 24 from 1 Jul, 2015 to 31 Oct, 2015 \cite{ritis1,ritis2}. The reference speed of 65 mph is reported. {Our aim is to analyze the long-term, large-scale periodicities manifested in these data by fitting high-order sparse AR models}. Given the huge length of the data and its high variability, the following pre-processing was made on the original data:
\begin{enumerate}
\item The data was downsampled by a factor of $4$ and averaged by the hour in order to reduce its daily variability, that is each lag corresponds to one hour.
\item The logarithm of speed was used for analysis and the mean was subtracted.  This reduces the high variability of speed due to rush hours and lower traffic during weekends and holidays.
\end{enumerate}

\begin{figure}[H]
\begin{center}
\noindent
\includegraphics[width=0.8\columnwidth]{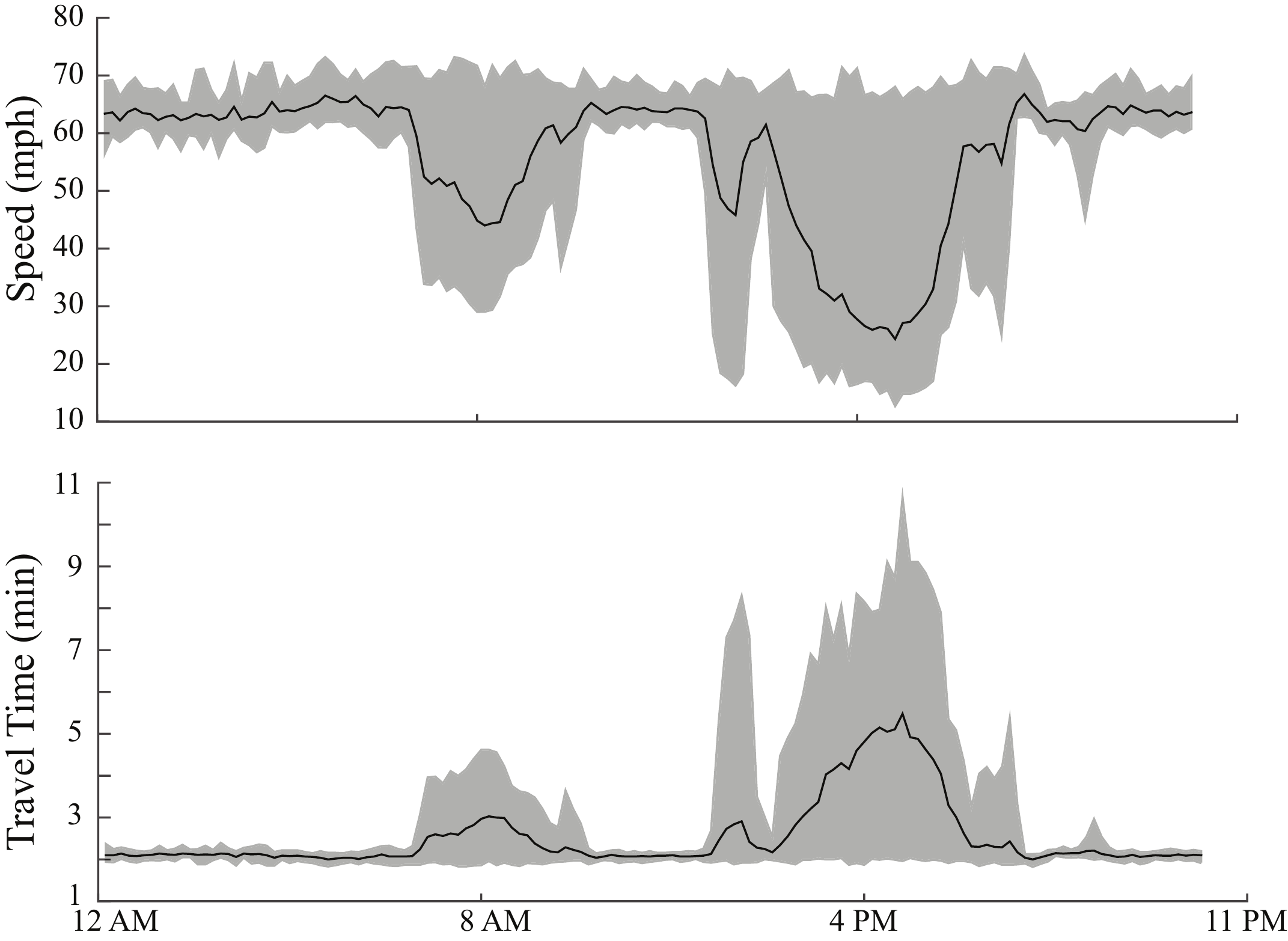}
\caption{{A sample of the speed and travel time data for I-495.}}\label{fig:ar_speed}
\end{center}
\vspace{-2mm}
\end{figure}

Figure \ref{fig:ar_speed} shows a typical average weekly speed and travel time in this dataset and the corresponding 25-75-th percentiles. As can be seen the data shows high variability around the rush hours of $8~\text{am}$ and $4~\text{pm}$. In our analysis, we used the first half of the data ($n=1500$) for fitting, from which the AR parameters and the distribution and variance of the innovations were estimated. The statistical tests were designed based on the estimated distributions, and the statistics were computed accordingly using the second half of the data. {We selected an order of $p = 200$ by inspection and noting that the data seems to have a periodicity of order $170$ samples.}

\begin{figure}[H]
\begin{center}
\noindent
\includegraphics[width=.9\columnwidth]{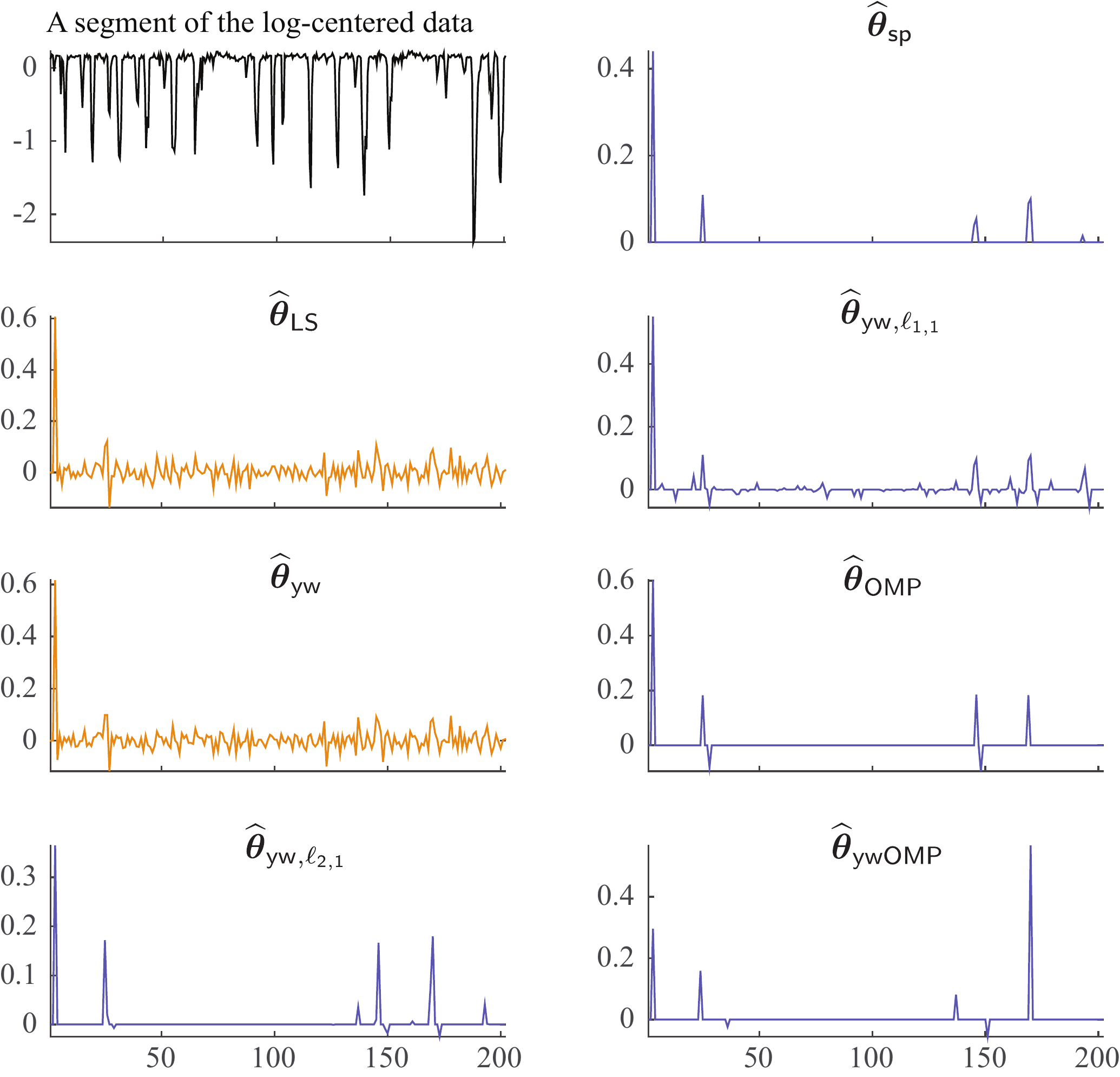}
\caption{{Estimates of $\boldsymbol{\theta}$ for the traffic speed data.}}\label{fig:traffic_estimate}
\end{center}
\vspace{-2mm}
\end{figure}

Figure \ref{fig:traffic_estimate} shows part of the data used in our analysis as well as the estimated parameters. The $\ell_1$-regularized LS ($\widehat{\boldsymbol{\theta}}_{\ell_1}$) and OMP ($\widehat{\boldsymbol{\theta}}_{\sf OMP}$) are consistent in selecting the same components of $\boldsymbol{\theta}$. These estimators pick up two major lags around which $\boldsymbol{\theta}$ has its largest components. The first lag corresponds to about $24$ hours which is mainly due to the rush hour periodicity on a daily basis. The second lag is around $150-170$ hours which corresponds to weekly changes in the speed due to lower traffic over the weekend. In contrast, the Yule-Walker and LS estimates do not recover these significant time lags. 

\begin{table}[h]
\centering
\vspace{-2mm}
\caption{Goodness-of-fit tests for the traffic speed data}
\label{tab:traffic_table}
\begin{tabular}{lllll}
\multicolumn{1}{l|}{\backslashbox{Estimate}{Test}} & CvM                  & AD                   & KS                  & SCvM \\ \cline{1-5} \hline
\multicolumn{1}{l|}{$\widehat{\boldsymbol{\theta}}_{\sf yw}$}  & \cellcolor{orange!70}0.012 & \cellcolor{orange!70}0.066 &  \cellcolor{orange!70}0.220 &  \cellcolor{orange!70} 0.05 \\
\multicolumn{1}{l|}{$\widehat{\boldsymbol{\theta}}_{\ell_1}$}  &  \cellcolor{blue!10}1.4$\times 10^{-7}$ &  \cellcolor{blue!10}2.1$\times 10^{-6}$ &   \cellcolor{blue!10}6.7$\times 10^{-4}$  &   0.25       \\
\multicolumn{1}{l|}{$\widehat{\boldsymbol{\theta}}_{\sf OMP}$}  & 0.017 &  0.082  & 0.220 &  1.49    \\
\multicolumn{1}{l|}{$\widehat{\boldsymbol{\theta}}_{\sf ywOMP}$}  &  0.025 &  0.122 &        0.270 &   0.14
\end{tabular}
\end{table}

Statistical tests for a selected subset of the estimators are shown in Table \ref{tab:traffic_table}. Interestingly, the $\ell_1$-regularized LS estimator significantly outperforms the other estimators in three of the tests. The Yule-Walker estimator, however, achieves the best SCvM test statistic.

\section{Concluding Remarks}\label{sec:conc}
In this chapter, we investigated sufficient sampling requirements for stable estimation of AR models in the non-asymptotic regime using the $\ell_1$-regularized LS and greedy estimation (OMP) techniques. We {have} further established the minimax optimality of the $\ell_1$-regularized LS estimator. Compared to the existing literature, our results provide several major contributions. {First, when $s \sim p^{\frac{1}{2} + \delta}$ for some $\delta \ge 0$, our results suggest an improvement of order $\mathcal{O}(p^{\delta} (\log p)^{3/2})$ in the sampling requirements for the estimation of univariate AR models with sub-Gaussian innovations using the LASSO, over those of \cite{wong2016regularized} and \cite{wu2016performance} which require $n \sim \mathcal{O}(p^2 (\log p)^2)$ for stable AR estimation.} {When specialized to a sub-Gaussian white noise process, i.e., establishing the RE condition of i.i.d. Toeplitz matrices, our results provide an improvement of order $\mathcal{O}(s/\log p)$ over those of \cite{Toeplitz}. Second, although OMP is widely used in practice, the choice of the number of greedy iterations is often ad-hoc. In contrast, our theoretical results prescribe an analytical choices of the number of iterations required for stable estimation, thereby promoting the usage of OMP as a low-complexity algorithm for AR estimation. Third, we established the minimax optimality of the $\ell_1$-regularized LS estimator for the estimation of sparse AR parameters.

We {further} verified the validity of our theoretical results through simulation studies as well as application to real financial and traffic data. These results show that the sparse estimation methods significantly outperform the widely-used Yule-Walker based estimators in fitting AR models to the data. Although we did not theoretically analyze the performance of sparse Yule-Walker based estimators, they seem to perform on par with the $\ell_1$-regularized LS and OMP estimators based on our numerical studies. Finally, as we will see in the next chapter our results provide a striking connection to estimation of sparse self-exciting discrete point process models. These models regress an observed binary spike train with respect to its history via Bernoulli or Poisson statistics, and are often used in describing spontaneous activity of sensory neurons. Our results {have shown} that in order to estimate a sparse history-dependence parameter vector of length $p$ and sparsity $s$ in a stable fashion, a spike train of length $n \sim \mathcal{O}(s^{2/3}p^{2/3}\log p)$ is required. {This leads us to conjecture} that these sub-linear sampling requirements are sufficient for a larger class of autoregressive processes, beyond those characterized by linear models.} {Finally, our minimax optimality result requires the sparsity level $s$ to grow at most as fast as $\mathcal{O}(n/(p\log p)^{1/2})$. We consider further relaxation of this condition, as well as the generalization of our results to sparse MVAR processes as future work.}

%{Moreover our results can be generalized to Multivariate AR models via a multivariate version of the concentration inequalities of  \cite{rudzkis1978large} in connjunction with the techniques of \cite{ wong2016regularized} and \cite{wu2016performance}}.

\chapter{Robust Estimation of Self-Exciting {Generalized Linear Models}}
\chaptermark{Robust Estimation of GLM's}

\label{chap:hawkes}

In this chapter, we close the gap in theory of compressed sensing for non-i.i.d. data by providing theoretical guarantees on stable estimation of self-exciting generalized linear models. We consider the problem of estimating self-exciting {generalized linear models} from limited binary observations, where the history of the process serves as the covariate. In doing so, we relax the assumptions of i.i.d. covariates and exact sparsity common in CS. Our results indicate that utilizing sparsity recovers important information about the intrinsic frequencies of such processes.  We analyze the performance of two classes of estimators, namely the $\ell_1$-regularized maximum likelihood and greedy estimators, for a canonical self-exciting process and characterize the sampling tradeoffs required for stable recovery in the non-asymptotic regime. Our results extend those of compressed sensing for linear and generalized linear models with i.i.d. covariates to those with highly inter-dependent covariates. We further provide simulation studies as well as application to real spiking data from the mouse's lateral geniculate nucleus and the ferret's retinal ganglion cells under different nonlinear forward models which agree with our theoretical predictions.

\section{Introduction}

The theory of compressed sensing (CS) has provided a novel framework for measuring and estimating statistical models governed by sparse underlying parameters {\cite{donoho2006compressed, candes2006compressive, candes2006stable, candes2008introduction, needell2009cosamp, bruckstein2009sparse}}. In particular, for linear models with random covariates and sparsity of the parameters, the CS theory provides sharp trade-offs between the number of measurement, sparsity, and estimation accuracy. Typical theoretical guarantees imply that when the number of random measurements are roughly proportional to sparsity, then stable recovery of these sparse models is possible.

Beyond those described by linear models, observations from binary phenomena form a large class of data in natural and social sciences. Their ubiquity in disciplines such as neuroscience, physiology, seismology, criminology, and finance has urged researchers to develop formal frameworks to model and analyze these data. In particular, the theory of point processes provides a statistical machinery for modeling and prediction of such phenomena. Traditionally, these models have been employed to predict the likelihood of self-exciting processes such as earthquake occurrences \cite{ogata1988statistical, vere1970stochastic}, but have recently found applications in several other areas. For instance, these models have been used to characterize heart-beat dynamics \cite{barbieri2005point, valenza2013point} and violence among gangs \cite{egesdal}. Self-exciting point process models have also found significant applications in analysis of neuronal data {\cite{brown2001analysis,smith2003estimating, brown2004multiple,  paninski2004maximum, truccolo2005point, paninski2007statistical, pillow2011model}.} 

In particular, point process models provide a principled way to regress binary spiking data with respect to extrinsic stimuli and neural covariates, and thereby forming predictive statistical models for neural spiking activity. {Examples include place cells in the hippocampus \cite{brown2001analysis}, spectro-temporally tuned cells in the primary auditory cortex \cite{calabrese2011generalized}, and spontaneous retinal or thalamic neurons spiking under tuned intrinsic frequencies \cite{Borowska,liets2003spontaneous}. Self-exciting point processes have also been utilized in assessing the functional connectivity of neuronal ensembles \cite{kim2011granger,brown_func_conn}.} When fitted to neuronal data, these models exhibit three main features: first, the underlying parameters are nearly sparse or compressible  \cite{Brown_pp, brown_func_conn}; second, the covariates are often highly structured and correlated; and third, the input-output relation is highly nonlinear. Therefore, the theoretical guarantees of compressed sensing do not readily translate to prescriptions for point process estimation.

Estimation of these models is typically carried out by Maximum Likelihood (ML) or regularized ML estimation {in discrete time, where the process is viewed as a Generalized Linear Model (GLM). In order to adjust the regularization level, empirical methods such as cross-validation are typically employed \cite{brown_func_conn}.} In the signal processing and information theory literature, sparse signal recovery under Poisson statistics has been considered in \cite{poi_sp_del} with application to the analysis of ranking data. In \cite{cs_poi}, a similar setting has been studied, with motivation from imaging by photon-counting devices. Finally, in theoretical statistics, high-dimensional $M$-estimators with decomposable regularizers, such as the $\ell_1$-norm, have been studied for GLMs \cite{Negahban}. 

A key underlying assumption in the existing theoretical analysis of estimating {GLMs} is the independence and identical distribution (i.i.d.) of covariates. This assumption does not hold for self-exciting processes, since the history of the process takes the role of the covariates. Nevertheless, regularized ML estimators show remarkable performance in fitting {GLMs} to neuronal data with history dependence and highly non-i.i.d. covariates.  In this chapter, we close this gap by presenting new results on robust estimation of compressible GLMs, relaxing the assumptions of i.i.d. covariates and exact sparsity common in CS. 

In particular, we will consider a {canonical GLM} and will analyze two classes of estimators for its underlying parameters: the $\ell_1$-regularized maximum likelihood and greedy estimators. We will present theoretical guarantees that extend those of CS theory and characterize fundamental trade-offs between the number of measurements, model compressibility, and estimation error of GLMs in the non-asymptotic regime. Our results reveal that when the number of measurements scale sub-linearly with the product of the ambient dimension and a generalized measure of sparsity (modulo logarithmic factors), then stable recovery of the underlying models is possible, even though the covariates solely depend on the history of the process. We will further discuss the extensions of these results to more general classes of {GLMs}. Finally, we will present applications to simulated as well as real data from {two classes of neurons exhibiting spontaneous activity}, namely the mouse's lateral geniculate nucleus and the ferret's retinal ganglion cells, which agree with our theoretical predictions. Aside from their theoretical significance, our results are particularly important in light of the technological advances in neural prostheses, which require robust neuronal system identification based on compressed data acquisition.

The rest of the chapter is organized as follows: In Section \ref{prelim}, we present our notational conventions, preliminaries and problem formulation. In Section \ref{theoretical}, we discuss the estimation procedures and state the main theoretical results of this chapter. Section \ref{simulations} provides numerical simulations as well as application to real data. In Section \ref{discussions}, we discuss the implications of our results and outline future research directions. Finally, we present the proofs of the main theoretical results and give a brief background on relevant statistical tests in Appendices \ref{appprf}--\ref{appks}.

\section{Preliminaries and Problem Formulation} \label{prelim}
%\subsection{Self-Exciting {Generalized Linear Models}}
We {first} give a brief introduction to self-exciting GLMs (see \cite{Daley2007} for a detailed treatment).

We consider a sequence of observations in the form of binary spike trains obtained by discretizing continuous-time observations (e.g. electrophysiology recordings), using bins of length $\Delta$. We assume that  not more than one event fall into any given bin. In practice, this can always be achieved by choosing $\Delta$ small enough. The binary observation at bin $i$ is denoted by $x_i$. The observation sequence can be modeled as the outcome of conditionally independent Poisson or Bernoulli trials, with a spiking probability given by $\mathbb{P}(x_i = 1) =: \lambda_{i|H_i}$, where $\lambda_{i|H_i}$ is the spiking probability at bin $i$ given the history of the process $H_i$ up to bin $i$.

These models are widely-used in neural data analysis and are motivated by the continuous time point processes with history dependent conditional intensity functions \cite{Daley2007}. For instance, given the history of a continuous-time point process $H_t$ up to time $t$, a conditional intensity of $\lambda(t | H_t) = \lambda$ corresponds to the {homogeneous Poisson} process.  As another example, a conditional intensity of $\lambda(t | H_t) = \mu + \int_{-\infty}^t \theta(t - \tau) dN(\tau)$ corresponds to a process known as the Hawkes process \cite{Hawkes_orig} with base-line rate $\mu$ and history dependence kernel $\theta(\cdot)$. Under the assumption of the orderliness of a continuous-time point process, a discretized approximation to these processes can be obtained by binning the process by bins of length $\Delta$, and defining the spiking probability by $\lambda_i := \lambda(i \Delta | H_{i\Delta}) \Delta + o(\Delta)$. In this chapter, we consider discrete {random processes} characterized by the spiking probability $\lambda_{i|H_i}$, which are either inherently discrete or employed as an approximation to continuous-time point process models.

Throughout the rest of the chapter, we drop the dependence of $\lambda_{i|H_i}$ on $H_i$ to simplify notation, denote it by $\lambda_i$ and refer to it {as} spiking probability. Given the sequence of {binary} observed data {$\mathbf{x}_1^n$}, the negative log-likelihood function under the {Bernoulli} statistics can be expressed as:

\begin{equation}
\label{L_def}
\mathfrak{L}(\boldsymbol{\theta}) = -\frac{1}{n} \sum\limits_{i=1}^n \left \{ x_i\log\lambda_{i}+(1-x_i)\log(1-\lambda_{i}) \right \}.
\end{equation}
Another common likelihood model used in the analysis of neuronal spiking data corresponds to Poisson statistics \cite{Brown_pp}, for which the negative log-likelihood takes the following form:
\begin{equation}
\label{poi_ll}
\mathfrak{L}(\boldsymbol{\theta}) :=  -\frac{1}{n} \sum\limits_{i=1}^n \left \{ x_i\log\lambda_{i}-\lambda_{i} \right \}.
\end{equation}
Throughout the chapter, we will focus on binary observations governed by Bernoulli statistics, whose negative log-likelihood is given in Eq. (\ref{L_def}). In applications such as electrophysiology in which neuronal spiking activities are recorded at a high sampling rate, the binning size $\Delta$ is very small and the Bernoulli and Poisson statistics coincide.

When the discrete process is viewed as an {approximation} to a continuous-time process, these log-likelihood functions are known as the Jacod log-likelihood approximations \cite{Daley2007}. We will present our analysis for the negative log-likelihood given by (\ref{L_def}), but our results can be extended to other statistics including (\ref{poi_ll}) (See {the remarks of Section \ref{theoretical}}).

Throughout this chapter {${\mathbf{x}_{-p+1}^{n}}$} will be considered as the observed spiking sequence which will be used for estimation purposes. A popular class of models for $\lambda_i$ is given by GLMs. In its general form, a GLM consists of two main components: an observation model and an equation expressing some (possibly nonlinear) function of the observation mean as a \textit{linear} combination of the covariates. In neural systems, the covariates consist of external stimuli as well as the history of the process. Inspired by spontaneous neuronal activity, we consider \textit{fully} self-exciting processes, in which the covariates are only functions of the process history. As for a canonical {GLM} inspired by the Hawkes process, we consider a process for which the spiking probability is a \textit{linear} function of the process history:
\begin{equation}
\label{CIF_models}
\lambda_i := \mu + \boldsymbol{\theta}' {\mathbf{x}_{i-p}^{i-1}},
\end{equation}
where $\mu$ is a positive constant representing the base-line rate, and $\boldsymbol{\theta}=[\theta_1,\theta_2,\cdots,\theta_p]'$ is a parameter vector denoting the history dependence of the process. {We further assume that the process is non-degenerate, i.e., it will not terminate in an infinite sequence of zeros.} We refer to this {GLM, viewed as a random process,} as the \textit{canonical self-exciting process}. Other popular models in the computational neuroscience literature include the log-link model where $\lambda_i = \exp ( \mu + \boldsymbol{\theta}'{\mathbf{x}_{i-p}^{i-1}} )$ and the logistic-link model where $\lambda_i = \frac{\exp(\mu + \boldsymbol{\theta}'{\mathbf{x}_{i-p}^{i-1}})}{1+\exp (\mu + \boldsymbol{\theta}'{\mathbf{x}_{i-p}^{i-1}})}$. The parameter vector $\boldsymbol{\theta}$ can be thought of as the binary equivalent of autoregressive (AR) parameters in linear AR models.

When fitted to neuronal spiking data, the parameter vector $\boldsymbol{\theta}$ exhibits a degree of sparsity \cite{Brown_pp, brown_func_conn}, that is, only certain lags in the history have a significant contribution in determining the statistics of the process. These lags can be thought of as the preferred or intrinsic {delays} in the spontaneous response of a neuron. {To be more precise, for a sparsity level $s < p$, we denote by $\boldsymbol{\theta}_s$ the best $s$-term approximation to $\boldsymbol{\theta}$. 

Finally, in this chapter, we are concerned with the compressed sensing regime where $n \ll p$, i.e., the observed data has a much smaller length than the ambient dimension of the parameter vector. The main estimation problem of this chapter is the following: \emph{given observations ${\mathbf{x}_{-p+1}^n}$ from the \textcolor{black}{canonical self-exciting process}, the goal is to estimate the unknown baseline rate $\mu$ and the $p$-dimensional $(s,\xi)$-compressible history dependence parameter vector $\boldsymbol{\theta}$ in a stable fashion (where the estimation error is controlled) when $n \ll p$.}

\section{Theoretical Results}\label{theoretical}
In this section, we consider two estimators for $\boldsymbol{\theta}$, namely, the $\ell_1$-regularized ML estimator and a greedy estimator, and present the main theoretical results of this chapter on the estimation error of these estimators. Note that when $\mu$ is not known, the following results can be applied to the augmented parameter vector $[\mu,\boldsymbol{\theta}']'$. We analyze the case of known $\mu$ for simplicity of presentation.

\subsection{$\ell_1$-Regularized ML Estimation}

The natural estimator for the parameter vector is the ML estimator, which is widely used in neuroscience \cite{Brown_pp}, which by virtue of (\ref{L_def}) is given by:
\begin{equation}
\label{ML_est_pp_L}
\widehat{\boldsymbol{\theta}}_{{\sf ML}}=\argmin\limits_{\boldsymbol{\theta}\in \boldsymbol{\Theta} } \mathfrak{L}(\boldsymbol{\theta}),
\end{equation}
where $\boldsymbol{\Theta}$ is the relaxed closed convex feasible region for which $0 \le \lambda_i \le 1$ given by the conditions:
%\begin{align}\label{eq:star}
%\nonumber & 1) \ 0<\mathbf{1}'\boldsymbol{\theta}<c_1<1,\\
%\nonumber & 2) \ 0<\pi_\min <\mu+\boldsymbol{\theta}'x_{i-p}^{i-1} <\pi_\max < 1/2,\text{ for } i=1,2,\cdots, \tag{$\star$}
%\end{align}
\vspace{-1mm}
\begin{align}\label{eq:star}
\begin{tabular}{l} ${0 < \pi_{\min} \le \mu -\mathbf{1}' \boldsymbol{\theta}^-}$,\\
 ${\mu +\mathbf{1}' \boldsymbol{\theta}^+ \leq \pi_\max < 1/2}$,
 \end{tabular} {\tag{$\star$}}
\end{align}
for some constants $\pi_{\min}$ and $\pi_{\max}$. {This first inequality incurs minimal loss of generality, as $\pi_{\min}$ can be chosen to be arbitrarily small. The restriction of $\pi_{\max} < 1/2$ ensures that the process is fast mixing and has mainly been adopted for technical convenience. This assumption incurs some loss of generality, as it excludes processes for which the maximum spiking probability exceeds $1/2$. However, due to the low spiking probability of typical neuronal activity, this loss is tolerable for the applications of interest in this chapter (see Section \ref{simulations})}.

In the regime of interest when $n \ll p$, the ML estimator is ill-posed and is typically regularized with a smooth norm. In order to capture the compressibility of the parameters, we consider  the $\ell_1$-regularized ML estimator:
\begin{equation}
\label{sp_est_pp_L}
\widehat{\boldsymbol{\theta}}_{{\sf sp}}:=\argmin\limits_{\boldsymbol{\theta}\in \boldsymbol{\Theta}} \quad \mathfrak{L}(\boldsymbol{\theta})+ \gamma_n\|\boldsymbol{\theta}\|_1.
\end{equation}
where $\gamma_n > 0$ is a regularization parameter. It is easy to verify that {the objective function and constraints in Eq.} (\ref{sp_est_pp_L}) are convex in $\boldsymbol{\theta}$ and hence $\widehat{\boldsymbol{\theta}}_{\sf sp}$ can be obtained using standard numerical solvers. Note that the solution to (\ref{sp_est_pp_L}) might not be unique. However, we will provide error bounds that hold for all possible solutions of (\ref{sp_est_pp_L}), with high probability.

It is known that ML estimates are asymptotically unbiased under mild conditions, and with $p$ fixed, the solution converges to the true parameter vector as $n \rightarrow \infty$. However, it is not clear how fast the convergence rate is for finite $n$ or when $p$ is not fixed and is allowed to scale with $n$. This makes the analysis of ML estimators, and in general regularized M-estimators, very challenging \cite{Negahban}. Nevertheless, such an analysis has significant practical implications, as it will reveal sufficient conditions on $n$ with respect to $p$ as well as a criterion to choose $\gamma_n$, which result in a stable estimation of $\boldsymbol{\theta}$. Finally, note that we are fixing the ambient dimension $p$ throughout the analysis. In practice, the history dependence is typically negligible beyond a certain lag and hence for a large enough $p$, {GLMs} fit the data very well.

\subsection{Greedy Estimation}
Although there exist fast solvers to convex problems of the type given by Eq. (\ref{sp_est_pp_L}), these algorithms are polynomial time in $n$ and $p$, and may not scale well with high-dimensional data.  This motivates us to consider greedy solutions for the estimation of $\boldsymbol{\theta}$. In particular, we will consider a generalization of the Orthogonal Matching Pursuit (OMP) \cite{zhang_omp,OMP} for general convex cost functions. A flowchart of this algorithm is given in Table \ref{tab:1}, which we denote by the Point Process Orthogonal Matching Pursuit (POMP) algorithm. At each iteration, the component in which the objective function has the largest deviation is chosen and added to the current support. The algorithm proceeds for a total of $s^{\star}$ steps, resulting in an estimate with $s^\star$ components.

The main idea behind the generalized OMP is in the greedy selection stage, where the absolute value of the gradient of the cost function at the current solution is considered as the selection metric. Consider an estimate $\widehat{\boldsymbol{\theta}}^{(k-1)}$ at the $(k-1)$-st stage of the generalized OMP for a quadratic cost function of the form $\| \mathbf{b} - \mathbf{A} \boldsymbol{\theta}\|_2^2$, with $\mathbf{b}$ and $\mathbf{A}$ denoting the observation vector and covariates matrix, respectively. Then, the gradient takes the form $\mathbf{A}' (\mathbf{b} - \mathbf{A} \widehat{\boldsymbol{\theta}}^{(k-1)})$ which is exactly the correlation vector between the residual error and the columns of $\mathbf{A}$ as in the original OMP algorithm. 

%\floatsetup[table]{capposition=top}
\begin{table}
\centering
\framebox{$\begin{array}{l}
\text{Input: } \mathfrak{L}(\boldsymbol{\theta}) , s^\star\\
\text{Output: } \widehat{\boldsymbol{\theta}}_{\sf POMP}^{(s^\star)}\\
\text{Initialization:}\Big\{\begin{array}{l}
\text{Start with the index set } S^{(0)}=\emptyset\\
\text{and the initial estimate }\widehat{\boldsymbol{\theta}}^{(0)}_{{\sf POMP}} = 0
\end{array}\\
\textbf{for } k=1,2,\cdots,s^\star\\
\text{  }\begin{array}{l}
j = \argmax \limits_i \left| \left( \nabla \mathfrak{L} \; \left(\widehat{\boldsymbol{\theta}}_{{\sf POMP}}^{(k-1)}\right) \right)_i\right|\\
S^{(k)}=S^{(k-1)}\cup \{j\}\\
\widehat{\boldsymbol{\theta}}_{{\sf POMP}}^{(k)} = \argmin \limits_{\support ({\boldsymbol{\theta}}) \subset S^{(k)}} \mathfrak{L}(\boldsymbol{\theta})
\end{array}\\
\textbf{end }\\
\end{array}$
}
\caption{\small{Point Process Orthogonal Matching Pursuit (POMP)}}
\label{tab:1}
\end{table}

\subsection{Theoretical Guarantees}

Recall that the parameter vector $\boldsymbol{\theta} \in \mathbb{R}^p$ is assumed to be $(s,\xi)$-compressible, so that $\sigma_s(\boldsymbol{\theta}) = \|\boldsymbol{\theta}-\boldsymbol{\theta}_S\|_1 = \mathcal{O} (s^{1-\frac{1}{\xi}})$, and the observed data are given by the vector ${\mathbf{x}_{-p+1}^n} \in \{0,1\}^{n+p-1}$, all in the regime of $s, n \ll p$. In the remainder of this chapter, we assume that $\boldsymbol{\theta} \in \boldsymbol{\Theta}$. The main theoretical result regarding the performance of the $\ell_1$-regularized ML estimator is given by the following theorem:
\begin{thm}
\label{negahban_lambda}
If $\sigma_s(\boldsymbol{\theta}) = \mathcal{O}(\sqrt{s})$, there exist constants $d_1,d_2,d_3$ and $d_4$ such that for $n>d_1s^{2/3}p^{2/3} \log p$ and a choice of $\gamma_n=d_2 \sqrt{\frac{\log p}{n}}$, any solution $\widehat{\boldsymbol{\theta}}_{{\sf sp}}$ to (\ref{sp_est_pp_L}) satisfies the bound
\begin{equation}
\label{thm2}
\left \|\widehat{\boldsymbol{\theta}}_{{\sf sp}}-\boldsymbol{\theta}\right \|_2 \leq d_3 \sqrt{\frac{s \log p}{n}}+ \sqrt{d_3\sigma_s(\boldsymbol{\theta})}\sqrt[4]{\frac{\log p}{n}},
\end{equation}
with probability greater than $1-\mathcal{O}\left(\frac{1}{n^{d_4}}\right)$.
\end{thm}

Similarly, the following theorem characterizes the performance bounds for the POMP estimate:

\begin{thm}
\label{thm_OMP}
If $\boldsymbol{\theta}$ is $(s,\xi)$-compressible for some $\xi < 1/2$, there exist constants $d_1',d_2',d_3'$ and $d_4'$ such that for $n>d_1's^{2/3} p^{2/3} \left(\log s \right)^{2/3} \log p$, the POMP estimate satisfies the bound
\begin{equation}
\label{thm2_OMP}
\left \|\widehat{\boldsymbol{\theta}}_{{\sf POMP}}-\boldsymbol{\theta}\right\|_2 \leq d_2' \sqrt{\frac{s \log s \log p}{n}} + d_3' \frac{\log s}{s^{{\frac{1}{\xi}-2}}}
\end{equation}
after $s^\star =\mathcal{O}(s \log s)$ iterations with probability greater than $1-\mathcal{O}\left(\frac{1}{n^{d'_4}}\right)$.
\end{thm}

{Full proofs of Theorems \ref{negahban_lambda} and \ref{thm_OMP} are given in Appendix \ref{appprf}.}

\noindent  \textit{\textbf{Remarks.}} An immediate comparison of the sufficient condition $n = \mathcal{O}(s^{2/3} p^{2/3} \log p)$ of Theorem \ref{negahban_lambda} with those of \cite{Negahban} for GLM models with i.i.d. covariates given by $n = \mathcal{O}(s \log p)$ reveals that a loss of order $\mathcal{O}(p^{2/3} s^{-1/3})$ is incurred due to the inter-dependence of the covariates. However, the sample space of $n$ i.i.d. covariates is $np$-dimensional, whereas in our problem the sample space is only $(n+p)$-dimensional. Hence, the aforementioned loss can be viewed as the price of self-averaging of the process accounting for the low-dimensional nature of the covariate sample space. \textcolor{black}{To the best of our knowledge, the dominant loss of $\mathcal{O}(p^{2/3})$ in both theorems does not seem to be significantly improvable, as self-exciting processes are known to converge quite slowly to their ergodic state \cite{reynaud2007some}. {On a related note, the analysis of the sampling requirements of linear AR models reveals a loss of {$\mathcal{O}(p^{1/2})$} in the number of measurements \cite{arpaper}.}}

The sufficient condition of Theorem \ref{thm_OMP} given by $n = \mathcal{O}(s^{2/3} p^{2/3} \left(\log s \right)^{2/3} \log p)$ implies an extra loss of $\left(\log s \right)^{2/3}$ due to the greedy nature of the solution. Theorem \ref{thm_OMP} also requires a high compressibility level of the parameter vector $\boldsymbol{\theta}$  ($\xi < 1/2$), whereas Theorem \ref{negahban_lambda} does not impose any extra restrictions on $\xi \in (0,1)$. Intuitively speaking, this comparison reveals the trade-off between computational complexity and compressibility requirements for convex optimization vs. greedy techniques, which is well-known for linear models \cite{bruckstein2009sparse}.

The constants $d_i, d'_i$, $i=1,\cdots,4$, $\alpha$ and $\beta$ are explicitly given in the proof of the theorems in Appendix \ref{appprf}. As for a typical numerical example, for $\pi_\min = 0.01$ and $\pi_\max=0.49$, the constants of Theorem \ref{negahban_lambda} can be chosen as $d_1 \approx 10^3, d_2 = 50, d_3 \approx 10^4$ and $d_4 = 4$. 

The main ingredient in the proofs of Theorems \ref{negahban_lambda} and \ref{thm_OMP} is inspired by the beautiful treatment of Negahban et al. in \cite{Negahban} in establishing the notion of Restricted Strong Convexity (RSC).  The major technical challenge for the \textcolor{black}{canonical self-exciting process}, as opposed to the GLM models with i.i.d. covariates in \cite{Negahban}, lies in the fact that the covariates are highly inter-dependent as they are formed by the history of the process. Hence, it is not straightforward to establish RSC with high probability, as the large deviation techniques used for i.i.d. random vectors {do} not hold. We establish the RSC for the \textcolor{black}{canonical self-exciting process} in two steps (see Lemma \ref{lemma1} in Appendix \ref{appprf}). First, we show that RSC holds for the expected value of the negative log-likelihood $\mathbb{E} [ \mathfrak{L}(\boldsymbol{\theta})]$, and then by invoking results on concentration of dependent random variables show that the negative log-likelihood $\mathfrak{L}(\boldsymbol{\theta})$ resides in a sufficiently small neighborhood of $\mathbb{E} [ \mathfrak{L}(\boldsymbol{\theta})]$ with high probability, and hence satisfies the RSC.

The remainder of the proof of Theorem \ref{negahban_lambda} establishes that upon satisfying the RSC, the estimation error can be suitably bounded (Proposition \ref{negahban_thm}, Appendix \ref{appprf}). Similarly, Theorem 2 is proven using the RSC of the {canonical self-exciting process} together with the results adopted from \cite{zhang_omp} on the performance of OMP for convex cost functions (Proposition \ref{prop_omp}, Appendix \ref{appprf}).

\medskip

\noindent {\textbf{Extensions.}} For simplicity and clarity of presentation, we have opted to present the proofs for the case of known $\mu$ and for the {canonical self-exciting process} as a canonical {GLM.} The following corollary extends our results to the case of unknown $\mu$.

\begin{corollary}\label{cor:1} The claims of Theorems \ref{negahban_lambda} and \ref{thm_OMP} hold when $\mu$ is not known, except for possibly slightly different constants.
\end{corollary}
\begin{proof}
The proof is given in Appendix \ref{appprf}.
\end{proof}

The {canonical self-exciting process} can be generalized to a larger class of {GLMs} by generalization of its spiking probability function. In a more general form we can consider a spiking probability function given by
\begin{equation*}
\label{nhp}
\lambda_i = \phi\left(\mu+\boldsymbol{\theta}'{\mathbf{x}_{i-p}^{i-1}}\right),
\end{equation*}
where $\phi(\cdot)$ is a possibly nonlinear function for which $0 < \lambda_i < 1$. In their continuous form, such processes are referred to as the \textit{nonlinear Hawkes process} \cite{nlhp}. Two of the commonly-used models in neural data analysis are the log-link and logistic-link models. Our prior numerical studies in \cite{kazemipour} revealed a similar performance improvement of  the $\ell_1$-regularized ML and the greedy solution over the ML estimate for the log-link model. Stationarity of these discrete processes can be proved similar to the canonical self-exciting process (see Appendix \ref{app:hawkes_psd}). The latter fact is key to extending our proofs to other models and is summarized by the following corollary:

\begin{corollary}\label{cor:2} Theorems \ref{negahban_lambda} and \ref{thm_OMP} hold when the spiking probability is given by $\lambda_i = \phi\left(\mu+\boldsymbol{\theta}'{\mathbf{x}_{i-p}^{i-1}}\right)$ for some continuous, bounded, convex and twice-differentiable function $\phi(\cdot)$ (e.g., $\phi(x) = \exp(x)$ or $\phi(x) = {\sf logit}^{-1}(x)$)  for which $0 < \lambda_i < 1/2$, except for  different constants.
\end{corollary}
\begin{proof}
The proof is given in Appendix \ref{appprf}.
\end{proof}

\section{Application to Simulated and Real Data}\label{simulations}

In this section, we study the performance of the conventional ML estimator, the $\ell_1$-regularized ML estimator, and the POMP estimator on simulated data as well as real spiking data recorded from the mouse's lateral geniculate nucleus (LGN) neurons and the ferret's retinal ganglion cells (RGC). {We have archived a MATLAB implementation of the estimators used in this chapter using the CVX package \cite{cvx} on the open source repository GitHub and made it publicly available \cite{hawkes_code}.}

\vspace{-3mm}
\subsection{Simulation Studies}

In order to simulate spiking data governed by the canonical self-exciting process, we sequentially generate spikes using (\ref{CIF_models}).   We have used $\mu = 0.1$, $\pi_\min = 0.01$, $\pi_\max = 0.49$, $p=1000$, $s=3$ and $n = 950$ for simulation purposes. Figure \ref{hawkes_fig} shows $500$ samples of the {canonical self-exciting process} generated using a history dependence parameter vector shown in Figure \ref{estimate_plot_synthetic}(a). The parameter vector $\boldsymbol{\theta}$ is compressible with a sparsity level of $s = 3$ and $\sigma_3(\boldsymbol{\theta}) = 0.05$. A value of $\gamma_n = 0.1$ is used to obtain the $\ell_1$-regularized ML estimate, which is slightly tuned around the theoretical estimate given by Theorem \ref{negahban_lambda}. Figures \ref{estimate_plot_synthetic}(b), \ref{estimate_plot_synthetic}(c), and \ref{estimate_plot_synthetic}(d) show the estimated history dependence parameter vectors using ML, $\ell_1$-regularized ML, and POMP, respectively. It can be readily visually observed that regularized ML and POMP significantly outperform the ML estimate in finding the correct values of $\boldsymbol{\theta}$. More specifically, the components at lags $405$ and $800$ (indicated by the gray arrows) are underestimated by the ML estimator, and their contribution is distributed among several falsely identified smaller lag components.

\begin{figure}[H]
\vspace{-2mm}
\centering
\includegraphics[width=.8\columnwidth]{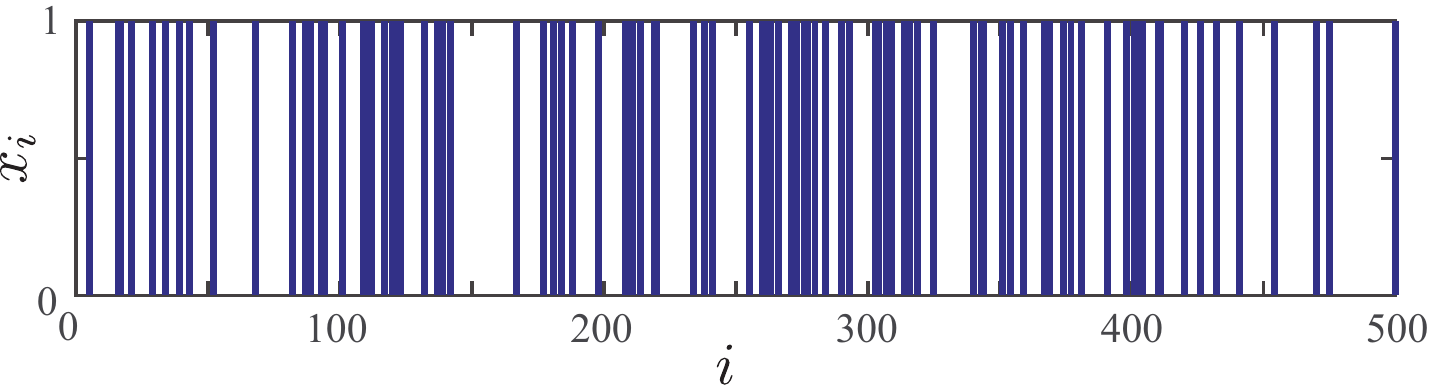}
\caption{{\small{A sample of the simulated {canonical self-exciting process}.}}}\label{hawkes_fig}
\vspace{-3mm}
\end{figure}

\begin{figure}[H]
\begin{center}
\includegraphics[width=.9\columnwidth]{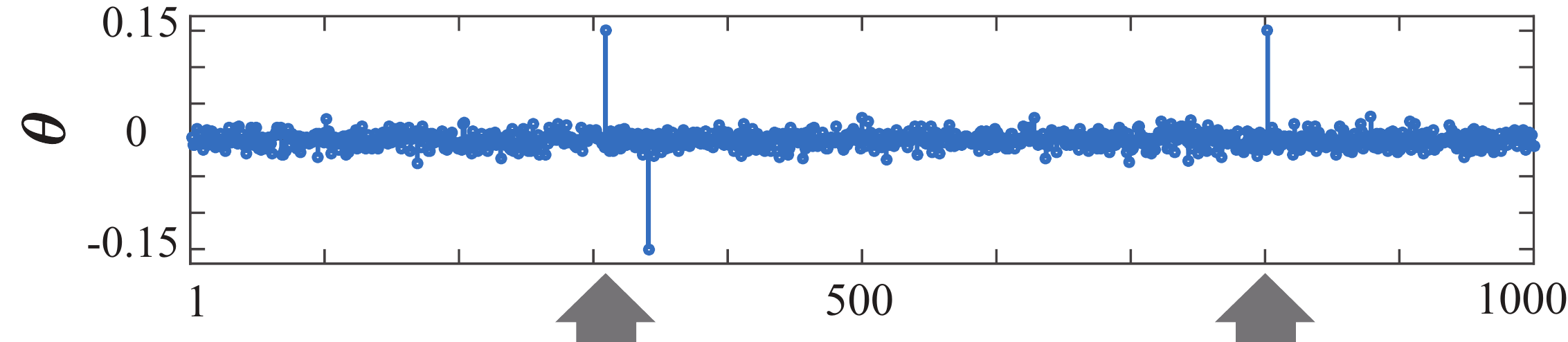}
\subcaption*{(a) True}
\vspace{3mm}
\includegraphics[width=.9\columnwidth]{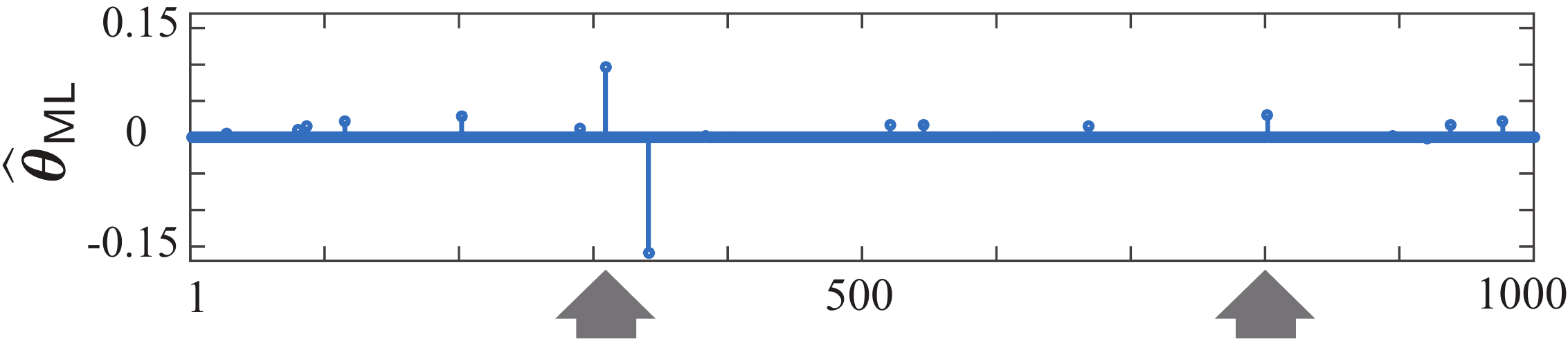}
\subcaption*{(b) ML}
\vspace{3mm}
\includegraphics[width=.9\columnwidth]{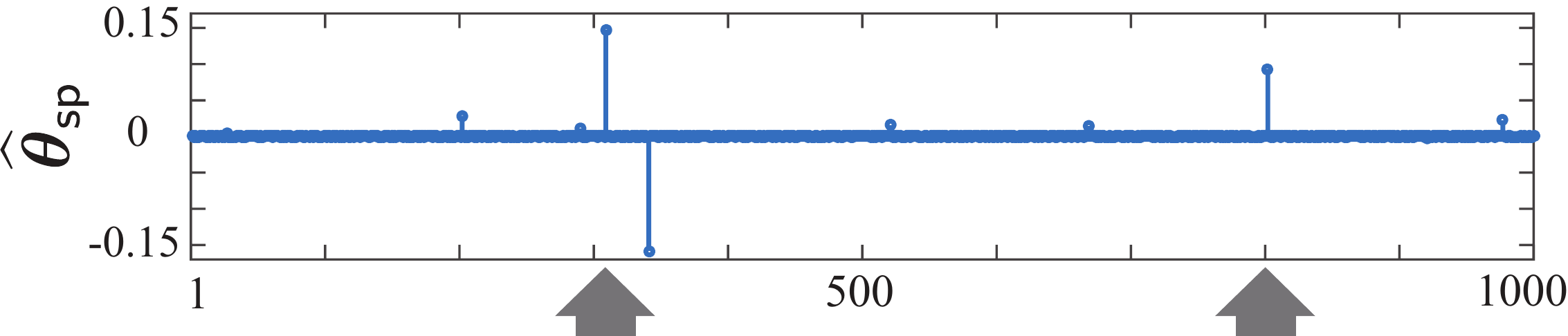}
\subcaption*{(c) $\ell_1$-regularized ML}
\vspace{3mm}
\includegraphics[width=.9\columnwidth]{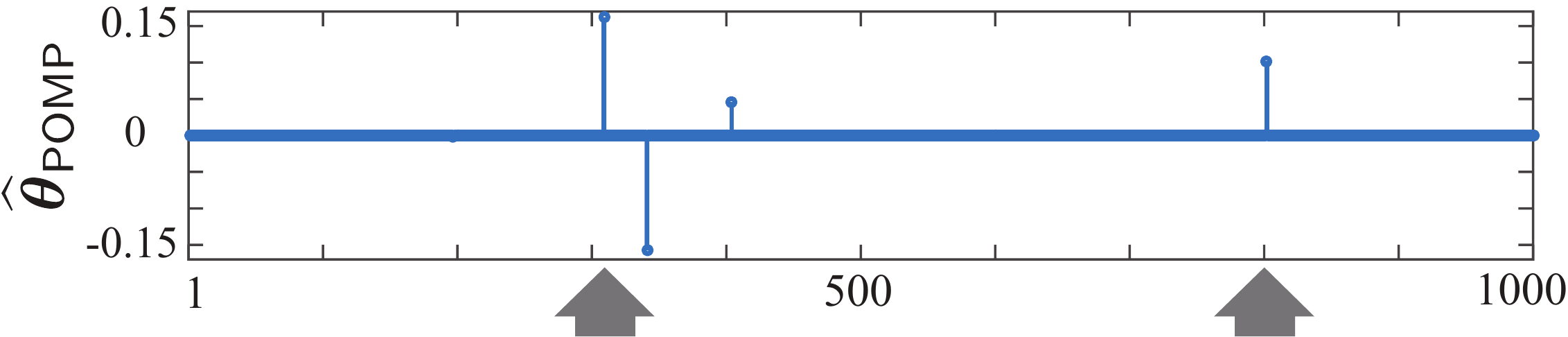}
\subcaption*{(d) POMP}
\end{center}
\vspace{-3mm}
\caption{{\small{(a) True parameters vs. (b) ML, (c) $\ell_1$-regularized ML, and (d) POMP estimates.}}}\label{estimate_plot_synthetic}
\vspace{-2mm}
\end{figure}

%\begin{figure}[H]
%\begin{center}
%\includegraphics[width=.95\columnwidth]{hawkes_estimate_true}
%\subcaption*{(a) True}
%\vspace{3mm}
%\includegraphics[width=.95\columnwidth]{hawkes_estimate_ml}
%\subcaption*{(b) ML}
%\vspace{3mm}
%\includegraphics[width=.95\columnwidth]{hawkes_estimate_sp}
%\subcaption*{(c) $\ell_1$-regularized ML}
%\vspace{3mm}
%\includegraphics[width=.95\columnwidth]{hawkes_estimate_pomp}
%\subcaption*{(d) POMP}
%\end{center}
%\vspace{-3mm}
%\caption{(a) True parameters vs. (b) ML, (c) $\ell_1$-regularized ML, and (d) POMP estimates.}\label{estimate_plot}
%\vspace{-5mm}
%\end{figure}

\begin{figure}[H]
\centering
\includegraphics[width=.9\columnwidth]{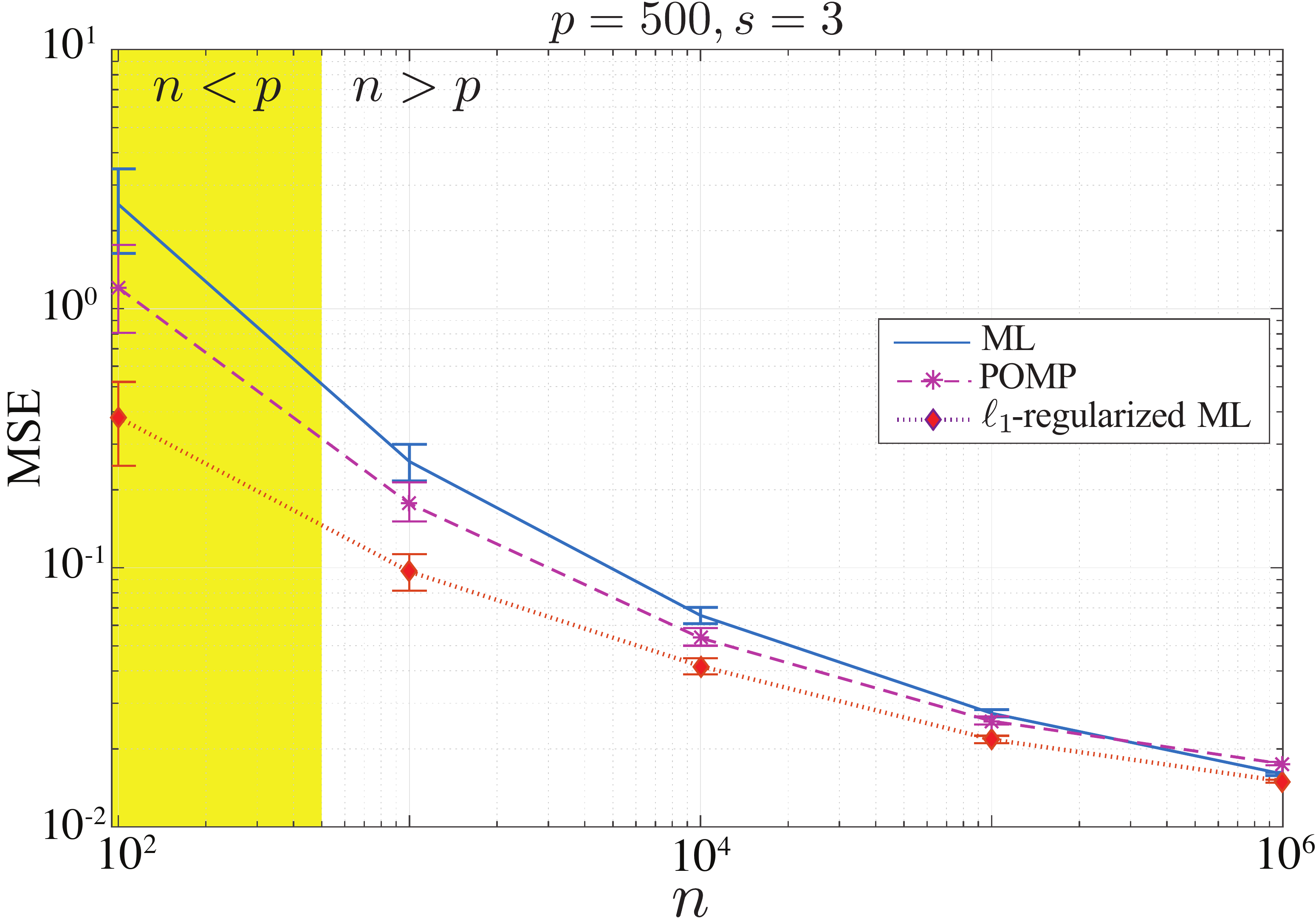}
\caption{{\small{MSE performance of the ML, $\ell_1$-regularized ML and POMP estimators.}}}\label{MSE_high_dim}
\vspace{-5mm}
\end{figure}

In order to quantify this performance gain, we repeated the same experiment {by generating  realizations corresponding to randomly chosen supports of size $s=3$ for $\boldsymbol{\theta}$ and spike trains of length $10^2 \le n \le 10^6$. In each case, the magnitudes of the components of $\boldsymbol{\theta}$ were chosen to satisfy the assumptions (\ref{eq:star}). For a given $\boldsymbol{\theta}$, the mean-square-error (MSE) of the estimate $\widehat{\boldsymbol{\theta}}$ is defined as $\widehat{\mathbb{E}}\{ \| \widehat{\boldsymbol{\theta}} - \boldsymbol{\theta} \|_2^2 \}$, where $\widehat{\mathbb{E}} \{ \cdot \}$ is the sample average over the realizations of the process.} Figure \ref{MSE_high_dim} shows the results of this simulation, where a similar systematic performance gain is observed. The left segment of the plot (shaded in yellow) and the right segment correspond to the compressive ($n < p$) and denoising ($n > p$) regimes, respectively. Error bars on the plot indicate 90\% quantiles of the MSE for this simulation obtained by multiple realizations. As it can be inferred from Figure \ref{MSE_high_dim}, the $\ell_1$-regularized ML and POMP have a systematic performance gain over the ML estimate {in the compressive regime, where $n \ll p$}, with the former outperforming the rest. In the denoising regime, the performance of the $\ell_1$-regularized and ML become closer, while the POMP saturates to a higher MSE floor. The latter observation can be explained by the fact that the POMP can only estimate $s^{\star}$ components (including those of $\boldsymbol{\theta}_s$), and fails to capture the $(p - s^{\star})$ compressible components. This results in an MSE floor above that obtained by ML, for large values of $n$.

\begin{figure}[H]
\begin{center}
\includegraphics[width=.8\columnwidth]{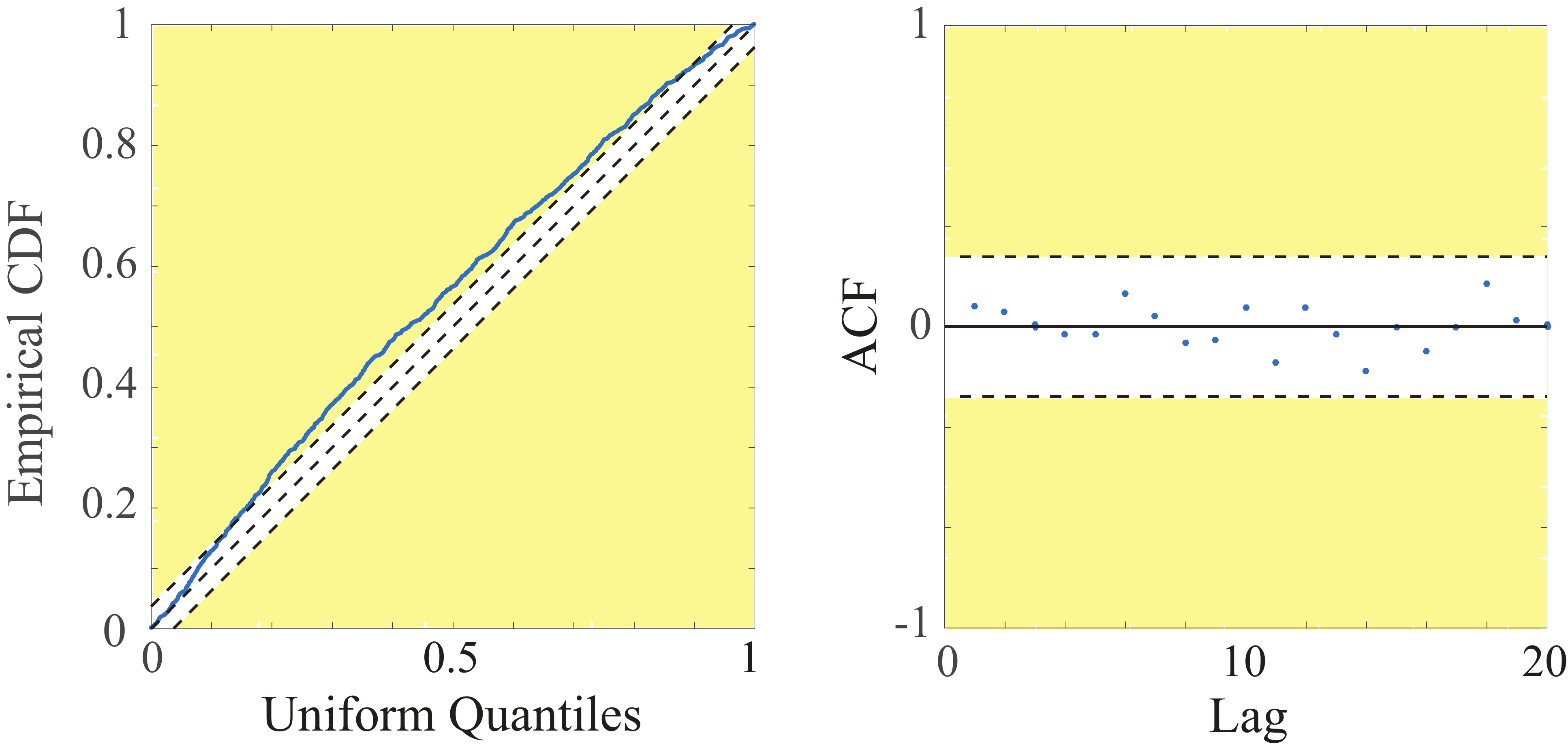}
\subcaption*{(a) ML}
\vspace{3mm}
\includegraphics[width=.8\columnwidth]{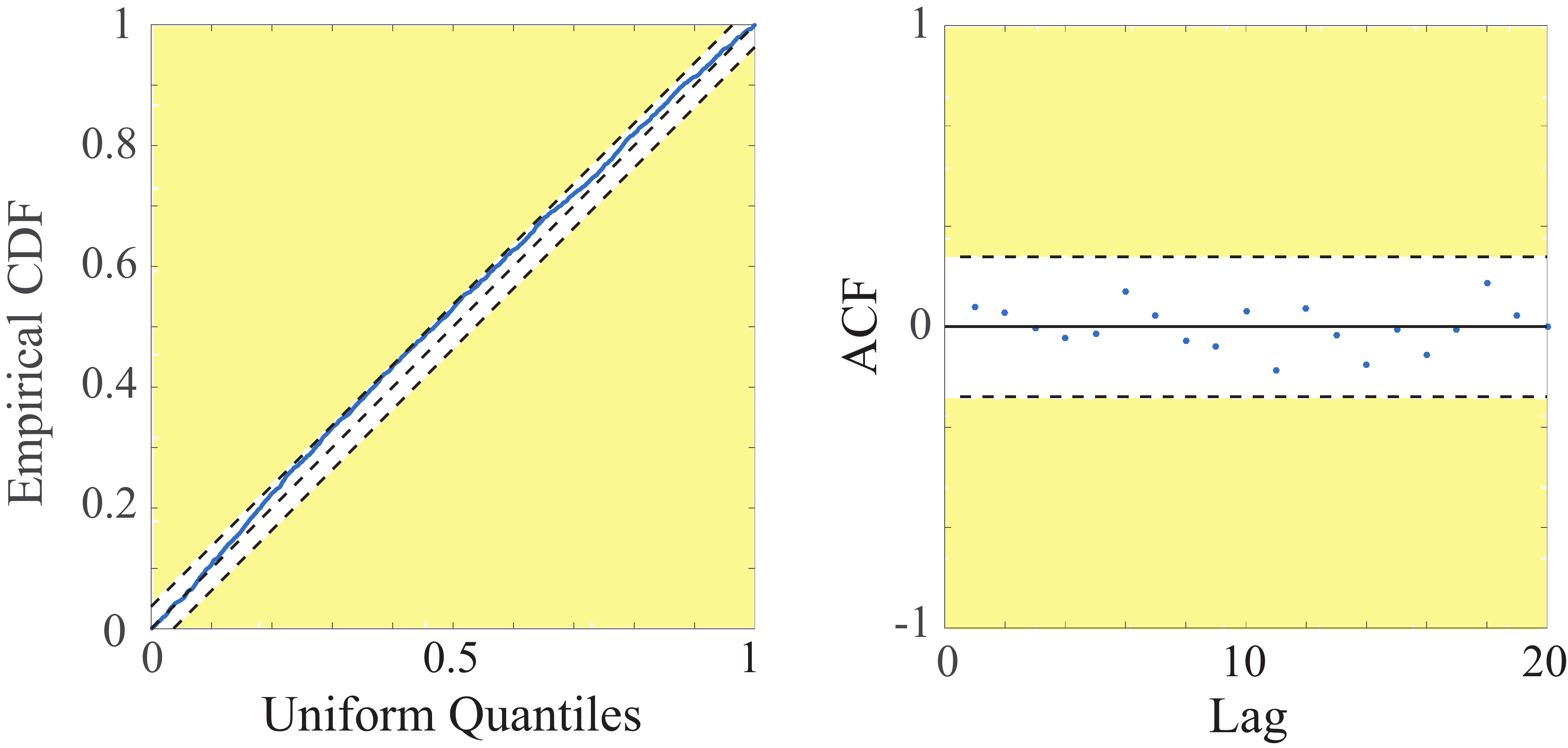}
\subcaption*{(b) $\ell_1$-regularized ML}
\vspace{3mm}
\includegraphics[width=.8\columnwidth]{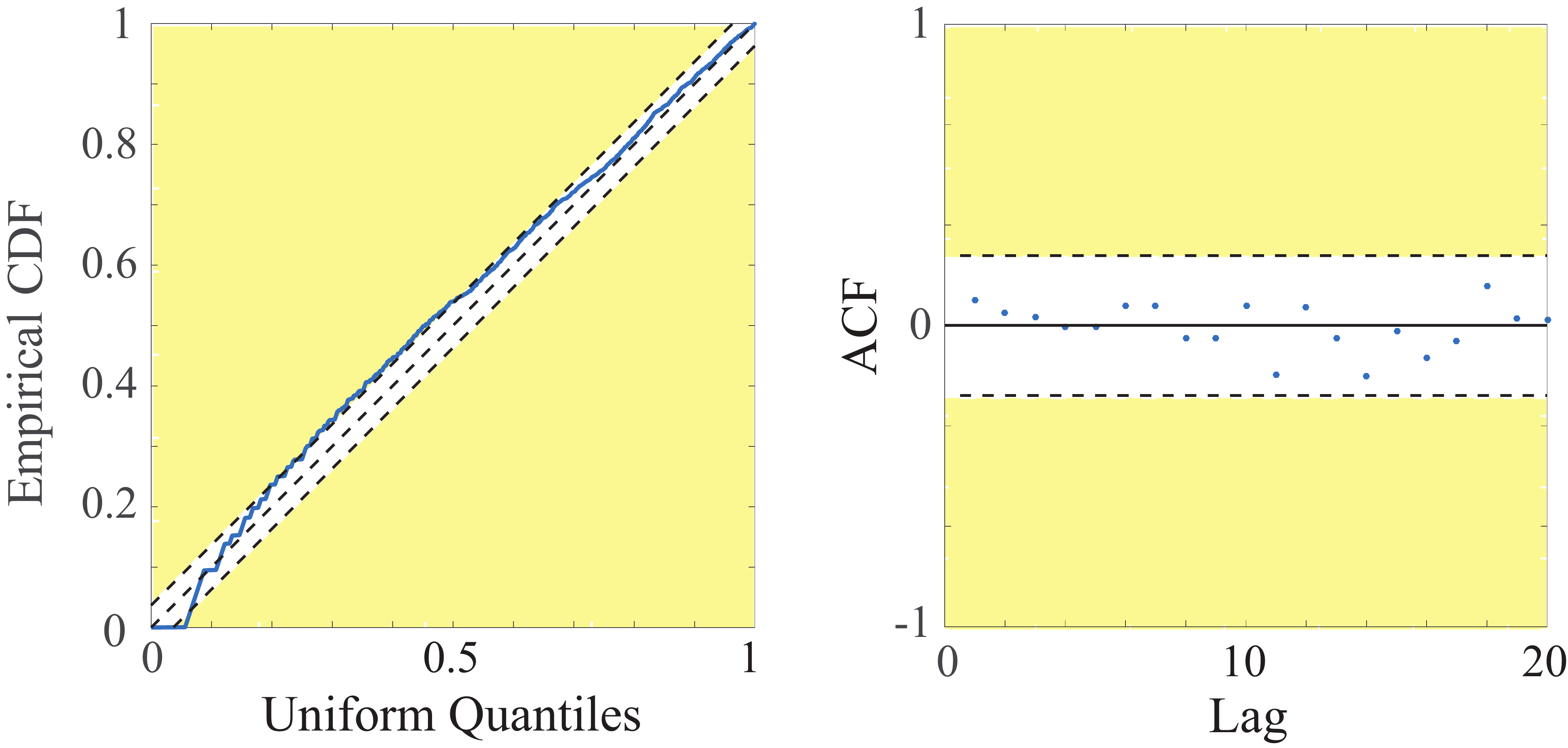}
\subcaption*{(c) POMP}
\end{center}
\vspace{-3mm}
\caption{{\small{KS and ACF tests at {$95\%$ confidence level}, for the ML, $\ell_1$-regularized ML and POMP estimates.}}}
\label{hawkes_ks_acf_synthetic}
\vspace{-5mm}
\end{figure}

{The MSE comparison in Figure \ref{MSE_high_dim}} requires one to know the true parameters. In practice, the true parameters are unknown, and statistical tests are typically used to assess the goodness-of-fit of the estimates to the observed data. We use the Kolmogorov-Smirnov (KS) test and the autocorrelation function (ACF) test to assess the goodness-of-fit. These tests are based on the time-rescaling theorem for point processes \cite{time_rescaling}, which states that if the time axis is rescaled using the {estimated} conditional intensity function of {the} inhomogeneous Poisson process, the resulting point process is a homogeneous Poisson process with unit rate. Thereby, one can test for the validity of the time-rescaling theorem via two statistical tests: the KS test reveals how close the empirical quantiles of the time-rescaled point process to the true quantiles of a unit rate Poisson process, and the ACF test reveals how close the ISI {distribution} of the time-rescaled process is to the true ISI distribution of a unit rate Poisson process. Details of these tests are given in Appendix \ref{appks}.

%%%%%%%%%%%%%%%%%%%%%%%%%%

%\begin{figure}[H]
%\begin{center}
%\includegraphics[width=.95\columnwidth]{sim_ks_ml}
%\subcaption*{(a) ML}
%\vspace{3mm}
%\includegraphics[width=.95\columnwidth]{sim_ks_sp}
%\subcaption*{(b) $\ell_1$-regularized ML}
%\vspace{3mm}
%\includegraphics[width=.95\columnwidth]{sim_ks_pomp}
%\subcaption*{(c) POMP}
%\end{center}
%\vspace{-3mm}
%\caption{KS and ACF tests at $95\%$ and $99\%$ confidence levels, respectively, for the ML, $\ell_1$-regularized ML and POMP estimates.}
%\label{hawkes_ks_ml}
%\vspace{-5mm}
%\end{figure}

Figure \ref{hawkes_ks_acf_synthetic} shows the KS and ACF tests {at a $95\%$ confidence level} for the ML $\ell_1$-regularized ML, and the POMP estimates from Figure \ref{estimate_plot_synthetic}. The yellow shades mark the regions below the specified confidence levels. The ML estimate {fails to pass the KS test, while the regularized and POMP estimates pass both tests.}

\subsection{Application to Spontaneous Neuronal Spiking Activity}

\subsubsection{{\textbf{Background and motivation}}}
Early studies of spontaneous neuronal activity from the cat's cochlear nucleus \cite{gerstein1960approach} marked a significant breakthrough in computational neuroscience by going beyond the so-called Poisson hypothesis, by which single neurons were assumed to be firing according to homogeneous Poisson statistics. The diversity of the ISIs deduced from the spontaneous activity of the cochlear neurons led to the development of more sophisticated statistical models based on renewal process theory, resulting in the Gamma and inverse Gaussian ISI descriptions of spontaneous neuronal activity \cite{gerstein1964random,tuckwell2005introduction}. Due to the analytical difficulties involved in working with these models, their generalization to a broader range of spiking statistics is not straightforward. 

In light of the more recent discoveries on the role of spontaneous neuronal activity in brain development \cite{xu2011instructive,blankenship2010mechanisms}, its relation to functional architecture \cite{tsodyks1999linking}, and its functional significance in a variety of modalities including retinal \cite{xu2011instructive}, visual \cite{luck1997neural}, auditory \cite{tritsch2007origin}, hippocampal \cite{sombati1995recurrent}, cerebellar \cite{aizenman1999regulation}, and thalamic \cite{pinault1998intracellular} function, the modeling and analysis of this phenomenon has sparked a renewed interest among researchers in recent years. In particular, models based on GLMs have shown to overcome the analytical difficulties of the abovementioned models based on renewal theory, and have been successfuly used in relating the spontaneous neuronal activity to instrinsic and extrinsic neural covariates \cite{paninski2004maximum,paninski2007statistical,time_rescaling,barbieri2001construction} as well as inferring the functional connectivity of neuronal ensembles \cite{kim2011granger,brown_func_conn}. The above-mentioned results rely on the accuracy of the ML estimation of these models. In addition, the estimated parameters are typically sparse. Therefore, the $\ell_1$-regularized ML and POMP estimators are expected to offer a more robust alternative than the ML, especially under the limited observation setting.

{In order to evaluate the performance of these estimators on real data, in the remainder of this section we will compare the performance of the ML, $\ell_1$-regularized ML, and POMP estimators in modeling the spontaneous spiking activity recorded from two different types of neurons, namely the mouse's lateral geniculate nucleus and the ferret's retinal ganglion cells.}

In the following analysis, the regularization parameter $\gamma_n$ was chosen using a two-fold cross-validation refinement around the value obtained from our theoretical results. {The length of the history components $p$ was chosen by first selecting a large enough $p$ as an upper bound for the expected correlation length of neuronal spontaneous activity (estimated as $\sim 1.5~\text{s}$), followed by reducing $p$ to the point where an increase in the history length  does not result in significantly detected history components.}

\subsubsection{{\textbf{Application to LGN spiking activity}}}
\label{nsp}
 We first compare the performance of the estimators on the LGN neurons.  The LGN is part of the thalamus in the brain, which acts as a relay from the retina to the primary visual cortex \cite{thalamus}. The data {were} recorded at $1ms$ resolution from the mouse LGN neurons using single-unit recording \cite{scholl2013emergence}. We used about $5$ seconds of data from one neuron for the analysis. In order to capture the history dependence governing the spontaneous spiking activity of the LGN neuron, we model the spiking probability using the \textcolor{black}{canonical self-exciting process} model with $p=100$ ($\Delta = 1ms$). {Figure \ref{real_fig} shows the spiking data used in the analysis.}

\begin{figure}[H]
\centering
\includegraphics[width=.8\columnwidth]{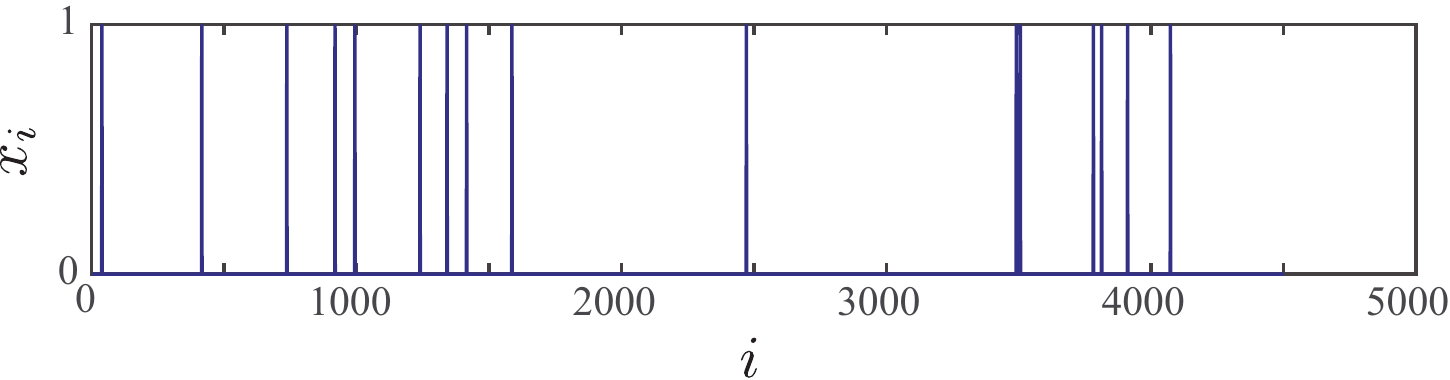}
\caption{\small{The LGN spiking data used in the analysis.}}\label{real_fig}
\vspace{-1mm}
\end{figure} 

\begin{figure}[H]
\begin{center}
\includegraphics[width=.8\columnwidth]{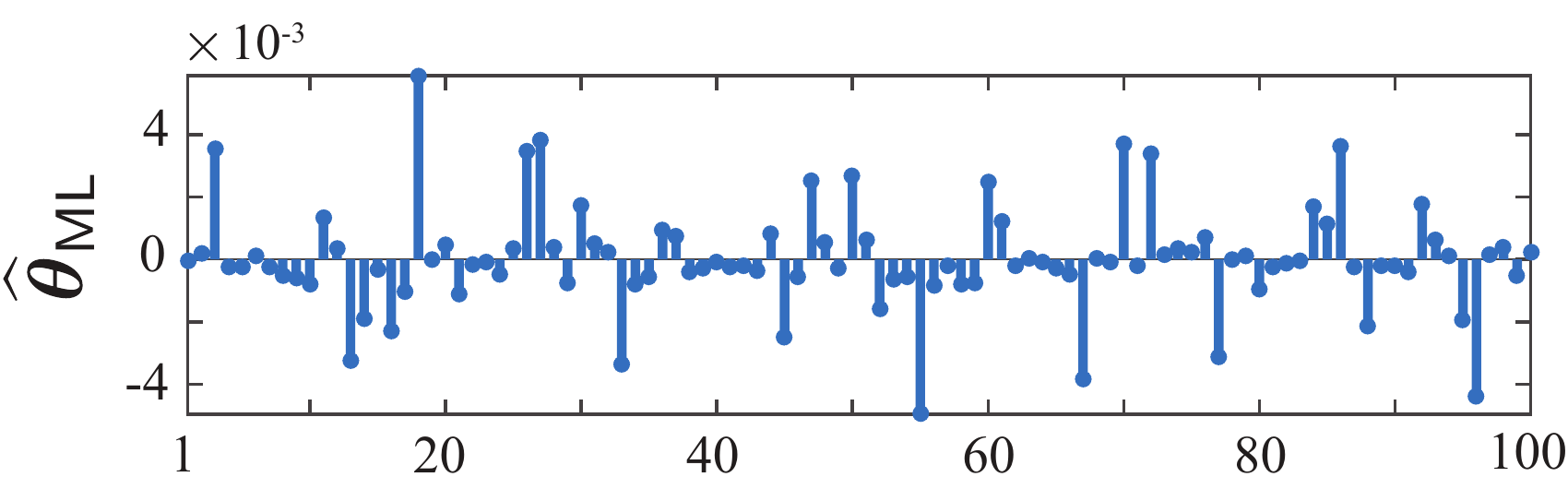}
\subcaption*{(a) ML}
\vspace{-1mm}
\includegraphics[width=.8\columnwidth]{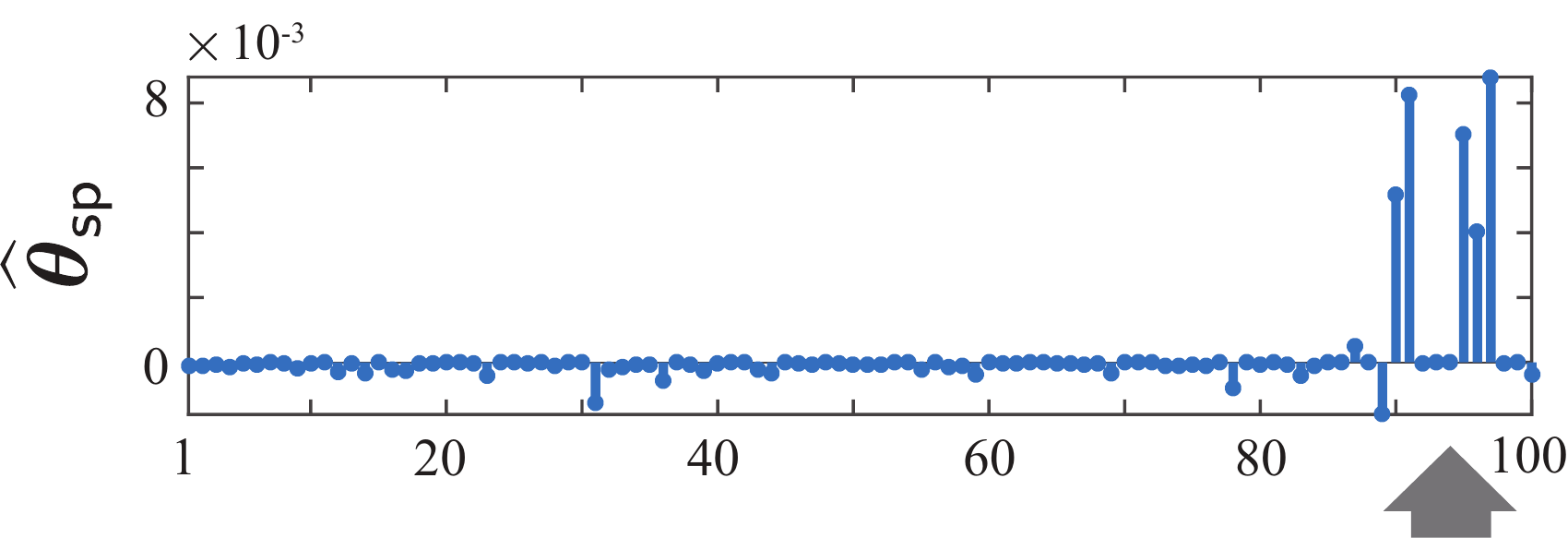}
\subcaption*{(b) $\ell_1$-regularized ML}
\vspace{3mm}
\includegraphics[width=.8\columnwidth]{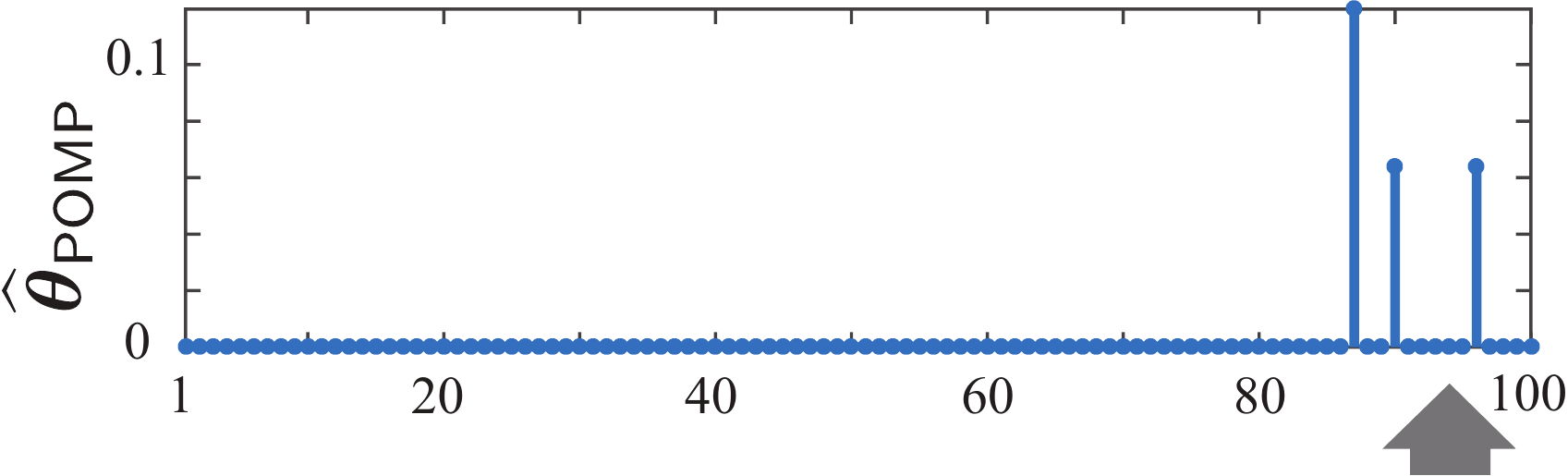}
\subcaption*{(c) POMP}
\end{center}
\vspace{-3mm}
\caption{\small{(a) ML, (b) $\ell_1$-regularized ML, and (c) POMP estimates of the LGN spiking parameters.}}
\label{lgn_mlvssp}
\vspace{-5mm}
\end{figure}

\begin{figure}[H]
\begin{center}
\includegraphics[width=.9\columnwidth]{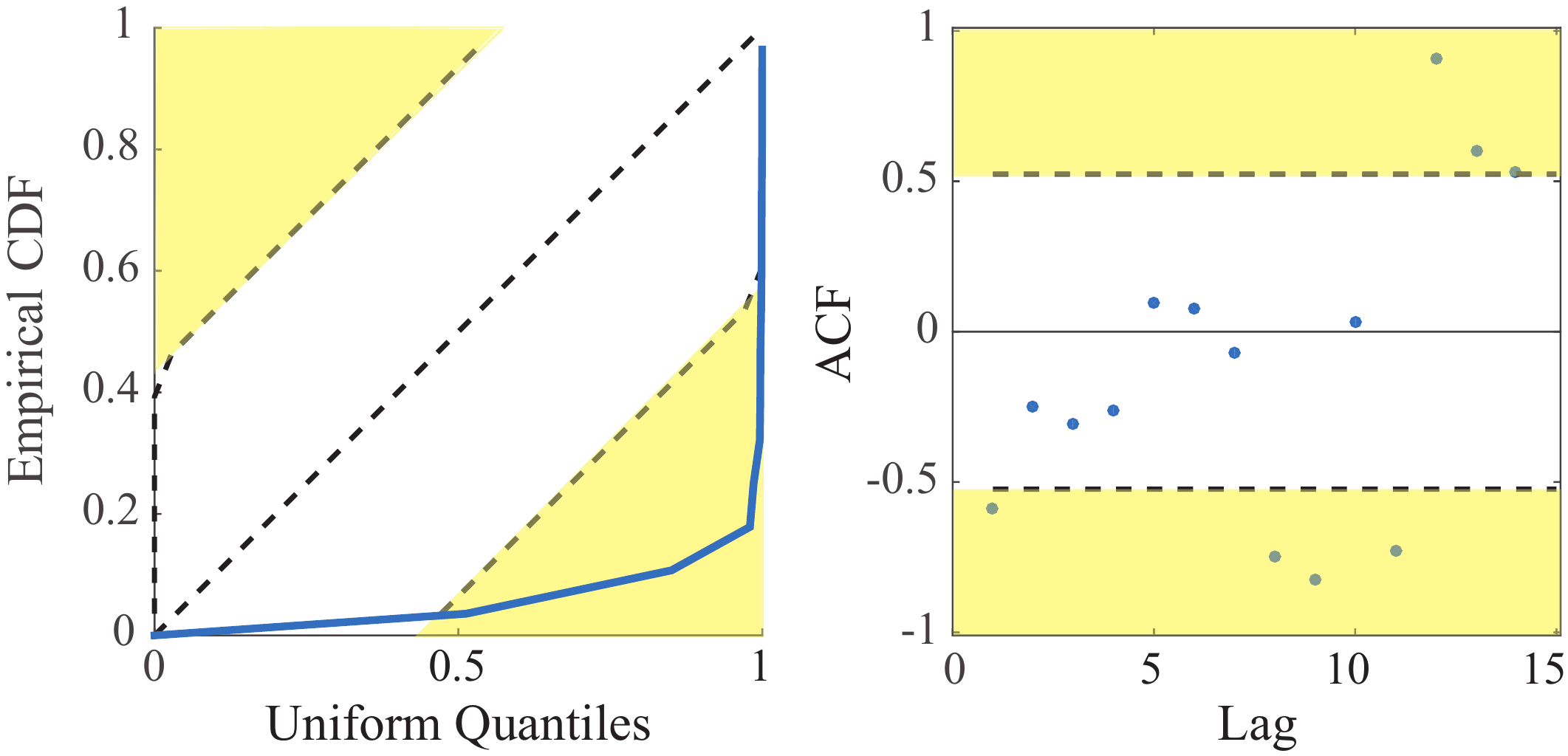}
\subcaption*{(a) ML}
\vspace{3mm}
\includegraphics[width=.9\columnwidth]{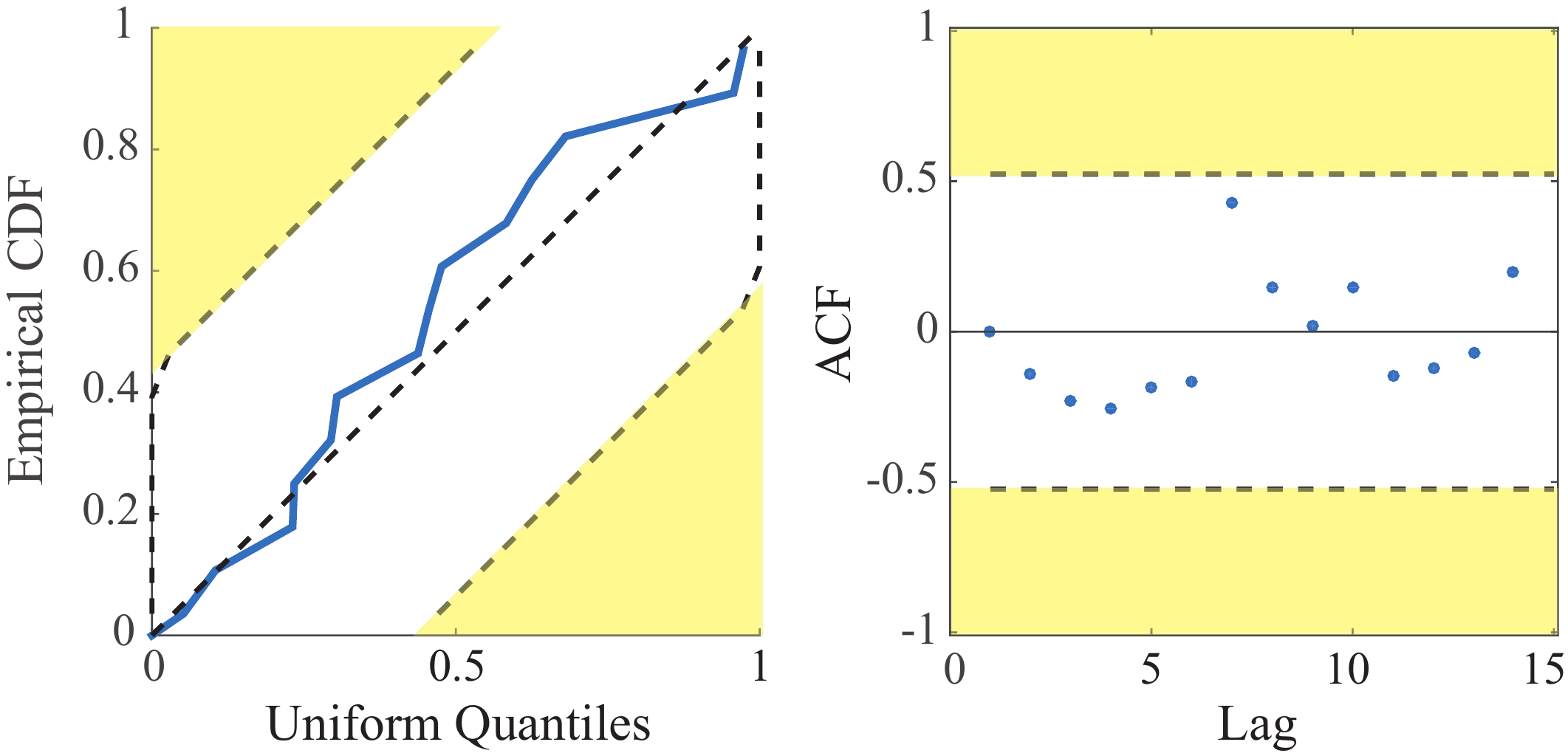}
\subcaption*{(b) $\ell_1$-regularized ML}
\vspace{3mm}
\includegraphics[width=.9\columnwidth]{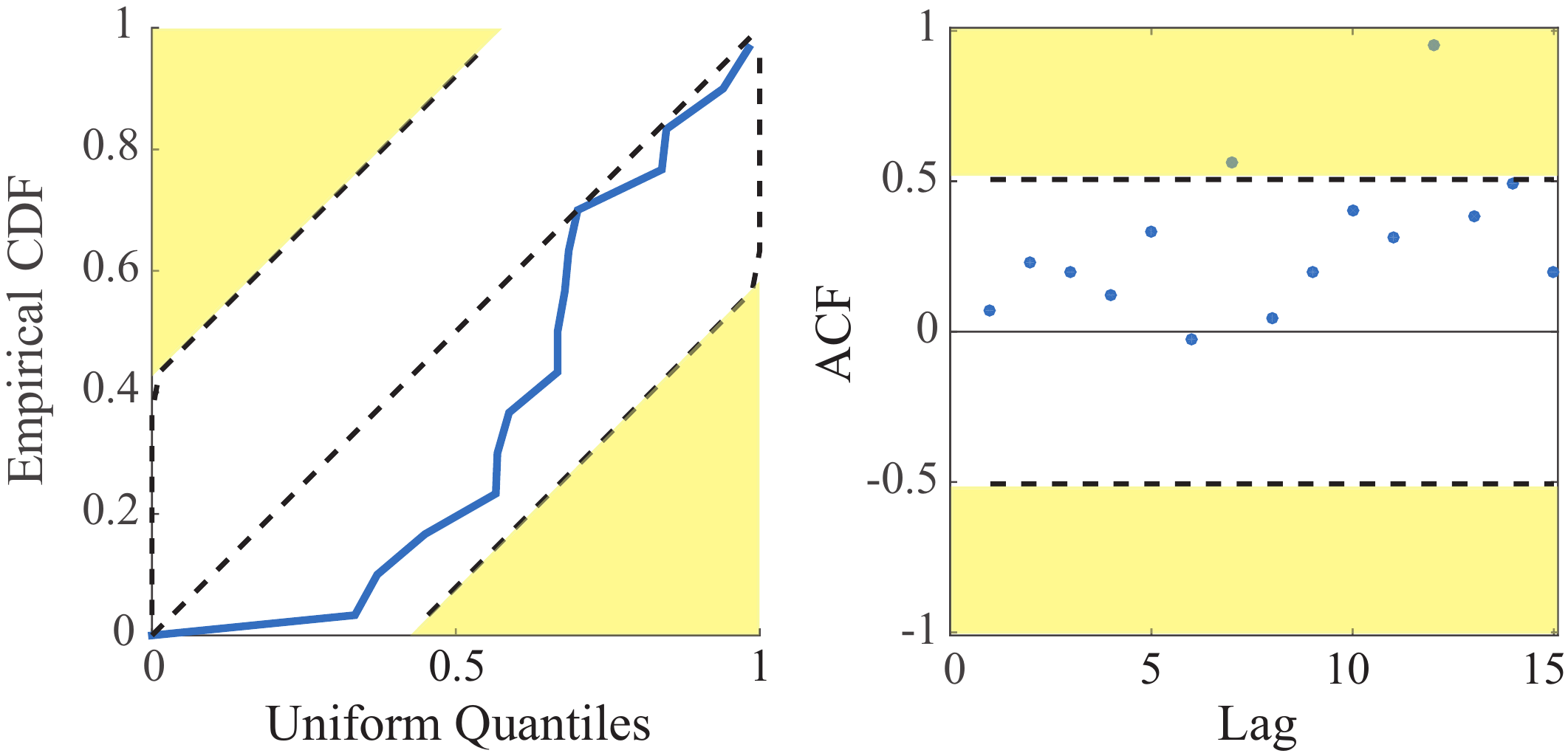}
\subcaption*{(c) POMP}
\end{center}
\vspace{-3mm}
\caption{\small{KS and ACF tests at $99\%$ confidence level, for the ML, $\ell_1$-regularized ML and POMP estimates.}}\label{lgn_ks_sp}
\vspace{-5mm}
\end{figure}

Figure \ref{lgn_mlvssp} shows the estimated history dependence parameter vectors using the three methods. Both the regularized ML (Figure \ref{lgn_mlvssp}(b)) and POMP (Figure \ref{lgn_mlvssp}(c)) estimates capture significant history dependence components around a lag of $90$--$95~\text{ms}$ (marked by the upward arrows). In \cite{Borowska}, an intrinsic neuronal oscillation frequency of around $10 \hertz$ has been reported in around $30\%$ of all classes of mouse retinal cells under experiment, using combined two-photon imaging and patch-clamp recording. Our results are indeed consistent with the above mentioned findings about the intrinsic spiking frequency of retinal neurons. \textcolor{black}{To see this, we consider the power spectral density of the \textcolor{black}{canonical self-exciting process} given by:
\begin{equation}
\label{bart_spec}
S(\omega) = \frac{1}{2\pi} \left( \pi_\star^2 \delta(\omega) + \frac{\pi_\star-\pi_\star^2}{\left(1-\mathbf{1}'\boldsymbol{\theta}\right)^2 \left|1 - \Theta(\omega)\right|^2} \right),
\end{equation}
where $\Theta(\omega)$ is the discrete-time Fourier transform of $\boldsymbol{\theta}$ and $\pi_\star = \mu/(1-\mathbf{1}'\boldsymbol{\theta})$ denotes the stationary distribution probability of spiking. The derivation of the power spectral density is given in Appendix \ref{app:hawkes_psd}. The power spectral density of the \textcolor{black}{canonical self-exciting process} resembles the Bartlett spectrum of the Hawkes process \cite{Bartlett1, Bartlett2, Hawkes_orig}, whose peaks correspond to the significant oscillatory components of the underlying process.} Our estimated parameter vectors $\boldsymbol{\theta}$ using the regularized ML and POMP have significant nonzero components around lags of $ 90 \le k \le 95$. As a result, $S(\omega)$ peaks at $\omega = \frac{2\pi}{k\Delta}$. Hence, $f=\frac{1}{k\Delta}$ is an estimate of the significant intrinsic frequency of the underlying self-exciting process. Using the estimated numerical values, the intrinsic frequency is around $10.5$--$11~\text{Hz}$, which is consistent with experimental findings of \cite{Borowska}. Compared to the method in \cite{Borowska}, our estimates are obtained using much shorter recordings of spiking activity and provide a principled framework to study the oscillatory behavior of LGN neurons using {sparse GLM estimation.}

Note that there is a difference in the orders of magnitudes of the POMP estimate compared to the ML and regularized ML estimates. This is due to the fact that the POMP estimate is exactly $s$-sparse, whereas the ML and regularized ML estimates consist of $p = 100$ non-zero values. In order to assess the goodness-of-fit of these estimates, we invoke the KS and ACF tests. Figure \ref{lgn_ks_sp} shows the corresponding KS and ACF test plots. As it is implied from Figure \ref{lgn_ks_sp}(a), the ML estimate fails both tests due to overfitting, whereas the regularized ML (Figure \ref{lgn_ks_sp}(b)) passes both tests at the specified confidence levels. The POMP estimate (Figure \ref{lgn_ks_sp}(c)), however, passes the KS test while marginally failing the ACF test. The latter observation implies that the seemingly negligible components of the parameter vector captured by the regularized ML estimate seem to be important in explaining the statistics of the observed data.

\subsubsection{{\textbf{Application to RGC spiking activity}}} \label{sec:rgc}
{We will next study the performance of the estimators on spiking} data recorded from the RGCs of neonatal and adult ferrets \cite{wong1993transient}. The retinal ganglion cells are located in the innermost layer of the retina. They integrate information from photoreceptors and project them into the brain \cite{bear2007neuroscience}. The data {were} recorded using a multi-electrode array from the ferret retina at $50~\mu s$ \cite{wong1993transient}. We used $5$ seconds of data from one neuron for the analysis (neuron 2, session 1, adult data set, CARMEN data base \cite{eglen2014data}). {Figure \ref{real_2_fig} shows a segment of the spiking data used in our analysis. The RGC activity in the adult ferret is characterized by bursts of activity with a mean firing rate of $9 \pm 7$~Hz, which are separated by $0.5$--$1~\text{s}$ intervals \cite{wong1993transient}.}

 {
\begin{figure}[H]
\centering
\includegraphics[width=.95\columnwidth]{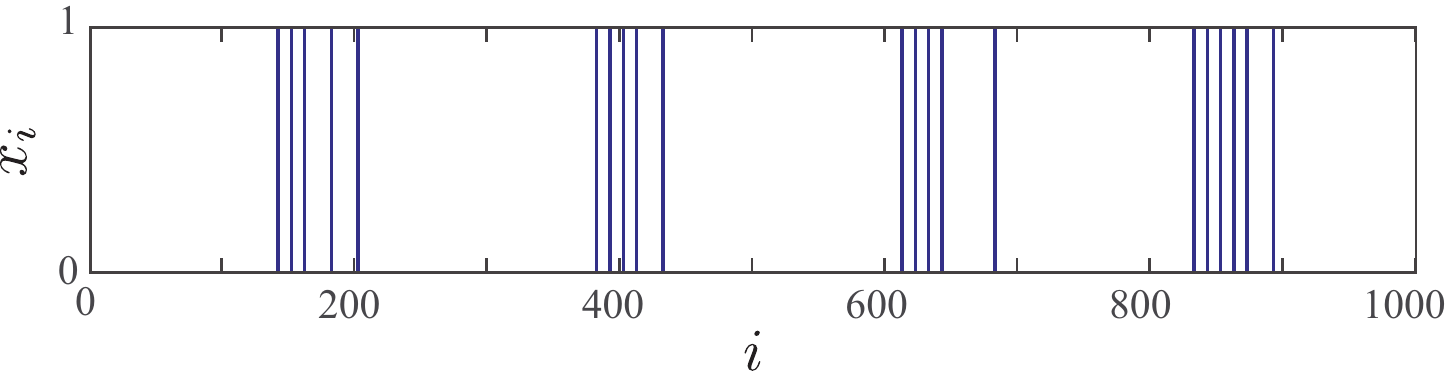}
\caption{\small{Segment of the RGC spiking data used in the analysis.}}\label{real_2_fig}
\end{figure} 

\begin{figure}[H]
\begin{center}
\includegraphics[width=.8\columnwidth]{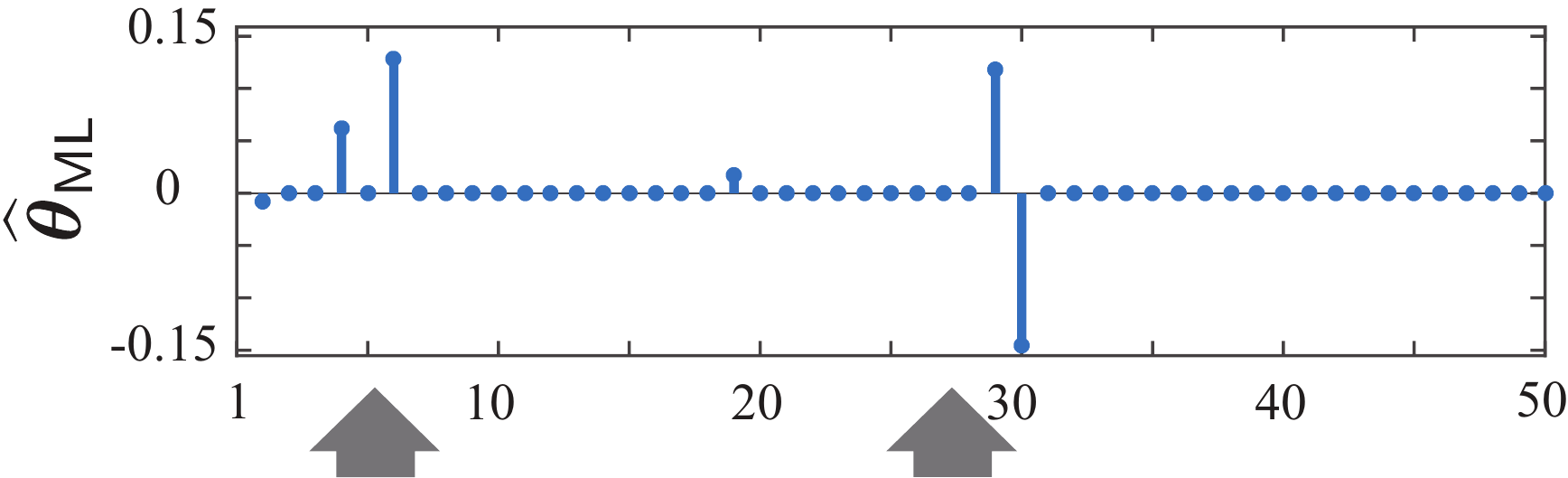}
\subcaption*{(a) ML}
\includegraphics[width=.8\columnwidth]{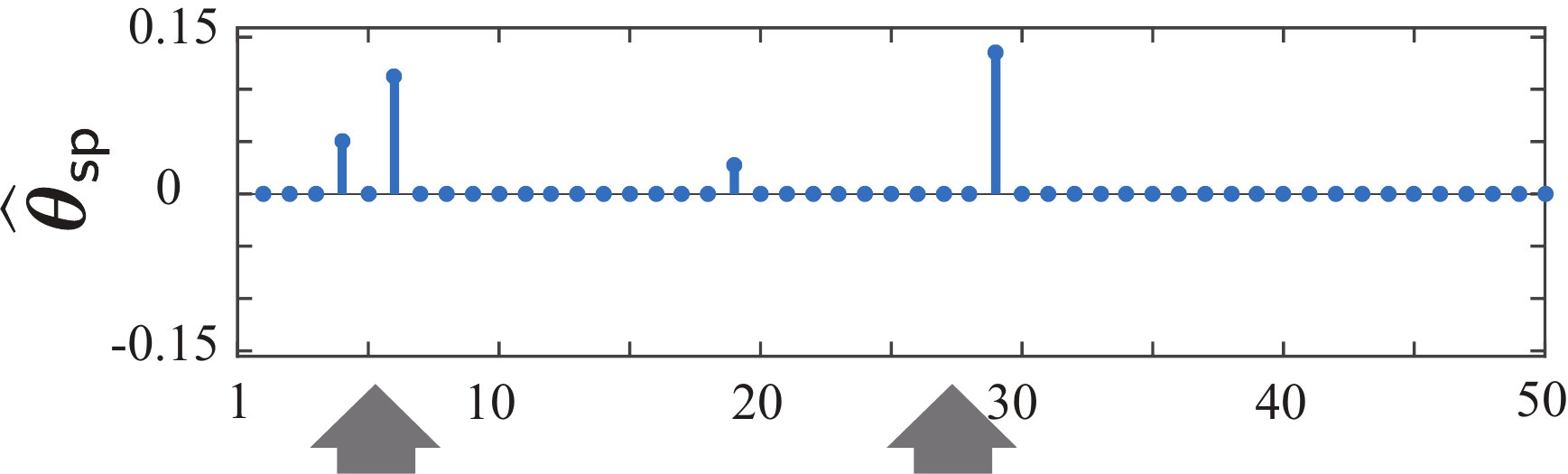}
\subcaption*{(b) $\ell_1$-regularized ML}
\includegraphics[width=.8\columnwidth]{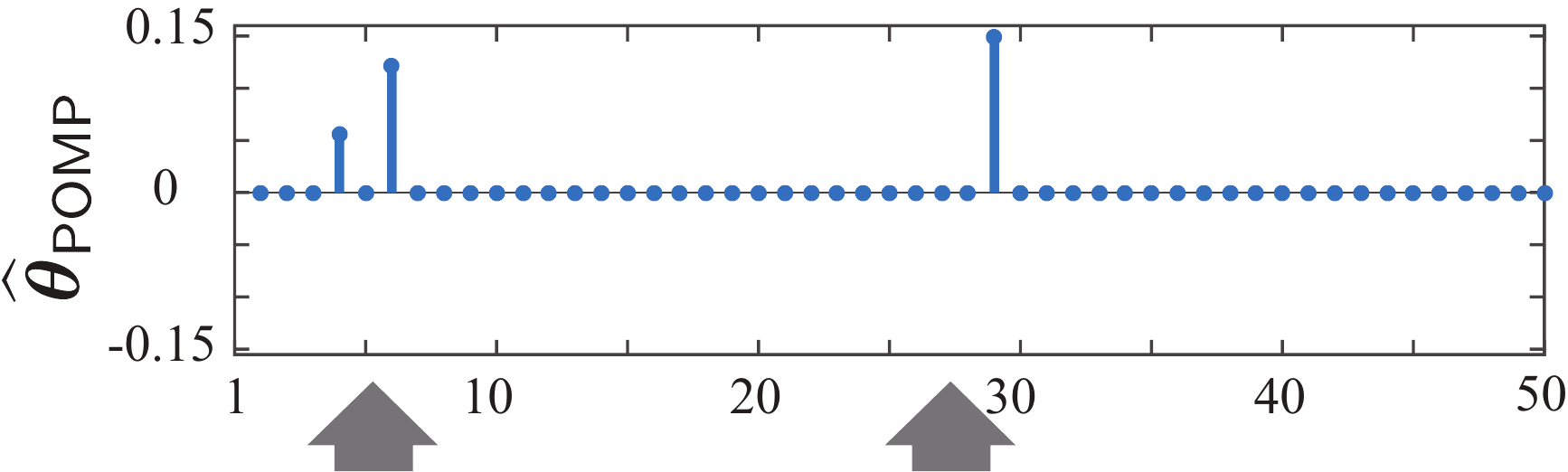}
\subcaption*{(c) POMP}
\end{center}
\caption{\small{(a) ML, (b) $\ell_1$-regularized ML, and (c) POMP estimates of the RGC spiking parameters {using the canonical self-exciting process model.}}}
\label{rgc_mlvssp_linear}
\end{figure}

In order to capture the history dependence governing the spontaneous spiking activity of the RGC neuron, we model the spiking probability using two different link models to further corroborate the generalization of our results to models beyond the canonical self-exciting process studied in this chapter. First, we consider the canonical self-exciting process model. We have chosen  $\pi_\max = 0.49$,  $p=50$ ($\Delta = 25~\text{ms}$) and $s_\star = 3$. The baseline parameter $\mu$ is estimated from the data and is set to be equal to half of empirical mean firing rate of the neuron.

\begin{figure}[H]
\begin{center}
\includegraphics[width=.8\columnwidth]{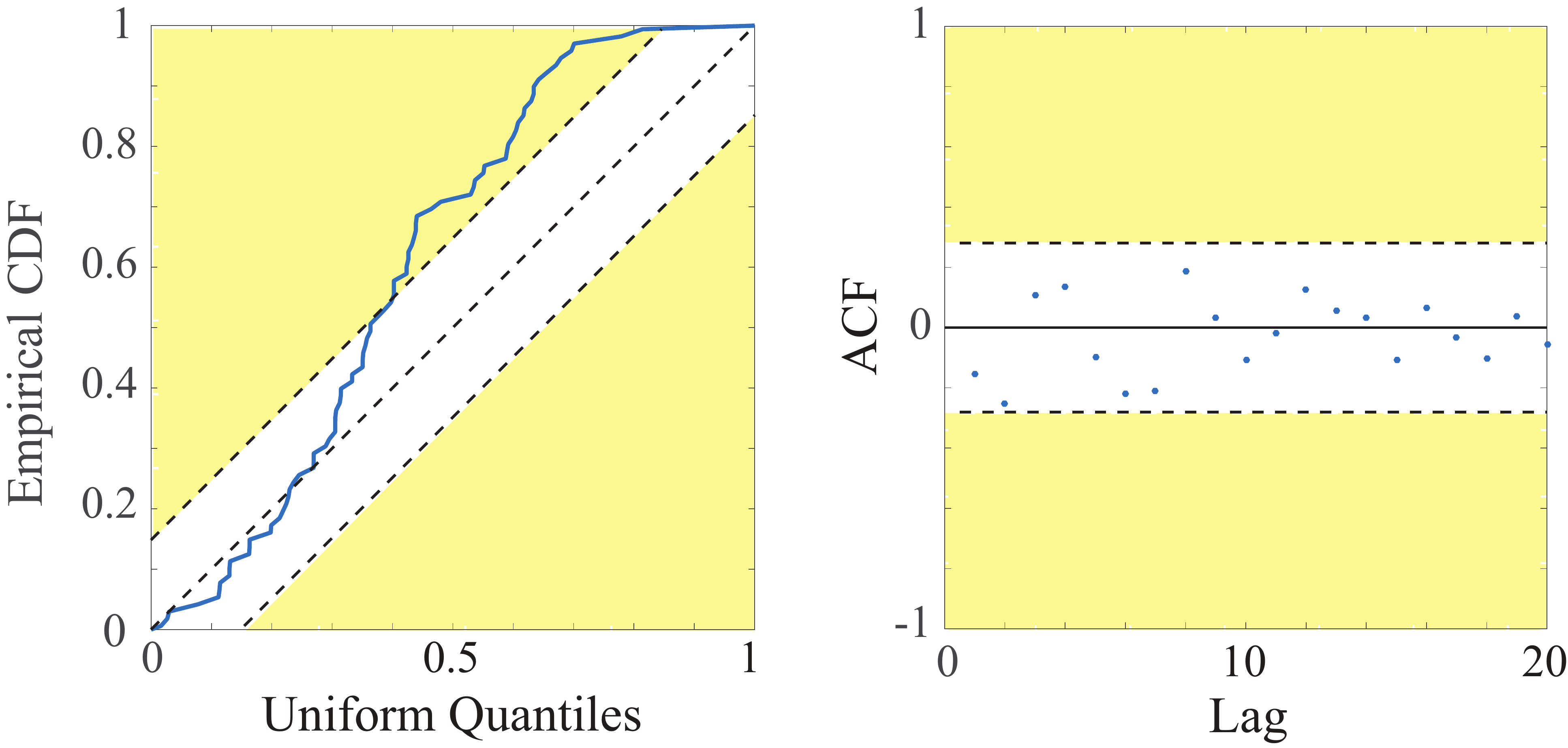}
\subcaption*{(a) ML}
\includegraphics[width=.8\columnwidth]{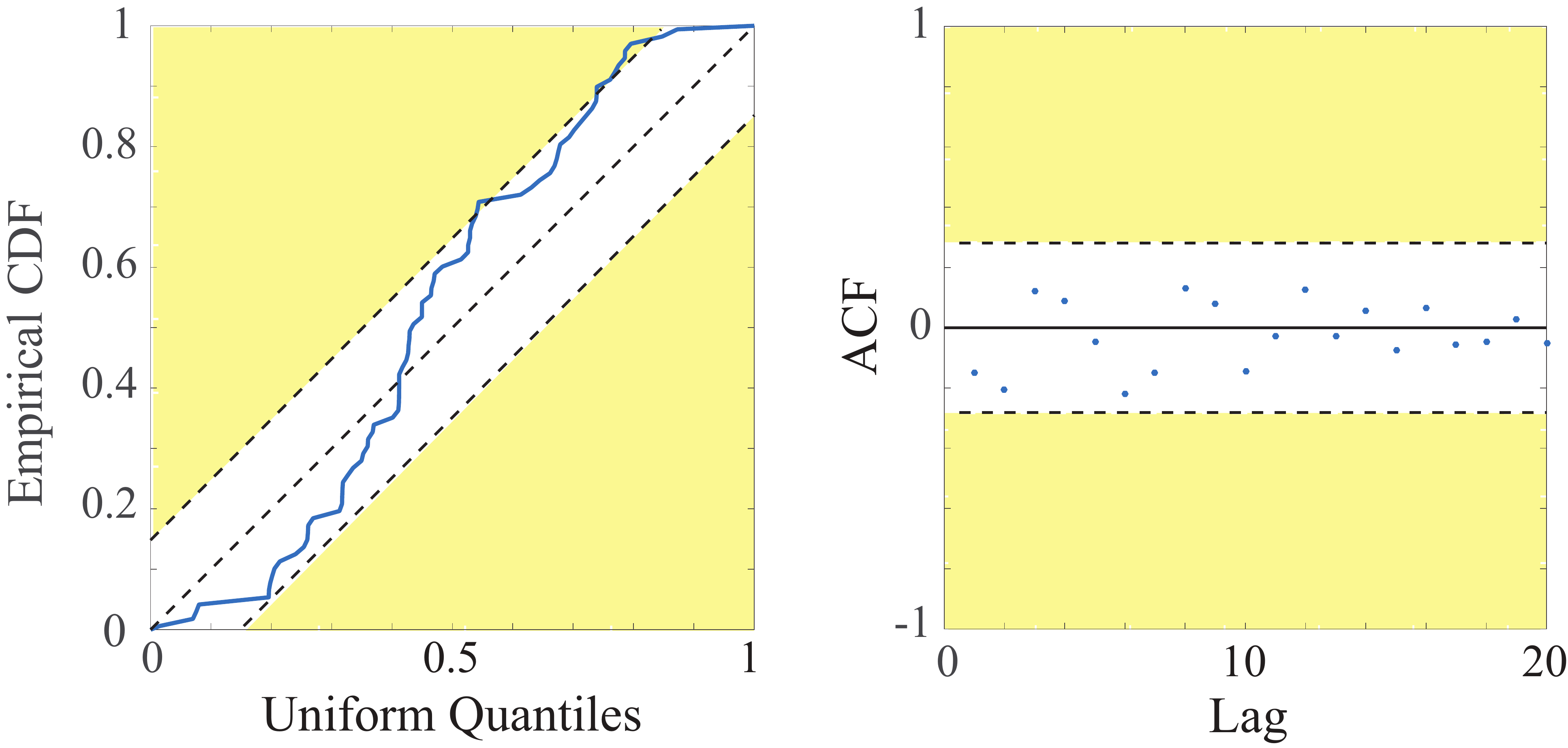}
\subcaption*{(b) $\ell_1$-regularized ML}
\includegraphics[width=.8\columnwidth]{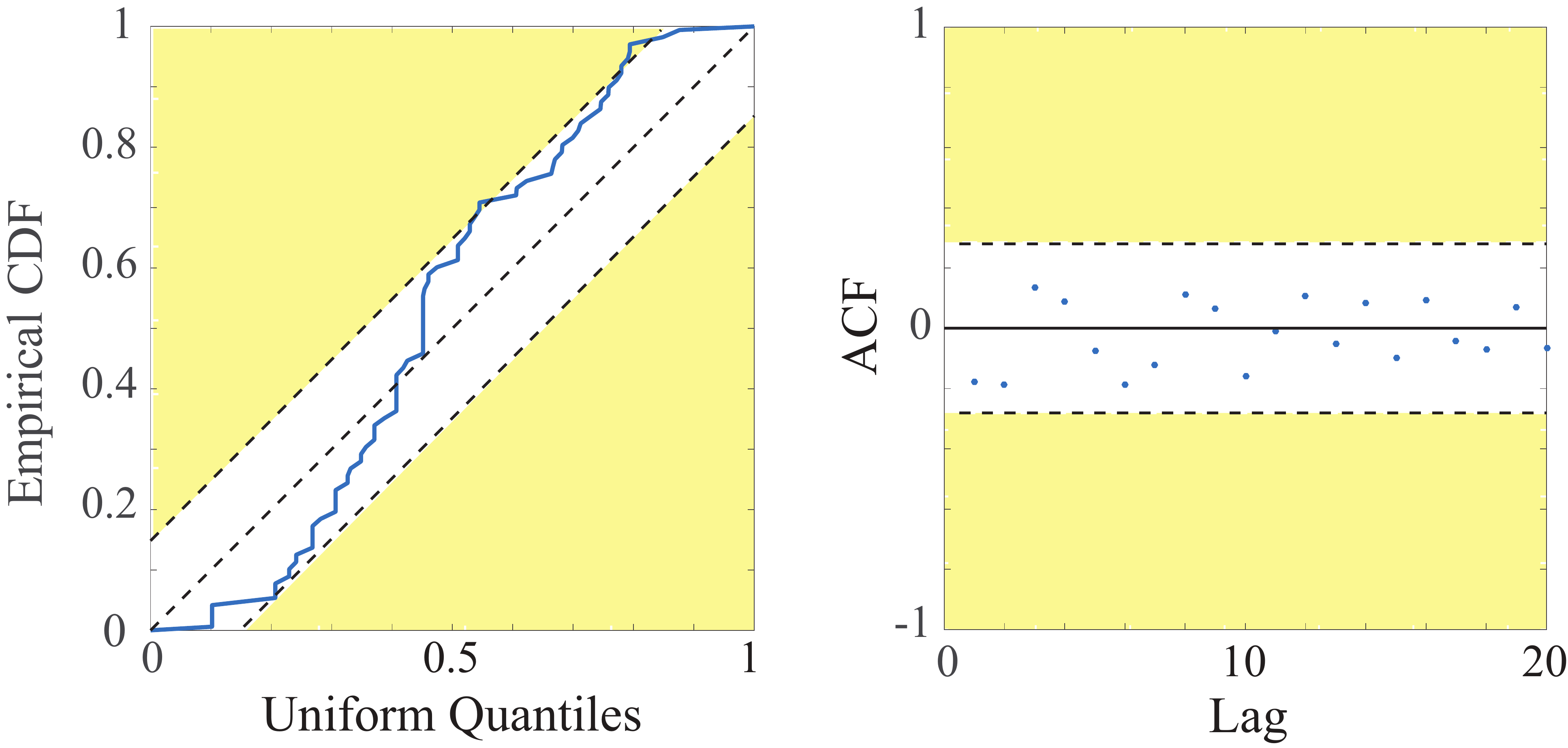}
\subcaption*{(c) POMP}
\end{center}
\caption{\small{KS and ACF tests at $95\%$ confidence level, for the ML, $\ell_1$-regularized ML and POMP estimates {using the canonical self-exciting process model.} }}\label{rgc_ks_sp_linear}
\end{figure}

\begin{figure}[H]
\begin{center}
\includegraphics[width=.8\columnwidth]{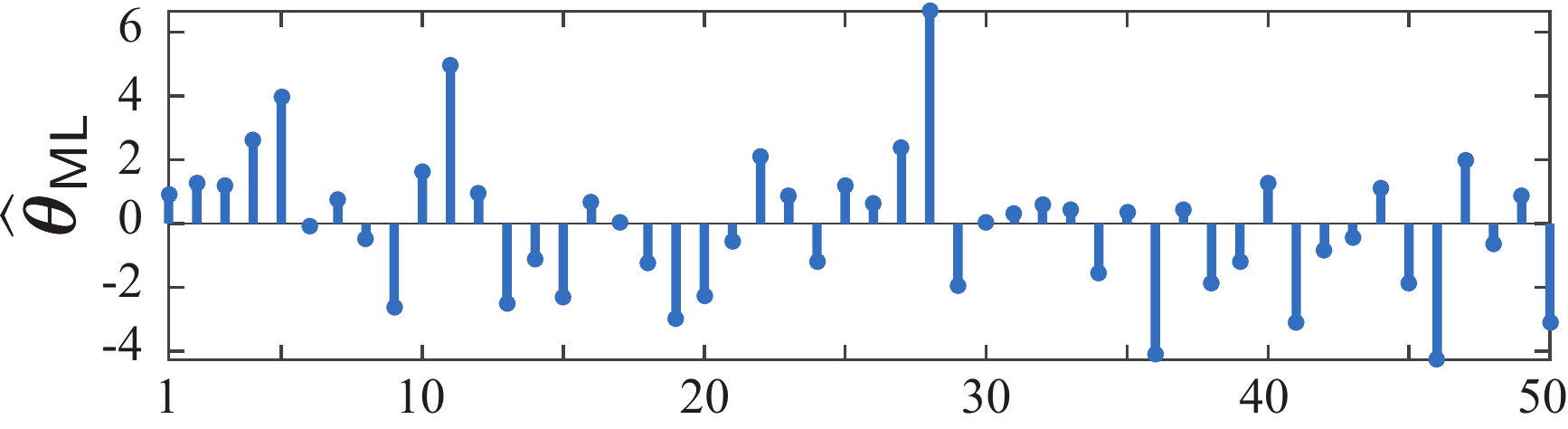}
\subcaption*{(a) ML}
\includegraphics[width=.8\columnwidth]{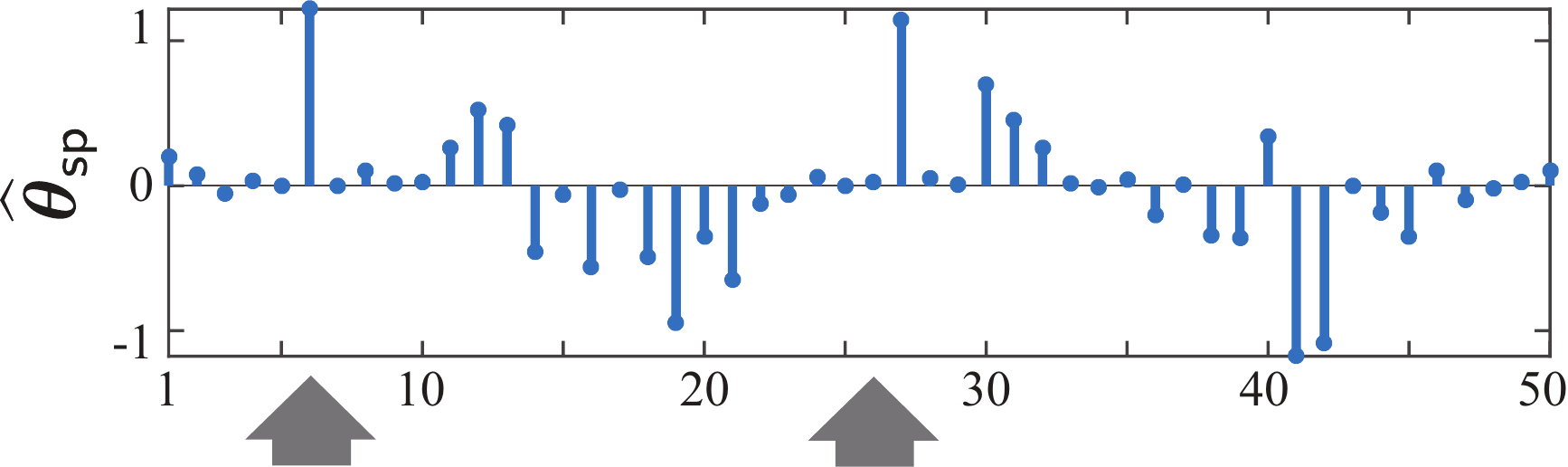}
\subcaption*{(b) $\ell_1$-regularized ML}
\includegraphics[width=.8\columnwidth]{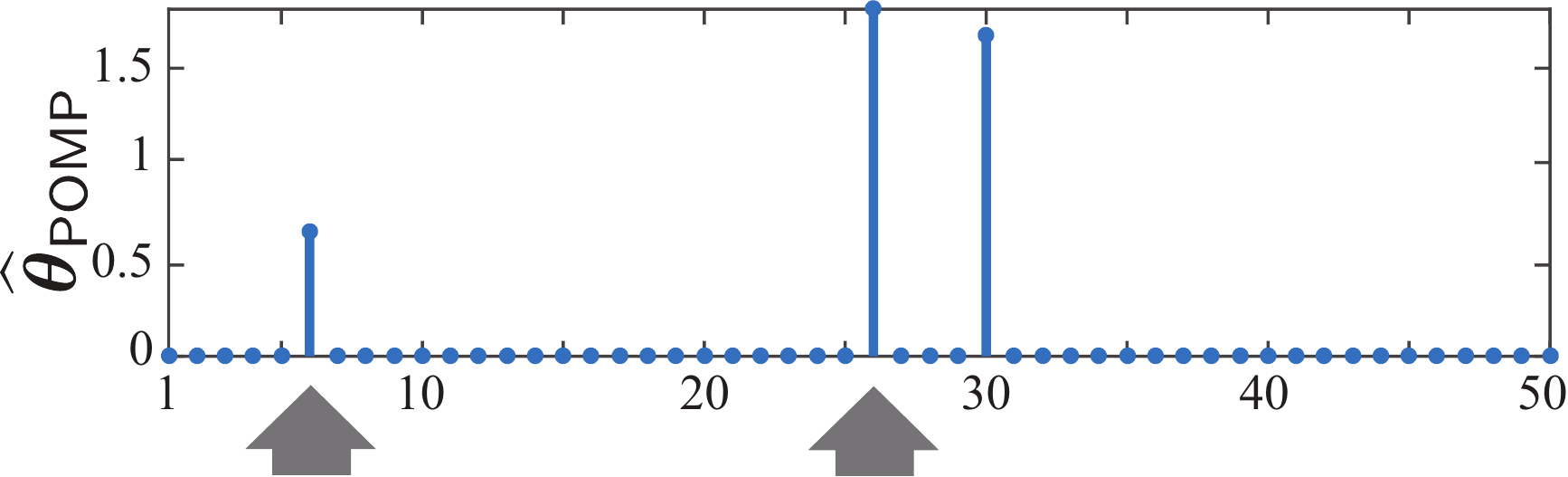}
\subcaption*{(c) POMP}
\end{center}
\caption{\small{(a) ML, (b) $\ell_1$-regularized ML, and (c) POMP estimates of the RGC spiking parameters  {using the logistic link model}.}}
\label{rgc_mlvssp}
\end{figure}

\begin{figure}[H]
\begin{center}
\includegraphics[width=.8\columnwidth]{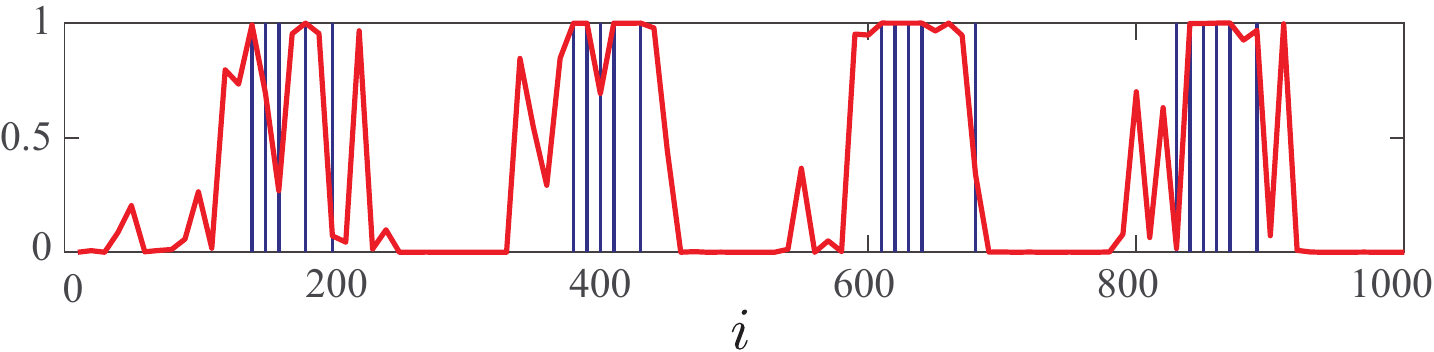}
\subcaption*{(a) ML}
\includegraphics[width=.8\columnwidth]{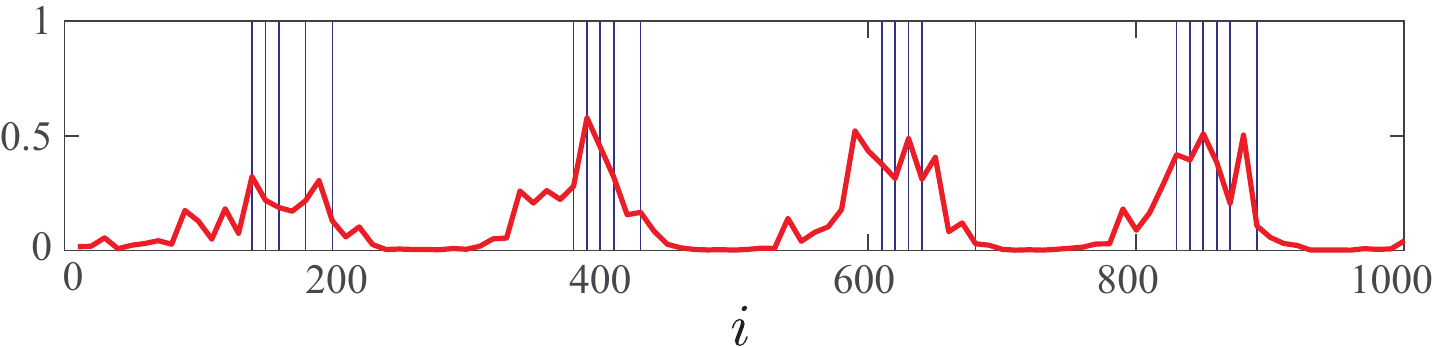}
\subcaption*{(b) $\ell_1$-regularized ML}
\includegraphics[width=.8\columnwidth]{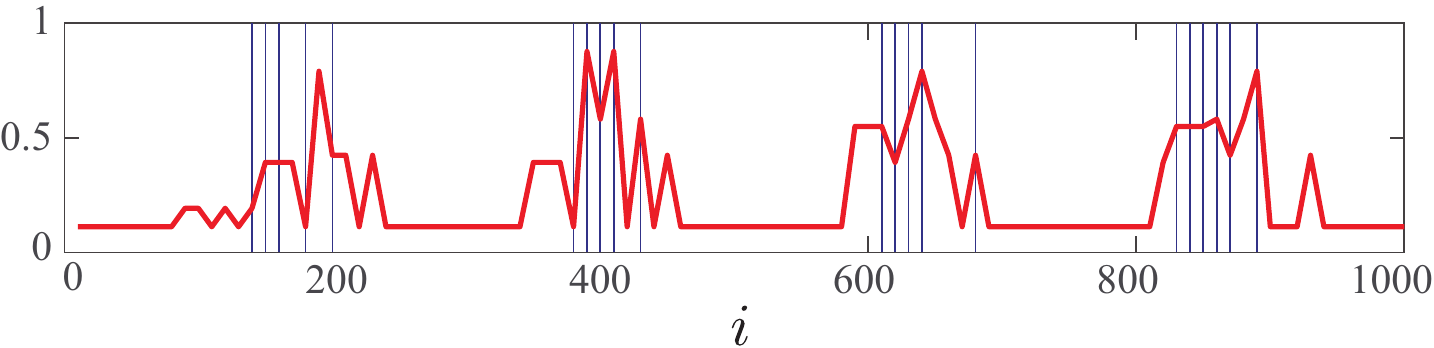}
\subcaption*{(c) POMP}
\end{center}
\caption{\small{(a) ML, (b) $\ell_1$-regularized ML, and (c) POMP estimates of the RGC spiking {probability using the unconstrained logistic link model}. Blue vertical lines show the locations of the spikes, and red traces show the estimated {probabilities}.}}
\label{rgc_rates}
\end{figure}

Figure \ref{rgc_mlvssp_linear} shows the estimated history components using the three estimators. {All three estimates} capture significant self-exciting history dependence components around the lags of $150~\text{ms}$ and $0.65$--$0.75~\text{s}$ (marked by the upward arrows). Invoking the foregoing argument for the LGN neuron regarding the power spectral density of the process (\ref{bart_spec}), these estimated lag components are consistent with the empirical estimates of \cite{wong1993transient}, as they indicate that the data can be characterized by a combination of $\frac{1}{150~\text{ms}} = 6.66$~Hz bursts separated by gaps of length $0.65$--$0.75~\text{s}$.
{The ML estimator predicts an extra self-inhibitory (negative) component, which results in over-fitting the data.}  This phenomenon can be observed by noting that the ML estimate fails the KS test shown in Figure \ref{rgc_ks_sp_linear}. 

We will next consider a logistic link model of the form $\lambda_i = \frac{\exp(\mu + \boldsymbol{\theta}'x_{i-p}^{i-1})}{C+\exp (\mu + \boldsymbol{\theta}' {\mathbf{x}_{i-p}^{i-1}})}$, with $C = 100$. This model is widely used in neuronal modeling literature (e.g., \cite{truccolo2005point, Brown_pp}), where the assumptions given by (\ref{eq:star}) are dropped and the optimization is performed in an unconstrained fashion. We adopt this approach and obtain all the estimates by dropping the assumptions of $(\star)$. Figure \ref{rgc_mlvssp} shows the estimated history components {using the unconstrained estimators. } {Compared to the canonical self-exciting process model with a linear link, both the regularized ML (Figure \ref{rgc_mlvssp}(b)) and POMP (Figure \ref{rgc_mlvssp}(c)) estimates capture similar significant self-exciting history dependence components, which are consistent across the two sets of estimates.}

{The KS and ACF test results for this case are very similar to Figure \ref{rgc_ks_sp_linear} are are thus omitted for brevity.} In order to further inspect the goodness-of-fit of these methods, we plot the estimated spiking probabilities in Figure \ref{rgc_rates}. The ML estimate shown in Figure \ref{rgc_rates}(a) overfits the spiking events by rapidly saturating the rate to either 0 and 1, which results in undesired high rate estimates where there are no spikes. On the contrary, the regularized ML (Figure \ref{rgc_rates}(b)) and POMP (Figure \ref{rgc_rates}(c)) provide a more reliable estimate of the rates consistent with the spiking events. {This analysis suggests that the sufficient assumptions of $(\star)$ are not necessary for the superior performance of the regularized and POMP estimators over that of ML.}

\section{Concluding Remarks}\label{discussions}
In this chapter, we studied the sampling properties of $\ell_1$-regularized ML and greedy estimators for a \textcolor{black}{canonical self-exciting process}. The main theorems provide non-asymptotic sampling bounds on the number of measurements, which lead to stable recovery of the parameters of the process. To the best of our knowledge, our results are the first of this kind, and can be readily generalized to various other classes of self-exciting {GLMs}, such as processes with logarithmic or logistic links. 

Compared to the existing literature, our results bring about two major contributions. First, we provide a theoretical underpinning for the advantage of $\ell_1$-regularization in ML estimation as well as greedy estimation in problems involving {binary} observations. These methods have been used in neuroscience in an ad-hoc fashion. Our results establish the utility of these techniques by characterizing the underlying sampling trade-offs. Second, our analysis relaxes the widely-assumed hypotheses of i.i.d. covariates. This assumption is often violated when working with history-dependent data such as neural spiking data. 

We also verified the validity of our theoretical results through simulation studies as well as application to real neuronal spiking data {from mouse's LGN and ferret's RGC neurons.} These results show that both the regularized ML and the greedy estimates significantly outperform the widely-used ML estimate. In particular, through making a connection with the spectrum of discrete point processes, we were able to quantify the estimation of the intrinsic firing frequency of LGN neurons. {In the spirit of easing reproducibility, we have archived a MATLAB implementation of the estimators studied in this work using the CVX package \cite{cvx} on the open source repository GitHub and made it publicly available \cite{hawkes_code}.}

One of the limitations of our analysis is the assumption that the spiking probabilities are bounded by $1/2$, which results in loss of generality. This assumption is made for the sake of theoretical analysis in bounding the mixing rate of the canonical self-exciting process. Our numerical experiments suggest that it is not necessary for the operation of the $\ell_1$-regularized and POMP estimators. We consider further inspection of the mixing properties of this process and thus relaxing this assumption as future work. Our future work also includes generalization of our analysis to multivariate GLMs, which will allow to infer network properties from multi-unit recordings of neuronal ensembles.

%
%\section{Acknowledgement}
% We would like to thank L. A. Kontorovich for helpful discussions regarding reference \cite{kontorovich2008concentration}.
%

\chapter{Fast and Stable Signal Deconvolution via Compressible State-Space Models} \label{chap:css}
\chaptermark{Compressible State-Space Deconvolution}

Common biological measurements are in the form of noisy convolutions of signals of interest with possibly unknown and transient blurring kernels. Examples include EEG and calcium imaging data. Thus, signal deconvolution of these measurements is crucial in understanding the underlying biological processes. The objective of this chapter is to develop fast and stable solutions for signal deconvolution from noisy, blurred and undersampled data, where the signals are in the form of discrete events distributed in time and space. Due to their smoothness properties, in these application Gaussian state-space models exhibit poor performance in recovering the states and the sharp transitions. In this chapter we consider the problem of estimating compressible state space models, where the sparsity lies in the transitions of the states.

We introduce compressible state-space models as a framework to model and estimate such discrete events. These state-space models admit abrupt changes in the states and have a convergent transition matrix, and are coupled with compressive linear measurements. We consider a dynamic compressive sensing optimization problem and develop a fast solution, using two nested Expectation Maximization algorithms, to jointly estimate the states as well as their transition matrices. Under suitable sparsity assumptions on the dynamics, we prove optimal stability guarantees for the recovery of the states and present a method for the identification of the underlying discrete events with precise confidence bounds. We present simulation studies as well as application to calcium deconvolution and sleep spindle detection,  which verify our theoretical results and show significant improvement over existing techniques. Our results show that by explicitly modeling the dynamics of the underlying signals, it is possible to construct signal deconvolution solutions that are scalable, statistically robust, and achieve high temporal resolution.  Our proposed methodology provides a framework for modeling and deconvolution of noisy, blurred, and undersampled measurements in a fast and stable fashion, with potential application to a wide range of biological data.

\section{Introduction}

In many signal processing applications such as estimation of brain activity from magnetoencephalography (MEG) time-series \cite{phillips1997meg}, estimation of time-varying networks \cite{kolar2010estimating},  electroencephalogram (EEG) analysis \cite{nunez1995neocortical}, calcium imaging \cite{vogelstein2010fast}, functional magnetic resonance imaging (fMRI) \cite{chang2010time}, and video compression \cite{jung2010motion}, the signals often exhibit abrupt changes that are blurred through convolution with unknown kernels due to intrinsic measurement constraints. Extracting the underlying signals from blurred and noisy measurements is often referred to as signal deconvolution. Traditionally, state-space models have been used for such signal deconvolution problems, where the states correspond to the unobservable signals. Gaussian state-space models in particular are widely used to model smooth state transitions. Under normality assumptions, posterior mean filters and smoothers are optimal estimators, where the analytical solution is given respectively by the Kalman filter and the fixed interval smoother \cite{anderson1979optimal,haykin2008adaptive}.

When applied to observations from abruptly changing states, Gaussian state-space models exhibit poor performance in recovering sharp transitions of the states due to their underlying smoothing property. Although filtering and smoothing recursions can be obtained in principle for non-Gaussian state-space models, exact calculations are no longer possible. Apart from crude approximations like the extended Kalman filter, several methods have been proposed including numerical methods for low-dimensional states \cite{kitagawa1998self}, Monte Carlo filters \cite{kitagawa1998self,hurzeler1998monte}, posterior mode estimation \cite{fruhwirth1994applied,fruhwirth1994data}, and fully Bayesian smoothing using Markov chain Monte Carlo simulation \cite{knorr1999conditional, shephard1997likelihood}. In order to exploit sparsity, several dynamic compressed sensing (CS) techniques, such as the Kalman filtered CS algorithm, have been proposed that typically assume partial information about the sparse support or estimate it in a greedy and online fashion \cite{vaswani2010ls, vaswani2008kalman, carmi2010methods, ziniel2013dynamic, zhan2015time}. However, little is known about the theoretical performance guarantees of these algorithms.

In this chapter, we consider the problem of estimating state dynamics from noisy and undersampled observations, where the state transitions are governed by autoregressive models with compressible innovations. Motivated by the theory of CS, we employ an objective function formed by the $\ell_1$-norm of the state innovations \cite{ba2012exact}. Unlike the traditional compressed sensing setting, the sparsity is associated with the dynamics and not the states themselves. In the absence of observation noise, the CS recovery guarantees are shown to extend to this problem \cite{ba2012exact}. However, in a realistic setting in the presence of observation noise, it is unclear how the CS recovery guarantees generalize to this estimation problem.

We will present stability guarantees for this estimator under a convergent state transition matrix, which confirm that the CS recovery guarantees can be extended to this problem. The corresponding optimization problem in its Lagrangian form is akin to the MAP estimator of the states in a linear state-space model where the innovations are Laplace distributed. This allows us to integrate methods from Expectation-Maximization (EM) theory and Gaussian state-space estimation to derive efficient algorithms for the estimation of states as well as the state transition matrix, which is usually unknown in practice. To this end, we construct two nested EM algorithms in order to jointly estimate the states and the transition matrix. The outer EM algorithm for state estimation is akin to the fixed interval smoother, and the inner EM algorithm uses the state estimates to update the state transition matrix \cite{shumway1982approach}. The resulting EM algorithm is recursive in time, which makes the computational complexity of our method scale linearly with temporal dimension of the problem. This provides an advantage over existing methods based on convex optimization, which typically scale super-linearly with the temporal dimension.

Our results are related to parallel applications in spectral estimation, source localization, and channel equalization \cite{fevrier1999reduced}, where the measurements are of the form $\mathbf{Y} = \mathbf{A} \mathbf{X}+ \mathbf{N}$, with $\mathbf{Y}$ is the observation matrix, $\mathbf{X}$ denotes the unknown parameters, $\mathbf{A}$ is the measurement matrix, and $\mathbf{N}$ is the additive noise. These problems are referred to as Multiple Measurement Vectors (MMV) \cite{cotter2005sparse} and Multivariate Regression \cite{obozinski2011support}. In these applications, solutions with row sparsity in $\mathbf{X}$ are desired. Recovery of sparse signals with Gaussian innovations is studied in \cite{zhang2011sparse}. Several recovery algorithms including the $\ell_1\!\!\!-\!\!\ell_q$ minimization methods, subspace methods and greedy pursuit algorithms \cite{davies2012rank} have been proposed for support union recovery in this setup. Our contributions are distinct in that we directly model the state innovations as a compressible sequence, for recovery of which we present both sharp theoretical guarantees as well as fast algorithms from state-space estimation.

Finally, we provide simulation results as well as applications to two experimentally-acquired data sets: calcium imaging recordings of neuronal activity, and EEG data during sleep. In the former, the deconvolution problem concerns estimating the location of spikes given the temporally blurred calcium fluorescence, and in the latter, the objective is to detect the occurrence and onset of sleep spindles. Our simulation studies confirm our theoretical predictions on the performance gain obtained by compressible state-space estimation over those obtained by traditional estimators such as the basis pursuit denoising. Our real data analyses reveal that our compressible state-space modeling and estimation framework outperforms two of the commonly-used methods for calcium deconvolution and sleep spindle detection. In the spirit of easing reproducibility, we have made MATLAB implementations of our codes publicly available \cite{code}.

\vspace{-2mm}
\section{Methods}\label{sec:tv_formulation}
In this section we introduce the experimental procedures of recording the data for analysis and establish our problem formulation and notational conventions and present our main theoretical analysis and algorithm development. 
\subsection{Experimental Procedures}
\subsubsection{Surgery} 2 hours before surgery, $0.1~\text{cc}$ dexamethasone ($2~\text{mg/ml}$, VetOne) was injected subcutaneously to reduce brain swelling during craniotomy. Anesthesia is induced with $4\%$ isoflurane (Fluriso, VetOne) with a calibrated vaporizer (Matrx VIP 3000). During surgery, isoflurane level was reduced to and maintained at a level of $1.5\%$--$2\%$. Body temperature of the animal is maintained at $36.0$ degrees Celsius during surgery. Hair on top of head of the animal was removed using Hair Remover Face Cream (Nair), after which Betadine (Purdue Products) and $70\%$ ethanol was applied sequentially 3 times to the surface of the skin before removing the skin. Soft tissues and muscles were removed to expose the skull. Then a custom designed 3D printed stainless headplate was mounted over left auditory cortex and secured with C\&B-Metabond (Parkell). A craniotomy with a diameter of around $3.5~\text{mm}$ was then performed over left auditory cortex. A three layered cover slip was used as cranial window, which is made by gluing (NOA71, Norland Products) $2$ pieces of $3~\text{mm}$ coverslips (64-0720 (CS-3R), Warner Instruments) with a $5~\text{mm}$ coverslip ($64$--$0700$ (CS-5R), Warner Instruments). Cranial window was quickly dabbed in kwik-sil (World Precision Instruments) before mounted $3~\text{mm}$ coverslips facing down onto the brain. After kwik-sil cured, Metabond was applied to secure the position of the cranial window. Synthetic Black Iron Oxide (Alpha Chemicals) was then applied to the hardened Metabond surface. $0.05~\text{cc}$ Cefazolin ($1$ gram/vial, West Ward Pharmaceuticals) was injected subcutaneously when entire procedure was finished. After the surgery the animal was kept warm under heat light for $30$ minutes for recovery before returning to home cage. Medicated water (Sulfamethoxazole and Trimethoprim Oral Suspension, USP $200~\text{mg}/40~\text{mg}$ per $5~\text{ml}$, Aurobindo Pharms USA; $6~\text{ml}$ solution diluted in $100~\text{ml}$ water) substitute normal drinking water for $7$ days before any imaging was performed.

\subsubsection{Awake two-photon imaging} Spontaneous activity data of population of layer $2/3$ auditory cortex (A1) neurons is collected from adult (3-month old) Thy1-GCaMP6s female mouse implanted with chronic window following the above procedure, using two-photon imaging. Acquisition is performed using a two-photon microscope (Thorlabs Bscope 2) equipped with a Vision 2 Ti:Sapphire laser (Coherent), equipped with a GaAsP photo detector module (Hamamatsu) and resonant scanners enabling faster high-resoluation scanning at $30$--$60~\text{Hz}$ per frame. The excitation wavelength was $920~\text{nm}$. Regions ($\sim 300~\mu\text{m}^2$) within A1 were scanned at $30~\text{Hz}$ through a $20\text{x}$, $0.95$ NA water-immersion objective (Olympus). During imaging the animal was head-fixed and awake. The microscope was rotated $45$ degrees and placed over the left A1 where window was placed.  An average image of field of view was generated by choosing a time window where minimum movement of the brain was observed and used as reference image for motion correction using TurboReg plugin in ImageJ. GCaMP6s positive cells are selected manually by placing a ring like ROI over each identified cell. Neuropil masks were generated by placing a $20~\mu\text{m}$ radius circular region over each cell yet excluding all cell soma regions. Traces of soma and neuropil were generated by averaging image intensity within respective masks at each time point. A ratio of $0.7$ was used to correct for neuropil contamination.

\subsubsection{Cell-attached patch clamp recordings and two-photon imaging} Recordings were performed in vitro in voltage clamp to simultaneously measure spiking activity and $\Delta F / F$. Thalamocortical slices containing A1 were prepared as previously described \cite{meng2015visual}. The extracellular recording solution consisted of artificial cerebral spinal fluid (ACSF) containing: $130$ NaCl, $3$ KCl, $1.25$ KH2PO4, $20$ NaHCO3, $10$ glucose, $1.3$ MgSO4, $2.5$ CaCl2 (pH $7.35$-$7.4$, in $95\%$ O2 – $5\%$ CO2). Action potentials were recorded extracellularly in loose-seal cell-attached configuration (seal resistance typically $20$--$30~\text{M}\Omega$) in voltage clamp mode. Borosilicate glass patch pipettes were filled with normal ACSF diluted $10\%$, and had a tip resistance of $\sim3$-$5~\text{M}\Omega$ in the bath. Data were acquired with a Multiclamp 700B patch clamp amplifier (Molecular Devices), low-pass filtered at $3$-$6~\text{kHz}$, and digitized at $10~\text{kHz}$ using the MATLAB-based software. Action potentials were stimulated with a bipolar electrode placed in L1 or L23 to stimulate the apical dendrites of pyramidal cells (pulse duration $1$-$5~\text{ms}$). Data were analyzed offline using MATLAB. Imaging was largely performed using a two-photon microscope (Ultima, Prairie Technologies) and a MaiTai DeepSee laser (SpectraPhysics), equipped with a GaAsP photo detector module (Hamamatsu) and resonant scanners enabling faster high-resoluation scanning at $30$-$60~\text{Hz}$ per frame. Excitation was set at $900~\text{nm}$. Regions were scanned at $30~\text{Hz}$ through a $40\text{x}$ water-immersion objective (Olympus). Cells were manually selected as ring-like regions of interest (ROIs) that cover soma but exclude cell nuclei, and pixel intensity within each ROI was averaged to generate fluorescence over time and changes in fluorescence ($\Delta F / F$) were then calculated.

\vspace{-2mm}
\subsection{Problem Formulation and Theoretical Analysis}

We consider the linear compressible state-space model given by
\begin{equation}
\label{eq:tv_lap_state_space}
\mathbf{x}_t = \mathbf{\Theta} \mathbf{x}_{t-1}+ \mathbf{w}_t, \qquad \mathbf{y}_t = \mathbf{A}_t \mathbf{x}_t + \mathbf{v}_t,
\end{equation}
where $(\mathbf{x}_t)_{t=1}^{T} \in \mathbb{R}^p$ denote the sequence of unobservable states, $\mathbf{\Theta}$ is the state transition matrix satisfying $\|\mathbf{\Theta}\|<1$, $\mathbf{w}_t \in \mathbb{R}^p$ is the state innovation sequence, $(\mathbf{y}_t)_{t=1}^T \in \mathbb{R}^{n_t}$ are the linear observations, $\mathbf{A}_t \in \mathbb{R}^{n_t \times p}$ denotes the measurement matrix, and {$\mathbf{v}_t \in \mathbb{R}^{n_t}$} denotes the measurement noise. The main problem is to estimate the unobserved sequence $(\mathbf{x}_t)_{t=1}^T$ (and possibly $\boldsymbol{\Theta}$), given the sequence of observations $(\mathbf{y}_t)_{t=1}^T$. This problem is in general ill-posed, when $n_t < p$, for some $t$. We therefore need to make additional assumptions in order to seek a stable solution.
 
We assume that the state innovations are sparse (resp. compressible), i.e. $\mathbf{x}_t-\mathbf{\Theta} \mathbf{x}_{t-1}$ is $s_t$-sparse (resp. $(s_t,\xi)$-compressible) with $s_1 \gg s_t$ for $t \in [T]\backslash\{1\}$. Our theoretical analysis pertain to the compressed sensing regime where $1 \ll s_t < n_t \ll p$.  We assume that the rows of $\mathbf{A}_t$ are a subset of the rows of $\mathbf{A}_1$, i.e. $\mathbf{A}_t = (\mathbf{A}_{1})_{n_t}$, and define $\widetilde{\mathbf{A}}_t = \sqrt{\frac{n_1}{n_t}}\mathbf{A}_t$. Other than its technical usefulness, the latter assumption helps avoid prohibitive storage of all the measurement matrices. In order to promote sparsity of the state innovations, we consider the dynamic $\ell_1$-regularization (dynamic CS from now on) problem defined as
\begin{align}
\label{eq:tv_prob_def_primal}
\minimize \limits_{(\mathbf{x}_t)_{t=1}^{T},\mathbf{\Theta}} \quad \displaystyle \sum_{t=1}^T \frac{\|\mathbf{x}_t-\mathbf{\Theta} \mathbf{x}_{t-1}\|_1}{\sqrt{s_t}} \ \ \text{s.t.} \ \ \|\mathbf{y}_t-\mathbf{A}_t\mathbf{x}_t\|_2 \leq \sqrt{\frac{n_t}{n_1}}\varepsilon.
\end{align}
where $\varepsilon$ is an upper bound on the observation noise, i.e., $\| v_t \|_2 \le \varepsilon$ for all $t$. Note that this problem is a variant of the dynamic CS problem introduced in \cite{ba2012exact}. We also consider the modified Lagrangian form of (\ref{eq:tv_prob_def_primal}) given by
\begin{equation}
\label{eq:tv_prob_def_dual}
\minimize_{(\mathbf{x}_t)_{t=1}^{T},\mathbf{\Theta}} \quad \lambda \sum_{t=1}^T \frac{\|\mathbf{x}_t-\mathbf{\Theta} \mathbf{x}_{t-1}\|_1}{\sqrt{s_t}} +  \frac{1}{n_t}\frac{\|\mathbf{y}_t-\mathbf{A}_t\mathbf{x}_t\|_2^2}{2\sigma^2}.
\end{equation}
for some constants $\sigma_2^2$ and $\lambda \ge 0$. Note that if $\mathbf{v}_t \sim \mathcal{N}(\mathbf{0}, n_t \sigma^2 \mathbf{I})$, then Eq. (\ref{eq:tv_prob_def_dual}) is akin to the maximum \emph{a posteriori} (MAP) estimator of the states in (\ref{eq:tv_lap_state_space}), assuming that the \textit{state} innovations were independent Laplace random variables with respective parameters $\lambda / \sqrt{s_t}$. We will later use this analogy to derive fast solutions to the optimization problem in (\ref{eq:tv_prob_def_dual}).

Uniqueness and exact recovery of the sequence $(\mathbf{x}_t)_{t=1}^T$ in the absence of noise was proved in \cite{ba2012exact} for $\mathbf{\Theta} =\mathbf{I}$, by an inductive construction of dual certificates. The special case $\mathbf{\Theta} =\mathbf{I}$ can be considered as a generalization of the total variation (TV) minimization problem \cite{poon2015role}. Our main result on stability of the solution of (\ref{eq:tv_prob_def_primal}) is the following:

\begin{thm}[Stable Recovery of Activity in the Presence of Noise]
\label{thm:tv_main}
Let $(\mathbf{x}_t)_{t=1}^T \in \mathbb{R}^p$ be a sequence of states with a known transition matrix $\mathbf{\Theta}=\theta \mathbf{I}$, where $|\theta|<1$ and  $\widetilde{\mathbf{A}}_t$, $t\geq 1$ satisfies RIP of order ${4s}$ with $\delta_{4s} <1/3$. Suppose that $n_1 > n_2 = n_3 = \cdots = n_T$. Then, the solution $(\widehat{\mathbf{x}}_t)_{t=1}^{T}$ to the dynamic CS problem (\ref{eq:tv_prob_def_primal}) satisfies
\begin{align}
\label{eq:tv_main_stable}
 \frac{1}{T} \sum_{t=1}^T \|\mathbf{x}_t-\widehat{\mathbf{x}}_t\|_2 \le\displaystyle \frac{1-\theta^T}{1-\theta}\left(12.6 \left( 1+\frac{1}{T} \sqrt{\frac{n_1}{n_2}}-\frac{1}{T} \right)\varepsilon + \frac{3}{T} \sum_{t=1}^T \frac{\sigma_{s_t}(\mathbf{x}_t-\mathbf{\Theta} \mathbf{x}_{t-1})}{\sqrt{s_t}}\right).
\end{align}
\end{thm}

\noindent \textit{\textbf{Proof Sketch.}} The proof of Theorem \ref{thm:tv_main} is based on establishing a modified cone and tube constraint for the dynamic CS problem (\ref{eq:tv_prob_def_primal}) and using the boundedness of the Frobenius norm of the inverse first-order differencing operator. Details of the proof are given in Appendix \ref{app:tv_main_proof}.

\noindent { \textbf{Remark 1.}} The first term on the right hand side of Theorem \ref{thm:tv_main} implies that the average reconstruction error of the sequence $(\mathbf{x}_t)_{t=1}^T$ is upper bounded proportional to the noise level $\varepsilon$, which implies the stability of the estimate. The second term is a measure of compressibility of the innovation sequence and vanishes when the sparsity condition is exactly met.

\noindent \textbf{Remark 2.} Similar to the general compressive sensing setup, the results of Theorem \ref{thm:tv_main} are general and hold even when the sparsity assume is not exactly met, namely when the compressibility (sparsity) assumption is violated, the second term on the right hand side could be of order of a constant. However in order for $\widetilde{\mathbf{A}}_t$ to satisfy RIP of order $4s$ one requires $n \sim \mathcal{O}(s \log p)$ which in the absence of sparsity ($s \sim \mathcal{O}(p)$) reduces to the denoising regime.

\noindent \textbf{Remark 3.} The theory of LASSO suggests a choice of the regularization parameter ${\lambda}  \sim \mathcal{O} \left(\sigma \sqrt{\frac{\log p}{n_t}}\right)$ \cite{Negahban}.  In the absence of sparsity (compressibility) $s \sim \mathcal{O}(n)$, one can think of two possible scenarios: In the high-dimensional setup where $n \gg 1$ , $\lambda$ will be very small and (\ref{eq:tv_prob_def_dual}) will be similar to solving the Maximum-Likelihood problem which is known to result in an unbiased estimator. In the low-dimensional case $\lambda$ is chosen via cross-validation which adapts the problem to the sparsity level.

\noindent \textbf{Remark 4.} Finally, the choice of the $\ell_1$-regularized maximum-likelihood estimation in (\ref{eq:tv_prob_def_dual}) is motivated by its superior performance over ML estimation in the compressive sensing literature and the applications of interest in this chapter where exhibit sparse activity patterns in time. It is well-known that Laplace distribution is not a compressible distribution \cite{gribonval2012compressible}. Therefore we have not made the assumption that the innovations $\mathbf{w}_t$ are Laplace distributed, but have made the observation that an $\ell_1$-regularized ML estimator is akin to the MAP estimator for a Laplace state-space model. Theorem \ref{thm:tv_main} suggests that a Laplace state-space model is asymptotically \textit{as good as} any other compressible state-space model up to a constant factor in the error bounds.

\vspace{-2mm}
\subsection{Fast Iterative Solution via the EM Algorithm}
\label{sec:tv_algorithm}
Due to the high dimensional nature of the state estimation problem, algorithms with polynomial complexity exhibit poor scalability. Moreover, when the state transition matrix is not known, the dynamic CS optimization problem (\ref{eq:tv_prob_def_dual}) is not convex in $\left( (\mathbf{x}_t)_{t=1}^T,\mathbf{\Theta} \right)$. Therefore standard convex optimization solvers cannot be directly applied. This problem can be addressed by employing the Expectation-Maximization algorithm \cite{shumway1982approach}. A related existing result considers weighted $\ell_1$-regularization to adaptively capture the state dynamics \cite{charles2013dynamic}. Our approach is distinct in that we derive a fast solution to (\ref{eq:tv_prob_def_dual}) via two nested EM algorithms, in order to jointly estimate the states and their transition matrix. The outer EM algorithm converts the estimation problem to a form suitable for the usage of the traditional Fixed Interval Smoothing (FIS) by invoking the EM interpretation of the Iterative Re-weighted Least Squares (IRLS) algorithms \cite{babadi_IRLS}. The inner EM algorithm performs state and parameter estimation efficiently using the FIS. We refer to our estimated as the Fast Compressible State-Space ({\sf FCSS}) estimator.

\subsubsection*{The outer EM loop of {\sf FCSS}}

In \cite{babadi_IRLS}, the authors established the equivalence of the IRLS algorithm as an instance of the EM algorithm for solving $\ell_1$-minimization problems via the Normal/Independent (N/I) characterization of the Laplace distribution. Consider the $\epsilon$-perturbed $\ell_1$-norm as
\begin{equation}
\label{eq:tv_eps_l1}
\|\mathbf{x}\|_{1,\epsilon} = \sqrt{x_1^2 + \epsilon^2}+\sqrt{x_2^2 + \epsilon^2}+\cdots+\sqrt{x_p^2 + \epsilon^2}.
\end{equation}
Note that for $\epsilon =0$, $\|\mathbf{x}\|_{1,\epsilon}$ coincides with the usual $\ell_1$-norm. We define the $\epsilon$-perturbed version of the dual problem (\ref{eq:tv_prob_def_dual}) by 
\begin{equation}
\label{eq:tv_eps_lag}
\minimize_{\left(\mathbf{x}_t\right)_{t=1}^T, \mathbf{\Theta}} \quad \lambda \sum_{t=1}^T \frac{\|\mathbf{x}_t-\mathbf{\Theta} \mathbf{x}_{t-1}\|_{1,\epsilon}}{\sqrt{s_t}} +  \frac{1}{n_t}\frac{\|\mathbf{y}_t-\mathbf{A}_t\mathbf{x}_t\|_2^2}{2\sigma^2}.
\end{equation}
If instead of the $\ell_{1,\epsilon}$-norm, we had the square $\ell_2$ norm, then the above problem could be efficiently solved using the FIS. The outer EM algorithm indeed transforms the problem of Eq. (\ref{eq:tv_eps_lag}) into this form. Note that the $\epsilon$-perturbation only adds a term of the order $\mathcal{O}(\epsilon p)$ to the estimation error bound of Theorem \ref{thm:tv_main}, which is negligible for small enough $\epsilon$ \cite{babadi_IRLS}.

The problem of Eq. (\ref{eq:tv_eps_lag}) can be interpreted as a MAP problem: the first term corresponds to the state-space prior $-\log$ $ p(\mathbf{x}_t | \mathbf{x}_{t-1}, \boldsymbol{\Theta}) = -\log p_{s_t}(\mathbf{x}_t -  \boldsymbol{\Theta} \mathbf{x}_{t-1})$, where $p_{s_t}(\mathbf{x}) \sim \exp$~$(-\lambda \|\mathbf{x}\|_{1,\epsilon} / \sqrt{s_t} )$ denoting the $\epsilon$-perturbed Laplace distribution; the second term is the negative log-likelihood of the data given the state, assuming a zero-mean Gaussian observation noise with covariance $\sigma^2 \mathbf{I}$. Suppose that at the end of the $l^{\sf th}$ iteration, the estimates $(\widehat{\mathbf{x}}_t^{(l)})_{t=1}^T,\widehat{\mathbf{\Theta}}^{(l)}$ are obtained, given the observations $(\mathbf{y}_t)_{t=1}^T$. As it is shown in Appendix \ref{app:tv_EM}, the outer EM algorithm transforms the optimization problem to:
\begin{align}
\label{eq:tv_prob_def_dual_irls}
\minimize_{\left(\mathbf{x}_t\right)_{t=1}^T, \mathbf{\Theta}} \quad &\frac{\lambda}{2} \sum_{j=1}^p\sum_{t=1}^T \frac{\left( \mathbf{x}_{t}-\mathbf{\Theta} \mathbf{x}_{t-1}\right)_j^2+\epsilon^2}{\sqrt{s_t}\sqrt{\left( \widehat{\mathbf{x}}_{t}^{(l)}-\widehat{\mathbf{\Theta}}^{(l)} \widehat{\mathbf{x}}_{t-1}^{(l)}\right)_j^2+\epsilon^2}}  +  \sum_{t=1}^T \frac{1}{n_t}\frac{\|\mathbf{y}_t-\mathbf{A}_t\mathbf{x}_t\|_2^2}{2\sigma^2}.
\end{align}
in order to find $(\widehat{\mathbf{x}}_t^{(l+1)})_{t=1}^T,\widehat{\mathbf{\Theta}}^{(l+1)}$. Under mild conditions, convergence of the solution of (\ref{eq:tv_prob_def_dual_irls}) to that of (\ref{eq:tv_prob_def_dual}) was established in \cite{babadi_IRLS}. The objective function of (\ref{eq:tv_prob_def_dual_irls}) is still not jointly convex in $\left( (\mathbf{x}_t)_{t=1}^T,\mathbf{\Theta} \right)$. Therefore, to carry out the optimization, i.e. the outer M step, we will employ another instance of the EM algorithm, which we call the inner EM algorithm, to alternate between estimating of $\left(\mathbf{x}_t\right)_{t=1}^T$ and $\mathbf{\Theta}$.

\subsubsection*{The inner EM loop of {\sf FCSS}}

Let $\mathbf{W}_t^{(l)}$ be a diagonal matrix such that
\[
(\mathbf{W}_t^{(l)})_{j,j} = {s_t}^{-1/2} \left \{\left( \widehat{\mathbf{x}}_t^{(l)} -{\widehat{\mathbf{\Theta}}}^{(l)} \widehat{\mathbf{x}}_{t-1}^{(l)}\right)_j^2+\epsilon^2\right\}^{-1/2}.
\] 
Consider an estimate $\widehat{\mathbf{\Theta}}^{(l,m)}$, corresponding to the $m^{\sf th}$ iteration of the inner EM algorithm within the $l^{\sf th}$ M-step of the outer EM. In this case, Eq. (\ref{eq:tv_prob_def_dual_irls}) can be thought of the MAP estimate of the Gaussian state-space model given by:
\begin{align}
\label{eq:tv_dynamic}
\begin{array}{l}
\mathbf{x}_t = \widehat{\mathbf{\Theta}}^{(l,m)} \mathbf{x}_{t-1}+ \mathbf{w}_t, \quad \mathbf{w}_t\sim \mathcal{N}\left(\mathbf{0}, \frac{1}{\lambda} {\mathbf{W}_t^{(l)}}^{-1} \right) \\
\mathbf{y}_t = \mathbf{A}_t \mathbf{x}_t + \mathbf{v}_t, \quad \mathbf{v}_t\sim \mathcal{N}(\mathbf{0},n_t\sigma^2 I)
\end{array},
\end{align}
In order to obtain the inner E step, one needs to find the density of $(\mathbf{x}_t)_{t=1}^T$ given $(\mathbf{y}_t)_{t=1}^T$ and $\widehat{\mathbf{\Theta}}^{(l,m)} \}$. Given the Gaussian nature of the state-space in Eq. (\ref{eq:tv_dynamic}), this density is a multivariate Gaussian density, whose means and covariances can be efficiently computed using the FIS. For all $t \in [T]$, the FIS performs a forward Kalman filter and a backward smoother to generate \cite{rauch1965maximum,anderson1979optimal}:
\begin{align*}
\mathbf{x}^{(l, m+1)}_{{t|T}} := {\mathbb{E}}\left\{\mathbf{x}_t|(\mathbf{y}_t)_{t=1}^T, \widehat{\mathbf{\Theta}}^{(l,m)}\right\},
\end{align*}
\begin{equation*}
\boldsymbol{\Sigma}^{(l,m+1)}_{t|T} := {\mathbb{E}}\left\{\mathbf{x}_t \mathbf{x}_t'|(\mathbf{y}_t)_{t=1}^T, \widehat{\mathbf{\Theta}}^{(l,m)}\right\},
\end{equation*}
and
\begin{equation*}
\mathbf{\Sigma}^{(l,m+1)}_{t-1,t|T}={\mathbf{\Sigma}}^{(l,m+1)}_{t,t-1|T}={\mathbb{E}}\left\{\mathbf{x}_{t-1} \mathbf{x}_t'|(\mathbf{y}_t)_{t=1}^T,\widehat{\mathbf{\Theta}}^{(l,m)}\right\}.
\end{equation*}

Note that due to the quadratic nature of all the terms involving $(\mathbf{x}_t)_{t=1}^T$, the outputs of the FIS suffice to compute the expectation of the objective function in Eq. (\ref{eq:tv_prob_def_dual_irls}), i.e., the inner E step, which results in:
\begin{align}
\label{eq:tv_Q_calculated}
\notag \maximize_{\mathbf{\Theta}} & -\frac{\lambda}{2}  \displaystyle \left ({\mathbf{\Theta}}\left( \sum_{t=1}^T   \mathbf{W}_t^{(l)} \left (\mathbf{x}^{(l, m+1)}_{t-1|T}{\mathbf{x}'}^{(l, m+1)}_{t-1|T} + \mathbf{\Sigma}_{t-1|T}^{(l,m+1)} \right )\right){\mathbf{\Theta}}^T\right)\\
&  + \frac{\lambda}{2} \operatorname{Tr} \displaystyle { \left ({\mathbf{\Theta}} \left (\sum_{t=1}^T \mathbf{W}_t^{(l)} \left ( \mathbf{x}^{(l, m+1)}_{t-1|T} {\mathbf{x}'}^{(l, m+1)}_{t|T} + \mathbf{x}^{(l, m+1)}_{t|T} {\mathbf{x}'}^{(l, m+1)}_{t-1|T} + 2\mathbf{\Sigma}_{t-1,t|T}^{(l,m+1)} \right ) \right) \right)},
\end{align}
to obtain $\widehat{\boldsymbol{\Theta}}^{(l,m)}$. The solution has a closed-form given by:
\begin{align}
\label{eq:tv_theta_upd}
\notag  \widehat{\mathbf{\Theta}}^{(l, m+1)} = \displaystyle & \left(   \sum_{t=1}^T 2 \mathbf{W}_t^{(l)}\left(\mathbf{x}^{(l, m+1)}_{t-1|T}{\mathbf{x}'}^{(l, m+1)}_{t-1|T} + \mathbf{\Sigma}_{t-1|T}^{(l,m+1)} \right)  \right)^{-1}\\
 &  \displaystyle \left( \sum_{t=1}^T \mathbf{W}_t^{(l)} \left(\mathbf{x}^{(l, m+1)}_{t-1|T} {\mathbf{x}'}^{(l, m+1)}_{t|T}+ \mathbf{x}^{(l, m+1)}_{t|T} {\mathbf{x}'}^{(l, m+1)}_{t-1|T} + 2 \mathbf{\Sigma}_{t-1,t|T}^{(l,m+1)}\right) \right).
\end{align}

\noindent \begin{minipage}{\columnwidth}
\begin{center}
\begin{algorithm}[H]
%\caption{SImultaneous DEconvolution, DEmixing and DEnoising of Compressive AutoRegressive models (SIDECAR)}
\caption{\small The Fast Compressible State-Space ({\sf FCSS}) Estimator}
\label{alg:DS3}
\begin{algorithmic}[1]
\Procedure{{\sf FCSS}}{}
\State Initialize: $\widehat{\mathbf{\Theta}}^{(0)}=\mathbf{0}$, $(\widehat{\mathbf{x}}_t^{(0)} = \mathbf{0})_{t=1}^T $, $(\mathbf{W}^{(0)}_t = \mathbf{0})_{t=1}^T$.
\Repeat
\State $l =0$ .
 %\Pisymbol{psy}{206} $N$
%\State \begin{align*}
%& Q\left((\mathbf{x}_t)_{t=1}^T,\mathbf{\Theta}\Big|(\widehat{\mathbf{x}}_t^{(l)})_{t=1}^T,\widehat{\mathbf{\Theta}}^{(l)}\right) =\\
%\notag -& {\lambda} \sum \limits_{j=1}^p \sum \limits_{t=1}^T \mathbb{E} \left\{ \left(\mathbf{u}_t\right)_j|(\widehat{\mathbf{x}}^{(l)}_t)_{t=1}^T, \widehat{•\mathbf{\Theta}}^{(l)} \right\}\frac{\left(\left(\mathbf{x}_t\right)_j-\mathbf{\Theta}_j \mathbf{x}_{t-1}\right)^2+\epsilon^2}{\sqrt{s_t}}\\
%\notag -& \sum_{t=1}^T \frac{1}{n_t}\frac{\|\mathbf{y}_t-\mathbf{A}_t\mathbf{x}_t\|_2^2}{2\sigma^2},
%\end{align*}
%\For{{$m = 0,\cdots, M-1$ or until convergence criterion met}}
\State \textbf{Outer E-step:}
\State $\mathbf{W}_t^{(l)} = {\sf diag} \Bigg \{ \frac{1}{\sqrt{s_t}\sqrt{\left( \widehat{\mathbf{x}}_t^{(l)} -{\widehat{\mathbf{\Theta}}}^{(l)} \widehat{\mathbf{x}}_{t-1}^{(l)}\right)_j^2+\epsilon^2}} \Bigg \}_{j=1}^p$.
\State \textbf{Outer M-step:}
\Repeat
\State $m =0$ .
\State \textbf{Inner E-Step:} Find the smoothed estimates  $\mathbf{x}^{(l,m+1)}_{{t|T}}$,  $\boldsymbol{\Sigma}^{(l,m+1)}_{{t|T}}$ and $\boldsymbol{\Sigma}^{(l,m+1)}_{{t-1,t|T}}$ using a Fixed Interval Smoother for
\begin{align*}
\left \{
\begin{array}{l}
\mathbf{x}_t = \widehat{\mathbf{\Theta}}^{(l,m)} \mathbf{x}_{t-1}+ \mathbf{w}_t, \quad \mathbf{w}_t\sim \mathcal{N}\left(\mathbf{0}, \frac{1}{\lambda} {\mathbf{W}_t^{(l)}}^{-1} \right) \\
\mathbf{y}_t = \mathbf{A}_t \mathbf{x}_t + \mathbf{v}_t, \quad \mathbf{v}_t\sim \mathcal{N}(\mathbf{0},n_t\sigma^2 I)
\end{array} \right. .
\end{align*}

% \begin{equation*}
%\widehat{\mathbf{x}}_t^{(l,m+1)} \leftarrow \mathbf{x}^{(l,m+1)}_{{t|T}}.
%\end{equation*}
\State \textbf{Inner M-Step:} Update $\widehat{\mathbf{\Theta}}^{(l, m+1)}$ via Eq. (\ref{eq:tv_theta_upd}).
\State $m \leftarrow m+1$.
\Until{convergence criteria met}
\State Update the estimates:
%\EndFor
\begin{equation*}
\widehat{\mathbf{\Theta}}^{(l+1)} \leftarrow \widehat{\mathbf{\Theta}}^{(l,m)}, \quad \left(\widehat{\mathbf{x}}_t^{(l+1)}\right)_{t=1}^{T} \leftarrow \left({\mathbf{x}}_{t|T}^{(l,m)}\right)_{t=1}^{T}.
\end{equation*}
\State $l \leftarrow l+1$.
\Until{convergence criteria met}

$\widehat{\mathbf{\Theta}} \leftarrow \widehat{\mathbf{\Theta}}^{(l)}.  \quad \left(\widehat{\mathbf{x}}_t \right)_{t=1}^{T} \leftarrow  \left(\widehat{\mathbf{x}}_t^{(l)}\right)_{t=1}^{T} $
\EndProcedure
\end{algorithmic}
\end{algorithm}
\end{center}
\end{minipage}

\medskip
\medskip
\medskip

\noindent \textbf{Remark 1.} 
The EM algorithm is known to converge linearly in the number of iterations \cite{mclachlan2007algorithm}. In applications of interest in this chapter, the stability assumption results in finite time constants and the EM algorithm can be assumed to converge in finite number of steps. As a result the complexity of the {\sf FCSS} algorithm is equal to complexity per EM iteration.
 By virtue of the FIS procedure, the compexity of the {\sf FCSS} algorithm is \emph{linear} in $T$, i.e., the observation duration. As we will show in Section \ref{sec:tv_sim}, this makes the {\sf FCSS} algorithm scale favorably when applied to long data sets.  Each iteration of the FIS algorithm requires two inversions which is of complexity $\mathcal{O}(p^3)$. Similarly, updating $\mathbf{\Theta}$ requires an inversion which is of complexity $\mathcal{O}(p^3)$. Altogether, complexity of the {\sf FCSS} algorithm amounts to $\mathcal{O}(p^3 T)$. For some applications of interest in this chapter, such as calcium deconvolution in the denoising regime and sleep spindle detection, the inversion is performed on a diagonal matrix and hence the complexity reduces to $\mathcal{O}(T)$.

\noindent \textbf{Remark 2.} In order to update $\mathbf{\Theta}$ in the inner M-step given by E. (\ref{eq:tv_theta_upd}), we have not specifically enforced the condition $\|\mathbf{\Theta}\| <1$ in the maximization step. This condition is required to obtain a convergent state transition matrix which results in the stability of the state dynamics. It is easy to verify that the set of matrices $\mathbf{\Theta}$ satisfying $\|\mathbf{\Theta}\| < 1- \eta$, is a closed convex set for small positive $\eta$, and hence one can perform the maximization in (\ref{eq:tv_theta_upd}) by projection onto this closed convex set. Alternatively, matrix optimization methods with operator norm constraints can be used \cite{mishra2013low}. We have avoided this technicality by first finding the global minimum and examining the largest eigenvalue. In the applications of interest in this chapter which follow next, the largest eigenvalue has always been found to be less than $1$.

\vspace{-2mm}
\section{Results}
\label{sec:tv_sim}
In this section, we study the performance of the {\sf FCSS} estimator on simulated data as well real data from two-photon calcium imaging recordings of neuronal activity and sleep spindle detection from EEG.

\vspace{-2mm}
\subsection{Application to Simulated Data}
We first apply the {\sf FCSS} algorithm to simulated data and compare its performance with the Basis Pursuit Denoising (BPDN) algorithm. The parameters are chosen as $p = 200, T = 200, s_1 =8,s_2 = 4,\epsilon = 10^{-10}$, and $\mathbf{\Theta} = 0.95 \mathbf{I}$. We define the quantity $1- n/p$ as the compression ratio. We refer to the case of $n_t=p$, i.e., no compression, as the denoising setting. The measurement matrix $\mathbf{A}$ is an $n_t \times p$ i.i.d. Gaussian random matrix, where $n_t$ is chosen such that $\frac{s_t}{n_t}$ is a fixed ratio. An initial choice of ${\lambda}  \geq 2\sqrt{2} \sigma \sqrt{\frac{\log p}{n_t}}$ is made inspired by the theory of LASSO \cite{Negahban}, which is further tuned using two-fold cross-validation.

Figures \ref{fig:tv_ds}--(a) and \ref{fig:tv_ds}--(b) show the estimated states as well as the innovations for different compression ratios for one sample component. In the denoising regime, all the innovations (including the two closely-spaced components) are exactly recovered. As we take fewer measurements, the performance of the algorithm degrades as expected. However, the overall structure of the innovation sequence is captured even for highly compressed measurements.
\begin{figure}[htb!]
\vspace{-2mm}
\centering     %%% not \center
\includegraphics[width=1\columnwidth]{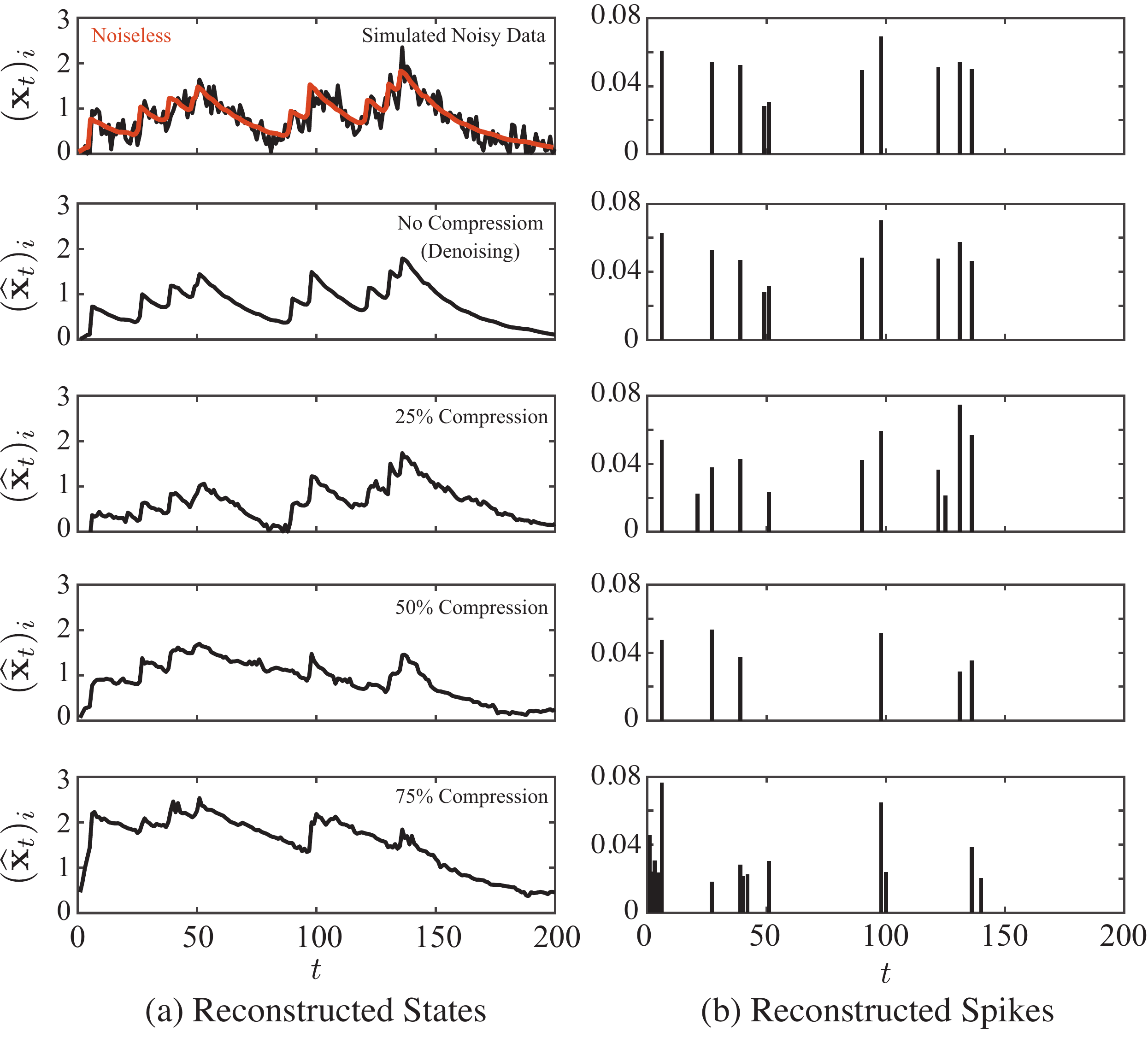}
\vspace{-2mm}
\caption{Reconstruction results of {\sf FCSS} on simulated data vs. compression levels. (a) reconstructed states, (b) reconstructed spikes. The {\sf FCSS} estimates degrade gracefully as the compression level increases.}\label{fig:tv_ds}
\vspace{-1mm}
\end{figure}

\begin{figure}[htb!]
\vspace{-2mm}
\centering     %%% not \center
\includegraphics[width=.7\columnwidth]{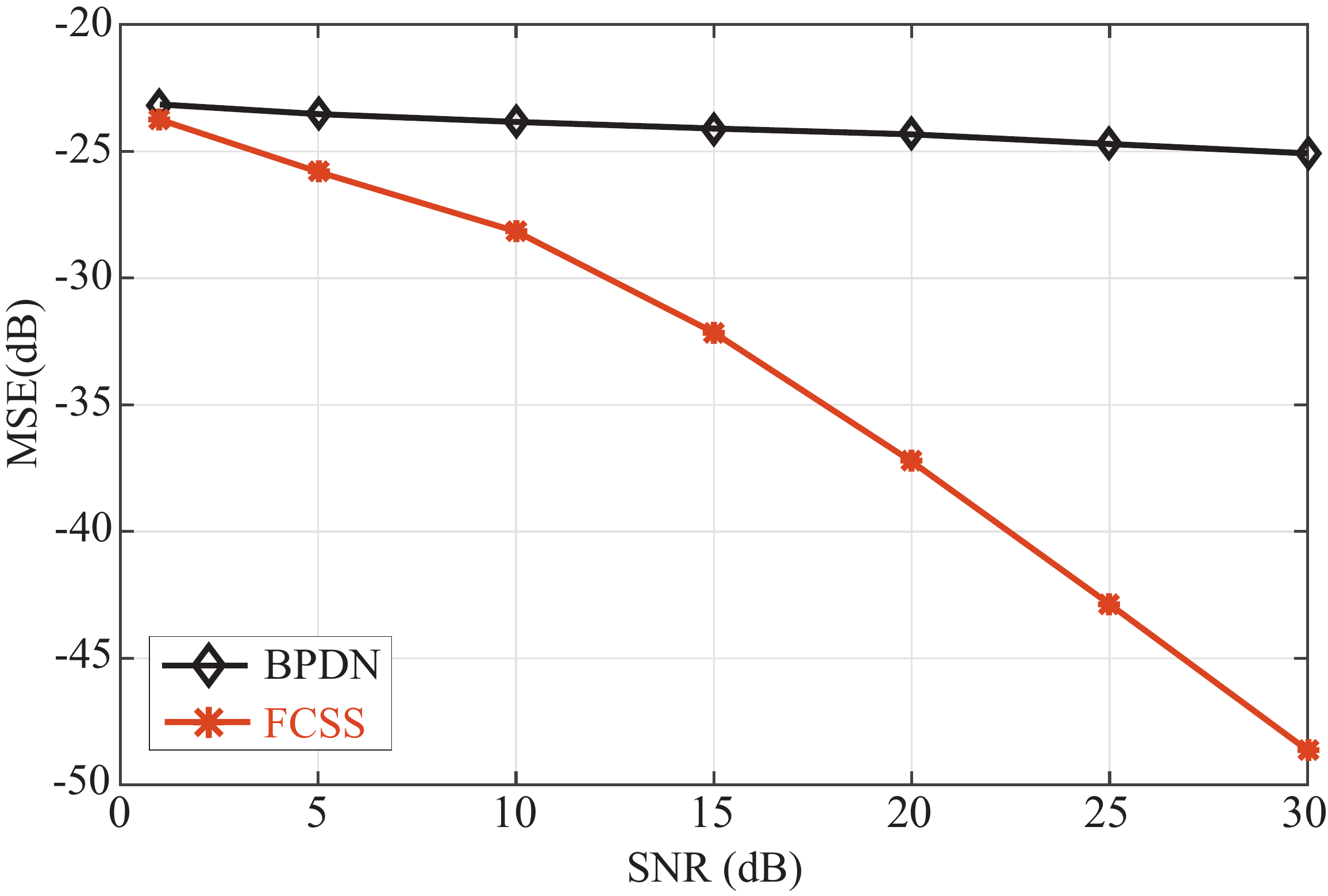}
\vspace{-2mm}
\caption{MSE vs. SNR comparison betwee {\sf FCSS} and BPDN. The {\sf FCSS} significantly outperforms the BPDN, even for moderate SNR levels.}
\label{fig:tv_mse_comp}
\vspace{-2mm}
\end{figure}

\begin{figure}[htb!]
\centering     %%% not \center
\includegraphics[width=.7\columnwidth]{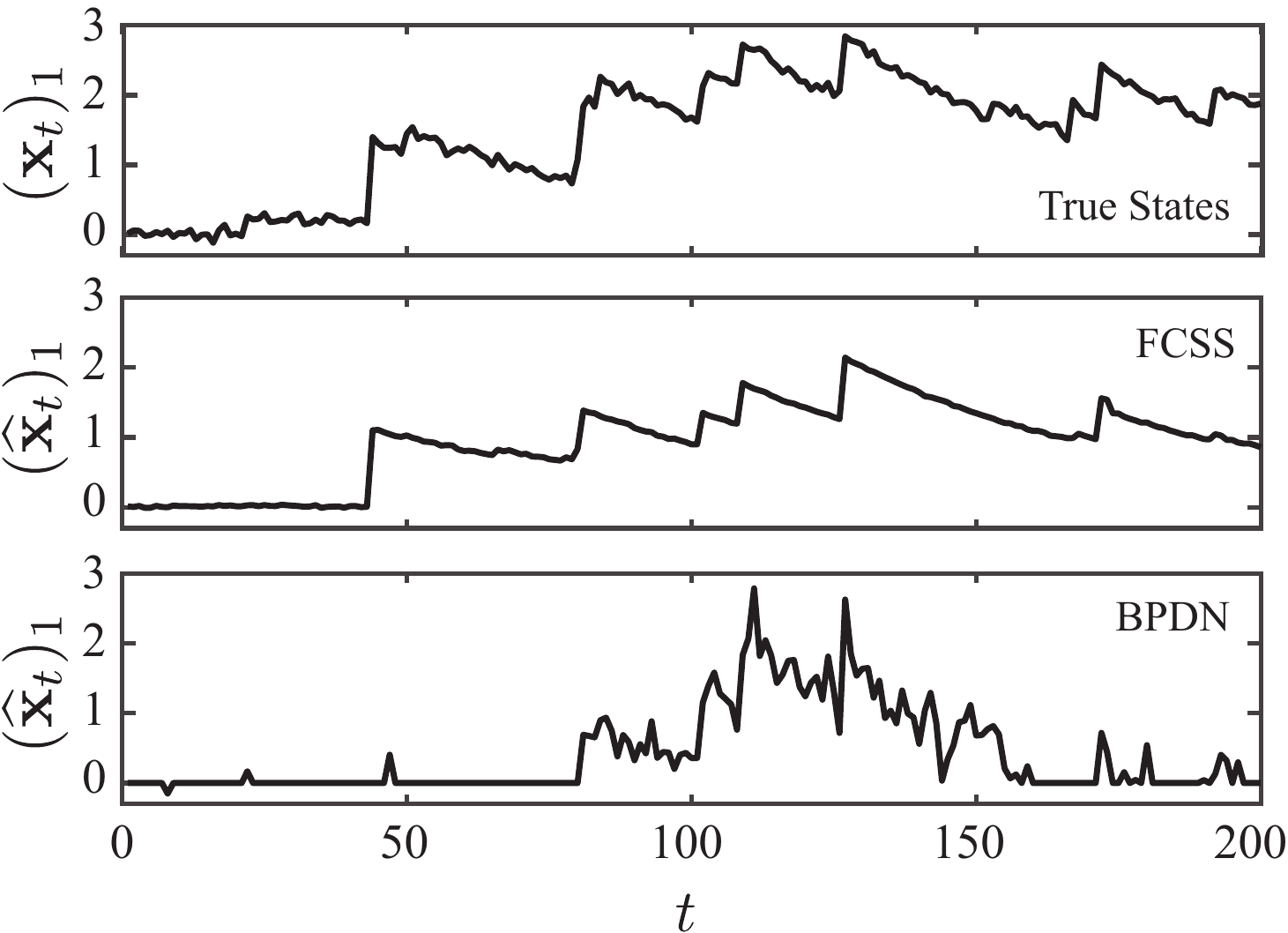}
\vspace{-2mm}
\caption{Example of state reconstruction results for {\sf FCSS} and BPDN. Top: true states, Middle: {\sf FCSS} state estimates, Bottom: BPDN estimates. The {\sf FCSS} reconstruction closely follows the true state evolution, while the BPDN fails to capture the state dynamics.}
\label{fig:tv_dcs_vs_ccs}
%\vspace{-3mm}
\end{figure}

\begin{figure}[htb!]
\centering     %%% not \center
\includegraphics[width=0.9\columnwidth]{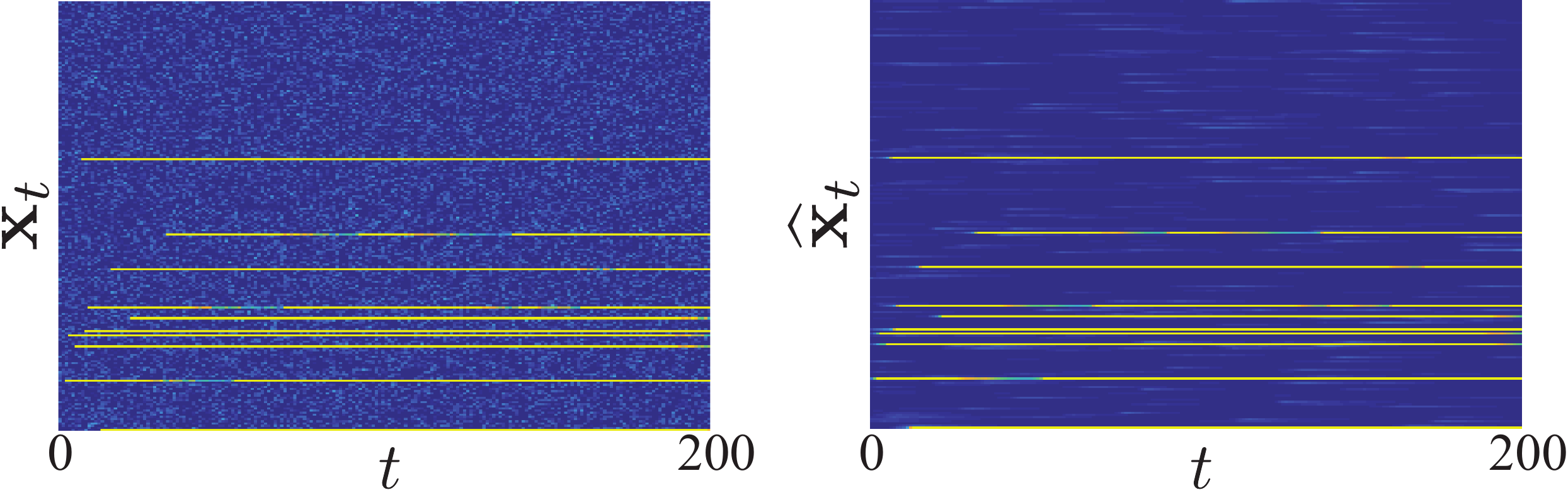}
\caption{Raster plots of the observed and estimated states via {\sf FCSS}. Left: noisy observations, Right: {\sf FCSS} estimates. The {\sf FCSS} significantly denoises the observed states.}
\label{fig:imagesc_sim}
\vspace{-4mm}
\end{figure}

Figure \ref{fig:tv_mse_comp} shows the MSE comparison of the {\sf FCSS} vs. BPDN, where the MSE is defined as $10\log_{10}\frac{1}{T}\sum_{t=1}^T \|\widehat{\mathbf{x}}_t-\mathbf{x}_t\|_2^2$. The {\sf FCSS} algorithm significantly outperforms BPDN, especially at high SNR values. Figure \ref{fig:tv_dcs_vs_ccs} compares the performance of {\sf FCSS} and BPDN on a sample component at a compression level of $n/p = 2/3$, in order to visualize the performance gain implied by Figure \ref{fig:tv_dcs_vs_ccs}.

Finally, Figure \ref{fig:imagesc_sim} shows the comparison of the estimated states for the entire simulated data in the denoising regime. As can be observed from the figure, the sparsity pattern of the states and innovations are captured while significantly denoising the observed states.

%Figure \ref{fig:tv_raster_simulated} compares the estimated Spikes and the Ground truth.
%\begin{figure}[htb!]
%\centering     %%% not \center
%\includegraphics[width=60mm]{tv_raster_chapter}
%\caption{Raster Plot of the Simulated Dynamics vs. the Estimated Spikes}
%\label{fig:tv_raster_simulated}
%\end{figure}

\vspace{-3mm}
\subsection{Application to Calcium Signal Deconvolution}

Calcium imaging takes advantage of intracellular calcium flux to directly visualize calcium signaling in living neurons. This is done by using calcium indicators, which are fluorescent molecules that can respond to the binding of calcium ions by changing their fluorescence properties and using a fluorescence or two-photon microscope and a CCD camera to capture the visual patterns \cite{smetters1999detecting, stosiek2003vivo}. Since spikes are believed to be the units of neuronal computation, inferring spiking activity from calcium recordings, referred to as calcium deconvolution, is an important problem in neural data analysis. Several approaches to calcium deconvolution have been proposed in the neuroscience literature, including model-free approaches such as sequential Monte Carlo methods \cite{vogelstein2009spike} and model-based approaches such as non-negative deconvolution methods \cite{vogelstein2010fast, pnevmatikakis2016simultaneous}. These approaches require solving convex optimization problems, which do not scale well with the temporal dimension of the data. In addition, they lack theoretical performance guarantees and do not provide clear measures for assessing the statistical significance of the detected spikes.

In order to construct confidence bounds for our estimates, we employ recent results from high-dimensional statistics \cite{van2014asymptotically}. We first compute the confidence intervals around the outputs of the {\sf FCSS} estimates using the node-wise regression procedure of \cite{van2014asymptotically}, at a confidence level of $1-\frac{\alpha}{2}$. We perform the node-wise regression separately for each time $t$. For an estimate $\widehat{\mathbf{x}}_t$, we obtain $\widehat{\mathbf{x}}_t^{\sf u}$ and $\widehat{\mathbf{x}}_t^{\sf l}$ as the upper and lower confidence bounds, respectively. Next, we partition the estimates into small segments, starting with a local minimum (trough) and ending in a local maximum (peak). For the $i^{\sf th}$ component of the estimate, let $t_{\sf min}$ and $t_{\sf max}$ denote the time index corresponding to two such consecutive troughs and peaks. If the difference $(\widehat{\mathbf{x}}_{t_{\sf max}}^{\sf l})_i - (\widehat{\mathbf{x}}_{t_{\sf min}}^{\sf u})_i$ is positive, the detected innovation component is declared significant (i.e., spike) at a confidence level of $1- \alpha$, otherwise it is discarded (i.e., no spike). We refer to this procedure as Pruned-{\sf FCSS} ({\sf PFCSS}).

\begin{figure*}[t!]
\centering     %%% not \center
\includegraphics[width=1\textwidth]{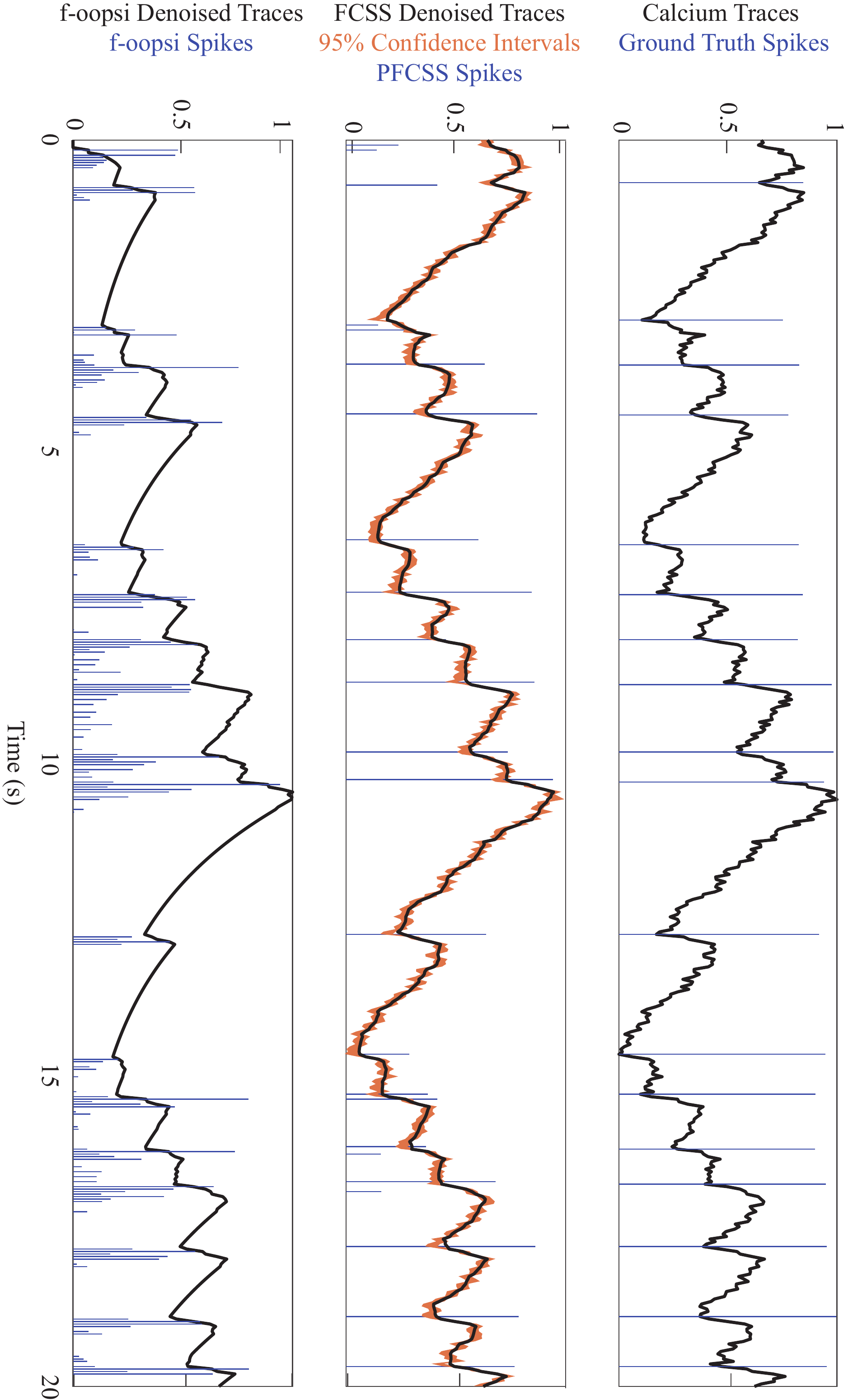}
\caption{Ground-truth performance comparison between {\sf PFCSS} and constrained f-oopsi. Top: the observed calcium traces (black) and ground-truth electrophysiology data (blue), Middle: {\sf PFCSS} state estimates (black) with $95\%$ confidence intervals (orange) and the detected spikes (blue), Bottom: the constrained f-oopsi state estimates (black) and the detected spikes (blue). The {\sf FCSS} spike estimates closely match the ground-truth spikes with only a few false detections, while the constrained f-oopsi estimates contain significant clustered false detections.}
\label{fig:tv_cal_simul_ss_vs_foopsi}
\end{figure*}

\begin{figure*}[htb!]
\centering     %%% not \center
\includegraphics[width=1\textwidth]{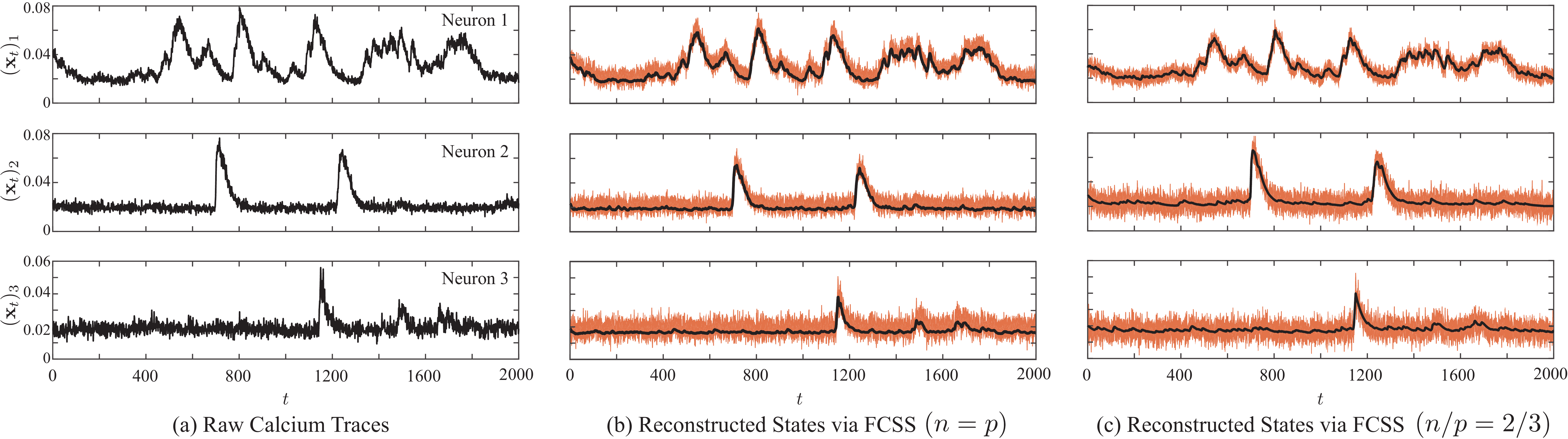}
\caption{Performance of {\sf FCSS} on large-scale calcium imaging data. Top, middle and bottom rows correspond to three selected neurons labeled as Neuron 1, 2, and 3, respectively. (a) raw calcium traces, (b) {\sf FCSS} reconstruction with no compression, (c) {\sf FCSS} reconstruction with $2/3$ compression ratio. Orange hulls show $95\%$ confidence intervals. The {\sf FCSS} significantly denoises the observed traces in both the uncompressed and compressed settings.\label{fig:tv_cal}}
\end{figure*}

\begin{figure*}[htb!]
\centering     %%% not \center
\includegraphics[width=1\textwidth]{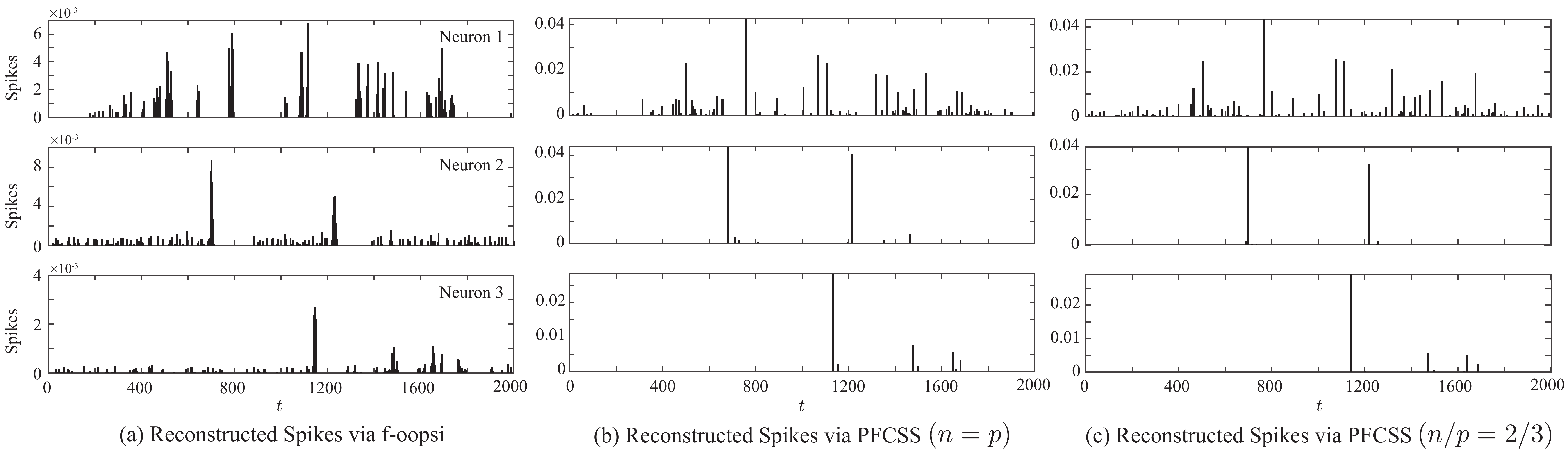}
\caption{Reconstructed spikes of {\sf PFCSS} and constrained f-oopsi from large-scale calcium imaging data. Top, middle and bottom rows correspond to three selected neurons labeled as Neuron 1, 2, and 3, respectively. (a) constrained f-oopsi spike estimates, (b) {\sf PFCSS} spike estimates with no compression, (c) {\sf PFCSS} spike estimates with $2/3$ compression ratio. The {\sf PFCSS} estimates in both the uncompressed and compressed settings are sparse in time, whereas the constrained f-oopsi estimates are in the form of clustered spikes. \label{fig:tv_cal_spikes}}
\end{figure*}

\begin{figure}[htb!]
\centering     %%% not \center
\includegraphics[width=1\columnwidth]{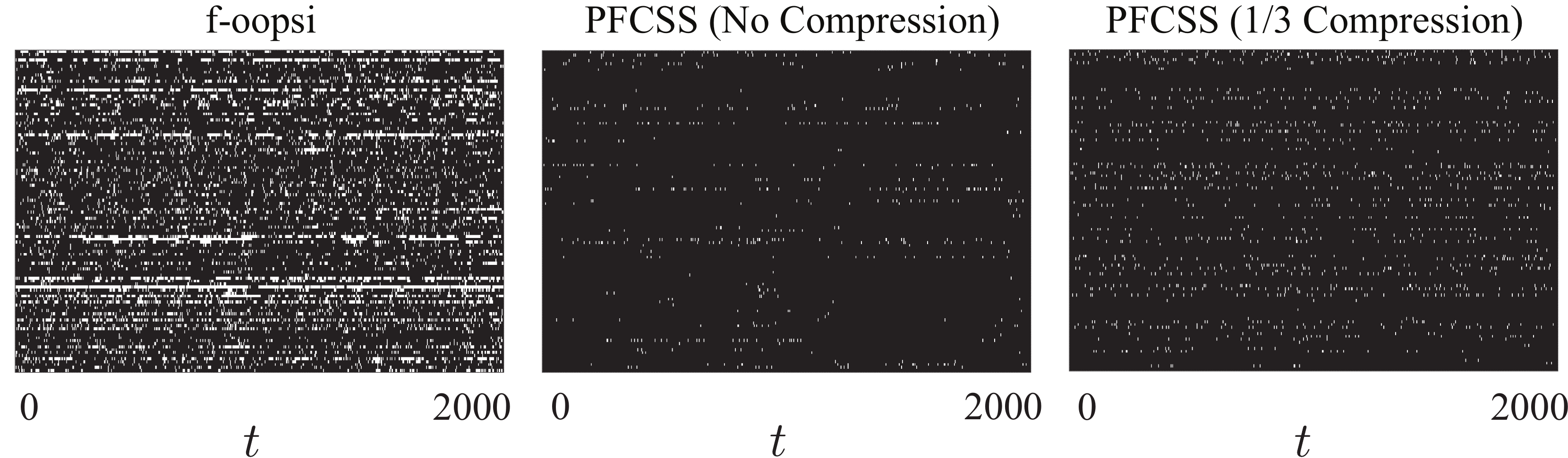}
\caption{Raster plot of the estimated spikes from large-scale calcium imaging data. Left: the constrained f-oopsi estimates, Middle: {\sf PFCSS} estimates with no compression, Right: {\sf PFCSS} estimates with a $\frac{1}{3}$ compression ratio. The {\sf PFCSS} estimates are spatiotemporally sparse, whereas the constrained f-oopsi outputs temporally clustered spike estimates.}
\label{fig:tv_raster_cal}
\end{figure}

We first apply the {\sf FCSS} algorithm for calcium deconvolution in a scenario where the ground-truth spiking is recorded \emph{in vitro} through simultaneous electrophysiology (cell-attached patch clamp) and two-photon calcium imaging. The calcium trace as well as the ground-truth spikes are shown for a sample neuron in Figure \ref{fig:tv_cal_simul_ss_vs_foopsi}--(a). The {\sf FCSS} denoised estimate of the states (black) and the detected spikes (blue) using $95\%$ confidence intervals (orange hulls) and the corresponding quantities for the constrained f-oopsi algorithm \cite{pnevmatikakis2016simultaneous} are shown in Figures \ref{fig:tv_cal_simul_ss_vs_foopsi}--(b) and --(c), respectively. Both algorithms detect the large dynamic changes in the data, corresponding to the spikes, which can also be visually captured in this case. However, in doing so, the f-oopsi algorithm incurs a high rate of false positive errors, manifested as clustered spikes around the ground truth events. Similar to f-oopsi, most state-of-the-art model-based methods suffer from high false positive rate, which makes the inferred spike estimates unreliable. Thanks to the aforementioned pruning process based on the confidence bounds, the {\sf PFCSS} is capable of rejecting the insignificant innovations, and hence achieve a lower false positive rate. One factor responsible for this performance gap can be attributed to the underestimation of the calcium decay rate in the transition matrix estimation step of f-oopsi. However, we believe the performance gain achieved by {\sf FCSS} is mainly due to the explicit modeling of the sparse nature of the spiking activity by going beyond the Gaussian state-space modeling paradigm. 

Next, we apply the {\sf FCSS} algorithm to large-scale \emph{in vivo} calcium imaging recordings, for which the ground-truth is not available due to measurement constraints. The data used in our analysis was recorded from $219$ spontaneously active neurons in mouse auditory cortex. The two-photon microscope operates at a rate of $30$ frames per second. We chose $T=2000$ samples corresponding to $1$ minute for the analysis. We chose $p=108$ well-separated neurons visually. We estimate the measurement noise variance by appropriate re-scaling of the power spectral density in the high frequency bands where the signal is absent. We chose a value of $\epsilon = 10^{-10}$. It is important to note that estimation of the measurement noise variance is critical, since it affects the width of the confidence intervals and hence the detected spikes. Moreover, we estimate the baseline fluorescence by averaging the signal over values within a factor of $3$ standard deviations of the noise. By inspecting Eq. (\ref{eq:tv_prob_def_dual}), one can see a trade-off between the choice of $\lambda$ and the estimate of the observation noise variance $\sigma^2$. We have done our analysis in both the compression regime, with a compression ratio of $1/3$ ($n/p = 2/3$), and the denoising regime. The measurements in the compressive regime were obtained from applying i.i.d. Gaussian random matrices to the observed calcium traces. The latter is done to motivate the use of compressive imaging, as opposed to full sampling of the field of view.

Figure \ref{fig:tv_cal}--(a) shows the observed traces for four selected neurons. The reconstructed states using {\sf FCSS} in the compressive and denoising regimes are shown in Figures \ref{fig:tv_cal}--(b) and --(c), respectively. The $90\%$ confidence bounds are shown as orange hulls. The {\sf FCSS} state estimates are significantly denoised while preserving the calcium dynamics. Figure \ref{fig:tv_cal_spikes} shows the detected spikes using constrained f-oopsi and {\sf PFCSS} in both the compressive and denoising regimes. Finally, Figure \ref{fig:tv_raster_cal} shows the corresponding raster plots of the reconstructed spikes for the entire ensemble of neurons. Similar to the preceding application on ground-truth date, the f-oopsi algorithm detects clusters of spikes, whereas the {\sf PFCSS} procedure results in sparser spike detection. This results in the detection of seemingly more active neurons in the raster plot. However, motivated by the foregoing ground-truth analysis, we believe that a large fraction of these detected spikes may be due to false positive errors. Strikingly, even with a compression ratio of $1/3$ the performance of the {\sf PFCSS} is similar to the denoising case. The latter observation corroborates the feasibility of compressed two-photon imaging, in which only a random fraction of the field of view is imaged, which in turn can result in higher acquisition rates.

In addition to the foregoing discussion on the comparisons in Figures \ref{fig:tv_cal_simul_ss_vs_foopsi}, \ref{fig:tv_cal}, \ref{fig:tv_cal_spikes}, and \ref{fig:tv_raster_cal}, two remarks are in order. First, the iterative solution at the core of {\sf FCSS} is linear in the observation length and hence significantly faster than the batch-mode optimization procedure used for constrained f-oopsi. Our comparisons suggest that the {\sf FCSS} reconstruction is at least 3 times faster than f-oopsi for moderate data sizes of the order of tens of minutes. Moreover, the vector formulation of {\sf FCSS} allows for easy parallelization {(without the need for GPU implementations)}, which allows simultaneous processing of ROI's without losing speed. {As a numerical example the results of Figure \ref{fig:tv_cal}--(b) took an average of 60-70 seconds to calculate for all $p=108$ ROI's and $T=2000$ frames on an Apple Macintosh desktop computer}. Second, using only about two-thirds of the measurements achieves similar results by {\sf FCSS} as using the full measurements.

\vspace{-2mm}
\subsection{Application to Sleep Spindle Detection}
In this section we use compressible state-space models in order to model and detect sleep spindles. A sleep spindle is a burst of oscillatory brain activity manifested in the EEG that occurs during stage 2 non-rapid eye movement (NREM) sleep. It consists of stereotypical $12$--$14~\text{Hz}$ wave packets that last for at least $0.5$ seconds \cite{de2003sleep}. The spindles occur with a rate of $2$--$5\%$ in time, which makes their generation an appropriate candidate for compressible dynamics. Therefore, we hypothesize that the spindles can be modeled using a combination of few echoes of the response of a second order compressible state-space model. As a result, the spindles can be decomposed as sums of modulated sine waves.

\begin{figure}[b!]
\centering     %%% not \center
\includegraphics[width=1\columnwidth]{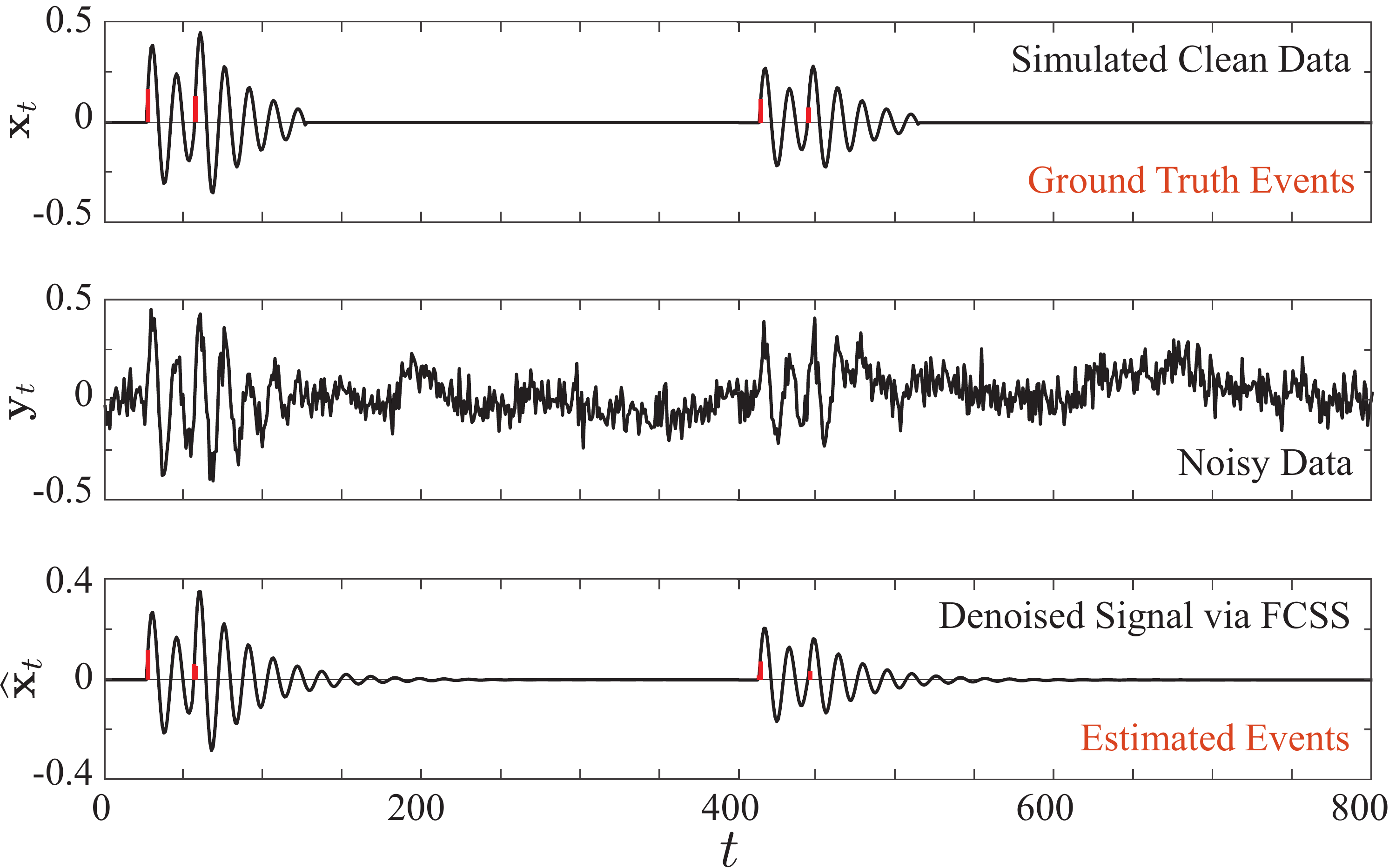}
\caption{Performance of {\sf FCSS} on simulated spindles. Top: simulated clean data (black) and ground-truth spindle events (red), Middle: simulated noisy data, Bottom: the denoised signal (black) and deconvolved spindle events (red). The {\sf FCSS} estimates are significantly denoised and closely match the ground-truth data shown in the top panel.}
\label{fig:tv_spindle_sim}
%\vspace{-5mm}
\end{figure}

\begin{figure*}[t!]
\centering     %%% not \center
\includegraphics[width=1\textwidth]{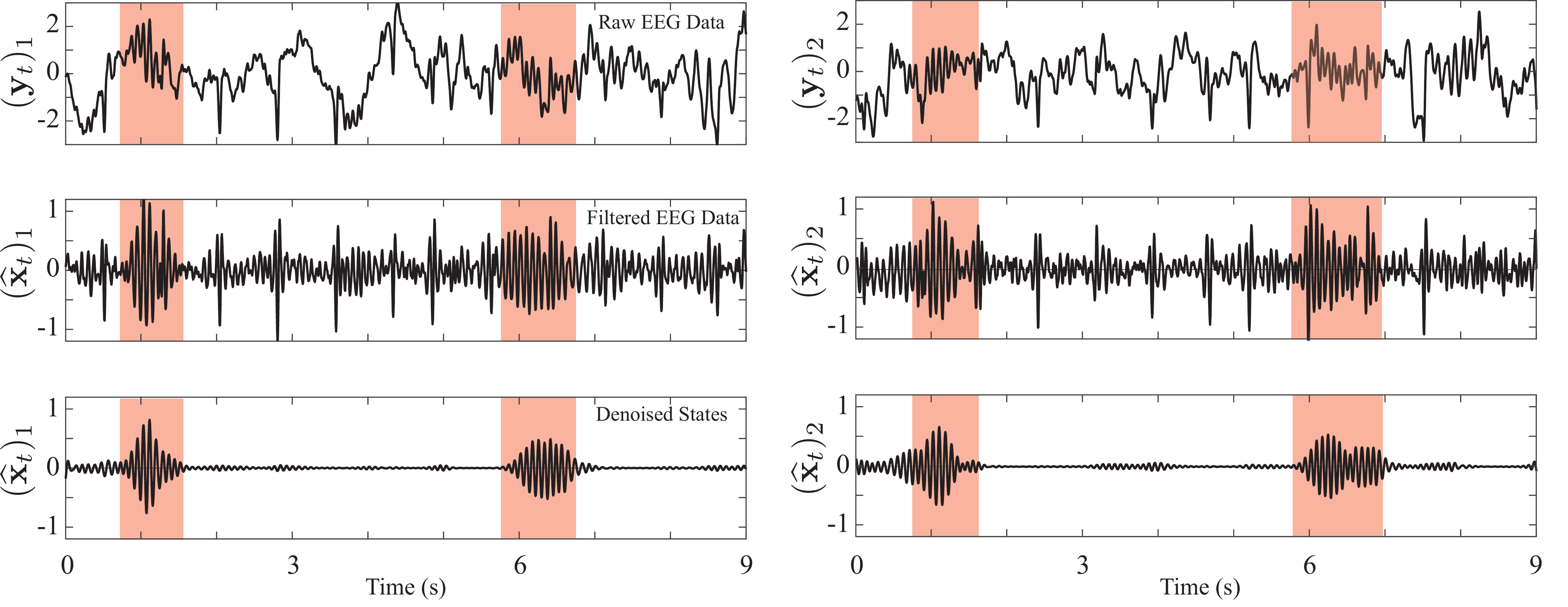}
\caption{Performance comparison between {\sf FCSS} and band-pass filtered EEG data. Left and right panels correspond to two selected electrodes labeled as 1 and 2, respectively. The orange blocks show the extent of the detected spindles by the expert. Top: raw EEG data, Middle: band-pass filtered EEG data in the $12$--$14~\text{Hz}$ band, Bottom: {\sf FCSS} spindle estimates. The {\sf FCSS} estimates closely match the expert annotations, while the band-pass filtered data contains significant signal components outside of the orange blocks.}
\label{fig:tv_spindle_dream}
\vspace{-3mm}
\end{figure*}

\begin{figure}[t!]
\centering     %%% not \center
\includegraphics[width=0.5\columnwidth]{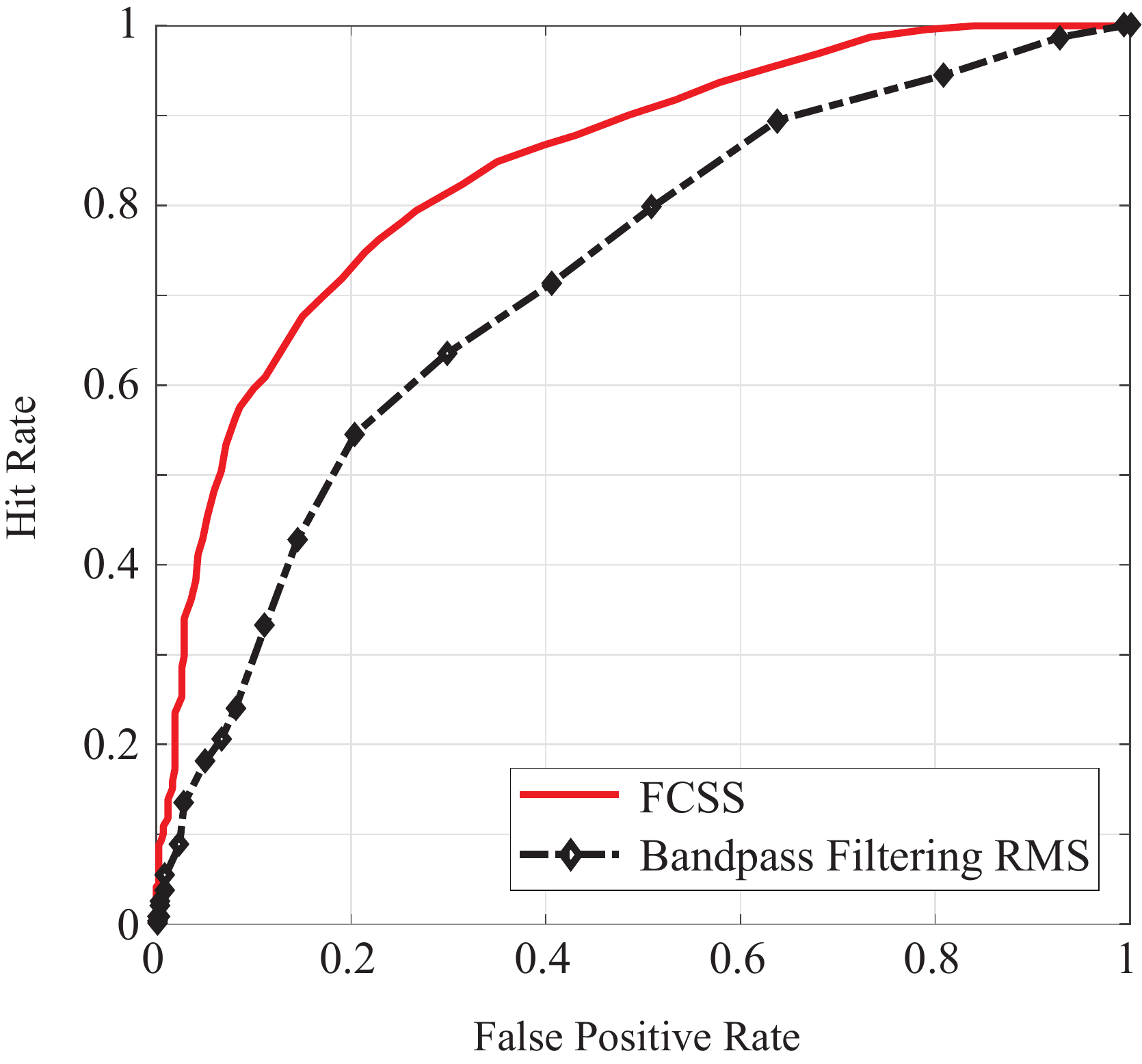}
\caption{The ROC curves for {\sf FCSS} (solid red) and bandpass-filtered RMS (dashed black). The {\sf FCSS} outperforms the widely-used band-pass filtered RMS method as indicated by the ROC curves.}
\label{fig:tv_spindle_roc}
\end{figure}

In order to model the oscillatory nature of the spindles, we consider a second order autoregressive (AR) model where the pole locations are given by $a e^{-j2\pi \frac{f}{f_s}}$ and $a e^{+j2\pi \frac{f}{f_s}}$, where $0<a<1$ is a positive constant controlling the effective duration of the impulse response, $f_s$ is the sampling frequency and $f$ is a putative frequency accounting for the dominant spindle frequency. The equivalent state-space model for which the MAP estimation admits the {\sf FCSS} solution is therefore:
%\vspace{-1mm}
\begin{equation}
\label{eq:tv_lap_state_space_spindle}
\begin{array}{ll}
\mathbf{x}_t = 2a\cos\left(\textstyle 2\pi \frac{f}{f_s}\right) \mathbf{x}_{t-1}-a^2 \mathbf{x}_{t-2}+ \mathbf{w}_t,\\
\mathbf{y}_t = \mathbf{A}_t \mathbf{x}_t + \mathbf{v}_t.
\end{array}
\vspace{-1mm}
\end{equation}

Note that impulse response of the state-space dynamics is given by $h_n  = a^n \cos\left(\textstyle 2\pi \frac{f}{f_s} n\right)u_n$, which is a stereotypical decaying sinusoid. By defining the augmented state $\widetilde{\mathbf{x}}_t = [\mathbf{x}_t', \mathbf{x}_{t-1}']'$, 
Eq. (\ref{eq:tv_lap_state_space_spindle}) can be expressed in the canonical form:
\begin{equation}
\label{eq:tv_lap_state_space_spindle_canonical}
\begin{array}{l}
\widetilde{\mathbf{x}}_t = \widetilde{\mathbf{\Theta}} \widetilde{\mathbf{x}}_{t-1}+ \widetilde{\mathbf{w}}_t, \qquad \mathbf{y}_t = \widetilde{\mathbf{A}}_t \widetilde{\mathbf{x}}_t + \mathbf{v}_t,
\end{array}
\end{equation}
where $\widetilde{\mathbf{w}}_t := [\mathbf{w}'_t
, \mathbf{0}']'$ , $\widetilde{\mathbf{A}}_t = \begin{array}{c:c}
[\mathbf{A}_t & \mathbf{0}]
\end{array}$ and
\[\widetilde{\mathbf{\Theta}} = \left[\begin{array}{c:c}
2a\cos\left(\textstyle 2\pi \frac{f}{f_s}\right)\mathbf{I} & -a^2 \mathbf{I}\\
\hdashline
\mathbf{I} & \mathbf{0}
\end{array}
\right].
\]

Eq. (\ref{eq:tv_Q_calculated}) can be used to update $\widetilde{\mathbf{\Theta}}$ in the M step of the {\sf FCSS} algorithm. However, $\widetilde{\mathbf{\Theta}}$ has a specific structure in this case, determined by $a$ and $f$, which needs to be taken into account in the optimization step. Let $\phi =: 2a\cos\left(2\pi \frac{f}{f_s}\right)$ and $\psi = a^2$, and let
\[
\sum_{t=1}^T   \widetilde{\mathbf{W}}_t^{(l)} \left ({\widetilde{\mathbf{x}}}^{(l, m+1)}_{t-1|T}{\operatorname{\widetilde{\mathbf{x}}}'}^{(l, m+1)}_{t-1|T} + \widetilde{\mathbf{\Sigma}}_{t-1|T}^{(l,m+1)} \right ) =:  \left [ \begin{tabular}{c:c} $\mathbf{A}$ & $\mathbf{B}$\\\hdashline$\mathbf{C}$ & $\mathbf{D}$ \end{tabular} \right],
\]
\begin{align}
\nonumber & \displaystyle \sum_{t=1}^T \widetilde{\mathbf{W}}_t^{(l)} \left(\widetilde{\mathbf{x}}^{(l, m+1)}_{t-1|T} {\operatorname{\widetilde{\mathbf{x}}}'}^{(l, m+1)}_{t|T}+ \widetilde{\mathbf{x}}^{(l, m+1)}_{t|T} {\operatorname{\widetilde{\mathbf{x}}}'}^{(l, m+1)}_{t-1|T} + 2 \widetilde{\mathbf{\Sigma}}_{t-1,t|T}^{(l,m+1)}\right) =: \left [ \begin{tabular}{c:c} $\mathbf{E}$ & $\mathbf{F}$\\\hdashline$\mathbf{G}$ & $\mathbf{H}$ \end{tabular} \right]\!.
\end{align}
\vspace{-2mm}

\noindent Then, Eq. (\ref{eq:tv_Q_calculated}) is equivalent to maximizing
\begin{align}
\label{eq:tv_estimate_af}
\maximize_{\phi,\psi} \ \ & \frac{\lambda}{2} (\phi^2 \text{Tr}(\mathbf{A}) \!-\! \phi \psi \text{Tr} (\mathbf{B}+\mathbf{C}) + \psi^2 \text{Tr} (\mathbf{D}) + \text{Tr}(\mathbf{A})) - {\textstyle \frac{\lambda}{2}} (\phi \text{Tr}(\mathbf{E}) - \psi \text{Tr}(\mathbf{G}) + \text{Tr}(\mathbf{F})). 
\end{align}
subject to $0 \le \psi \le 1$ and $\phi^2 \le 4 \psi$, which can be solved using interior point methods. In our implementation, we have imposed additional priors of $f \sim {\sf Uniform} (12, 14)$ Hz and $a \sim {\sf Uniform} (0.95,0.99)$, which simplifies the constraints on $\phi$ and $\psi$ to 
\[
0.95^2 \leq \psi \leq 0.99^2, \ \ \text{and} \ \  4\psi \cos^2\left(\textstyle 2\pi \frac{12}{f_s}\right)\leq \phi^2 \leq 4\psi \cos^2\left(\textstyle 2\pi \frac{14}{f_s}\right).
\]
 Given the convexity of the cost function in (\ref{eq:tv_estimate_af}), one can conclude from the KKT conditions that if the global minimum is not achieved inside the region of interest it must be achieved on the boundaries.

%\[
%\text{Tr} (xE + yG + F) [E, F; G, H])
%Tr {(x^2A + xyB + xyC + y^2D) + C}
%\]
%\[
%2xA + y(B + C) = E
%x(B+C) +2yD = G
%\]

Figure \ref{fig:tv_spindle_sim}, top panel, shows two instances of simulated spindles (black traces) with parameters $f_s = 200$ Hz, $f = 13$ Hz and $a = 0.95$, with the ground truth events generating the wave packets shown in red. The middle panel shows the noisy version of the data with an SNR of $-7.5$ dB. The noise was chosen as white Gaussian noise plus slowly varying (2 Hz) oscillations to resemble the slow oscillations in real EEG data. As can be observed the simulated signal exhibits visual resemblance to real spindles, which verifies the hypothesis that spindles can be decomposed into few combinations of wave packets generated by a second order AR model. The third panel, shows the denoised data using {\sf FCSS}, which not only is successful in detecting the the ground-truth dynamics (red bars), but also significantly denoises the data.

%\begin{figure*}[htb!]
%\centering     %%% not \center
%\includegraphics[width=87mm]{tv_spindle_dream_chapter}
%\caption{Real EEG Data from Dream Dataset and Estimated Spindles}
%\label{fig:tv_spindle_dream}
%\end{figure*}

We next apply {\sf FCSS} to real EEG recordings from stage-2 NREM sleep. Manual scoring of sleep spindles can be very time-consuming, and achieving accurate manual scoring on a long-term recording is a highly demanding task with the associated risk of decreased diagnosis. Although automatic spindle detection would be attractive, most available algorithms sensitive to variations in spindle amplitude and frequency that occur between both subjects and derivations, reducing their effectiveness \cite{nonclercq2013sleep, schimicek1994automatic}. Moreover most of these algorithms require significant pre- and post-processing and manual tuning. Examples include algorithms based on Empirical Mode Decomposition (EMD) \cite{yang2006detection, fraiwan2012automated, causa2010automated}, data-driven Bayesian methods \cite{babadi2012diba}, and machine learning approaches \cite{duman2009efficient, ventouras2005sleep}. Unlike our approach, none of the existing methods consider modeling the generative dynamics of spindles, as transient sparse events in time, in the detection procedure. 

The data used in our analysis is part of the recordings in the DREAMS project \cite{devuyst2011automatic}, recorded using a 32-channel polysomnograpgh. We have used the EEG channels in our analysis. The data was recorded at a rate of $f_s=200$ Hz for 30 minutes. The data was scored for sleep spindles independently by two experts. We have used expert annotations to separate regions which include spindles for visualization purposes. For comparison purposes, we use a bandpass filtered version of the data within $12$--$14~\text{Hz}$, which is the basis of several spindle detection algorithms \cite{schimicek1994automatic,clemens2005overnight, huupponen2007development}, hallmarked by the widely-used bandpass filtered root-mean-square (RMS) method \cite{schimicek1994automatic}.

Figure \ref{fig:tv_spindle_dream} shows the detection results along with the bandpass filtered version of the data for two of the EEG channels. The red bars show the expert markings of the onset and offset of the spindle events. The {\sf FCSS} simultaneously captures the spindle events and suppressed the activity elsewhere, whereas the bandpass filtered data produces significant activity in the $12$--$14~\text{Hz}$ throughout the observation window, resulting in high false positives. To quantify this observation, we have computed the ROC curves of the {\sf FCSS} and bandpass filtering followed by root mean square (RMS) computation in Figure \ref{fig:tv_spindle_roc}, which confirms the superior performance of the {\sf FCSS} algorithm over the data set. The annotations of one of the experts has been used for as the ground truth benchmark.

\vspace{-2mm}
\section{Discussion}
\label{sec:disc}
In this section, we discuss the implication of our techniques in regard to the application domains as well as existing methods.

\vspace{-3mm}
\subsection{Connection to existing literature in sparse estimation}
Contrary to the traditional compressive sensing, our linear measurement operator does not satisfy the RIP \cite{ba2012exact}, despite the fact that $\mathbf{A}_t$'s satisfy the RIP. Nevertheless, we have extended the near-optimal recovery guarantees of CS to our compressible state-space estimation problem via Theorem \ref{thm:tv_main}. Closely related problems to our setup are the super-resolution and sparse spike deconvolution problems \cite{candes2014towards,duval2015exact}, in which abrupt changes with minimum separation in time are resolved in fine scales using coarse (lowpass filtered) frequency information, which is akin to working in the compressive regime.

Theoretical guarantees of CS require the number of measurements to be roughly proportional to the sparsity level for stable recovery \cite{Negahban}. These results do not readily generalize to the cases where the sparsity lies in the dynamics, not the states per se. Most of the dynamic compressive sensing techniques such as Kalman filtered compressed sensing, assume partial information about the support or estimate them in a greedy and often ad-hoc fashion \cite{vaswani2010ls, vaswani2008kalman, carmi2010methods, ziniel2013dynamic, zhan2015time}. As another example, the problem of recovering discrete signals which are approximately sparse in their gradients using compressive measurements, has been studied in the literature using Total Variation (TV) minimization techniques \cite{needell2013stable,poon2015role}. For one-dimensional signals, since the gradient operator is not orthonormal, the Frobenius operator norm of its inverse grows linearly with the discretization level \cite{needell2013stable}. Therefore, stability results of TV minimization scale poorly with respect to discretization level. In higher dimensions, however, the fast decay of Haar coefficients allow for near-optimal theoretical guarantees \cite{cohen1999nonlinear}. A major difference of our setup with those of CS for TV-minimization is the \emph{structured} and \emph{causal} measurements, which unlike the non-causal measurements in \cite{needell2013stable}, do not result in an overall measurement matrix satisfying RIP. We have considered dynamics with convergent transition matrices, in order to generalize the TV minimization approach. To this end, we showed that using the state-space dynamics one can infer temporally global information from local and causal measurements. Another closely related problem is the fused lasso \cite{tibshirani2005sparsity} in which sparsity is promoted both on the covariates and their differences. 

\vspace{-3mm}
\subsection{Application to calcium deconvolution} In addition to scalability and the ability to detect abrupt transitions in the states governed by discrete events in time (i.e., spikes), our method provides several other benefits compared to other spike deconvolution methods based on state-space models, such as the constrained f-oopsi algorithm. First, our sampling-complexity trade-offs are known to be optimal from the theory of compressive sensing, whereas no performance guarantee exists for constrained f-oopsi. Second, we are able to construct precise confidence intervals on the estimated states, whereas constrained f-oopsi does not produce confidence intervals over the detected spikes. A direct consequence of these confidence intervals is estimation of spikes with high fidelity and low false alarm. Third, our comparisons suggest that the {\sf FCSS} reconstruction is at least 3 times faster than f-oopsi for moderate data sizes of the order of tens of minutes. Finally, our results corroborate the possibility of using compressive measurement for reconstruction and denoising of calcium traces. From a practical point of view, a compressive calcium imaging setup can lead to higher scanning rate as well as better reconstructions, which allows monitoring of larger neuronal populations \cite{pnevmatikakis2013sparse}. Due to the structured nature of our sampling and reconstruction schemes, we can avoid prohibitive storage problems and benefit from parallel implementations.

%{It is well-known that for smooth spaces of order one in $\ell_1(\mathbb{R})$ where the changes are not abrupt, dynamic $\ell_1$-regularization methods such as TV minimization correspond to Haar thresholding \cite{cohen1999nonlinear}. As a result Gaussian state-space models cannot result in significant improvements over Haar thresholding.}

\vspace{-3mm}
\subsection{Application to sleep spindle detection}
Another novel application of our modeling and estimating framework is to case sleep spindle generation as a second-order dynamical system governed by compressive innovations, for which {\sf FCSS} can be efficiently used to denoise and detect the spindle events. Our modeling framework suggest that spectrotemporal spindle dynamics cannot be fully captured by just pure sinusoids via bandpass filtering, as the data consistently contains significant $12$--$14~\text{Hz}$ oscillations almost everywhere (See Figure \ref{fig:tv_spindle_dream}). Therefore, using the bandpass filtered data for further analysis purposes clearly degrades the performance of the resulting spindle detection and scoring algorithms. The {\sf FCSS} provides a robust alternative to bandpass filtering in the form of model-based denoising.

In contrast to state-of-the-art methods for spindle detection, our spindle detection procedure requires minimal pre- and post-processing steps. We expect similar properties for higher order AR dynamics, which form a useful generalization of our methods for deconvolution of other transient neural signals. In particular, K-complexes during the stage 2 NREM sleep form another class of transient signals with high variability. A potential generalization of our method using higher order models can be developed for simultaneous detection of K-complexes and spindles.

\vspace{-2mm}
\section{Concluding Remarks}\label{sec:tv_conc}

In this chapter, we considered estimation of compressible state-space models, where the state innovations consist of compressible discrete events. For dynamics with convergent state transition dynamics, using theory of compressed sensing we provided an optimal error bound and stability guarantees for the dynamic $\ell_1$-regularization algorithm which is akin to the MAP estimator for a Laplace state-space model. We also developed a fast and low-complexity iterative algorithm, namely {\sf FCSS}, for estimation of the states as well as their transition matrix. We further verified the validity of our theoretical results through simulation studies as well as application to spike deconvolution from calcium traces and detection of sleep spindles from EEG data. Our methodology has two unique major advantages: first, we have proven theoretically why our algorithm performs well, and characterized its error performance. Second, we have developed a fast algorithm, with guaranteed convergence to a solution of the deconvolution problem, which for instance, is $\sim 3$ times faster than the widely-used f-oopsi algorithm in calcium deconvolution applications.

While we focused on two specific application domains, our modeling and estimation techniques can be generalized to apply to broader classes of signal deconvolution problems: we have provided a framework to model transient phenomena which are driven by sparse generators in time domain, and whose event onsets are of importance. Examples include heart beat dynamics and rapid changes in the covariance structure of neural data (e.g., epileptic seizures). In the spirit of easing reproducibility, we have made a MATLAB implementation of our algorithm publicly available \cite{code}.

\chapter{Multiplicative Updates for Optimization Problems with Dynamics}
\chaptermark{Multiplicative Updates with Positivity Constraints}
\label{chap:multiplicative}
In this chapter we consider the problem of optimizing general {convex} objective functions with nonnegativity constraints. Using the Karush-Kuhn-Tucker (KKT) conditions for the nonnegativity constraints we will derive fast multiplicative update rules for several problems of interest in signal processing, including nonnegative deconvolution, point-process smoothing, ML estimation for Poisson observations, nonnegative least squares and nonnegative matrix factorization (NMF). Our algorithm can also account for temporal and spatial structure and regularization . We will analyze the performance of our algorithm on simultaneously recorded neuronal calcium imaging and electrophysiology data.
%\boldmath


% \fi

\section{Introduction}
The advent of big data has given rise to new challenges in signal processing. Fast and scalable solvers for solving large optimization problems remains a big challenge of optimization theory. In this chapter we consider the problem of solving general optimization problems under nonnegativity constraints. Such optimization problems arise in many applications of interest. Examples include nonnegative matrix factorization for images of objects \cite{lee1999learning}, Poisson image reconstruction \cite{willett2010poisson}, point process smoothing for stimulus-response experiments in neurophysiology \cite{smith2003estimating}, nonnegative least squares \cite{lawson1995solving} and nonnegative calcium deconvolution \cite{kazemipour2017fast}. In this chapter we will use the KKT conditions \cite{boyd2004convex} to provide a unified framework for solving such optimization problems with nonnegativity constraints. As we will see these conditions naturally lead to multiplicative updates with suitable convergence in many applications.

Multiplicative updates have been used for solving ML and MAP estimation as well as KL-divergence minimization. Many of these algorithms are special cases of the so-called proximal backward-forward scheme \cite{palomar2010convex}. These algorithms try to find fixed points of a set of equations resulting from setting gradients of the objective function to zero. A With the help of parallel computing and graphics processing units (GPUs), these iterative methods can be solved very fast. Therefore, they become increasingly important. An important application of these multiplicative updates is the Richardson-Lucy (RL) algorithm for image deconvolution \cite{lucy1974iterative}, which is widely used in astronomy and microscopy \cite{shepp1982maximum}. The RL algorithm recovers the ML estimate of a sample under Poisson statistics \cite{richardson1972bayesian}.

Multiuplicative updates are commonly contrasted with gradient descent methods. Their update steps do not necessarily follow the direction of the steepest descent. Multiplicative updates are argued to be insensitive to noise and  more flexible \cite{yan2013general}. Despite fast early convergenece multiplicative updates are claimed to converge slowly in later stages \cite{white1994image}. However, this argument has been refuted for Poisson image reconstruction  \cite{yan2013general}, the Weiszfeld problem \cite{palomar2010convex} and NMF \cite{lee2001algorithms} by showing their equivalence to a Majorization Minimization (MM) algorithm which has linear convergence in iterations \cite{wu2016convergence}. In contrast, both multiplicative updates and gradient descent based algorithms such as the proximal-gradient method have sublinear rate of convergence \cite{palomar2010convex} in general. Moreover, with specific choices of the stepsize, in many cases such as the Weiszfeld problem these algorithms have proven to be equivalent  \cite{palomar2010convex} . These findings suggest that slow convergence of multiplicative updates in some cases is due to absence of strong convexity in the objective function.

An advantage of multiplicative updates over gradient descent based algorithms is their flexibility in terms of adapting to the objective functions without the need for calculation dual functions or tuning extra parameters such as the step-size. Despite the recent breakthroughs in choosing these parameters \cite{kingma2014adam}, each step in calculation of the step size is usually as costly as an iteration of the algorithm which is not as effective for big data problems. In addition many problems such as image reconstruction and calcium deconvolution \cite{pnevmatikakis2016simultaneous} are spatially separable and are easily parallelized.

Finally, temporal dynamics and penalization play an important role in signal recovery from noisy data. Examples include state-space estimations, video reconstruction and total variation denoising problems. Apart from special cases, the solutions to these problems are generally batch mode and computationally demanding. In this chapter we provide a unified framework for generalizations of multiplicative updates to the problems with nonnegativity constraints and dynamics by adapting the update rules to different forms of penalties.  We have empirically found that multiplicative updates show superior convergence properties and speed to gradient descent methods for models that include dynamics and penalization.

%The rest of the chapter is organized as follows: In Section  \ref{sec:slapmi_formulation} we introduce our notation and the problem formulation. In Section \ref{sec:slapmi_solution} we will describe the general structure of our multiplicative update rules. In the three subsequent sections we will provide several examples and applications of our algorithm. We will provide concluding remarks in Section \ref{sec:slapmi_conclusions}. 

\section{Problem Formulation}
\label{sec:slapmi_formulation}
%We will use the notation $x_i^j$ to denote the vector $[x_i,\cdots,x_j]^T$ and $\widehat{(.)}$ for the estimated values. By $c_\epsilon$ we mean an absolute constant which only depends on $\epsilon$ and by $\mathcal{O}_\epsilon(.)$ we mean that the constants depend on $\epsilon$. 

We consider a convex optimization problem of the form
\begin{align}
\label{eq:slapmi_main}
\minimize_{\mathbf{X} \succeq 0} \mathcal{F}(\mathbf{X}):= \mathcal{L}(\mathbf{X}) + \lambda \mathcal{P}(\mathbf{X}),
\end{align}
where $\mathcal{L}(.)$ denotes a convex objective function and $\mathcal{P}(.)$ denotes a suitable penalty function. Typically $\mathcal{L}(.)$ is a negative log-likelihood and $\mathcal{P}(.)$ is a smooth norm. Additionally we make the assumption that both $\mathcal{L}$ and $\mathcal{P}$ are differentiable with respect to $\mathbf{X}$ on the positive orthant, 

Among the algorithms used for solving (\ref{eq:slapmi_main}) one can name the primal-dual algorithm and proximal gradient method. For specific  choices of the penalty functions  $\ell_1$ and $\ell_2$ (Tikhonov) regularization several fast algorithms exist. However these algorithms cannot be easily generalized to arbitrary penalties or temporal dynamics. In some cases such as the gradient based methods they require knowledge of the proximal map or have extra parameters such as the step size to be tuned and chosen. Calculation of the step size is usually as costly as a few iterations of the algorithm and could slow them down. However, our approach to solving (\ref{eq:slapmi_main}) does not require tuning of extra parameters and is very simple to implement. We will next discuss our solution.

\section{Solution to the Main Optimization Problem}
\label{sec:slapmi_solution}
In this section we will introduce our solution to (\ref{eq:slapmi_main}) via multiplicative updates.
The Lagrangian form of (\ref{eq:slapmi_main}) is given by
\begin{align}
\label{eq:slapmi_main_lag}
\minimize_{{\mathbf{X}},{\mathbf{S} \succeq 0}}  \mathcal{F}(\mathbf{X}) + \mathbf{S} \odot \mathbf{X}.
\end{align}
Assuming convexity and zero duality gap, the KKT conditions  for (\ref{eq:slapmi_main_lag}) can be expressed as
\begin{align}
\label{eq:slapmi_main_kkt}
&{\mathbf{X}^\star \succeq 0}, \;\; {\mathbf{S}^\star \succeq 0},\\
\label{eq:slapmi_kkt_mult}
&{\mathbf{S}^\star } \odot {\mathbf{X}^\star} = \mathbf{0},\\
\label{eq:slapmi_kkt_grad}
&\nabla_{\mathbf{X}}  \mathcal{F}(\mathbf{X}) + \mathbf{S} = \mathbf{0}.
\end{align}
In the rest of the chapter, we drop the subscripts and arguments whenever they can be understood from the context. Multiplying (\ref{eq:slapmi_kkt_grad}) by $\mathbf{X}$ and using (\ref{eq:slapmi_kkt_mult}) we obtain:
\begin{align}
\label{eq:slapmi_main_mult}
 \nabla  \mathcal{F}(\mathbf{X}) \odot \mathbf{X}= \mathbf{0}.
\end{align}
%\begin{align*}
%\label{eq:slapmi_mult_main}
%\left( \nabla_{\mathbf{X}} \mathcal{L}(\mathbf{X}) + \lambda \nabla_{\mathbf{X}} \mathcal{P}(\mathbf{X}) \right) \odot \mathbf{X}= \mathbf{0}.
%\end{align*}
Our solution  to (\ref{eq:slapmi_main}) looks for a positive fixed point of (\ref{eq:slapmi_main_mult}). Therefore giving us the multiplicative update rule
\begin{equation}
\label{eq:slapmi_main_update_rule}
\mathbf{X}^{(k+1)} \leftarrow   \left( \nabla  \mathcal{F}(\mathbf{X}^{(k)})\right)^{-} \oslash \left( \nabla  \mathcal{F}(\mathbf{X}^{(k)})\right)^{+} \odot \mathbf{X}^{(k)}.
\end{equation}
In all application introduced in this chapter we initialize the algorithm with a positive solution, the choice of which depends on the application. The update rule will then ensure the solution remains positive. In order to provide more insight into our algorithm we will next provide several examples and applications.

In applications of interest in this chapter we consider temporal dynamics in $\mathbf{X}$, hence referring to our algorithm by FAst DEconvolution (FADE) algorithm. In the spirit of easing reproducibility, we have made MATLAB implementations of our codes publicly available \cite{code_mult}.

\section{Examples and Application to Real Data}
In this Section we will provide examples of the multiplicative updates in different applications of interest. 
\subsection{Nonnegative Deconvolution}

In its simplest form the nonnegative deconvolution problem can be formalized by considering the state-space model given by
\begin{equation}
\label{eq:slapmi_lap_state_space}
\mathbf{x}_t = \mathbf{\Theta} \mathbf{x}_{t-1}+ \mathbf{w}_t, \qquad \mathbf{y}_t = \mathbf{A}_t \mathbf{x}_t + \mathbf{v}_t,
\end{equation}
where $\mathbf{w}_t \succeq 0$ models the innovations at time  $t \in [T]$. Usually, the observation noise is assumed to be i.i.d normal, i.e. $\mathbf{v}_t \sim \mathcal{N}(0,\mathbf{\Sigma}_t )$ and the measurement matrices $\mathbf{A}_t$ are assumed to conserve positivity. For this problem we can identify $\mathbf{W} = \mathbf{W}_{[T]}$ and
\begin{align*}
\mathcal{L}(\mathbf{W})=\sum_{t=1}^T \left\| \mathbf{y}_t - \mathbf{A}_t \mathbf{x}_t\right\|_{\mathbf{\Sigma}_t}^2 = \sum_{t=1}^T \left\| \mathbf{y}_t - \mathbf{A}_t \sum_{\tau = 0}^{t-1} \mathbf{\Theta}^\tau \mathbf{w}_{t-\tau} \right\|_{\mathbf{\Sigma}_t}^2,
\end{align*}
from which we can calculate
\begin{align*}
\left(\nabla_{\mathbf{w}_t} \mathcal{L}(\mathbf{W})\right)^{+}= \sum_{\tau \geq t} \left(\mathbf{\Theta}^{\tau-t}\right)^T \mathbf{A}_\tau^T \mathbf{\Sigma}_\tau^{-1}\mathbf{y}_\tau,\\
\left(\nabla_{\mathbf{w}_t} \mathcal{L}(\mathbf{W})\right)^{-}= \sum_{\tau \geq t} \left(\mathbf{\Theta}^{\tau-t}\right)^T \mathbf{A}_\tau^T \mathbf{\Sigma}_\tau^{-1} \mathbf{A}_\tau \mathbf{x}_\tau.
\end{align*}

Typically one can use a smooth norm in order to enforce prior assumptions on the spikes, for example one can use a sparsity inducing prior $\mathcal{P} = \|\mathbf{W}\|_{1,1}$, for which $\left( \nabla \mathcal{P}\right)^+  = \mathbf{1}$ and $\left( \nabla \mathcal{P}\right)^-  = \mathbf{0}$. The choice of the penalty function on the spikes is arbitrary and could differ from application to application. In applications where such information is not readily available, one would like to enforce minimal assumptions on the spikes and hence would want to enforce non-informative priors. The most famous example of such priors is known as Jeffrey's prior \cite{jeffreys1946invariant}. However this problem is an active area of research as there is no unanimously agreed upon choice of non-informative priors.

\begin{figure}[t!]
\begin{center}
\includegraphics[width=.7\columnwidth]{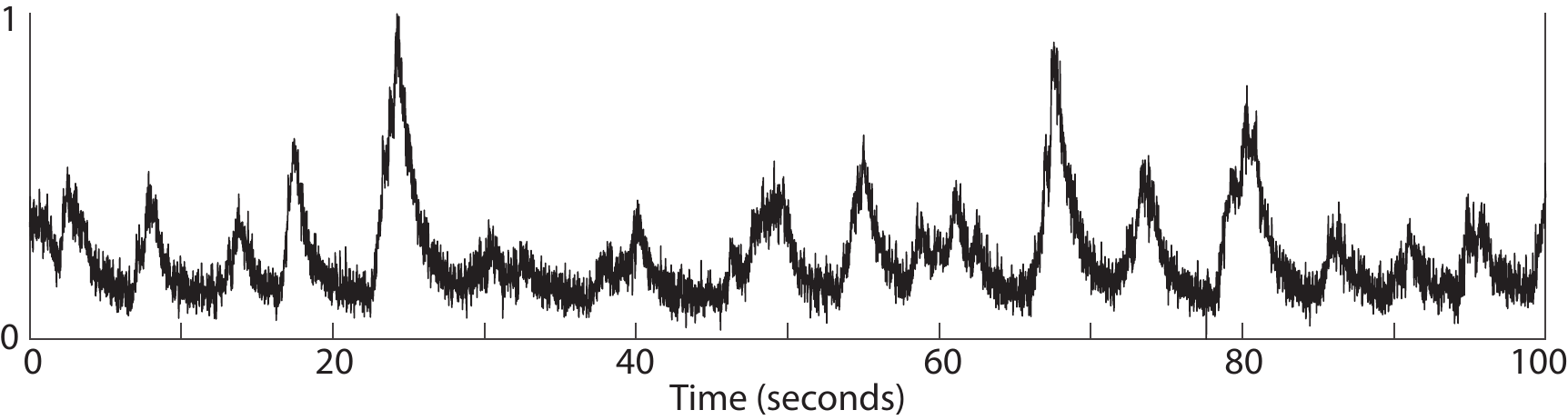}
\subcaption*{(a) Normalized calcium traces}
\vspace{3mm}
\includegraphics[width=.7\columnwidth]{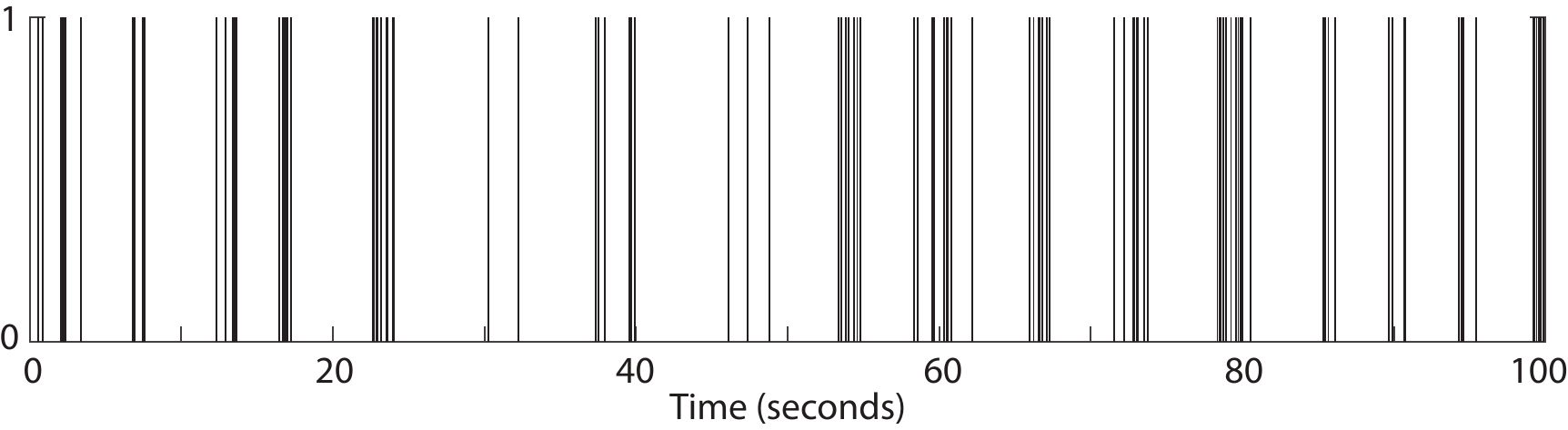}
\subcaption*{(b) Ground-truth spikes}
\vspace{3mm}
\includegraphics[width=.7\columnwidth]{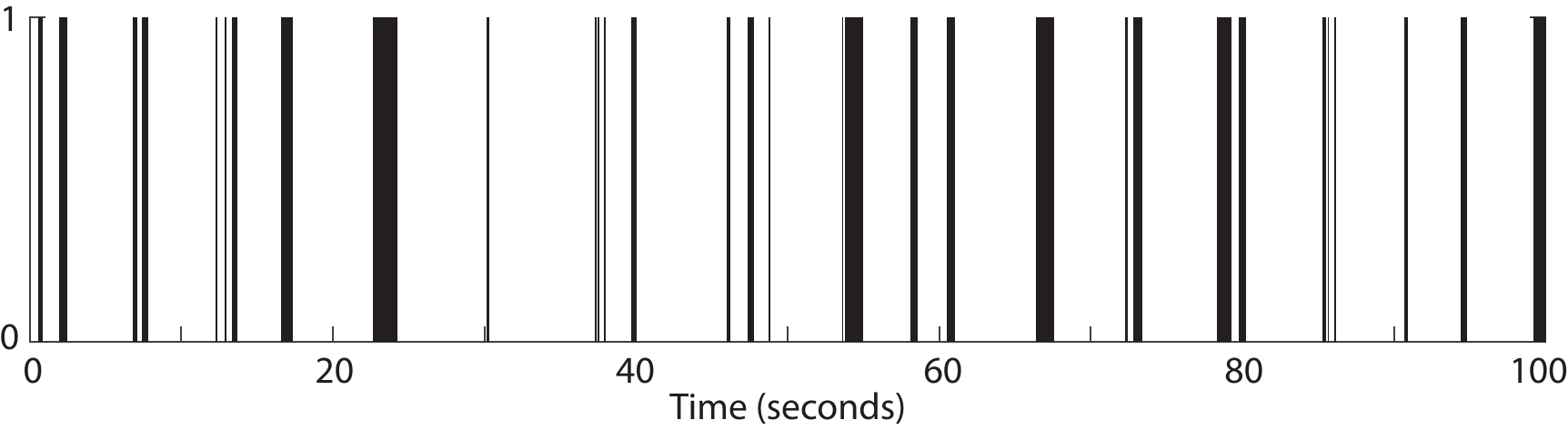}
\subcaption*{(c) Deconvolved spikes}
\end{center}
\vspace{-3mm}
\caption{\small{Application of the FADE algorithm to calcium deconvolution problem.} }\label{fig:ca_deconvolution}
\vspace{-2mm}
\end{figure}
\subsection{Application to Calcium Deconvolution}
Calcium imaging is used to visualize currents associated with action potentials in living neurons. This is done using fluorescent molecules that change their fluorescence properties upon binding calcium, and using a one- or two-photon fluorescence microscope to record these changes  \cite{smetters1999detecting, stosiek2003vivo}. Inferring action potentials (spikes) from calcium recordings, referred to as calcium deconvolution, is an important problem in neural data analysis. For the special case of calcium imaging we have  $\mathbf{\Sigma} = \sigma^2\mathbf{I}$, $\mathbf{A}_t = \mathbf{I}$ and $\mathbf{\Theta} = \theta \mathbf{I}$. Here the baseline is assumed to have been estimated and subtracted separately, but can be estimated similarly. We refer to  \cite{kazemipour2017fast} for details on estimation of the unknown parameters $\sigma^2$ and $\theta$ and a list of methods used for calcium deconvolution. These approaches require solving convex optimization problems, which do not scale well with the temporal dimension of the data. 

Figure \ref{fig:ca_deconvolution} shows application of the FADE algorithm to simultaneously recorded imaging and electrophysiology data. The algorithm has covnverged (less than 0.5\% change in spikes) in 28 iterations. The data is a 100 second interval from the spikefinder challenge \cite{spikefinder} (dataset 3, neuron 1). We have used an AR(2) model and an $\ell_{0.5,1}$ penalty on the spikes in order to enforce temporal sparsity. The spikes have been obtained by simply thresholding the deconvolved spikes at 3$\sigma$, where $\sigma$ is the estimated standard deviation of the observation noise. A  comparison of the performance of our algorithm with many other methods is provided on the spikefinder challenge website \cite{spikefinder}.

One can use spatial regularization on elements of $\mathbf{w}_t$ in this setup as well as compressive sensing regimes for when $\mathbf{A}$ satisfies the restricted isometry property RIP \cite{kazemipour2017fast}. We refer to \cite{kazemipour2017fast} for a more detailed discussion.

\subsection{Poisson Image Reconstruction and Point Process Smoothing }
\label{eq:slapmi_pp_smoothing}
State-space models with Poisson observations have also been studied in many applications of interest. In neuroscience, temporal dynamics of stimulus-response experiments in neurophysiology have been modeled using a Poisson state-space model. In emission tomography, dynamics of the photons hitting the detectors can be modeled with Poisson noise models.
Without loss of generality we consider the state-space model given by
\begin{equation}
\label{eq:slapmi_poisson_state_space}
\mathbf{x}_t = \mathbf{\Theta} \mathbf{x}_{t-1}+ \mathbf{w}_t, \qquad \mathbf{y}_t \sim \text{Poisson}\left( \phi \left( \mathbf{A}\mathbf{x}_t + \mathbf{b}_t \right) \right),
\end{equation}
where $\mathbf{w}_t \succeq 0$  and $\mathbf{b}_t \succeq 0 $ model the spikes and baseline rates at time  $t \in [T]$ respectively and $\phi(.)$ is a bijective convex function. Common examples include $\phi(x) = \exp(x)$, $\phi(x)= \frac{\exp(x)}{1+\exp(x)}$ and $\phi(x) = x$. We assume the latter in our derivations due to space considerations.

Several approaches have been proposed in the literature for finding the MAP solution to (\ref{eq:slapmi_poisson_state_space}). We refer to \cite{harmany2010spiral} for a detailed list of these methods. In \cite{smith2003estimating} the authors used  the maximum a posteriori derivation of the Kalman filter and proposed an approximate expectation maximization (EM) approach to this problem by Gaussian approximations of the posterior likelihood. This EM approach has several shortcomings. First, it requires solving a nonlinear system of equations which could potentially be computationally costly. Second, it only accounts for Gaussian spikes. Third its performance heavily depends on the Poisson rate model, especially when the rates are small, which is the usual case for spiking activities. In these  cases usually $\phi(x) = \exp(x)$ is considered for stability of approximations. Moreover due to nonlinear recursive filtering nature of the problem, the performance of the Gaussian approximation quickly degrades as the dimension of the latent space goes beyond 2 or 3.  Similarly, in \cite{harmany2010spiral} the authors proposed SPIRAL which uses  a  Gaussian approximation to $\mathcal{L}$ and is a gradient-based solution to (\ref{eq:slapmi_poisson_state_space}). Except for the special cases of $\ell_1$ and TV penalties, calculation of the Gaussian model is tedious leading to slow convergence. In \cite{sussillo2016lfads} the authors introduce a variational auto-encoder (gradient descent based) model to retrieve the low-dimensional temporal factors.

In applications such as fluorescence microscopy, it is also common to use to use variance stabilizing transforms \cite{harmany2010spiral} such as square root filtering \cite{roweis1999unifying} in order to make Gaussian approximations to the Poisson distribution. In the high photon regime such transformations are not necessary as one can use infinite divisibility property of the Poisson distribution for Gaussian approximations. However one would then need to deal with complications arising from equality of the mean and the covariance matrices for such approximations.
In contrast, our algorithm gives an exact solution, is fast, can account for any rate model and suitably scales with the problem dimensions.

The Gaussian approximations could then be used as an input to a Kalman smoother if the innovations (spikes) follow a half-normal or Gaussian distribution. Despite the fact that our solutions are faster, exact and do not involve approximations, for Gaussian state-spaces the Kalman smoother provides a smoothed estimate of the covariances which could be used for building confidence intervals, whereas the covariances are not a direct output of the multiplicative updates.

Considering the MAP estimator for $\mathbf{W}= \mathbf{W}_{[T]}$ we can identify
\begin{equation} 
\mathcal{L}(\mathbf{W}) = \sum_{t=1}^T \mathbf{1}^T \left( \mathbf{A}\mathbf{x}_t + \mathbf{b_t} \right) -  \mathbf{y}_t^T \log \left( \mathbf{A}\mathbf{x}_t + \mathbf{b_t} \right),
\end{equation}
for which we have
\begin{align*}
\left(\nabla_{\mathbf{w}_t} \mathcal{L}(\mathbf{W})\right)^{+}= \sum_{\tau \geq t} \left(\mathbf{\Theta}^{\tau-t}\right)^T \mathbf{A}_\tau^T \mathbf{1},
\end{align*}
and
\begin{align*}
\left(\nabla_{\mathbf{w}_t} \mathcal{L}(\mathbf{W})\right)^{-}= \sum_{\tau \geq t} \left(\mathbf{\Theta}^{\tau-t}\right)^T \mathbf{A}_\tau^T \left( \mathbf{y}_t \oslash \left( \mathbf{A}\mathbf{x}_t + \mathbf{b_t} \right) \right).
\end{align*}
The penalty function and the corresponding terms can be calculated similar to the nonnegative deconvolution problem. A similar update rule can be derived for the baseline. The special case of $\mathbf{\Theta} = \mathbf{0}$ (no dynamics with the convention $\mathbf{0}^0 = \mathbf{I}$) and $\lambda = 0$ (no penalization) is known as the Richardson-Lucy (RL) iterations. The RL algorithm has also been used with TV seminorm regularization in  \cite{dey2006richardson}.  Similar to the RL algorithm we can use FADE for blind deconvolution,  when the measurement matrix $\mathbf{A}$ is unknown. In this setup one can alternatively update $\mathbf{A}$ and $\mathbf{X}$.  We can also used FADE, for estimation of GLM models for self-exciting point process models \cite{kazemipour2017robust}.

\subsection{Combination with Other Constraints}
In many applications of interest the optimization problem could also include several inequality constraints. For example in fluorescence microscopy the maximum changes of the fluorescence level with respect to baseline (also referred to as $\mathbf{\frac{\Delta \mathbf{F}}{\mathbf{F}}}$) is controlled by the properties of the indicator in use. In these situations we need to satisfy the KKT conditions for the extra constraints. Here we will introduce an adaptive method in order to achieve this goal.
Consider the modified problem setup of Section \ref{eq:slapmi_pp_smoothing} given by 
\begin{equation}
\label{eq:slapmi_ss}
\begin{array}{ll}
\left( \frac{\Delta \mathbf{F}}{\mathbf{F}} \right)_t = \mathbf{\Theta} \left( \frac{\Delta \mathbf{F}}{\mathbf{F}} \right)_{t-1} + \mathbf{w}_t\\
\mathbf{y}_t \sim \text{Poisson}\Big(\mathbf{A} b_t ( 1+ \left( \frac{\Delta \mathbf{F}}{\mathbf{F}} \right)_t ) \Big) 
\end{array},
\end{equation}
where $b_t \ge 0$ denotes the known baseline fluorescence  at time $t$, on top of which $\left( \frac{\Delta \mathbf{F}}{\mathbf{F}} \right)_t$ lies. In addition to nonnegativity constraints we need to account for the following constraints
%Without loss of generality we will focus on the special case of $\mathbf{\Theta} = \theta \mathbf{I}$, where $\theta = \exp(-1/\tau_\min)$ and $\tau_\min$ is the shortest (fastest) time constant of $ \frac{\Delta \mathbf{F}}{\mathbf{F}} $ traces.  Such a choice of $\theta$ allows for capturing the fastest dynamics. Due to space considerations we leave detailed explanation of baseline estimation to future work and assume it is known. Therefore we will be dealing with a convex optimization problem. 
\begin{align}\label{eq:slapmi_star}
\begin{tabular}{ll}
$\left( \frac{\Delta \mathbf{F}}{\mathbf{F}} \right)_t  \preceq c_f $ & for all $t$
 \end{tabular} {\tag{$\mathcal{\star}$}}.
\end{align}
The constant $c_f$ is a characteristic of the indicator used and is assumed to be known.
In order to  enforce (\ref{eq:slapmi_star}) we proceed as in Algorithm \ref{alg:slapmi_reg}.

\noindent \begin{minipage}{\columnwidth}
\begin{center}
\begin{algorithm}[H]
%\caption{SImultaneous DEconvolution, DEmixing and DEnoising of Compressive AutoRegressive models (SIDECAR)}
\caption{\small Multiplicative Updates with Adaptive Regularization}
\label{alg:slapmi_reg}
\begin{algorithmic}[1]
\Procedure{{ Multiplicative Updates}}{}
\State Initialize: $\mathcal{P}\left( \frac{\Delta \mathbf{F}}{\mathbf{F}} \right)_t  = \| \left( \frac{\Delta \mathbf{F}}{\mathbf{F}} \right)_{[T]} \|_{\infty,\infty} $, $\lambda = 0$, $\lambda_0 = 0.01$, $i=0$.

\Repeat
 \If {$\max \left( \frac{\Delta \mathbf{F}}{\mathbf{F}} \right)_t \ge c_f$ and $i=0$}
 \State  $\lambda \leftarrow \lambda_0$, $i \leftarrow 1$
 \EndIf
 \If {$\lambda >0$}
\State Set $\lambda \leftarrow \lambda  \frac{ \left\| \left( \frac{\Delta \mathbf{F}}{\mathbf{F}} \right)_{[T]} \right\|_{\infty,\infty}}{c_f}$
\EndIf
\State Update $\mathbf{W}$.
\Until{convergence criteria met}
\EndProcedure
\end{algorithmic}
\end{algorithm}
\end{center}
\end{minipage}\\

The main idea behind Algorithm \ref{alg:slapmi_reg} is that when the constraints are violated the complimentary slackness condition should be met for the optimal dual variable $\lambda$ in Lagrangian form of the problem, meaning that the optimal solution should satisfy $ \| \left( \frac{\Delta \mathbf{F}}{\mathbf{F}} \right)_{[T]} \|_{\infty,\infty} = c_f$,  which is equivalent to finding  a fixed point of updates for the dual (regularization) variable $\lambda$.

\section{Other Examples}
\subsection{Dynamic Nonnegative Least Square (NLS)}
 The NLS problem can in general be formulated
% \label{eq:slapmi_nnls_state_space}
\begin{align*}
&\mathbf{Y} = \mathbf{A} \mathbf{X} + \mathbf{V}, \qquad \mathbf{V} \sim\mathcal{N}(0,\sigma^2\mathbf{I}),\\
&\mathcal{L}(\mathbf{X}) = \|\mathbf{Y} - \mathbf{A} \mathbf{X}\|_2^2, \;  \left(\nabla \mathcal{L} \right)^+ = \mathbf{A}^T\mathbf{Y}, \;  \nabla \left(\mathcal{L}\right)^- = \mathbf{A}^T \mathbf{A} \mathbf{Y}.
\end{align*}
The most famous algorithm for solving the NLS problem is the active set method \cite{lawson1995solving} which does not account for temporal dynamics in $\mathbf{x}_t$ or other forms of penalty. In these settings our update rules are very similar to the nonnegative deconvolution problem. A very useful example from  the compressed sensing literature is the Multiple Measurement Vector (MMV) problem (without the positivity constraint) \cite{chen2005sparse}. A commonly used penalty in this setup is the $\|\mathbf{X}\|_{2,1}$  which enforces row sparsity.
\subsection{Dynamic Nonnegative Matrix Factorization (NMF)}
The NMF problem is very similar to the NLS problem except that the matrix $\mathbf{A}$ is not known. In this case we can alternatively update our estimates of $\mathbf{A}$ and $\mathbf{X}$ \cite{boyd2011distributed}. 
\begin{align*}
\label{eq:slapmi_nmf_state_space}
& \mathbf{Y} = \mathbf{A} \mathbf{X} + \mathbf{V}, \qquad\mathbf{V} \sim\mathcal{N}(0,\sigma^2\mathbf{I})\\
& \mathcal{L}(\mathbf{X}) = \|\mathbf{Y} - \mathbf{A} \mathbf{X}\|_2^2,\\
& \left( \nabla_{\mathbf{X}} \mathcal{L}\right)^+ = \mathbf{A}^T\mathbf{Y}, \quad \left( \nabla_{\mathbf{X}} \mathcal{L}\right)^- = \mathbf{A}^T \mathbf{A} \mathbf{Y} \\
&  \left(\nabla_{\mathbf{A}} \mathcal{L}\right)^+ = \mathbf{Y} \mathbf{X}^T, \quad  \left(\nabla_{\mathbf{A}} \mathcal{L}\right)^- =  \mathbf{Y} \mathbf{X}^T \mathbf{X}
\end{align*}
In the absence of penalization or dynamics we recover the multiplicative updates of \cite{lee1999learning}. Our update rules can also account for the dynamic case where
\begin{align*}
& \mathbf{X}_t = \alpha \mathbf{X}_{t-1} + \mathbf{W}_t, \qquad \mathbf{W}_t \succeq \mathbf{0}\\
& \mathbf{Y}_t = \mathbf{A} \mathbf{X}_t + \mathbf{V}_t, \qquad \mathbf{V}_t \sim \mathcal{N}(0,\sigma^2 \mathbf{I})
\end{align*}
For example one can account for sparsely changing temporal factors by considering a Laplacian distribution on $\mathbf{W}_t$.

\section{Concluding Remarks}
\label{sec:slapmi_conclusions}
In this chapter we considered convex optimization problems with nonnegativity constraints and provided unified multiplicative updates for them using the KKT conditions. These updates are easy to implement and parallelizable on a CPU. They do not require tuning of extra parameters such as the step size, exhibit fast convergence in practice and can account for temporal dynamics and smooth penalties without slowing down. 

%Finally, these updates exhibit fast convergence in practice. In applications with temporal dynamics, these updates can be considered as belief propagation algorithms. We have empirically observed  that in applications with temporal dynamics, the computational complexity of these updates are sublinear in time and only depend on the time constants of the problem.

Although in the absence of convexity the KKT conditions no longer hold, we have empirically observed that our updates exhibit good performance when the problem has simple nonconvexities. As an example one can model calcium saturation in the calcium deconvolution problem by adopting the calcium hill model given by  $\mathbf{y}_t = \alpha \frac{\mathbf{x}_t}{\mathbf{x}_t + \mathbf{c}} + \mathbf{v}_t$ \cite{yasuda2004imaging}. These observations suggest that suitable initializations result in convergence to a suitable local minimum. As another example one can combine the multiplicative updates with the IRLS algorithm \cite{chartrand2008iteratively} for $\ell_q$, $q <1$ minimization problems. The convergence of the IRLS algorithm was shown in the literature by showing an equivalence to a special case of the EM algorithm \cite{babadi_IRLS}. We applied this generalization to calcium imaging data using a nonconvex penalty.

Finally, the positivity constraint can easily be relaxed in the general form of the problems in two ways:  First, any variable $\mathbf{X}$ can be decomposed into $\mathbf{X} = \mathbf{X}^+ - \mathbf{X}^-$, where both $\mathbf{X}^+$ and $\mathbf{X}^-$ are positive. Second, generalized positivity and negativity could be defined with respect to the boundary of the convex set of feasible solutions, i.e. any point point inside/outside the feasibility set could be considered as positive/negative. Generalized positive and negative terms in the decompositions could be redefined similarly. Therefore by looking for a generalized positive fixed point of the gradient of the log-likelihood, the multiplicative updates can be generalized to a larger class of problems with not necessarily positivity constraints. We leave full details of these extensions and examples and their convergence properties to future work.

\chapter{Megapixel Two-photon imaging at kHz Framerates}
\chaptermark{Projection Laser Microscopy}
\label{chap:slapmi}

Two-photon laser scanning microscopy enables high-resolution imaging within scattering specimens such as the brain, but the sequential acquisition of image voxels fundamentally limits its speed. We developed a two-photon imaging technique that scans lines of excitation across a focal plane at multiple angles, recovering diffraction-limited images from relatively few incoherently multiplexed measurements. We use a static image as a prior for image reconstruction to track rapid brightness changes in neural activity sensors, and sparsity as a prior to track diffusing particles at over 1.4 billion voxels per second. We imaged glutamate sensor transients across hundreds of individually resolved dendritic spines in mouse cortex at framerates over 1kHz. This method surpasses a physical limit on the speed of sequential two-photon imaging imposed by fluorescence lifetime, enabling recordings that are not possible by raster scanning.

\section{Introduction} \label{slapmi:intro}

The study of brain function relies on measurement tools that achieve high spatial resolution over large volumes at high rates. Activity travels through neural circuits on a timescale of milliseconds. Interacting elements such as neurons or synapses are scattered over distances many times their size, requiring imaging volumes containing millions of voxels. The living brain is opaque, meaning that optical tools for monitoring brain activity must be insensitive to light absorption and scattering. Two-photon imaging achieves this insensitivity by using nonlinear absorption to confine fluorescence excitation to the high-intensity focus of a laser and prevent excitation by scattered light. All emitted fluorescence can be assigned to that focus regardless of scattering, without forming an optical image. Instead, an image is produced by scanning the focus in space. However, this serial approach to image acquisition creates a tradeoff between achievable framerates and pixel counts per frame. Common fluorophores have fluorescence lifetimes of approximately 3 ns, and brighter fluorophores tend to have longer lifetimes \cite{strickler1962relationship, striker1999photochromicity}. Consequently, approximately 10 ns must pass between consecutive measurements to obtain distinct samples. The maximum achievable framerate for a 1-megapixel field of view (FOV) under raster-scanning fluorescence imaging is therefore approximately 100 Hz. In practice, pixel rates have been further limited by factors such as photodamage, fluorophore saturation, and scanner technology  \cite{yang2017vivo}.

Calcium indicator transients have been widely adopted as a proxy for neuronal and synaptic activity \cite{yang2017vivo, szalay2016fast, chen2011functional}, in part because their large amplitude and slow speed can be recorded at low frame rates \cite{peterka2011imaging}. However, calcium transients are imperfect reporters of the biophysical quantities underlying neuronal communication, such as action potentials and synaptic events \cite{peterka2011imaging, sobczyk2005nmda, hao2012depolarization}. Faster reporters and imaging methods are needed to more directly monitor such signals at speeds commensurate with computations in the brain.

The pixel rate bottleneck for raster imaging can be avoided by more efficient sampling \cite{yang2017vivo, prevedel2016fast}. In most activity imaging paradigms, a dense pixel-based representation of the sample is recorded then reduced to a lower-dimensional space, e.g. by selecting regions of interest \cite{pnevmatikakis2016simultaneous, vogelstein2010fast}. By sampling this low-dimensional representation more directly, the equivalent result can be obtained with many fewer measurements. Several methods have been developed that can record from sample volumes more efficiently than raster scanning: 

Random Access imaging using Acousto-Optic Deflectors/Lenses4 \cite{nadella2016random, bullen1997high} enables sampling of any subset of points in the sampling region, with a fixed access time required to move the excitation focus between points.  If the desired points are sparse enough in space, the time saved by not sampling the intervening area significantly outweighs the access time costs. Multifocal multiphoton (MMP) methods scan a fixed pattern of focal points through the sample, allowing multiple subvolumes to be acquired simultaneously. If a single-pixel detector is used, the resulting image volumes sampled by each focus are superposed \cite{yang2016simultaneous}. A spatially resolved detector such as a camera, eye \cite{bewersdorf1998multifocal} or multianode photomultiplier \cite{kim2007multifocal} can also be used, but the resolution of resulting images is significantly degraded by scattering. Extended Depth of Field (EDoF) methods collapse the axial dimension of the sample, allowing projections of volumes to be acquired at the rate of two-dimensional images. EDoF excitation has been achieved at high resolutions using scanned Bessel beams \cite{lu2016video, botcherby2006scanning, dufour2006two, theriault2014extended}. Several other imaging methods involving patterned two-photon illumination have been described \cite{prevedel2016fast, field2016superresolved}.

Except random access imaging, the above techniques share a common feature: Signals from different locations in the sample are deliberately mixed, enabling a given volume to be measured with fewer samples. For most analyses, the underlying signals are then unmixed computationally. We refer to this approach as ``Projection Microscopy" because it deliberately projects multiple resolution elements into each measurement. Projection Microscopy has advantages over random access imaging when the specimen can move unpredictably, such as in awake animals or diffusing particles, when regions of interest are not well approximated by single points, such as cell membranes, or when it is difficult to rapidly select regions of interest, such as dendritic spines.

Recovery of signals from mixed measurements is common in imaging. Several methods combine images with distinct optical transfer functions to improve resolution \cite{gustafsson1999extended, heintzmann1999laterally, preibisch2014efficient, neil1997method}. In all imaging methods, finite resolution can cause pixels to contain signals from multiple sources, such as a neuron and its surrounding neuropil. It is common for analyses to explicitly model these mixed sources (e.g. \cite{pnevmatikakis2016simultaneous}). Source recovery is often posed as an optimization problem, using implicit or explicit regularization to impose a desired statistical structure on the recovered signals, such as independence \cite{brown2001analysis} or sparsity \cite{pnevmatikakis2016simultaneous}. Compressive Sensing \cite{Eldar} is a framework for the acquisition and unmixing of signals that admit a sparse representation in some basis. By acquiring mixed signals and regularizing for sparsity during recovery, structured systems with many fewer measurements than unknowns can be accurately recovered if the measurements are conducted appropriately \cite{candes2008introduction}. Highly coherent measurements (i.e. ones that tend to mix a given source with others in the same way) make recovery of the underlying sources ambiguous, while incoherent measurements can guarantee accurate recovery \cite{candes2008introduction, foucart2013mathematical}. The prior information used for recovery can take many forms, such as the distribution or dynamics of the sources \cite{babacan2010bayesian} \cite{kazemipour2017fast}. In medical imaging and microscopy, compressive sensing using sparsity has greatly increased imaging speed and reduced radiation dose \cite{chen2008prior, lustig2007sparse, pegard2016compressive}. 

In this Chapter, we recover neural activity from highly incoherent mixed measurements using strong priors obtained from a detailed structural image of the sample and a model of indicator dynamics. These priors constrain the space of solutions enough that regularization for sparsity is unnecessary.  We also image and track diffusing particles within a more canonical compressive sensing framework. The goal of our work has been to advance optical neurophysiology by developing a practical high-speed two-photon microscope. This microscope retains the high spatial resolution and insensitivity to scattering of conventional two-photon imaging, uses known spatiotemporal structure of samples to increase framerate, is insensitive to sample motion, enables highly accurate source recovery by performing efficient incoherent measurements, and adapts easily to a variety of experiments.

\section{Results} \label{slapmi:results}
\subsection{Scanned Line Angular Projection Microscopy} \label{slapmi:results_slapmi}
We built a microscope that scans line foci across the focal plane at four different angles, obtaining linear projections of the sample. We chose line foci because they: 
\begin{enumerate}
\item Are simple to produce optically,
\item efficiently sample a compact area by scanning 
\item achieve diffraction-limited spatial resolution
\item produce two-photon excitation more efficiently than non-contiguous foci of the same area
\item are a low-coherence basis for the sample space. 
\end{enumerate}

We named this approach Scanned Line Angular Projection Microscopy (SLAPMi). SLAPMi samples the entire FOV with four line scans. The frame time is proportional to the resolution, compared to resolution squared for a raster scan, resulting in greatly increased frame rates.  For example, SLAPMi images a 250 $\times$ 250 $\mu$m FOV with diffraction-limited optical resolution at a framerate of 1016 Hz, corresponding to over 1.49 billion pixels per second, recovered from 5 million multiplexed measurements.

To reduce optical power at the sample and enable control over degree of parallelization, we included a spatial light modulator (SLM) in an amplitude modulation geometry. This configuration selects an arbitrary pattern in the focal plane for imaging, and discards the remaining excitation light, making SLAPMi a random access microscope. This reduces excitation power in sparse samples and allows users to artificially introduce sparsity into densely labeled samples. Unlike AOD-based imaging, the SLAPMi framerate is independent of the number of pixels imaged, up to the entire field of view. When imaging sufficiently sparse samples, the SLM and other optimizations for excitation efficiency allow SLAPMi to use average powers lower than conventional raster scanning.

\subsection{Particle Localization and Tracking} \label{slapmi:particles}
Recovering images from mixed measurements requires identifying a representation of the sample with rank lower than the number of measurements. For sparsely labeled samples, this representation can be the nonzero pixels in the image \cite{lustig2007sparse, antipa2016single}. We demonstrated this regime by localizing and tracking fluorescent particles.  Each particle produces a bump of signal on each of the four scan axes corresponding to its position, and each frame consists of the superposition of signals for all particles in view. If particles are sufficiently sparse, three measurement axes are sufficient to localize all particles in the FOV. 

Using the microscope’s known projection matrix {(see \ref{slapmi:methods})}, we can recover the maximum likelihood image by incorporating dynamics into Richardson-Lucy (RL) deconvolution as discussed in Chapter \ref{chap:multiplicative} \cite{lucy1974iterative, richardson1972bayesian}. For large numbers of labeled pixels, maximum likelihood reconstruction results in spurious peaks in the recovered image, which are reduced by solvers that enforce sparsity. 

As a demonstration, we recorded the motion of thousands of 500 nm fluorescent beads within a 250 $\mu$m (diameter) $\times$ 250 $\mu$m (depth) imaging volume (87 equally-spaced planes; 1.016kHz frame rate, 5080 measurements per plane, 10 Hz volume rate), (Figure \ref{fig:slapmi_particles1}, supplemental movie 1). Particles were readily tracked within the recovered volume movies using established tracking methods for volumetric data (Figure \ref{fig:slapmi_particles2}).}

\begin{figure}[H]
\begin{center}
\noindent
\includegraphics[width=.9\columnwidth]{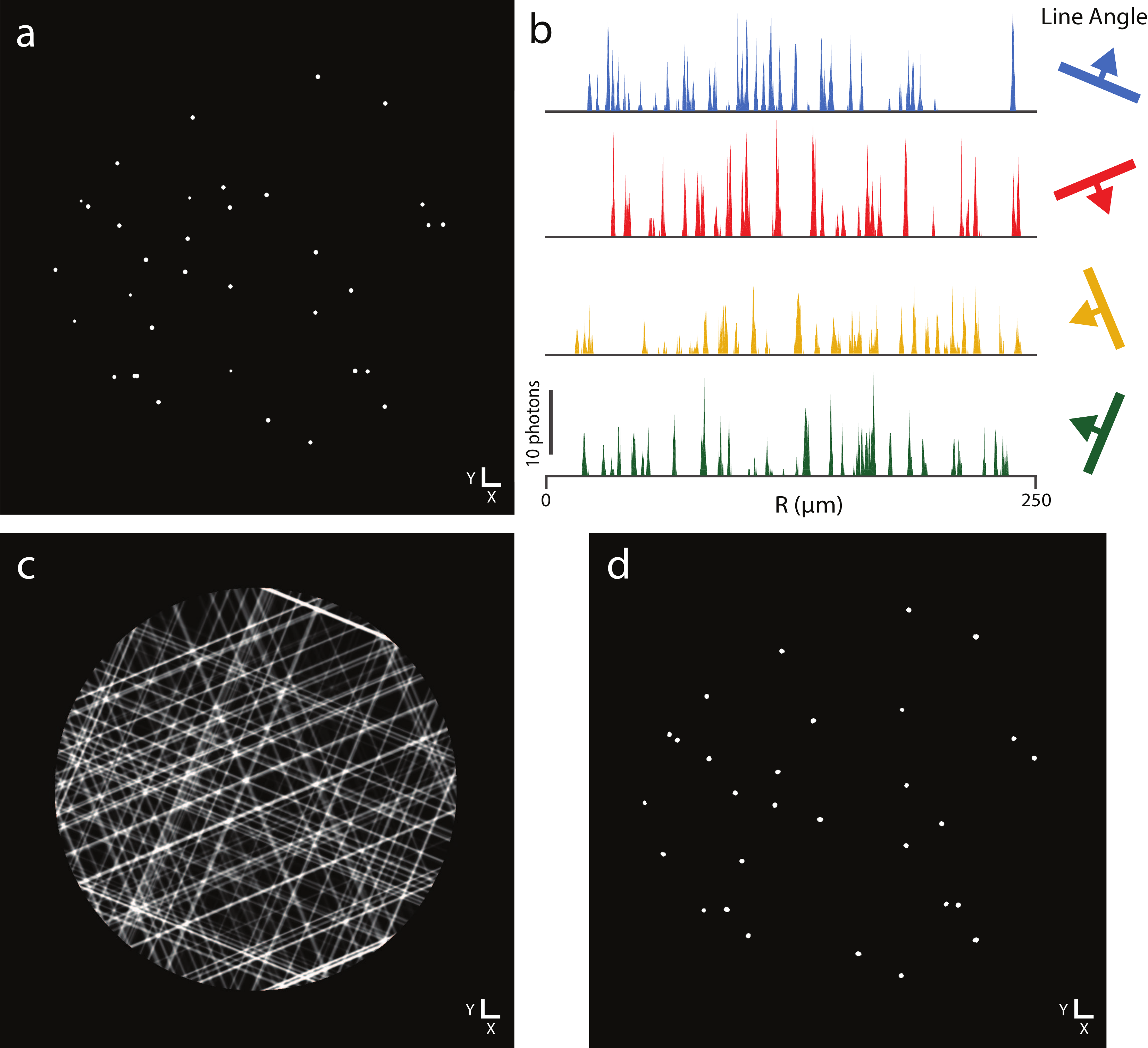}
\caption{Particle localization using SLAPMi: a) 2D raster image of 500nm fluorescent beads on a glass surface. b) 2D SLAPMi measurement of the same sample, consisting of projections along the four scan axes. R denotes position on the scan axis. c) Backprojection of the measurements in b, corresponding to one iteration of RL deconvolution. d) Recovered image following 20 iterations of pruned RL deconvolution }\label{fig:slapmi_particles1}
\end{center}
\vspace{-5mm}
\end{figure}

\begin{figure}[H]
\begin{center}
\noindent
\includegraphics[width=1\columnwidth]{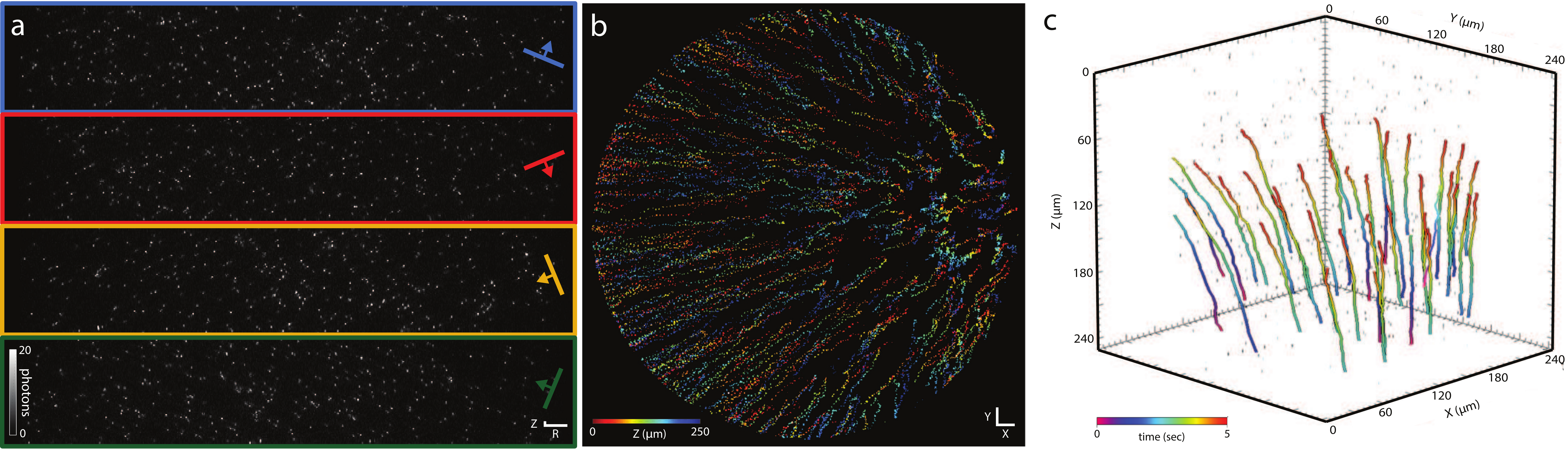}
\caption{Particle tracking using SLAPMi: a) 3D SLAPMi measurement of a 250 $\times$ 250 $\times$ 250 250 $\mu$m volume of 500 nm fluorescent beads in water.  b) Superposition of 40 volume images acquired with SLAPMi, showing convection and diffusion of beads. Color denotes depth. c) Example particle tracks obtained with a common software package
}\label{fig:slapmi_particles2}
\end{center}
\vspace{-5mm}
\end{figure}

\subsection{Imaging Neural Activity} \label{slapmi:inVivo}
To recover neural activity with SLAPMi, we adopted a sample representation in which fixed spatial components vary in brightness over time \cite{pnevmatikakis2016simultaneous, vogelstein2010fast}. The spatial components are obtained from a separate raster-scanned volume image. We identify compartments of labelled neurons with a manually trained pixel classifier (Ilastik \cite{sommer2011ilastik}), and a skeletonization-based algorithm that divides neurites into short segments {(Figure \ref{fig:slapmi_inVitro1}-b)}, resulting in up to 1000 segments per plane {(see Methods)}. Source recovery consists of assigning an intensity to each segment at each frame according to the following model:

\begin{align*}
& \left( \frac{\Delta \mathbf{F}}{\mathbf{F}} \right)_t = \theta \left( \frac{\Delta \mathbf{F}}{\mathbf{F}} \right)_{t-1} + \mathbf{w}_t\\
& \mathbf{y}_t \sim \text{Poisson}\Big(\mathbf{A} b_t ( 1+ \left( \frac{\Delta \mathbf{F}}{\mathbf{F}} \right)_t ) \Big)\\
& \subjectto \mathbf{W} \succeq \mathbf{0},
\end{align*}
where $\mathbf{Y}$ are the measurements (number of lixels $\times$ frames), $\mathbf{P}$ is the projection matrix and $\mathbf{S}$ is the segmented image obtained by a 2P raster scan. $\mathbf{P}$ and $\mathbf{S}$ are measured separately. We estimate $\mathbf{W}$, the innovations, by minimizing the negative log-likelihood as discussed in chapter \ref{chap:multiplicative}. Importantly, in this paradigm, regularization is unnecessary because the segmentation is low rank. 

\subsection{In Vitro Validation}
We validated SLAPMi in experiments with rat hippocampal cultures under conditions where ground truth activity was known.  In the first of these experiments, we co-cultured cells expressing a cytosolic fluorophore (tdTomato) with cells expressing the glutamate sensor Venus-iGluSnFR. The mixed culture was imaged using a single detector, while stimulating with a field electrode. Stimulation triggers transients only in Venus-iGluSnFR-expressing cells, and no change in tdTomato brightness. We collected separate two-channel raster images to verify the identity of the imaged cells, and quantified recovered signals in tdTomato-expressing cells to assess spatial crosstalk.

SLAPMi reliably reported stimulation-induced transients only in Venus-iGluSnFR labeled sample voxels {(Figure \ref{fig:slapmi_inVitro1}, supplemental movie 2)}.

In the second experiment, we imaged yGluSnFR-expressing neurons while uncaging glutamate at two locations at different times, to assess the timing precision of recovered signals {(Figure \ref{fig:slapmi_inVitro2}, supplemental movie 3)}.
We quantified the onset time of transients at each pixel of the reconstructed image.

\begin{figure}[H]
\begin{center}
\noindent
\includegraphics[width=1\columnwidth]{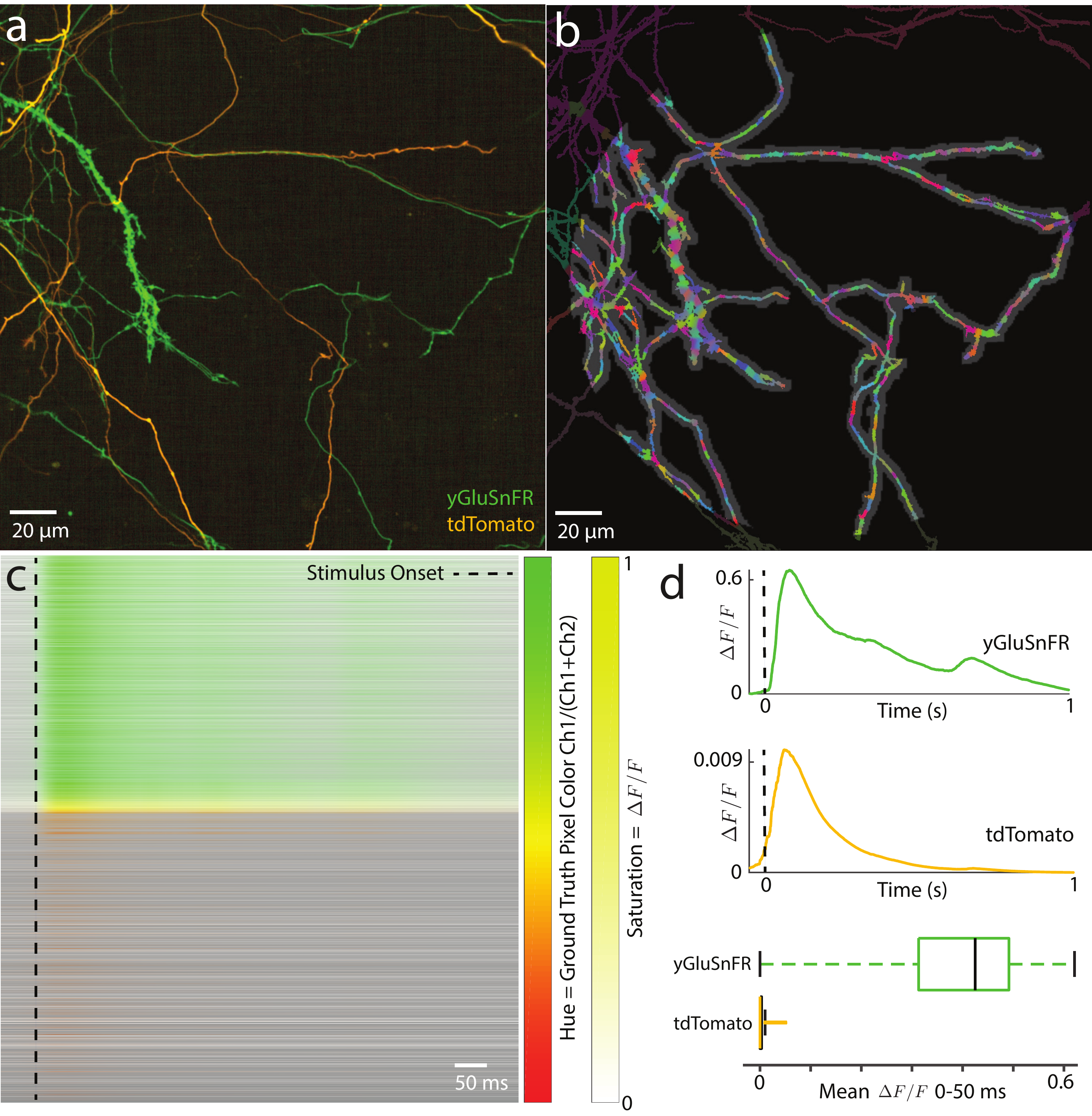}
\caption{In Vitro validation, red/green experiment: a) Superposed 2D raster image of cultures expressing yGluSnFR (Green, channel 1) and tdTomato  (Red, both channels). The activity was imaged only on channel 1, b)	Segmentation for a). The SLM boundaries are tinted in gray, c)	Recovered pixelwise $\frac{\Delta F}{F}$ for a). Pixels are sorted according to their color (redness). The value (darkness) represents square root of $F_0$, d)	Mean $\frac{\Delta F}{F}$  for most green and most red pixels (1000 pixels each). Top and Middle: Spatial mean $\frac{\Delta F}{F}$  for most green/red pixels. Bottom: Temporal mean $\frac{\Delta F}{F}$  (0-50 ms after stimulus) quantiles for most green/red pixels.
}\label{fig:slapmi_inVitro1}
\end{center}
\vspace{-5mm}
\end{figure}

\begin{figure}[H]
\begin{center}
\noindent
\includegraphics[width=1\columnwidth]{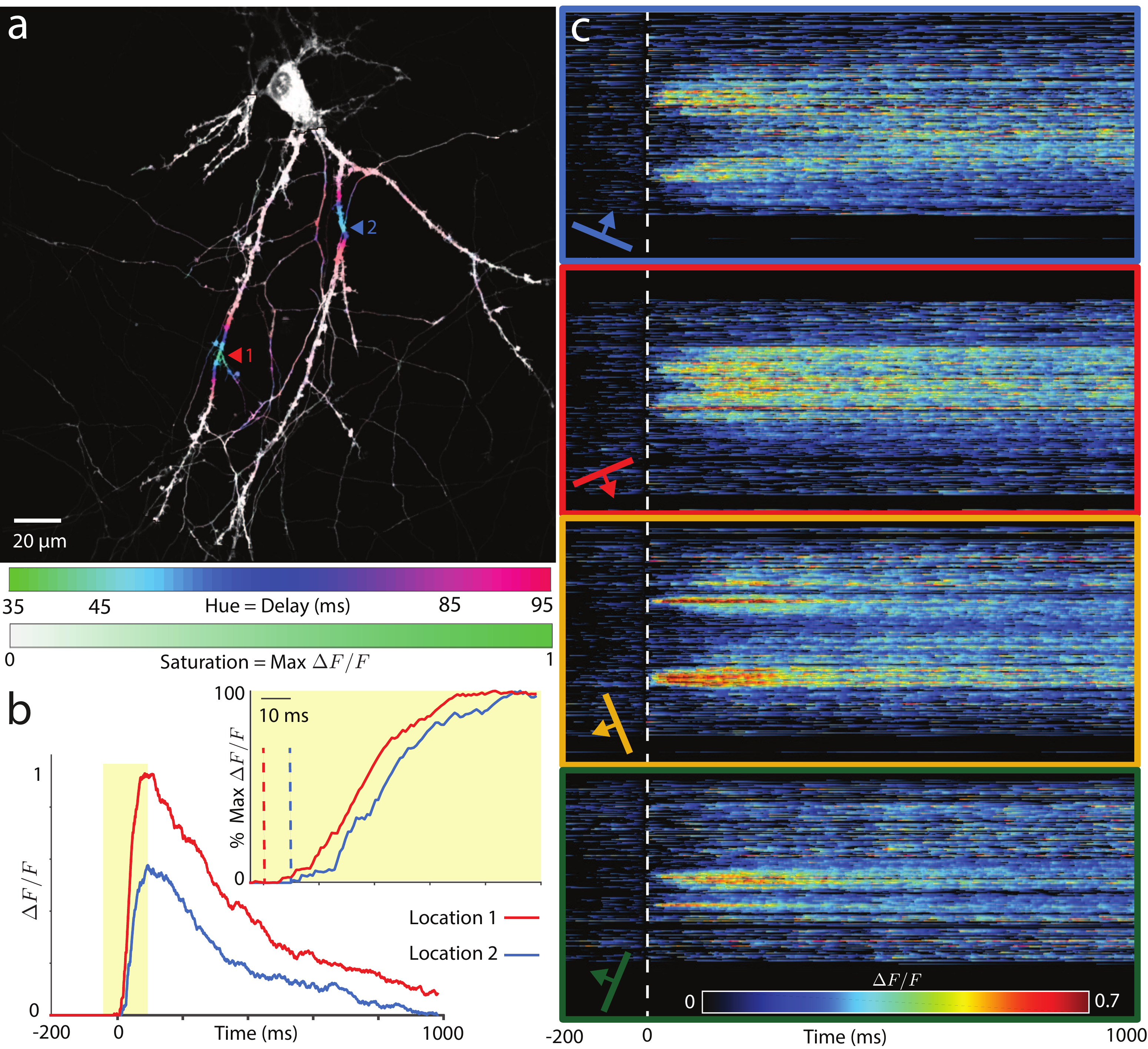}
\caption{In Vitro validation, two-spot uncaging experiment: a)	Raster image for 2spot uncaging, with uncaging locations (1 and 2) denoted by red and blue arrows, colored according to delay to half-max $\frac{\Delta F}{F}$,
b)	Dff traces for the two uncaging locations. Inset: zoomed in $\frac{\Delta F}{F}$ traces normalized to peak $\frac{\Delta F}{F}$. Dashed lines denote the stimulus onset,
c)	2D SLAPMi lixel-space measurement of the sample, consisting of projections along the four scan axes, showing diffusion. The dashed line marks the stimulus onset in uncaging location 1.
}\label{fig:slapmi_inVitro2}
\end{center}
\vspace{-5mm}
\end{figure}

\subsection{In Vivo Activity Imaging}

SLAPMi inherits the resolution enhancement and insensitivity to scattering that make two-photon imaging effective in vivo. However, we were concerned that tissue heating from light absorption could limit the practicality of two-photon projection microscopy. Economy of illumination power is critical for biological imaging \cite{laissue2017assessing}. Conventional two-photon imaging is limited by brain heating under common configurations \cite{podgorski2016brain}. Two-photon projection microscopy is even more limited by heating, because higher degrees of parallelization require a matched increase in power to maintain nonlinear excitation efficiency.  Lower degrees of parallelization make two-photon excitation more efficient and source recovery more effective. In general, multiphoton projection methods benefit by using the lowest degree of parallelization compatible with an experiment's required framerate \cite{podgorski2016brain, sofroniew2016large}. 

To reduce optical power at the sample and control degree of parallelization, we included a spatial light modulator (SLM) at an intermediate focal plane in an amplitude modulation geometry. This configuration selects an arbitrary pattern in the focal plane for imaging, and discards the remaining excitation light, thereby combining benefits of random access imaging and projection microscopy. Unlike AOD-based methods, random access SLAPMi has no access time. By retaining a buffer area surrounding each region of interest, SLAPMi remains insensitive to sample motion with no reduction in imaging rate. When imaging sparse samples such as dendrites, the SLM and other optimizations for excitation efficiency allow us to use average powers lower than conventional raster scanning.

%\textcolor{red}{
%GCaMP
%	To demonstrate SLAPMi in vivo, we imaged dendrites in the visual cortex of awake mice
%Tuning curves SLAPMi vs linear galvos.
%Glutamate
%Voltage
%NMF without segmentation
%In some cases, a voxel-space representation of activity components is not necessary for a given experiment, for example when extracting principal components of population activity (REFs, see also Laura Waller lightfield paper, Vaziri seeded demixing). In such cases, the SLAPMi measurement matrix can be directly factorized into temporal components and their corresponding projections in the measurement space.}

\section{Methods} \label{slapmi:methods}

\subsection{Solver}
The solver is a slight modification of the one introduced in Chapter \ref{chap:multiplicative}, where the baseline was assumed to be known. The baseline is meant to capture model mismatches in the data due to imperfect alignment and motion correction. Optionally, we allow a regularization term to impose prior knowledge of the indicator (maximum $\Delta F /F$) though this was not necessary for the in vitro and in vivo datasets presented here. The RL iterations are known to amplify noise for large number of iterations, especially in very sparse samples such as the particle localization and tracking data. In such scenarios a small $\ell_q$-regularization ($q \le 1$) proved to be helpful in pruning such artifacts. (Figure \ref{fig:slapmi_particles1}). We also implemented an optional damping, which is known to be useful in suppressing the artifacts due to low photon counts \cite{white1994image}.

The only user-supplied parameters involved in recovery are:
\begin{itemize}
\item The decay time constant, a property of the indicator,
\item the number of segments (not needed),
\item the convergence threshold of the solver, or, number of multiplicative iterations,
\item the damping parameter which determines the level of suppression of the objective function for the low-photon counts.
\end{itemize}

Mixed measurements can introduce spurious correlations into recovered signals, even when the mixing matrix is invertible, because measurement noise is shared among recovered sources. In simulations of SLAPMi imaging and source recovery, we accurately recovered signal amplitudes and correlations between sources without significant bias (See Section \ref{slapmi:methods_sim}). Recovery was robust to errors in segmentation and kinetics. 
In addition to the existing solver, we tried several solvers and models including:
\begin{itemize}
\item A Kalman filtering and smoothing approach by fitting a Gaussian approximation to the state space model \cite{smith2003estimating}. The point process smoothing is very slow and exhibits weak approximation abilities for ambient dimensions higher than a few.
\item	Nonnegative least squares with temporal dynamics, which did not suit the Poisson noise model.
\item	Gradient descent methods such as spiral \cite{harmany2010spiral}, which happened to be sensitive to the step size.
\item We also considered a multiplicative baseline model and a rank 1 baseline model for the data, which showed inferior performance.
\end{itemize}

\subsection{Simulations}
\label{slapmi:methods_sim}
We evaluated performance of the solver under different conditions using simulated data. We used the delta PSF for the simulations. Unless otherwise stated, we use the following default parameters:
p = 500 segments,  40 random generations of the dynamics,  uniformly distributed between $0.5-5$.
The solver was initialized by a constant positive solution for the spikes and a time constant of  100 frames,
We assumed a photon count of 100 Photons per frame per segment which is close to the measured datasets,T = 500 frames with one spike per segments at frame 100, segments placed randomly within a 500 $\times$ 500 pixel circle at the center of the field of view and the maximum allowed $\Delta F/F $  for the solver set to 10. Segments are assumed to be 10 $\times$10 pixel squares. The null distributions were obtained by random shuffling of time points for each seed.

\textbf{Evaluation metric:} we report the Pearson correlation between the estimated signal and its ground-truth counterpart over  time for each segment. 

In the first simulation (Figure \ref{fig:corr_vs_ppp}) we evaluated the effect of sample brightness on performance of the solver. This was done by fixing the expected Poisson rates per segment per frame to vary in the range $0.1$ to $10^5$. As a comparison, typical values of this parameter in the recorded datasets in this work are around 10-100. Increasing the photon budget improves the reconstructions.

%We first increased the expected photons per segment per frame from $0.1$ to $10^5$ (Figure \ref{fig:corr_vs_ppp}). In the samples of interest in this paper the typical value is around 100. As can be noted, the estimates improve by increasing the photon budget.
\begin{figure}[H]
\begin{center}
\noindent
\includegraphics[width=0.9\columnwidth]{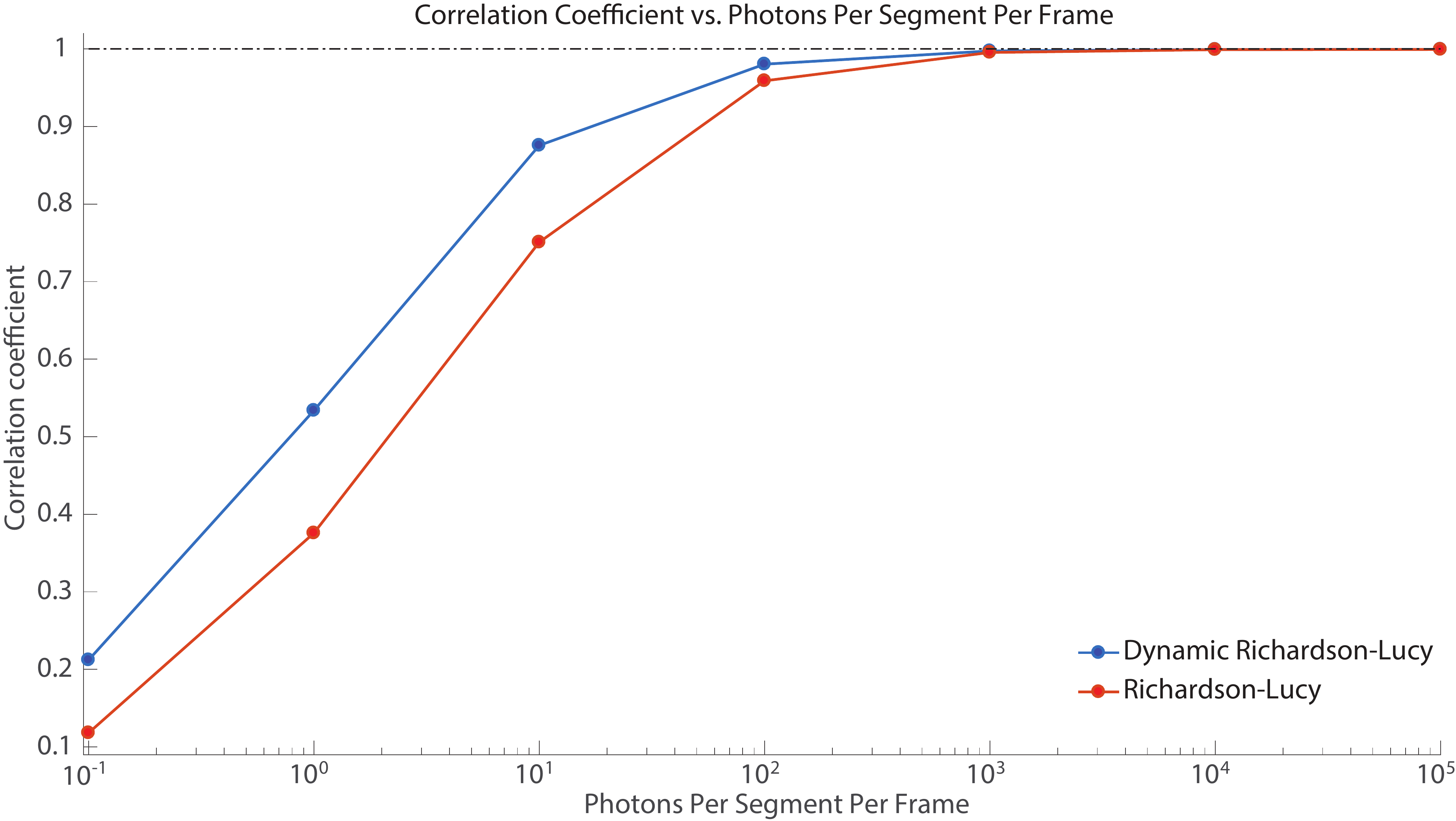}
\caption{Effect of sample brightness on the solver.
}\label{fig:corr_vs_ppp}
\end{center}
\vspace{-5mm}
\end{figure}
In the second simulation (Figure \ref{fig:corr_vs_nseeds}), we changed the number of sources in the range 10-1000, in both the generation of the ground truth and solver. Increasing the number of sources decreases solver performance.

\begin{figure}[H]
\begin{center}
\noindent
\includegraphics[width=0.9\columnwidth]{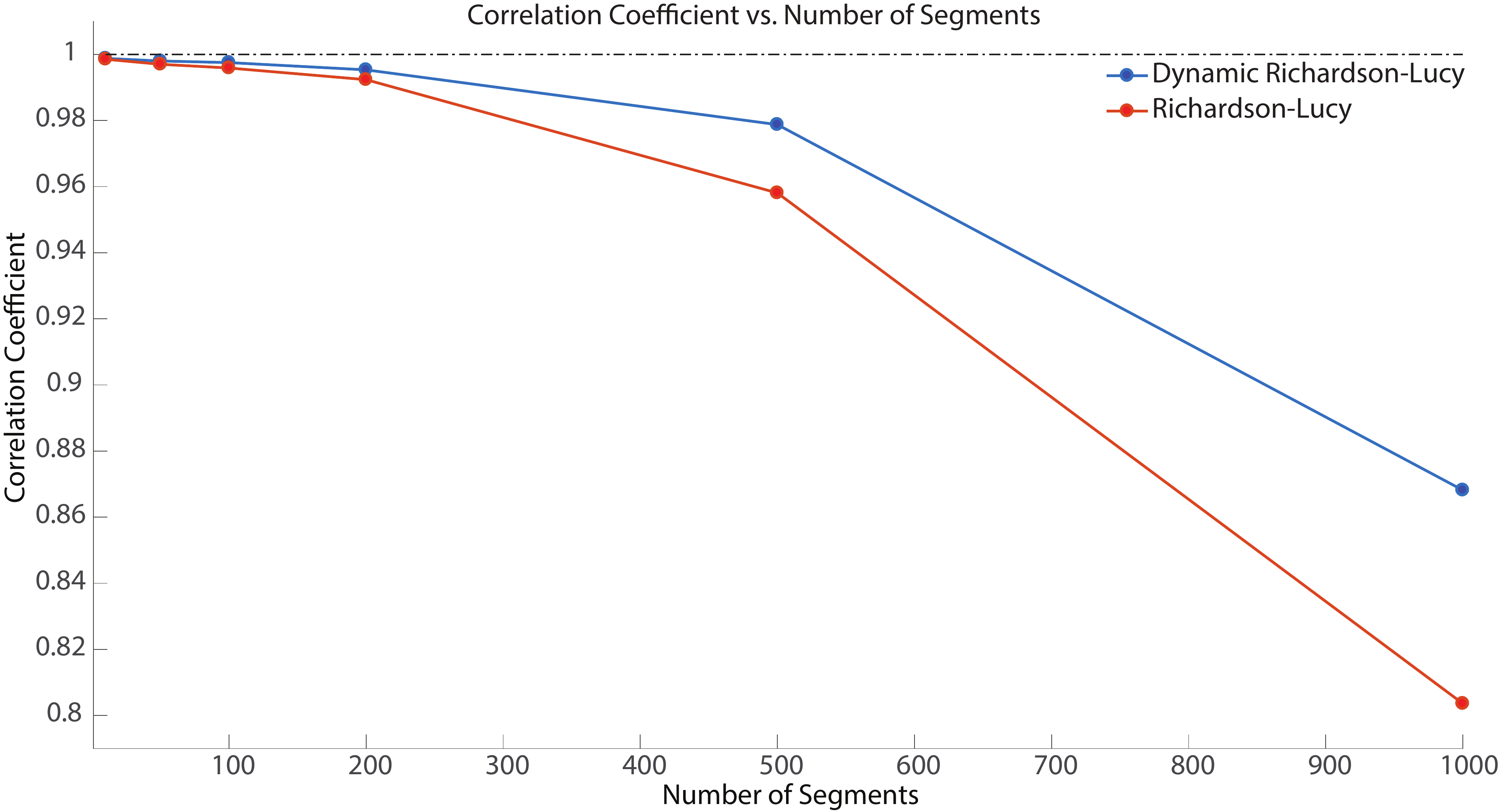}
\caption{Effect of number of sources on the solver.
}\label{fig:corr_vs_nseeds}
\end{center}
\vspace{-5mm}
\end{figure}

In the third simulation (Figure \ref{fig:corr_vs_tau}), we evaluated performance using different estimated decay time constants. The ground truth activity was generated using a time constant of 100 frames. We varied the time-constant in the range 0-1000. The solver performs best when the correct time constant is used. The solver is insensitive to underestimation of the time-constant, as this can be compensated with additional spikes but does reduce the denoising benefit of the dynamics model. A time constant of 0 corresponds to normal Richardson-Lucy iterations. Substantial overestimation of the time constant degrades the performance of the solver. The null distributions for different time constants is shown in Figure \ref{fig:corr_vs_tau_null}.

\begin{figure}[H]
\begin{center}
\noindent
\includegraphics[width=0.9\columnwidth]{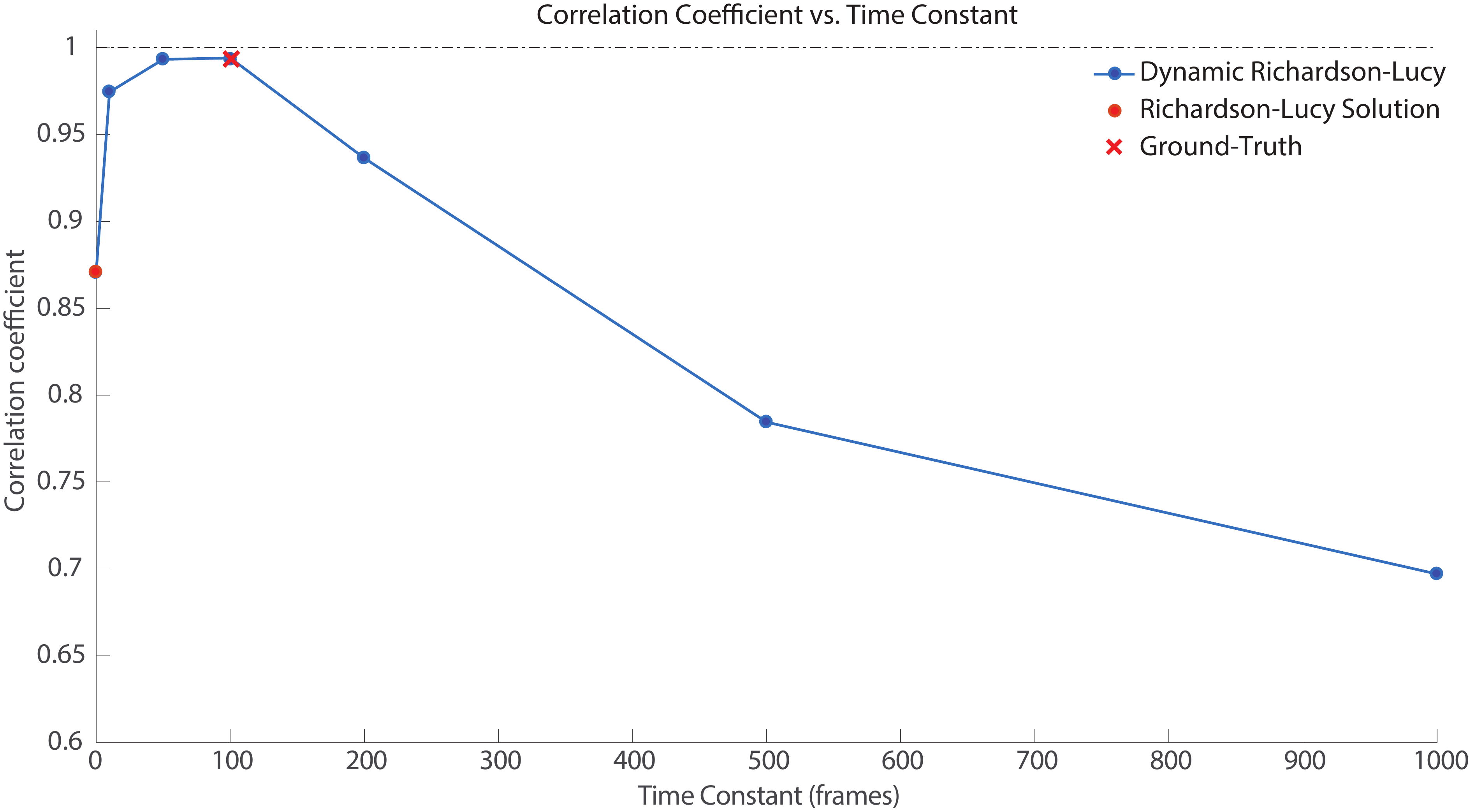}
\caption{The effect of erroneous decay time-constant on the solver.
}\label{fig:corr_vs_tau}
\end{center}
\vspace{-5mm}
\end{figure}

\begin{figure}[H]
\begin{center}
\noindent
\includegraphics[width=0.9\columnwidth]{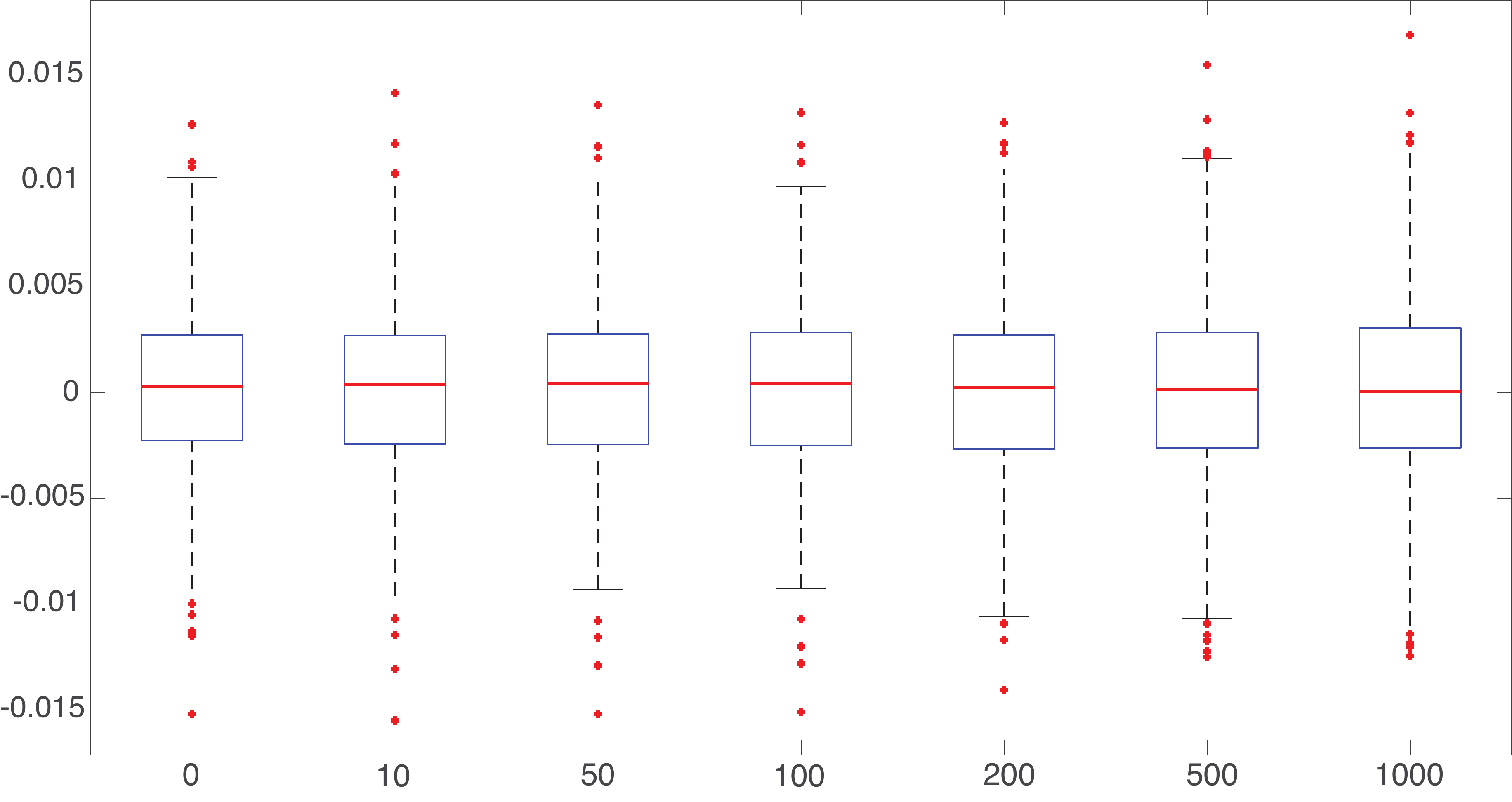}
\caption{Null distributions for different time-constants.
}\label{fig:corr_vs_tau_null}
\end{center}
\vspace{-5mm}
\end{figure}

In the fourth simulation (Figure \ref{fig:corr_vs_unsuspected}), we evaluated the performance of the solver in the case where not all sources in the field of view were included in the segmentation. We added 0-50\% additional unsegmented sources. This is meant to simulate fluorescence activity within the 'ON' region of the SLM but not detected in the reference image, a possibility in sensors with very low baseline fluorescence. Increasing unsegmented activity degrades the performance of the solver.

%In the next simulation we considered the possibility of missing some active spots in the segmentation process by adding 0-50\% of unsegmented activity (Figure \ref{fig:corr_vs_unsuspected}. 
The null distributions for different levels of unsegmented activities are shown in Figures \ref{fig:corr_vs_unsuspected_null_dynamic} and \ref{fig:corr_vs_unsuspected_null_RL}.

\begin{figure}[H]
\begin{center}
\noindent
\includegraphics[width=0.9\columnwidth]{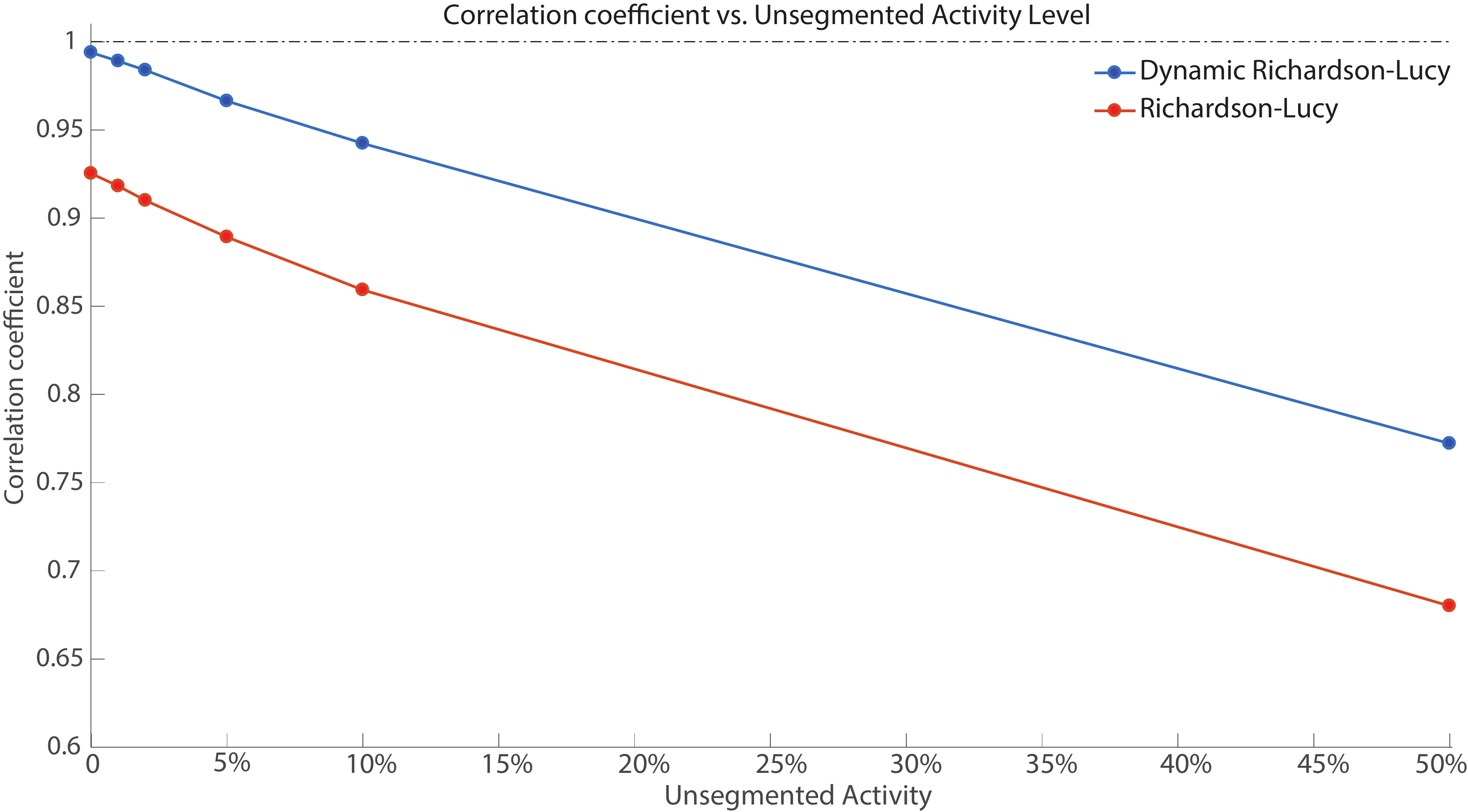}
\caption{The effect of unsegmented activity on the reconstructions.
}\label{fig:corr_vs_unsuspected}
\end{center}
\vspace{-5mm}
\end{figure}

\begin{figure}[H]
\begin{center}
\noindent
\includegraphics[width=0.9\columnwidth]{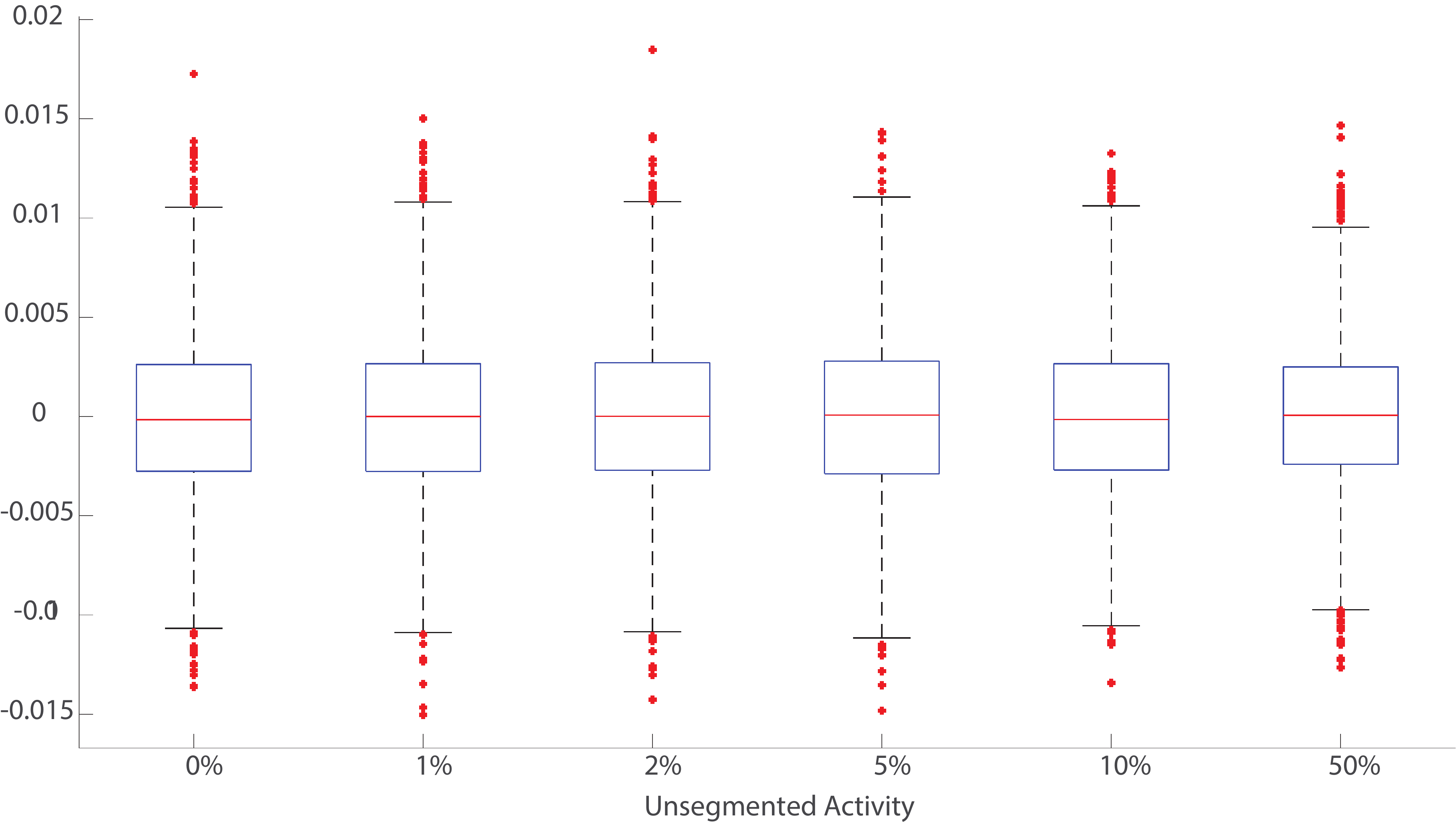}
\caption{Null distributions for different levels of unsegmented activity.
}\label{fig:corr_vs_unsuspected_null_dynamic}
\end{center}
\vspace{-5mm}
\end{figure}

\begin{figure}[H]
\begin{center}
\noindent
\includegraphics[width=0.9\columnwidth]{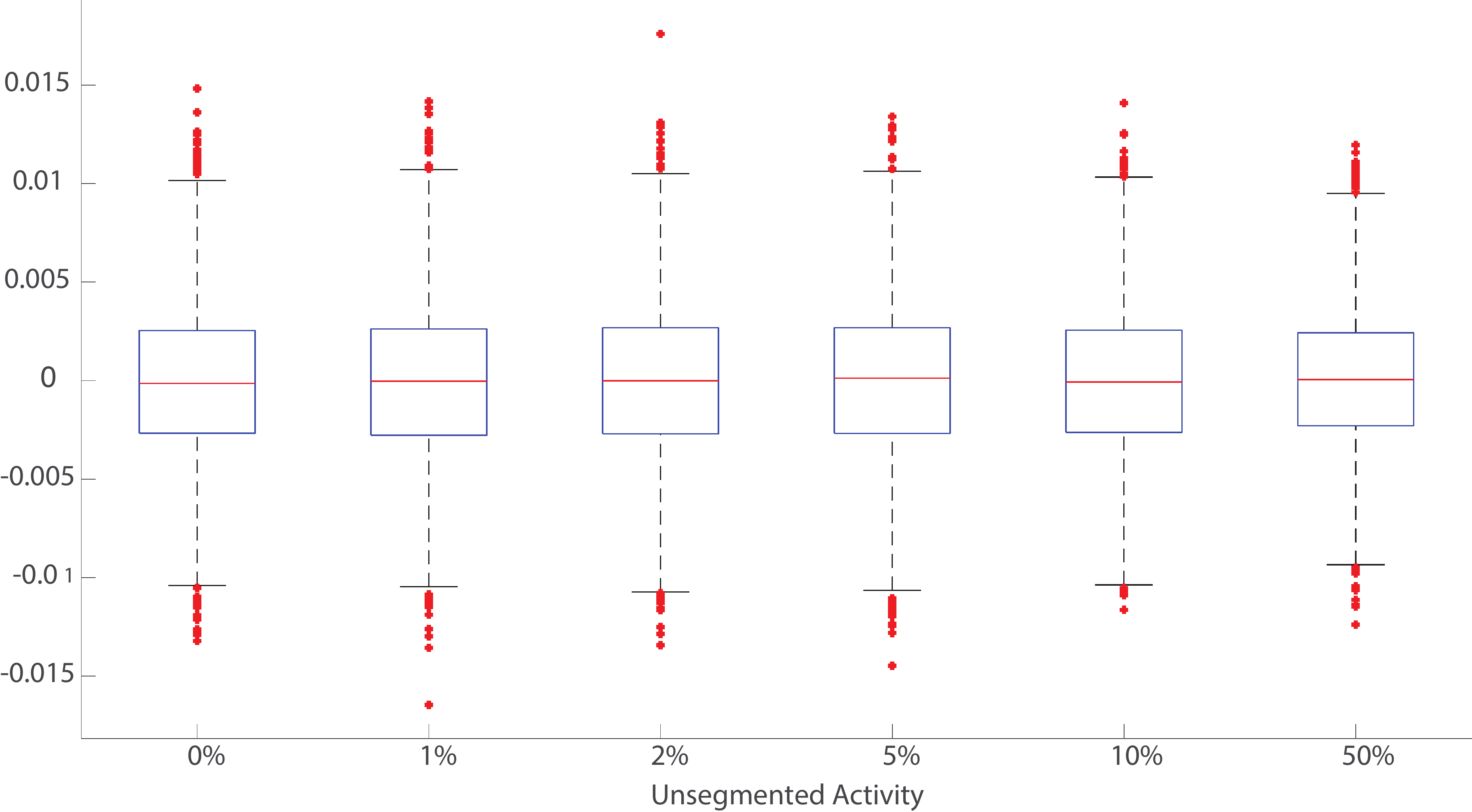}
\caption{Null distributions for the unsegmented activity for the frame by frame solver.
}\label{fig:corr_vs_unsuspected_null_RL}
\end{center}
\vspace{-5mm}
\end{figure}

In the fifth simulation (Figure \ref{fig:corr_vs_nshifts}) we evaluated sensitivity of the solver to alignment errors. In order to do so, we shifted the reference image by 0-10 pixels after registration. Better alignment improves the reconstructions.

\begin{figure}[H]
\begin{center}
\noindent
\includegraphics[width=0.9\columnwidth]{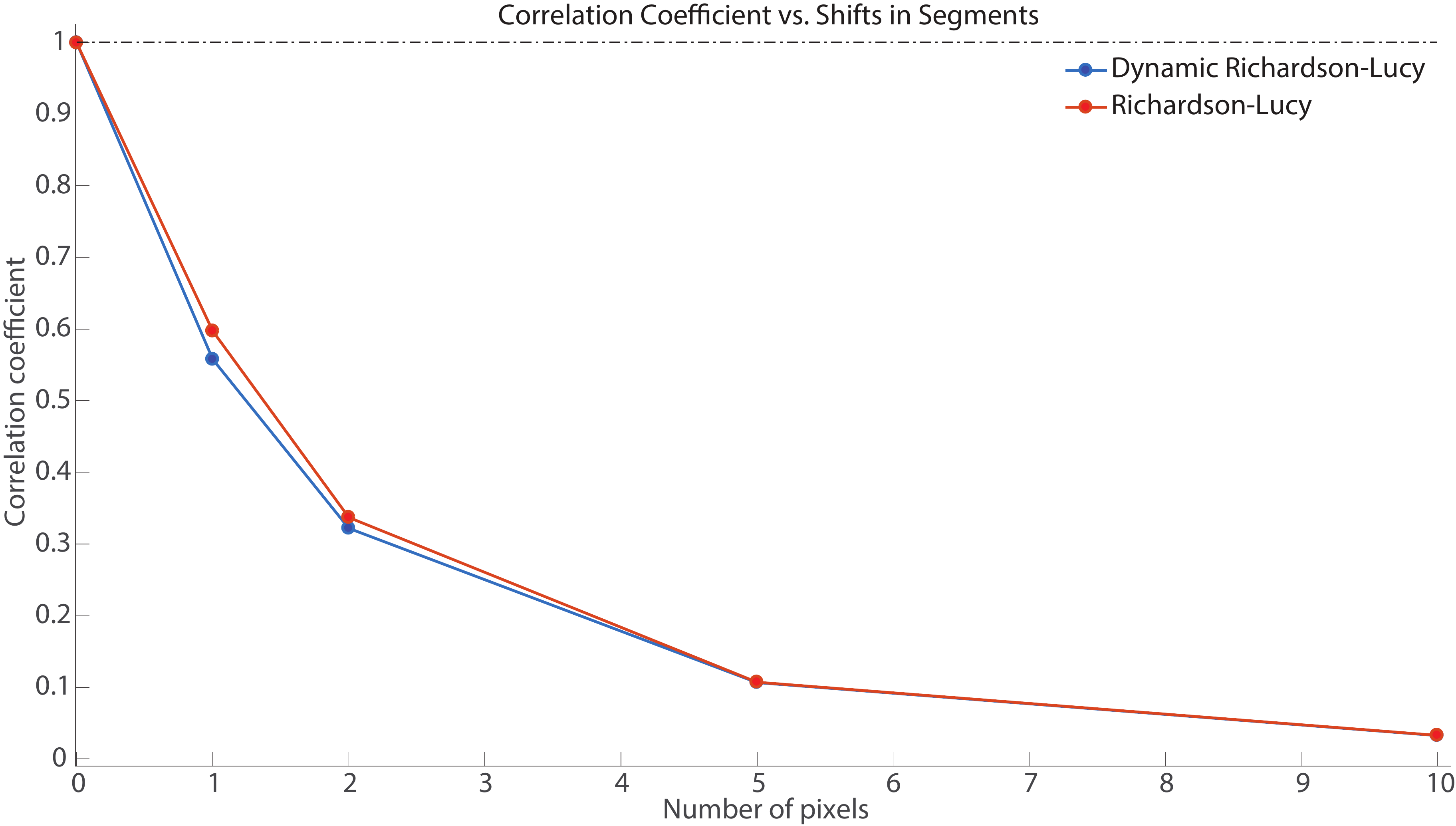}
\caption{The effect of alignment (shifts in pixels) on the reconstructions.
}\label{fig:corr_vs_nshifts}
\end{center}
\vspace{-5mm}
\end{figure}

In the sixth simulation (Figure \ref{fig:corr_vs_alternate_seg}) we evaluated performance of the solver when the reconstruction segmentations do not match the ground-truth activity segmentations. Alternate segmentations lead to a degraded performance which can be partially compensated for by using a finer segmentation than the ground truth.

%We considered 10 alternate segmentations for a neuron. The ground-truth activity was compared with the alternate segmentations. The performance of the algorithm degrades with when the ground-truth segmentations are not available. This can be partly resolved by Increasing the number of segments (Figure \ref{fig:corr_vs_alternate_seg}).

\begin{figure}[H]
\begin{center}
\noindent
\includegraphics[width=0.9\columnwidth]{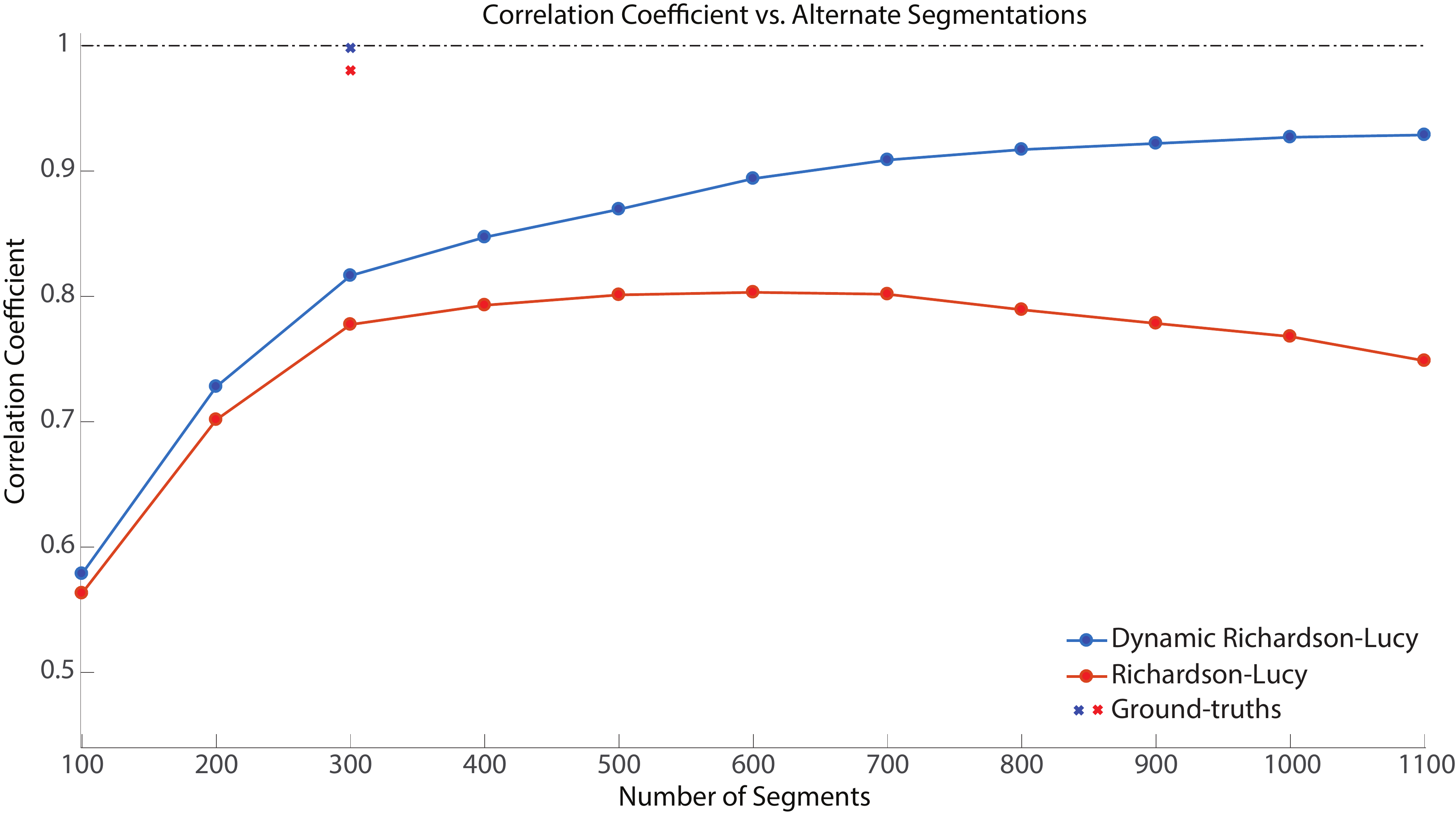}
\caption{The effect of alternate segmentations on the reconstructions.
}\label{fig:corr_vs_alternate_seg}
\end{center}
\vspace{-5mm}
\end{figure}

We next quantified the undesired correlations introduced to the activities (Figures \ref{fig:corr_mtx_dyn} and \ref{fig:corr_mtx_RL}) as a result of the solver for the dynamic RL solver compared to the RL iterations. We generated 3000 frames of activity with random spiking patterns using vines and the extended onion method \cite{lewandowski2009generating}, resulting in correlated  activity with correlations in the range -0.6 -0.6. As can be noted, recovered activities using the dynamic RL iterations significantly improves.

%Finally, in Figures \ref{fig:corr_mtx_dyn} and \ref{fig:corr_mtx_RL} we quantified the amount of undesired correlations introduced to or removed from the segments. For this purpose, we simulated 3000 frames of activity with random spiking patterns in time. We generated correlated  spikes  with correlations in the range -0.6-0.6 using the vine method \cite{lewandowski2009generating}. We then scattered the correlations estimated  from the activities versus the ground-truth correlations.

\begin{figure}[H]
\begin{center}
\noindent
\includegraphics[width=0.7\columnwidth]{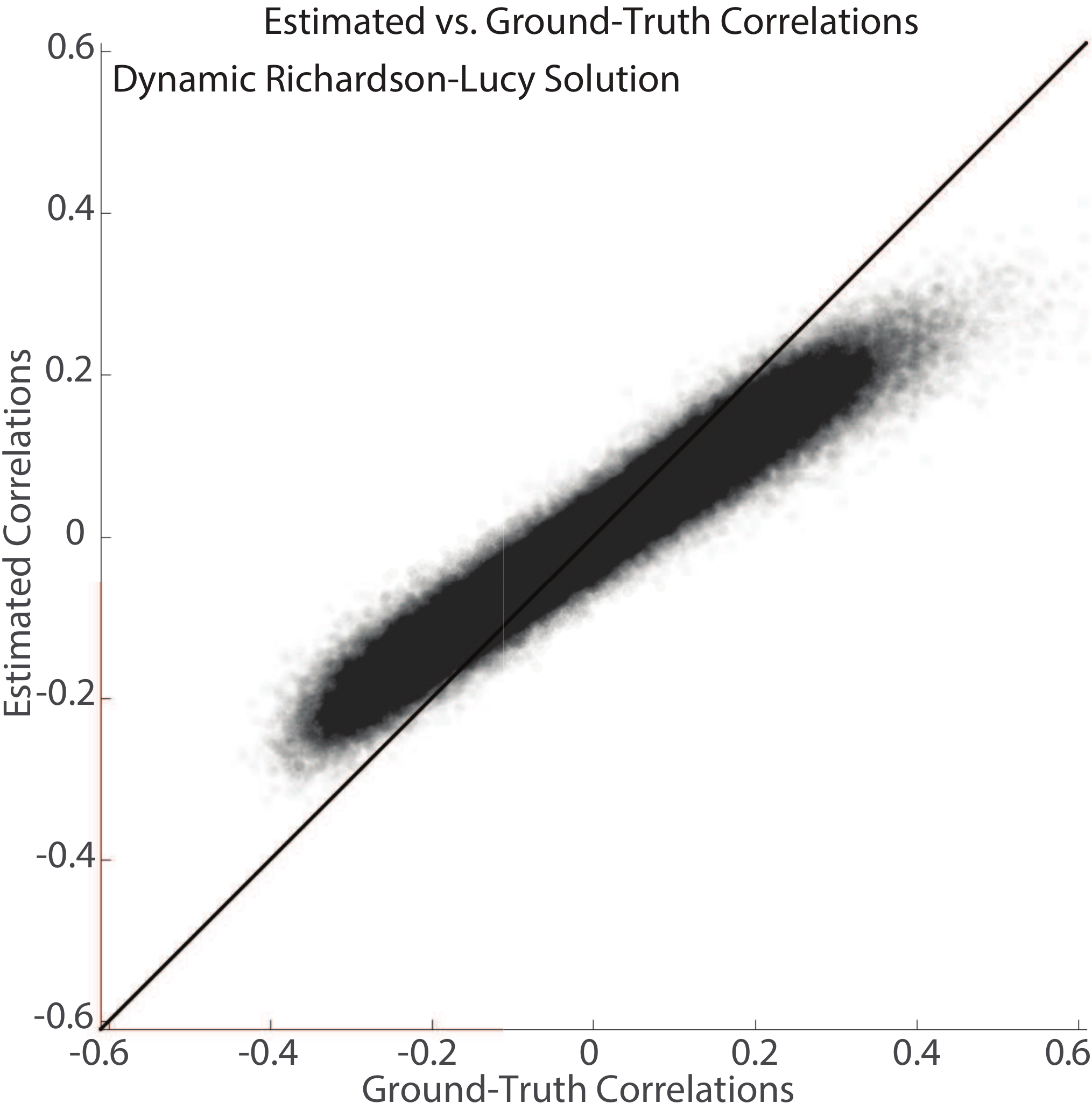}
\caption{Recovered activity correlations using the dynamic solver .
}\label{fig:corr_mtx_dyn}
\end{center}
\vspace{-5mm}
\end{figure}

\begin{figure}[H]
\begin{center}
\noindent
\includegraphics[width=0.7\columnwidth]{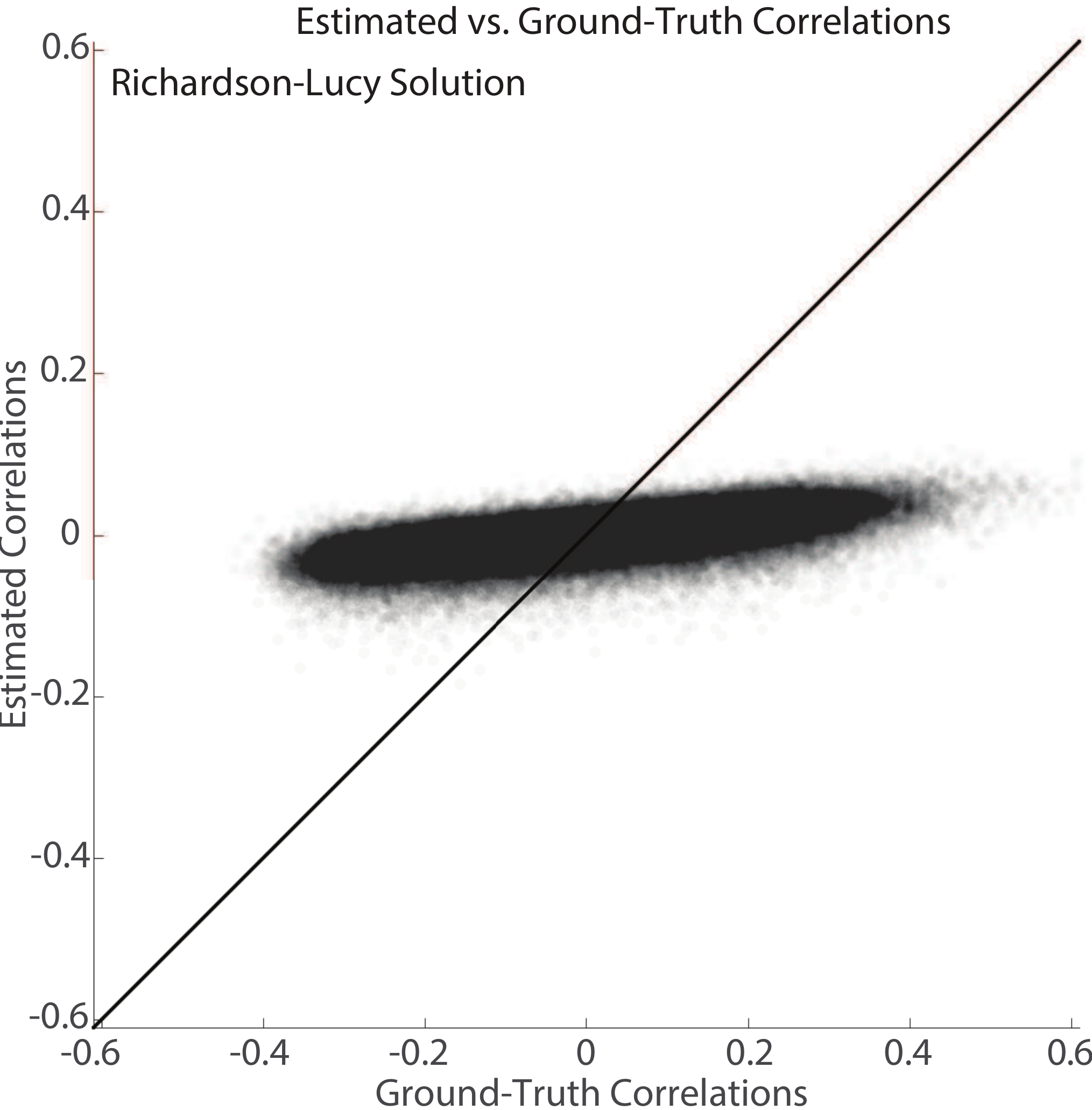}
\caption{Recovered activity correlations using the frame by frame solver.
}\label{fig:corr_mtx_RL}
\end{center}
\vspace{-5mm}
\end{figure}

\subsection{Software}
Software was written in Matlab and LabView FPGA. SLAPMi interfaces with ScanImage (Vidrio Technologies) to perform raster scanning.
% Plans for the microscope, instructions for use, example datasets, and all code used in data acquisition, reduction, source recovery, and analysis, are available at \cite{}

\subsection{Measurement Matrix}
The projection matrix ($\mathbf{P}$) is measured in an automated calibration step using a thin ($\ll$1$\mu$m) fluorescent film. Images of the excitation focus in the film, collected by a camera, allow a correspondence to be made between the positions of galvanometer scanners and the location of the resulting line focus. The raster scanning focus is also mapped, allowing us to create a model of the line foci transformed into the space of the sample image obtained by the raster scan.

\subsection{Motion Registration and Alignment}
Recorded SLAPMi data are spatially registered to compensate for sample motion. As with raster imaging, translations of a single resolution element can be sufficient to impact recovery of activity in fine structures, necessitating precise registration.
Accurate registration and source recovery rely on accurate, minimally warped reference raster images. SLAPMi was designed to efficiently interface with freely available ScanImage software (Vidrio Technologies), which is used to collect raster stacks. These images can be warped by many factors, including nonlinearity in the scan pattern, sample motion, and the ‘rolling shutter’ artifact of the raster scan. We estimate and compensate for warping by collecting two sets of reference images interleaved, one with each of the two galvos acting as the fast axis. To compensate for motion and activity variations during reference image acquisition we obtain a large number of stacks and rely on consensus between aligned stacks to reject these artifacts.
Registration of SLAPMi recordings is performed by identifying the 3D translation that maximizes the sum of 1-dimensional correlations (or optionally, Dynamic Time Warping distances) between the recorded signal and the expected projection of the reference image on each of the four scan axes. If the SLM is not used, this objective is maximized using cross-correlations, where the SLM is used, we perform an iterative multiscale grid search.

\section{Concluding Remarks}

In vivo imaging techniques are becoming increasingly specialized, with different methods best suited for different organisms and experimental parameters \cite{yang2017vivo}. Until now, random access imaging has been the most effective method for imaging hundreds of target sites, spanning hundreds of microns, hundreds of times per second, in scattering tissue. SLAPMi performs such measurements in highly dynamic samples at rates exceeding 1kHz. SLAPMi improves upon existing projection microscopy techniques by having lower coherence, higher frame rates, and random access excitation. Methods that scan a static multifocal pattern or Bessel beam have tended to mix sample voxels coherently, making unmixing difficult where objects overlap in their measurements. SLAPMi's angular projections, in contrast, ensure that no two voxels are always mixed together in measurements. Scanned Bessel beam imaging records volume projections as fast as 2D scans, but the 2D scan rate is not increased. SLAPMi records from planes with just four 1D scans. The approach of scanning excitation patterns across an SLM dramatically reduces power usage, and could be used similarly with other projection schemes provided the excitation pattern lies in the focal plane.

SLAPMi has several limitations. Recovered traces become less precise as the number of distinct objects e.g. neurons, imaged increases above 1000. As with all projection methods, certain sample structures are adversarial to source recovery due to the compressed nature of the measurements. In particular, regions where many distinct objects are packed closely together may not be uniquely determined by SLAPMi measurements, and might not be accurately recovered even when the total number of objects imaged is less than 1000. Dim objects surrounded by extremely bright objects are difficult to recover, because Poisson noise originating from bright sources may exceed signal from the dim source on all projection axes. The SLM allows dense or bright regions to be selectively dimmed, and can ameliorate these issues. SLAPMi has a maximum frame rate limited by the 2D galvanometer scanners. In samples compatible with random access imaging having a very small number of target sites, random access imaging may allow higher frame rates. 

SLAPMi achieves the same lateral resolution as raster scanning, but axial sectioning is reduced, as two photon intensity of lines drops off as $\frac{1}{z}$, compared to $\frac{1}{z^2}$ for points. In general we find SLAPMi produces extremely robust results without parameter tuning in a wide variety of sample preparations, making it a practical alternative to other imaging methods.

%\textcolor{red}{
%Applications of particle tracking; Resolution is sharp, superresolution possible with same methods
%Discussion of future of voltage imaging}

\chapter{Conclusions and Future Work}
\label{chap:conclusions}

In this thesis we revisited and made improvements over several theoretical aspects of compressive sensing for nonlinear and dynamic models.
 
From a theoretical perspective, in Chapters \ref{chap:ar} and \ref{chap:hawkes} we derived minimax optimal sampling-complexity tradeoffs for autoregressive processes, point processes and generalized linear models where the covariates do not satisfy the conventional i.i.d. assumptions. We consider extension of these theoretical results to multivariate processes as future work. The results on point processes were motivated by their applications in characterizing the self-exciting and history dependence nature of neural spiking activities. Our results on autoregressive processes started from a class project for the Compressive 
Sensing class taught by Professor Piya Pal at UMD. The main idea behind the theoretical guarantees came from  the theoretical results on convergence of eigenvalues of covariance matrices that I learned during the course I took on adaptive filter theory taught by Professor Ali Olfat at University of Tehran. 

From an algorithm design point of view, in Chapter \ref{chap:css} we introduced the idea of compressible state-space models with sparse innovations and provided a fast, optimal recovery algorithm for such models. These models have huge applications in modeling biological signals. Many beautiful intuitions about such state-space models came from suggestions of Professor Prakash Narayan during my research proposal exam at UMD. We consider generalization of these models to heavy-tailed innovations as future work.

Finally, in Chapter \ref{chap:slapmi} we developed a two-photon imaging technique that scans lines of excitation across the sample at multiple angles, recovering high-resolution images from relatively few incoherently multiplexed measurements. By combining traditional Poisson image reconstruction techniques with temporal dynamics we managed to reconstruct neural activity in behaving animals at framerates higher than 1 kHz for Megapixel fields of view. This research was in collaboration with Janelia research campus. Many intuitions about the reconstruction algorithms and the experiment came from the suggestions of several group leaders at Janelia. Combined with the evolution of fast sensors we consider studying neuronal dynamical systems using SLAPMi as future work.

\appendix
\chapter{Proof of Theoretical Results}
\label{chap:proofs}
\section{Proofs of Main Theorems for Autoregressive Processes }
\subsection{The Restricted Strong Convexity of the matrix of covariates}
The first element of the proofs of both Theorems \ref{thm:ar_1} and \ref{thm_OMP} is to establish the Restricted Strong Convexity (RSC) for the matrix $\mathbf{X}$ of covariates formed from the observed data. First, we investigate the closely related {Restricted} Eigenvalue (RE) condition. Let $[\lambda_{\sf min}(s)$, $\lambda_{\sf max}(s)]$ be the smallest interval containing the singular values of $\frac{1}{n} (\mathbf{X}_S' \mathbf{X}_S)$, where $\mathbf{X}_S$ is a sub-matrix $\mathbf{X}$ over an index set $S$ of size $s$.

\begin{definition}[Restricted Eigenvalue Condition]
\label{RE_def}
A matrix $\mathbf{X}$ is said to satisfy the RE condition of order $s$ if $\lambda_{\sf min}(s) > 0$.
\end{definition}
Although the RE condition only restricts $\lambda_{\sf min}(s)$, in the following analysis we also keep track of $\lambda_{\sf max}(s)$, which appears in some of the bounds. Establishing the RSC for $\mathbf{X}$ proceeds in a sequence of lemmas (Lemmas \ref{eig_conv}--\ref{RE_RSC} culminating in Lemma \ref{lem:rsc}). We first show that the RE condition holds for the true covariance of an AR process:
\begin{lemma}[from \cite{grenander1958toeplitz}]
\label{eig_conv}
Let $\mathbf{R} \in \mathbb{R}^{k \times k}$ be the $k \times k$ covariance matrix of a stationary process with power spectral density $S(\omega)$, and denote its maximum and minimum eigenvalues by $\phi_{\max}(k)$ and $\phi_{\sf min}(k)$, respectively. Then, $\phi_{\max}(k)$ is increasing in $k$, $\phi_{ \sf min}(k)$ is decreasing in $k$, and we have
\begin{equation}
\phi_{\sf min}(k) \downarrow \inf_{\omega}S(\omega), \quad \mbox{and} \quad \phi_{\sf max}(k) \uparrow \sup_{\omega}S(\omega).
\end{equation}
\end{lemma}
\noindent This result gives us the following corollary:
\begin{corollary}[Singular Value Spread of $\mathbf{R}$]
\label{cor:eig_conv}
Under the {sufficient stability assumption}, the singular values  of the covariance $\mathbf{R}$ of an AR process lie in the interval $\left [\frac{\sigma^2_{\sf w}}{8 \pi}, \frac{\sigma^2_{\sf w}}{2 \pi \eta^2} \right]$.
\end{corollary}
\begin{proof}
For an AR($p$) process
\[
S(\omega) = \frac{1}{2\pi}\frac{\sigma^2_{\sf w}}{|1- \sum_{\ell=1}^p \theta_{\ell} e^{-j\ell\omega}|^2}.
\]
Combining $\|\boldsymbol{\theta}\|_1 \leq 1-\eta < 1$ with Lemma \ref{eig_conv} proves the claim.
\end{proof}

Note that by Lemma \ref{eig_conv}, the result of Corollary \ref{eig_conv} not only holds for AR processes, but also for \textit{any} stationary process satisfying $\inf_\omega S(\omega) >0$ and $\sup_\omega S(\omega) < \infty$, i.e., a process with finite spectral spread.

We next establish conditions for the RE condition to hold for the empirical covariance $\widehat{\mathbf{R}}$:

\begin{lemma}\label{lem:re}
If the singular values of $\mathbf{R}$ lie in the interval $[\lambda_{\sf min}, \lambda_{\sf max}]$, then $\mathbf{X}$ satisfies the RE condition of order { $s_\star$} with parameters ${\lambda}_{\sf min}(s_\star) = \lambda_{\sf min} - t s_\star$ and ${\lambda}_{\sf max} (s_\star)= \lambda_{\sf max} +ts_\star$, where $t = \max_{i,j} |\widehat{R}_{ij}-R_{ij}|$.
\end{lemma}
\begin{proof}
Let $\widehat{\mathbf{R}} = \frac{1}{n} (\mathbf{X}^T \mathbf{X})$. For every $s_\star$-sparse $\boldsymbol{\theta}$ we have
\[
\boldsymbol{\theta}' \widehat{\mathbf{R}} \boldsymbol{\theta} \geq \boldsymbol{\theta}' {\mathbf{R}} \boldsymbol{\theta} - t \|\boldsymbol{\theta}\|_1^2 \geq (\lambda_{\sf min} - t s_\star) \|\boldsymbol{\theta}\|_2^2,
\]
\[\boldsymbol{\theta}' \widehat{\mathbf{R}} \boldsymbol{\theta} \leq \boldsymbol{\theta}' {\mathbf{R}} \boldsymbol{\theta} + t \|\boldsymbol{\theta}\|_1^2 \leq (\lambda_{\sf max} + t s_\star) \|\boldsymbol{\theta}\|_2^2,\]
which proves the claim.
\end{proof}

We will next show that $t$ can be suitably controlled with high probability. Before doing so, we state a {key} result of Rudzkis \cite{rudzkis1978large} regarding the concentration of second-order empirical sums from stationary processes:
\begin{lemma}
\label{conc_biased}
Let $\mathbf{x}_{-p+1}^n$ be samples of a stationary process which satisfies 
\vspace{-.2cm}
\begin{equation}
\label{eq:wold}
x_k = \sum_{j= -\infty}^\infty b_{j-k} w_j,
\vspace{-.2cm}
\end{equation}
\noindent where $w_k$'s are i.i.d random variables with
\begin{equation}
\label{eq: bounded_moments}
|\mathbb{E}(|w_j|^k)| \leq ({\tilde{c} \sigma_{\sf w}})^k k!, \ k=2, 3, \cdots,
\end{equation} 
for some constant $\tilde{c}$ and
\vspace{-.2cm}
\begin{equation}
\label{eq: abs_sum}
\sum_{j=-\infty}^\infty |b_j| < \infty.
\end{equation}
\vspace{-.2cm}
Then, the \textit{biased} sample autocorrelation given by
\[\widehat{r}^b_k=\frac{1}{n+k}\sum_{i,j=1,j-i=k}^{n+k}x_ix_j\]
\vspace{-.2cm}
\noindent satisfies
\begin{equation}
\label{conc_ineq_biased}
\mathbb{P}(|\widehat{r}^b_k - r^b_k|>t) \leq c_1 (n+k) \exp \left(-{\frac{c_2}{\sigma_{\sf w}} \frac{t^2 (n+k)}{c_3 \sigma_{\sf w}^3 + t^{3/2} \sqrt{n+k}}}\right),
\end{equation}
for { positive absolute constants $c_1$, $c_2$ and $c_3$ which are independent of the dimensions of the problem. In particular, if $x_k = w_k$, i.e., a sub-Gaussian white noise process, $c_3$ vanishes.}
\end{lemma}
\begin{proof}
The lemma is a special case of Theorem 4 under Condition 2 of Remark 3 in \cite{rudzkis1978large}. {For the special case of $x_k = w_k$, the constant $H$ in Lemma 7 of \cite{rudzkis1978large} and hence $c_3$ vanish.}
\end{proof}

{Using the result of} Lemma \ref{conc_biased}, we can control $t$ and establish the RE condition for $\widehat{\mathbf{R}}$ as follows:
\begin{lemma}\label{lem:re2}
Let $m$ be a positive integer. Then, $\mathbf{X}$ satisfies the RE condition of order $(m+1)s$ with a constant $\lambda_{\sf min}/2$ with probability at least
\vspace{-.4cm}
\begin{equation}
1 - c_1 p^2 (n+p) \exp \left(-\frac{c_4 \sqrt{\frac{n}{s}}}{1 + c_5 \frac{n+p}{\left(\frac{n}{s}\right)^{3/2}}}\right),
\end{equation}
\vspace{-.2cm}
\noindent where $c_1$ is the same as in Lemma \ref{conc_biased}, $c_4 = \frac{c_2}{\sigma_{\sf w}} \sqrt{ \frac{\lambda_{\sf min}}{2(m+1)}}$ and $c_5 = \frac{c_3 \sigma_{\sf w}^3}{\left( \frac{\lambda_{\sf min}}{2(m+1)}\right)^{3/2}}$.
\end{lemma}

\begin{proof}
First, note that for the given AR process, condition (\ref{eq:wold}) is verified by the Wold decomposition of the process, condition (\ref{eq: bounded_moments}) results from the sub-Gaussian assumption on the innovations, and condition (\ref{eq: abs_sum}) results from the stability of the process. Noting that
\vspace{-.2cm}
\begin{equation}
\widehat{R}_{i,i+k} = \frac{1}{n}\sum_{i=1}^n x_i x_{i+k} = \frac{1}{n}\sum_{i,j=1,j-i=k}^{n+k}x_ix_j = \frac{n+k}{n}\widehat{r}^b_k, 
\end{equation}
for $i=1,\cdots,n$ and $k = 0, \cdots, p-1$, Eq. (\ref{conc_ineq_biased}) implies:
\begin{equation}
\displaystyle \mathbb{P}\left ( |\widehat{R}_{i,i+k} - {R}_{i,i+k}|> \tau \right) \leq c_1 (n+k) \exp \left( - \frac{ c_2 \sqrt{\tau n}}{\frac{c_3 \sigma_{\sf w}^4 (n+k)}{\tau^{3/2} n^{3/2}} + \sigma_{\sf w}}\right).
\end{equation}
By the union bound and $k \le p$, we get:
\begin{align}\label{eq:max}
\displaystyle \mathbb{P}\left(\max_{i,j}|\widehat{R}_{ij}-R_{ij}|> \tau \right) \leq c_1 p^2 (n+p) \exp \left( - \frac{ c_2 \sqrt{\tau n}}{\frac{c_3 \sigma_{\sf w}^4 (n+p)}{\tau^{3/2} n^{3/2}} + \sigma_{\sf w}}\right).
\end{align}
Choosing $\tau= \frac{\lambda_{\sf min}}{2(m+1)s}$ and invoking the result of Lemma \ref{lem:re} establishes the result of the lemma.
\end{proof}

We next define the closely related notion of the Restricted Strong Convexity (RSC):

\begin{definition}[Restricted Strong Convexity \cite{Negahban}]
\label{RSC_def}
Let 
\begin{equation}
\label{cone_condition}
\mathbb{V}:= \{\mathbf{h}\in \mathbb{R}^p | \|\mathbf{h}_{S^c}\|_1 \leq 3\| \mathbf{h}_S\|_1+4\|\boldsymbol{\theta}_{S^c}\|_1\}.
\end{equation}
Then, $\mathbf{X}$ is said to satisfy the RSC condition of order $s$ if there exists a positive $\kappa > 0$ such that
\begin{equation}
\frac{1}{n}\mathbf{h}' \mathbf{X}' \mathbf{X} \mathbf{h} = \frac{1}{n}\|\mathbf{X}\mathbf{h}\|_2^2 \geq \kappa \|\mathbf{h}\|_2^2, \;\;\;\; \forall \mathbf{h} \in \mathbb{V}.
\end{equation}
\end{definition}

The RSC condition can be deduced from the RE condition according to the following result:
\begin{lemma}[Lemma 4.1 of \cite{bickel2009simultaneous}]
\label{RE_RSC}
If $\mathbf{X}$ satisfies the RE condition of order $s_\star = (m+1)s$ with a constant $\lambda_{\sf min}((m+1)s)$, then the RSC condition of order $s$ holds with
\vspace{-.2cm}
\begin{equation}
\kappa = {\lambda_{\sf min}((m+1)s)}\left( 1- 3 \sqrt{\frac{\lambda_{\sf max}(ms)}{m\lambda_{\sf min}\left((m+1)s\right)}}\right)^2.
\end{equation}
\end{lemma}

We can now establish the RSC condition of order $s$ for $\mathbf{X}$:
\begin{lemma}\label{lem:rsc}
The matrix of covariates $\mathbf{X}$ satisfies the RSC condition of order $s$ with a constant $\kappa = \frac{\sigma^2_{\sf w}}{16 \pi}$ with probability at least
\vspace{-.3cm}
\begin{equation}
\label{ar:eq_rsc}
1 - c_1 p^2 (n+p) \exp \left(-\frac{c_{\eta} \sqrt{\frac{n}{s}}}{1 + c'_{\eta} \frac{n+p}{\left(\frac{n}{s}\right)^{3/2}}}\right),\end{equation}
where $c_\eta = \frac{c_2 \eta}{\sqrt{16 \pi ( 72 + \eta^2)}}$ and $c'_\eta = \frac{c_3 (16 \pi (72 + \eta^2))^{3/2}}{\eta^3}$.
\end{lemma}
\begin{proof}
Choosing $m = \lceil \frac{72}{\eta^2} \rceil$, and using Lemmas \ref{lem:re}, \ref{lem:re2}, and \ref{RE_RSC} establishes the result. {Note that if $x_k = w_k$, i.e., a sub-Gaussian white noise process, then $c_3$ and hence $c'_\eta$ vanish.}
\end{proof}

We are now ready prove Theorems \ref{thm:ar_1} and \ref{thm_OMP}.

\subsection{Proof of Theorem \ref{thm:ar_1}}

We first establish the so-called vase (cone) condition for the error vector $\mathbf{h} = \widehat{\boldsymbol{\theta}}_{\ell_1}-{\boldsymbol{\theta}}$:

\label{app:ar_main}
\begin{lemma}
For a choice of the regularization parameter $\gamma_n \ge \| \nabla \mathfrak{L}(\boldsymbol{\theta}) \|_\infty = \frac{2}{n} \| \mathbf{X}' \left(\mathbf{x}_1^n-\mathbf{X}\boldsymbol{\theta}\right) \|_{\infty}$, the optimal error $\mathbf{h} = \widehat{\boldsymbol{\theta}}_{\ell_1}-{\boldsymbol{\theta}}$ belongs to the vase
\begin{equation}
\mathbb{V}:= \{\mathbf{h}\in \mathbb{R}^p | \|\mathbf{h}_{S^c}\|_1 \leq 3\| \mathbf{h}_S\|_1+4\|\boldsymbol{\theta}_{S^c}\|_1\}.
\end{equation}
\end{lemma}

\begin{proof}
Using several instances of the triangle inequality we have:
\begin{align*}
0 & \geq  \frac{1}{n} \left(\|\mathbf{x}_1^n-\mathbf{X}(\boldsymbol{\theta}+\mathbf{h})\|_2^2- \|\mathbf{x}_1^n-\mathbf{X}\boldsymbol{\theta}\|_2^2 \right)+  \gamma_n \left( \|\boldsymbol{\theta}+\mathbf{h}\|_1 - \|\boldsymbol{\theta}\|_1 \right)\\
& \geq -  \frac{1}{n} \| \mathbf{X}^T \left(\mathbf{x}_1^n-\mathbf{X}\boldsymbol{\theta}\right) \|_{\infty} \|\mathbf{h}\|_1 + \gamma_n \left( \|\boldsymbol{\theta}_S+\mathbf{h}_{S^c}+\mathbf{h}_S +\boldsymbol{\theta}_{S^c}\|_1 - \|\boldsymbol{\theta}\|_1 \right)\\
& \geq -  \frac{\gamma_n}{2} (\|\mathbf{h}_{S^c}\|_1+\|\mathbf{h}_S\|_1) +\gamma_n \left( \|\boldsymbol{\theta}_S+\mathbf{h}_{S^c}\|_1-\|\mathbf{h}_S +\boldsymbol{\theta}_{S^c}\|_1 - \|\boldsymbol{\theta}\|_1 \right)\\
& = -  \frac{\gamma_n}{2} (\|\mathbf{h}_{S^c}\|_1+\|\mathbf{h}_S\|_1) + \gamma_n (\|\boldsymbol{\theta}_S\|_1+\|\mathbf{h}_{S^c}\|_1-\|\mathbf{h}_S\|_1-\|\boldsymbol{\theta}_{S^c}\|_1-\|\boldsymbol{\theta}_{S^c}\|_1- \|\boldsymbol{\theta}_S\|_1)\\
& =  \frac{\gamma_n}{2} (\|\mathbf{h}_{S^c}\|_1-3\|\mathbf{h}_{S}\|_1 - 4\|\boldsymbol{\theta}_{S^c}\|_1).
\end{align*}
\end{proof}

The following result of Negahban et al. \cite{Negahban} allows us to characterize the desired error bound: 
\begin{lemma}[Theorem 1 of \cite{Negahban}]
\label{RSC_thm}
If $\mathbf{X}$ satisfies the RSC condition of order $s$ with  a constant $\kappa > 0$ and $\gamma_n \ge \| \nabla \mathfrak{L}(\boldsymbol{\theta}) \|_\infty$, then any optimal solution $\widehat{\boldsymbol{\theta}}_{\ell_1}$ satisfies
\begin{align}\label{eq:rsc_bound}
\nonumber & \|\widehat{\boldsymbol{\theta}}_{\ell_1}-\boldsymbol{\theta}\|_2 \leq  \frac{2\sqrt{s}\gamma_n}{\kappa}+ \sqrt{\frac{2\gamma_n \sigma_s(\boldsymbol{\theta})}{\kappa}} \tag{$\star$}.
%\\
%\nonumber & \|\widehat{\boldsymbol{\theta}}_{\ell_1}-\boldsymbol{\theta}\|_1 \leq \frac{6{s}\gamma_n}{\kappa}. 
\end{align}
\end{lemma}

In order to use Lemma \ref{RSC_thm}, we need to control $\gamma_n = \| \nabla \mathfrak{L}(\boldsymbol{\theta}) \|_\infty$. We have:
\begin{equation}
\nabla \mathfrak{L}(\boldsymbol{\theta}) = \frac{2}{n} \mathbf{X}' (\mathbf{x}_1^n-\mathbf{X}\boldsymbol{\theta}),
\end{equation}
It is easy to check that by the uncorrelatedness of the innovations $w_k$'s, we have
\begin{equation}
\label{eq:expec=0}
\mathbb{E} \left[ \nabla \mathfrak{L}(\boldsymbol{\theta}) \right] = \frac{2}{n} \mathbb{E} \left[ \mathbf{X}' (\mathbf{x}_1^n-\mathbf{X}\boldsymbol{\theta}) \right]=\frac{2}{n}\mathbb{E} \left[ \mathbf{X}' \mathbf{w}_1^n \right]= \mathbf{0}.
\end{equation}
Eq. (\ref{eq:expec=0}) is known as the orthogonality principle. We next show that $ \nabla \mathfrak{L}(\boldsymbol{\theta})$ is concentrated around its mean. We can write 
\begin{equation*}
\left(\nabla \mathfrak{L}(\boldsymbol{\theta})\right)_i = \frac{2}{n} \mathbf{x}^{{n-i}'}_{{-i+1}} \mathbf{w}_1^n,
\end{equation*}
and observe that the $j$th element in this expansion is of the form $y_j = x_{n-i-j+1}w_{n-j+1}$. It is easy to check that the sequence $y_1^n$ is a martingale with respect to the filtration given by
\begin{equation*}
\label{filtration}
\mathcal{F}_{j}=\sigma \left( \mathbf{x}_{-p+1}^{n-j+1} \right),
\end{equation*}
where $\sigma(\cdot)$ denote the sigma-field generated by the random variables $x_{-p+1}, x_{-p+2}, \cdots, x_{n-j+1}$. 
We use the following concentration result for sums of dependent random variables \cite{van_de_geer}:
\begin{lemma}
\label{hoeff_dep}
Fix $n\geq 1$. Let $Z_j$'s be sub-Gaussian $\mathcal{F}_j$-measurable random variables, satisfying for each $j=1,2,\cdots,n$,
\begin{equation*}
\mathbb{E}\left[Z_j|\mathcal{F}_{j-1}\right] = 0, \;\; \text{almost surely},
\end{equation*}
then there exists a constant $c$ such that for all $t>0$,
\begin{equation*}
\mathbb{P} \left( \left| \frac{1}{n}  \sum_{j=1}^n Z_j - \mathbb{E}[Z_j] \right| \geq t \right)\leq \exp\left(-\frac{nt^2}{c^2}\right).
\end{equation*}
\end{lemma}
\begin{proof}
This is a special case of Theorem 3.2 of \cite{van_de_geer} or Lemma 3.2 of \cite{geer2000empirical}, for sub-Gaussian-weighted sums of random variables. The constant $c$ depends on the sub-Gaussian constant of $Z_i$'s.
\end{proof}
Since $y_j$'s are a product of two independent sub-Gaussian random variables, they are sub-Gaussian as well. Lemma \ref{hoeff_dep} implies that
\vspace{-.2cm}  
\begin{equation}
\label{bound}
{\mathbb{P}\left( |\nabla \mathfrak{L}(\boldsymbol{\theta})_i|  \ge t \right) \leq  \exp\left(-\frac{nt^2}{c^2_0 \sigma^4_{\sf w}}\right).}
\end{equation}
where $c^2_0 := \frac{c^2}{\sigma_{\sf w}^4}$ is an absolute constant. By the union bound, we get:
\begin{equation}
\label{ubound}
{\mathbb{P}\Big(\left\| \nabla \mathfrak{L}(\boldsymbol{\theta})\right\|_\infty \ge t \Big) \leq  \exp\left(-\frac{t^2n}{c_0^2 \sigma^4_{\sf w}}+\log p\right).}
\end{equation}
Let $d_4$ be any positive integer. Choosing $t = c_0 \sigma^2_{\sf w}\sqrt{{1+d_4}}\sqrt{\frac{\log p}{n}}$, we get:
\vspace{-.2cm}
\begin{align}
\label{grad_bound}
\notag \mathbb{P}\left(\left\| \nabla \mathfrak{L}(\boldsymbol{\theta})\right\|_\infty \ge c_0 \sigma^2_{\sf w}\sqrt{{{1+d_4}}}\sqrt{\frac{\log p}{n}} \right) \leq \frac{2}{n^{d_4}}.
\end{align}
Hence, a choice of $\gamma_n = d_2 \sqrt{\frac{\log p}{n}}$ with $d_2 := c_0 \sigma^2_{\sf w}\sqrt{{1 + d_4}}$, satisfies $\gamma_n \ge \| \nabla \mathfrak{L}(\boldsymbol{\theta}) \|_\infty$ with probability at least $1 - \frac{2}{n^{d_4}}$. {Let $d_0:= \frac{(3+d_4)^2}{c_\eta^2}$} and {$d_1 = \frac{4 c'_\eta ( 3 + d_4)}{c_\eta}$}. Using Lemma \ref{lem:rsc}, the fact that {$n > s \max \{d_0 (\log p)^2, d_1 {(p \log p)^{1/2}}\}$} by hypothesis, and $p > n$ we have that the RSC of order $s$ hold for $\kappa = \frac{\sigma^2_{\sf w}}{16 \pi }$ with a probability at least $1- {\frac{2c_1}{p^{d_4}}} - \frac{1}{p^{d_4}}$. Combining these two assertions, the claim of Theorem 1 follows for $d_3 = 32 \pi c_0 \sqrt{1+d_4}$. \QEDB

\subsection{Proof of Theorem \ref{thm_OMP}}\label{prf:aromp}

The proof is mainly based on the following lemma, adopted from Theorem 2.1 of \cite{zhang_omp}, stating that the greedy procedure is successful in obtaining a reasonable $s^\star$-sparse approximation, if the cost function satisfies the RSC:
\begin{lemma}\label{prop_omp}
Let $s^\star$ be a constant such that
\begin{equation}
\label{sstar}
{s^\star \geq {4 \rho s}\log {20 \rho s}},
\end{equation}
and suppose that $\mathfrak{L}(\boldsymbol{\theta})$ satisfies RSC of order $s^\star$ with a constant $\kappa > 0$. Then, we have
\begin{equation*}
\left \|\widehat{\boldsymbol{\theta}}^{(s^\star)}_{{\sf OMP}}-\boldsymbol{\theta}_S \right \|_2 \leq \frac{\sqrt{6} \varepsilon_{s^\star}}{\kappa},
\end{equation*}
where $\eta_{s^\star}$ satisfies
\begin{equation}
\label{eps_bound}
\varepsilon_{s^\star} \leq \sqrt{s^\star+s} \|\nabla\mathfrak{L}(\boldsymbol{\theta}_S)\|_\infty .
\end{equation}
\end{lemma}
\begin{proof}
The proof is a specialization of the proof of Theorem 2.1 in \cite{zhang_omp} to our setting with the spectral spread ${\rho=1/4\eta^2}$.
\end{proof}

In order to use Lemma \ref{prop_omp}, we need to bound $\|\nabla\mathfrak{L}(\boldsymbol{\theta}_S)\|_\infty$. We have:
\begin{align*}
\mathbb{E} \left[\nabla\mathfrak{L}(\boldsymbol{\theta}_S)\right] &= \frac{1}{n}\mathbb{E} \left[\mathbf{X}' (\mathbf{x}_1^n-\mathbf{X}\boldsymbol{\theta}_S)\right] = \frac{1}{n}\mathbb{E} \left[\mathbf{X}' \mathbf{X} (\boldsymbol{\theta}-\boldsymbol{\theta}_S)\right] = \mathbf{R} (\boldsymbol{\theta}-\boldsymbol{\theta}_S) \le \frac{\sigma^2_{\sf w}}{2\pi\eta^2} \varsigma_s(\boldsymbol{\theta}) \mathbf{1},
\end{align*}
where in the second inequality we have used (\ref{eq:expec=0}), and the last inequality results from Corollary \ref{cor:eig_conv}. 
Let $d'_4$ be any positive integer. Using the result of Lemma \ref{hoeff_dep} together with the union bound yields:  
\begin{align*}
\displaystyle \mathbb{P}\left(\|\nabla\mathfrak{L}(\boldsymbol{\theta}_S)\|_\infty \geq  c_0 \sigma^2_{\sf w}\sqrt{{1+d'_4}} \sqrt{\frac{\log p}{n}} + \frac{\sigma^2_{\sf w}\varsigma_s(\boldsymbol{\theta})}{2 \pi \eta^2} \right) \leq \frac{2}{n^{d'_4}}.
\end{align*}
Hence, we get the following concentration result for $\varepsilon_{s^\star}$:
\begin{align}
\label{eps_prob}
\mathbb{P} \left(\varepsilon_{s^\star} \geq \sqrt{s^\star+s} \left(  c_0 \sigma^2_{\sf w}\sqrt{{1+d'_4}} \sqrt{\frac{\log p}{n}} + \frac{\sigma^2_{\sf w}\varsigma_s(\boldsymbol{\theta})}{2 \pi \eta^2}\right)\right)
\leq\frac{2}{n^{d'_4}}.
\end{align}
Noting that by (\ref{sstar}) we have $s^\star+s \le \frac{4 s \log s}{\eta^2} $. Let {$d'_0 = \frac{4(3+d'_4)^2}{\eta^2 c_\eta^2}$} and {$d'_{1} = \frac{16 c'_\eta ( 3 + d_4)}{c_\eta}$}. By the hypothesis of $\varsigma_s(\boldsymbol{\theta}) \le {A} s^{1- \frac{1}{\xi}}$ for some constant $A$, and invoking the results of Lemmas \ref{lem:rsc} and \ref{prop_omp}, we get:
\begin{align*}
\left \|\widehat{\boldsymbol{\theta}}^{(s^\star)}_{{\sf OMP}}-\boldsymbol{\theta}_S \right\|_2 &\leq d'_{2} \sqrt{\frac{s \log s \log p}{n}} + d''_{2} \sqrt{s\log s} \varsigma_s(\boldsymbol{\theta})\leq  d'_{2} \sqrt{\frac{s \log s \log p}{n}} + d''_{2} \frac{\sqrt{\log s}}{s^{\frac{1}{\xi}-\frac{3}{2}}},
\end{align*}
where $d'_{2} = \frac{16 \pi c_0 \sqrt{24 (1+d'_4)}}{\eta}$ and $d''_{2} = \frac{A }{ \pi \eta^3}$, 
with probability at least  $1- {\frac{2c_1}{p^{d'_4}}} - {\frac{1}{p^{d'_4}}} -\frac{2}{n^{d'_4}}$. Finally, we have: 
\begin{align*}
\left \|\widehat{\boldsymbol{\theta}}^{(s^\star)}_{{\sf OMP}}-\boldsymbol{\theta}\right\|_2 &= \left \|\widehat{\boldsymbol{\theta}}^{(s^\star)}_{{\sf OMP}}-\boldsymbol{\theta}_S +\boldsymbol{\theta}_S -\boldsymbol{\theta} \right \|_2  \leq \left \|\widehat{\boldsymbol{\theta}}^{(s^\star)}_{{\sf OMP}}-\boldsymbol{\theta}_S \right \|_2 + \|\boldsymbol{\theta}_S-\boldsymbol{\theta}\|_2.
\end{align*}
Choosing $d'_{3} = 2 d''_{2}$ completes the proof. \QEDB

\subsection{Proof of Proposition \ref{thm:ar_minimax}}
\label{prf:minimax}
\noindent Consider the event defined by
\begin{align*}
\mathcal{A}:=\left \{\max_{i,j}|\widehat{R}_{ij}-R_{ij}| \leq \tau \right \}.
\end{align*}
Eq. (\ref{eq:max}) in the proof of Lemma \ref{lem:re2} implies that:
\begin{align*}
\mathbb{P}(\mathcal{A}^c) \leq c_1 p^2 (n+p) \exp \left( - \frac{ c_2 \sqrt{\tau n}}{\frac{c_3 \sigma_{\sf w}^4 (n+p)}{\tau^{3/2} n^{3/2}} + \sigma_{\sf w}}\right).
\end{align*}
By choosing $\tau$ as in the proof of Theorem \ref{thm:ar_1}, we have
\begin{align*}
\mathcal{R}^2_{\sf est} (\widehat{\boldsymbol{\theta}}_{\sf minimax}) & \leq \mathcal{R}^2_{\sf est} (\widehat{\boldsymbol{\theta}}_{\ell_1}) = \sup_{\mathcal{H}} \left(\mathbb{E}\left[\|\widehat{\boldsymbol{\theta}}_{\ell_1}-\boldsymbol{\theta}\|_2^2\right]\right)\\
& \leq \mathbb{P}(\mathcal{A})  d^2_3 {\frac{s \log p}{n}} +\sup_{\mathcal{H}} \mathbb{E}_{{A}^c} \left [  \|\widehat{\boldsymbol{\theta}}_{\ell_1}-{\boldsymbol{\theta}}\|_2^2\right]\\
&\leq d^2_3 {\frac{s \log p}{n}}+ \displaystyle 8(1-\eta)^2 c_1 \exp \left(-\frac{ c_2 \sqrt{\tau n}}{\frac{c_3 \sigma_{\sf w}^4 (n+p)}{\tau^{3/2} n^{3/2}} + \sigma_{\sf w}} + 3 \log p\right),
\end{align*}
where the second inequality follows from Theorem \ref{thm:ar_1}, and the third inequality follows from the fact that $\| \widehat{\boldsymbol{\theta}}_{\ell_1} - \boldsymbol{\theta} \|_2^2 \le 4 (1-\eta)^2$ by the sufficient stability assumption. For {$n > s \max \{ d_0 (\log p)^2, d_1 {(p \log p)^{1/2}\}}$}, the first term will be the dominant, and thus we get $\mathcal{R}_{\sf est} (\widehat{\boldsymbol{\theta}}_{\sf minimax}) \le 2 d_3 \sqrt{\frac{s \log p}{n}}$, for large enough $n$.

As for a lower bound on $\mathcal{R}_{\sf est} (\widehat{\boldsymbol{\theta}}_{\sf minimax})$, we take the approach of \cite{goldenshluger2001nonasymptotic} by constructing a family of AR processes with sparse parameters $\boldsymbol{\theta}$ for which the minimax risk is optimal modulo constants. In our construction, we assume that the innovations are Gaussian. The key element of the proof is the Fano's inequality:
\begin{lemma}[Fano's Inequality]
\label{lem:ar_fano_ineq}
Let $\mathcal{Z}$ be a class of densities with a subclass $\mathcal{Z}^\star$ of densities $f_{\boldsymbol{\theta}_i}$, parameterized by $\boldsymbol{\theta}_i$, for $i \in \{0,\cdots,2^M \}$. Suppose that for any two distinct $\boldsymbol{\theta}_1, \boldsymbol{\theta}_2 \in \mathcal{Z}^\star$, 
$\mathcal{D}_{\sf KL}(f_{\boldsymbol{\theta}_1}\| f_{\boldsymbol{\theta}_2}) \leq \beta$ for some constant $\beta$. Let $\widehat{\boldsymbol{\theta}}$ be an estimate of the parameters. Then
\vspace{-.2cm}
\begin{equation}
\label{eq:fano}
\sup_j \mathbb{P}(\widehat{\boldsymbol{\theta}}\neq {\boldsymbol{\theta}_j}|H_j) \geq 1 - \frac{\beta+\log2}{M},
\end{equation}
where $H_j$ denotes the hypothesis that $\boldsymbol{\theta}_j$ is the true parameter, and induces the probability measure $\mathbb{P}(.|H_j)$.
\end{lemma}

{Consider a class $\mathcal{Z}$ of AR processes with $s$-sparse parameters over any subset $S \subset \{1,2,\cdots,p\}$ satisfying $|S| =s$,} with parameters given by
\begin{equation}
\label{eq:ar_minimax_params}
\theta_\ell = \pm e^{-m} \mathbbm{1}_S(\ell),
\end{equation}
where $m$ remains to be chosen. We also add the all zero vector $\boldsymbol{\theta}$ to $\mathcal{Z}$. For a fixed $S$,  we have $2^s + 1$ such parameters forming a subfamily $\mathcal{Z}_S$. Consider the maximal collection of ${p \choose s}$ subsets $S$ for which any two subsets differ in at least $s/4$ indices. The size of this collection can be identified by $A(p, \frac{s}{4}, s)$ in coding theory, where $A(n,d,w)$ represents the maximum size of a binary code of length $n$ with minimum distance $d$ and constant weight $w$ \cite{macwilliams1977theory}. We have
\[
A(p, {\textstyle \frac{s}{4} }, s) \ge \frac{p^{\frac{7}{8}s - 1}}{s!},
\]
for large enough $p$ (See Theorem 6 in \cite{graham1980lower}). Also, by the Gilbert-Varshamov bound \cite{macwilliams1977theory}, there exists a subfamily $\mathcal{Z}_S^\star \subset \mathcal{Z}_S$, of cardinality $|\mathcal{Z}_S^\star| \geq 2^{\lfloor s/8 \rfloor}+1$, such that any two distinct $\boldsymbol{\theta}_1,\boldsymbol{\theta}_2 \in \mathcal{Z}_S^\star$ differ at least in $s/16$ components. Thus  for $\boldsymbol{\theta}_1, \boldsymbol{\theta}_2 \in \mathcal{Z}^\star :=  \resizebox{!}{0.3cm}{$\displaystyle \bigcup_S$} \mathcal{Z}^\star_S$, we have
\vspace{-.2cm}
\begin{equation}
\label{eq:ar_theta_lowerbd}
\|\boldsymbol{\theta}_1-\boldsymbol{\theta}_2\|_2 \geq \frac{1}{4} \sqrt{s} e^{-m} =: \alpha,
\end{equation}
and $| \mathcal{Z}^\star | \ge \frac{p^{\frac{7}{8}s - 1}}{s!} 2^{\lfloor s/8 \rfloor}$. For an arbitrary estimate $\widehat{\boldsymbol{\theta}}$, consider the testing problem between the $\frac{p^{\frac{7}{8}s - 1}}{s!} 2^{\lfloor s/8 \rfloor}$ hypotheses $H_j: \boldsymbol{\theta}=\boldsymbol{\theta}_j \in \mathcal{Z}^\star$, using the minimum distance decoding strategy. Using Markov's inequality we have
\begin{align}
\label{eq:ar_sup_lowerbd}
 \sup_{\mathcal{Z}} \mathbb{E}\left[\|\widehat{\boldsymbol{\theta}}-{\boldsymbol{\theta}}\|_2\right] &\geq \sup_{\mathcal{Z}^\star} \mathbb{E}\left [\|\widehat{\boldsymbol{\theta}}-{\boldsymbol{\theta}}\|_2\right] \geq \frac{\alpha}{2} \sup_{\mathcal{Z}^\star}\mathbb{P}\left(\|\widehat{\boldsymbol{\theta}}-{\boldsymbol{\theta}}\|_2 \geq \frac{\alpha}{2} \right) = \frac{\alpha}{2}\sup_{j} \mathbb{P}\left(\widehat{\boldsymbol{\theta}}\neq {\boldsymbol{\theta}_j}|H_j \right).
 \end{align}
Let $f_{\boldsymbol{\theta}_j}$ denote joint probability distribution of $\{x_k\}_{k=1}^n$ conditioned on $\{x_k\}_{k=-p+1}^0$ under the hypothesis $H_j$. Using the Gaussian assumption on the innovations, for $i \neq j$, we have
\begin{align}
\label{eq:ar_kl_bound}
\notag
 \mathcal{D}_{\sf KL}(f_{\boldsymbol{\theta}_i}\|f_{\boldsymbol{\theta}_j}) & \leq \sup_{i\neq j} \mathbb{E}\left[\log \frac{f_{\boldsymbol{\theta}_i}}{f_{\boldsymbol{\theta}_j}} |H_i\right]\\
\notag & \leq \sup_{i \neq j}\displaystyle \mathbb{E}\left[-\frac{1}{2\sigma^2_{\sf w}} \sum_{k=1}^n \left( \left(x_k-\boldsymbol{\theta}_i'\mathbf{x}_{k-p}^{k-1} \right)^2 - \left(x_k-\boldsymbol{\theta}_j'\mathbf{x}_{k-p}^{k-1} \right)^2 \right)  \Big| H_i\right]\\
\notag & \leq \sup_{i \neq j} \frac{n}{2\sigma^2_{\sf w}} \mathbb{E}\left[ \left((\boldsymbol{\theta}_i-\boldsymbol{\theta}_j)'\mathbf{x}_{k-p}^{k-1}\right)^2 \Big| H_i \right]\\
\notag & = \frac{n}{2\sigma^2_{\sf w}}\sup_{i \neq j} (\boldsymbol{\theta}_i-\boldsymbol{\theta}_j)' \mathbf{R} (\boldsymbol{\theta}_i-\boldsymbol{\theta}_j)\\
& \leq \frac{n \lambda_{\sf max}}{2\sigma^2_{\sf w}}\sup_{i \neq j} \|\boldsymbol{\theta}_i-\boldsymbol{\theta}_j\|_2^2 \leq \frac{n s e^{-2m}}{64 \pi \eta^2} =: \beta.
\end{align}

Using Lemma \ref{lem:ar_fano_ineq}, (\ref{eq:ar_theta_lowerbd}), (\ref{eq:ar_sup_lowerbd}) and (\ref{eq:ar_kl_bound}) yield:
\begin{equation*}
\sup_{\mathcal{Z}} \mathbb{E}\left[\|\widehat{\boldsymbol{\theta}}-{\boldsymbol{\theta}}\|_2 \right] \geq \frac{\sqrt{s}  e^{-m}}{8}\left(1-\frac{2\left( \frac{  n s e^{-2m}}{64 \pi \eta^2}+\log 2\right)}{s \log p}\right).
\end{equation*}
for $p$ large enough so that $\log p \ge \frac{\log s - \frac{9}{8}}{\frac{3}{8} - \frac{1}{s}}$. Choosing $m = \frac{1}{2} \log \left ( \frac{n}{8 \pi \eta^2 \log p} \right )$ gives us the claim of Proposition \ref{thm:ar_minimax} with $L = \frac{d_3}{\eta \sqrt{2\pi}}$ for large enough $s$ and $p$ such that $s \log p \ge \log (256)$. The hypothesis of $s  \le \frac{1- \eta}{\sqrt{8\pi} \eta} \sqrt{\frac{n}{\log p}}$ guarantees that for all $\boldsymbol{\theta}\in \mathcal{Z}^\star$, we have $\|\boldsymbol{\theta}\|_1 \leq 1-\eta$. \QEDB

%In order to calculate the one-step prediction error assume that $\widehat{\boldsymbol{\theta}}_{\ell_1}$ is padded with zeros (possibly with infinitely many) so that it matches the order of $\boldsymbol{\theta}$. Now we have
%\begin{equation*}
%e_{n+1} =  \underbrace{\left( {\boldsymbol{\theta}}^{(p)}-\widehat{\boldsymbol{\theta}}_{{\ell_1}}^{(p)} \right)'x_{n-p+1}^{n}}_{\mathcal{E}_1} + \underbrace{\sum_{i=p+1}^{\infty}{\boldsymbol{\theta}}_ix_{n-i+1}}_{\mathcal{E}_2}+w_{n+1}.
%\end{equation*}
%Therefore
%\begin{equation}
%\mathbb{E}\left[e_{n+1}^2 \right] = \mathbb{E}\left[\left(\mathcal{E}_1+\mathcal{E}_2\right)^2\right]+1 \leq 2 \left(\mathbb{E}\left[\mathcal{E}_1^2\right] +\mathbb{E}\left[\mathcal{E}_2^2\right] \right)+1.
%\end{equation}
%We will now bound each term on the right hand side separately. We have
%
%
%\begin{align}
%\label{eq:ar_pred_mse}
%\notag \mathbb{E}\left[\mathcal{E}_2^2\right] & = \mathbb{E}\left[ \left( \left( {\boldsymbol{\theta}}^{(p)}-\widehat{\boldsymbol{\theta}}_{{\ell_1}} \right)'x_{n-p+1}^{n}\right)^2 \right]\\
%& \leq  \| {\boldsymbol{\theta}}^{(p)}-\widehat{\boldsymbol{\theta}}_{{\ell_1}}\|_2^2
%\end{align}

%\section{Generalization to Stable ]AR Models}
%\label{app:ar_generalization}

\vspace{-3mm}
\subsection{Generalization to stable AR processes}
We consider relaxing the sufficient stability assumption of $\| \boldsymbol{\theta} \|_1 \le 1- \eta < 1$ to $\boldsymbol{\theta}$ being in the set of stable AR processes. Given that the set of all stable AR processes is not necessarily convex, the LASSO and OMP estimates cannot be obtained by convex optimization techniques. Nevertheless, the results of Theorems \ref{thm:ar_1} and \ref{thm_OMP} can be generalized to the case of stable AR models:

\begin{corollary}\label{cor:ar_1} The claims of Theorems \ref{thm:ar_1} and \ref{thm_OMP} hold when $\boldsymbol{\Theta}$ is replaced by the set of stable AR processes, except for possibly slightly different constants.
\end{corollary}

\begin{proof} Note that the stability of the process guarantees boundedness of the power spectral density. The result follows by simply replacing the bounds $\left [\frac{\sigma^2_{\sf w}}{8 \pi}, \frac{\sigma^2_{\sf w}}{2 \pi \eta^2} \right]$ on the singular values of the covariance matrix $\mathbf{R}$ in Corollary \ref{cor:eig_conv} by $[\inf_\omega S(\omega), \sup_\omega S(\omega)]$.
\end{proof}

\section{Proofs of Main Theorems for GLM's} \label{appprf}

\subsection{Roadmap of the Proofs}
This appendix contains the proofs of Theorems \ref{negahban_lambda} and \ref{thm_OMP}, as well as Corollaries \ref{cor:1} and \ref{cor:2}. Before presenting the proofs, we establish some of the basic properties of the canonical self-exciting process (Proposition \ref{prop:hawkes_properties}) as well as our notational conventions as preliminaries. We then state a key result, namely Lemma \ref{lemma1}, which is at the core of the proofs of Theorems \ref{negahban_lambda} and \ref{thm_OMP}. The proofs are presented via a sequence of three propositions (Propositions \ref{negahban_thm}--\ref{prop_omp}) based on existing results in the literature, in conjunction with Proposition \ref{prop:hawkes_properties} and Lemma \ref{lemma1}. Therefore, Appendix \ref{appprf} is stand-alone modulo the proofs of Proposition \ref{prop:hawkes_properties} and Lemma \ref{lemma1}.

The proofs of Proposition \ref{prop:hawkes_properties} and Lemma \ref{lemma1} are presented in Appendix \ref{app:hawkes_psd}. In particular, the proof of Lemma \ref{lemma1} follows from two propositions (Propositions \ref{eig_conv} and \ref{prop:hawkes_concentration}). Therefore, Appendix \ref{app:hawkes_psd} is stand-alone modulo the proofs of Propositions \ref{eig_conv} and \ref{prop:hawkes_concentration}.

While Proposition \ref{eig_conv} is a well-known result, Proposition \ref{prop:hawkes_concentration} requires a careful proof, which is presented in Appendix \ref{app:c} and relies on an existing result on the concentration of dependent random variables (Proposition \ref{conc_dep}).

\subsection{Preliminaries}
We state some useful properties of the canonical self-exciting process in the form of the following proposition:
\begin{prop} \label{prop:hawkes_properties}[Properties of the Canonical Self-Exciting Process] The canonical self-exciting process is stationary and we have
\begin{align}
\notag & \pi_\star = \frac{\mu}{1-\mathbf{1}'\boldsymbol{\theta}} >0, \quad  \mu > 0 \Rightarrow  \mathbf{1}' \boldsymbol{\theta}  <1 ,\quad \mu + \mathbf{1}' \boldsymbol{\theta}^+ < 1,\\
\notag & S(\omega) = \frac{1}{2\pi} \left( \pi_\star^2 \delta(\omega) + \frac{\pi_\star-\pi_\star^2}{\left(1-\mathbf{1}'\boldsymbol{\theta}\right)^2 \left|1 - \Theta(\omega)\right|^2} \right),\\
\notag & S(\omega) \ge  \frac{\pi_\star (1 - \pi_\star)}{2 \pi (1+2\pi_\max)^4} =: \kappa_l ,
\end{align}
where $\pi_\star$ denotes the stationary probability of spiking, $S(\omega)$ denotes the power spectral density of the process, and  $\boldsymbol{\theta}^\pm = \max\{\pm \boldsymbol{\theta},\mathbf{0}\}$.
\end{prop}
\begin{proof}
The proof is given in Appendix \ref{app:hawkes_psd}.
\end{proof}

% and calculate its stationary probability $\pi_\star$ to be:
%\begin{equation}
%\label{p_star}
%\pi_\star = \frac{\mu}{1-\mathbf{1}'\boldsymbol{\theta}}.
%\end{equation}

The {stationarity gap of $1-\mathbf{1}'\boldsymbol{\theta}$  plays an important role in controlling the convergence rate of the process to its stationary distribution.}
Throughout the proof, we will also use the notation
\begin{align*}
\mathbb{S}_p(t) :=\left\{\boldsymbol{\nu}\mid\|\boldsymbol{\nu}\|_p=t\right\}.
\end{align*}
to denote the $p$-norm ball of radius $t$.  For simplicity of notation, we also define the $n$-sample empirical expectation as follows:
\[
\hat{\mathbb{E}}_n \{ f(x_{\cdot}) \} := \frac{1}{n} \sum_{i=1}^n f(x_i)
\]
for any measurable function $f(x_{\cdot})$. Note that the subscript $x_{\cdot}$ refers to an index in the set $\{ 1,2, \cdots, n\}$.

\subsection{Establishing the Restricted Strong Convexity}

The proof of Theorems \ref{negahban_lambda} and \ref{thm_OMP} relies on establishing the Restricted Strong Convexity (RSC) for the negative log-likelihood given by (\ref{L_def}).  Recall that if the log-likelihood is twice differentiable with respect to $\boldsymbol{\theta}$, the {RSC} property implies the existence of a lower quadratic bound on the negative log-likelihood:

\begin{equation}
\label{RSC}
\mathfrak{D_L}({\boldsymbol{\psi}},\boldsymbol{\theta}):= \mathfrak{L}(\boldsymbol{\theta}+{\boldsymbol{\psi}})-  \mathfrak{L}(\boldsymbol{\theta})-{\boldsymbol{\psi}}'\nabla\mathfrak{L}(\boldsymbol{\theta})\geq \kappa\|{\boldsymbol{\psi}}\|_2^2,
\end{equation}
for a positive constant $\kappa > 0$ and all ${\boldsymbol{\psi}}\in\mathbb{R}^p$ satisfying:
\begin{equation}
\label{eq:cone}
\|\boldsymbol{\psi}_{S^{c}} \|_1 \leq 3 \|{\boldsymbol{\psi}}_S\|_1 + 4 \| \boldsymbol{\theta}_{S^{c}}\|_1.
\end{equation}

for any index set $S \subset \{1,2,\cdots,p\}$ of cardinality $s$. The latter condition is known as the cone constraint.

The following key lemma establishes the Restricted Strong Convexity condition for the {canonical self-exciting process}:

\begin{lemma}[Restricted strong convexity of the {canonical self-exciting process}]\label{lemma1}
Let ${\mathbf{x}_{-p+1}^n}$ denote a sequence of samples from the {canonical self-exciting process} with parameters $\{\mu,\boldsymbol{\theta}\}$ satisfying the conditions given by (\ref{eq:star}). Then, for $n \ge d_1 s^{2/3} p^{2/3} \log p$, the negative log-likelihood function $\mathfrak{L}(\boldsymbol{\theta})$ satisfies the RSC property with a positive constant $\kappa > 0$ with probability at least $1-2 \exp\left(-\frac{c\kappa^2n^3}{s^2p^2}\right)$, for some constant $c$, and both $\kappa$ and $c$ are only functions of $d_1$, $c_1$, and $\pi_\max$.
\end{lemma}
\begin{proof}
The proof is given in Appendix \ref{app:hawkes_psd}.
\end{proof}

Lemma \ref{lemma1} can be viewed as the key result in the proofs of Theorems \ref{negahban_lambda} and \ref{thm_OMP} which follow next.

\subsection{{Proof of Theorem \ref{negahban_lambda}}}

Given the results of Lemma \ref{lemma1} and Proposition \ref{negahban_thm}, it only remains to establish an upper bound on $\gamma_n$. To this end, we establish a suitable upper bound on $\| \nabla \mathfrak{L}(\boldsymbol{\theta})\|_\infty$ which holds with high probability and provides the appropriate scaling of $\gamma_n$. From Eq. (\ref{L_def}), we have
\begin{equation}
\label{llgrad}
\nabla \mathfrak{L}(\boldsymbol{\theta}) = \frac{1}{n}\sum_{i=1}^n \left[ x_i-(\mu+\boldsymbol{\theta}'  {\mathbf{x}_{i-p}^{i-1}})\right] \frac{{\mathbf{x}_{i-p}^{i-1}}}{{\lambda_i (1-\lambda_i)}}.
\end{equation}
We proceed in two steps:
\medskip

\noindent {\bf Step 1.} We first show that

\begin{equation}
\label{expec=0}
\mathbb{E} \left[ \nabla \mathfrak{L}(\boldsymbol{\theta}) \right]=\mathbf{0}.
\end{equation}
To see this, we use the law of iterated expectations on the $i$th term as follows:
\begin{align}
\label{martingale}
\notag \mathbb{E} \left[ [x_i-(\mu+\boldsymbol{\theta}'  {\mathbf{x}_{i-p}^{i-1}})]\frac{{\mathbf{x}_{i-p}^{i-1}}}{{\lambda_i (1-\lambda_i)}}\right] &=\mathbb{E} \left[ \mathbb{E} \left[ x_i-(\mu+\boldsymbol{\theta}'  {\mathbf{x}_{i-p}^{i-1}})\frac{{\mathbf{x}_{i-p}^{i-1}}}{{\lambda_i (1-\lambda_i)}}\bigg | {\mathbf{x}_{i-p}^{i-1}} \right]\right]\\
&= \mathbb{E} \left[ \underbrace{\mathbb{E} \left[ x_i-(\mu+\boldsymbol{\theta}'  {\mathbf{x}_{i-p}^{i-1}})\Big | {\mathbf{x}_{i-p}^{i-1}} \right]}_{0} \frac{{\mathbf{x}_{i-p}^{i-1}}}{{\lambda_i (1-\lambda_i)}} \right] = 0.
\end{align}
Summing over $i$, establishes (\ref{expec=0}).

\noindent {\bf Step 2.} We next show that the summation given by (\ref{llgrad}) is concentrated around its mean. The iterated expectation argument used in establishing (\ref{martingale}) implies that {the sequence
\begin{equation*}
\left\{ \left[ x_i-(\mu+\boldsymbol{\theta}'  {\mathbf{x}_{i-p}^{i-1}})\right] \frac{{\mathbf{x}_{i-p}^{i-1}}}{{\lambda_i (1-\lambda_i)}} \right\}_{i=1}^n
\end{equation*}
}
%\begin{equation*}
%\left(\nabla \mathfrak{L}(\boldsymbol{\theta})\right)_i = \left[ x_i-(\mu+\boldsymbol{\theta}'  {\mathbf{x}_{i-p}^{i-1}})\right] \frac{{\mathbf{x}_{i-p}^{i-1}}}{{\lambda_i (1-\lambda_i)}}
%\end{equation*}
is a martingale with respect to the filtration given by
\begin{equation*}
\label{filtration}
\mathcal{F}_{i}=\sigma \left( {\mathbf{x}_{-p+1}^i} \right),
\end{equation*}
where $\sigma(\cdot)$ denote the sigma-field generated by the random variables in its argument.

By Lemma \ref{hoeff_dep} we have
\begin{equation}
\label{bound}
\mathbb{P}\left( \left | \left ( \nabla \mathfrak{L}(\boldsymbol{\theta})\right)_i\right|  \ge t \right) \leq  \exp(-c nt^2).
\end{equation}
By the union bound, we get:
\begin{equation}
\label{ubound}
\mathbb{P}\Big(\left\| \nabla \mathfrak{L}(\boldsymbol{\theta})\right\|_\infty \ge t \Big) \leq  \exp(-c t^2n+\log p).
\end{equation}
Choosing $t = \sqrt{\frac{{1+\alpha_1}}{c}}\sqrt{\frac{\log p}{n}}$ for some $\alpha_1>0$ yields
\begin{align}
\label{grad_bound}
\notag \mathbb{P}\left(\left\| \nabla \mathfrak{L}(\boldsymbol{\theta})\right\|_\infty \ge \sqrt{\frac{{1+\alpha_1}}{c}}\sqrt{\frac{\log p}{n}} \right) &\leq 2 \exp(-\alpha_1\log p)\\
&\leq \frac{2}{n^{\alpha_1}}.
\end{align}
Hence, a choice of $\gamma_n = d_2 \sqrt{\frac{\log p}{n}}$ with $d_2 := \sqrt{\frac{1 + \alpha_1}{c}}$ satisfies (\ref{reg}) with probability at least $1 - \frac{2}{n^{\alpha_1}}$.
Combined with the result of Lemma \ref{lemma1} for $n > d_1 s^{2/3} p^{2/3} \log p$, we have that the RSC is satisfied with a constant $\kappa$ with a probability at least $1 - \frac{1}{p^{\alpha_2}} \ge 1 - \frac{1}{n^{\alpha_2}}$ for some constant $\alpha_2$. The latter results in conjunction with Proposition \ref{negahban_thm} establishes the claim of Theorem 1. \QEDB

\noindent \textit{\textbf{Remark.}} The choice of $\pi_\min$ does not affect the proof of Theorem \ref{negahban_lambda}, and can be chosen as $0$ in defining the set $\boldsymbol{\Theta}$, thereby relaxing the first inequality in ($\star$). However, as we will show below, the assumption of $\pi_{\min} > 0$ is required for the proof of Theorem \ref{thm_OMP}.

\subsection{{Proof of Theorem \ref{thm_OMP}}}\label{prf_thm2}

The proof is mainly based on the following proposition, adopted from Theorem 2.1 of \cite{zhang_omp}, stating that the greedy procedure is successful in obtaining a reasonable $s^\star$-sparse approximation, if the cost function satisfies the RSC:
\begin{prop}\label{prop_omp}
Suppose that $\mathfrak{L}(\boldsymbol{\theta})$ satisfies RSC with a constant $\kappa > 0$. Let $s^\star$ be a constant such that
\begin{equation}
\label{sstar}
s^\star \geq \frac{4s}{\pi_\min^2\kappa}\log \frac{20s}{\pi_\min^2 \kappa}= \mathcal{O}(s\log s),
\end{equation}
Then, we have
\begin{equation*}
\left \|\widehat{\boldsymbol{\theta}}^{(s^\star)}_{{\sf POMP}}-\boldsymbol{\theta}_s \right \|_2 \leq \frac{\sqrt{6} \epsilon_{s^\star}}{\kappa},
\end{equation*}
where $\epsilon_{s^\star}$ satisfies
\begin{equation}
\label{eps_bound}
\epsilon_{s^\star} \leq \sqrt{s^\star+s} \|\nabla\mathfrak{L}(\boldsymbol{\theta}_s)\|_\infty .
\end{equation}
\end{prop}
\begin{proof}
The proof is a specialization of the proof of Theorem 2.1 in \cite{zhang_omp} to our setting.
\end{proof}

Recall that Lemma \ref{lemma1} establishes the RSC for the negative log-likelihood function. In order to complete the proof of Theorem \ref{thm_OMP}, it only remains to upper bound $\|\nabla\mathfrak{L}(\boldsymbol{\theta}_s)\|_\infty$. {Let $\lambda_{i,s} := \mu+\boldsymbol{\theta}'_s  \mathbf{x}_{i-p}^{i-1}$.} We have
\begin{align*}
\mathbb{E} \left[\nabla\mathfrak{L}(\boldsymbol{\theta}_s)\right] &= \mathbb{E} \left [ \frac{1}{n}\sum_{i=1}^n \left[ x_i-(\mu+\boldsymbol{\theta}'_s  {\mathbf{x}_{i-p}^{i-1}})\right] \frac{{\mathbf{x}_{i-p}^{i-1}}}{{\lambda_{i,s}(1-\lambda_{i,s})}} \right ]\\
&=\frac{1}{n}\sum_{i=1}^n  \displaystyle \mathbb{E} \left [ \mathbb{E} \left [ (\boldsymbol{\theta}-\boldsymbol{\theta}_s)'  {\mathbf{x}_{i-p}^{i-1}} \Big | {\mathbf{x}_{i-p}^{i-1}} \right ] \frac{{\mathbf{x}_{i-p}^{i-1}}}{{\lambda_{i,s}(1-\lambda_{i,s})}} \right ]\\
& \le  c_2 \sigma_s(\boldsymbol{\theta}) \mathbf{1}.
\end{align*}
where we have used the fact that $0 \le x_i \le 1$ for all $i$, and {$c_2 := \frac{1}{\pi_\min (1-\pi_\max)}$}. 
Invoking the result of Proposition \ref{hoeff_dep} together with the union bound yields:  
\begin{align*}
\mathbb{P}\left(\|\nabla\mathfrak{L}(\boldsymbol{\theta}_s)\|_\infty \geq  c_1 \sqrt{\frac{\log p}{n}} + c_2 \sigma_s(\boldsymbol{\theta}) \right) \leq \frac{2}{n^{\beta_1}}.
\end{align*}
for some constants $c_1$ and $\beta_1$. Hence, we get the following concentration result for $\epsilon_{s^\star}$:
\begin{equation}
\label{eps_prob}
\mathbb{P}\left(\epsilon_{s^\star} \geq \sqrt{s^\star+s} \left(  c_1 \sqrt{\frac{\log p}{n}} + c_2 \sigma_s(\boldsymbol{\theta})\right)\right) \leq\frac{2}{n^{\beta_1}}.
\end{equation}
Noting that by (\ref{sstar}) we have $s^\star+s = \mathcal{O}(s\log s) \leq c_0 s \log s$, for some constant $c_0$, and invoking the result of Lemma \ref{lemma1}, we get:
\begin{align*}
\left \|\widehat{\boldsymbol{\theta}}^{(s^\star)}_{{\sf POMP}}-\boldsymbol{\theta}_S \right\|_2 &\leq d'_2 \sqrt{\frac{s \log s \log p}{n}} + d'_3 s\log s \sigma_s(\boldsymbol{\theta}) \leq  d'_2 \sqrt{\frac{s \log s \log p}{n}} + d'_3 \frac{\log s}{s^{\frac{1}{\xi}-2}},
\end{align*}
where $d'_2 = \sqrt{c_0} c_1$ and $d'_3 = \sqrt{c_0} c_2$.
with probability $\left(1-\exp\left(-\frac{c\kappa^2n^3}{s^2 (\log s)^2 p^2 }\right)\right)\left(1-\frac{2}{n^{\beta_1}}\right)$. Choosing $n > d'_1 s^{2/3} (\log s)^{2/3} p^{2/3}\log p$ establishes the claimed success probability of Theorem \ref{thm_OMP}. Finally, we have: 
\begin{align*}
\left \|\widehat{\boldsymbol{\theta}}^{(s^\star)}_{{\sf POMP}}-\boldsymbol{\theta}\right\|_2 &= \left \|\widehat{\boldsymbol{\theta}}^{(s^\star)}_{{\sf POMP}}-\boldsymbol{\theta}_s +\boldsymbol{\theta}_s -\boldsymbol{\theta} \right \|_2  \leq \left \|\widehat{\boldsymbol{\theta}}^{(s^\star)}_{{\sf POMP}}-\boldsymbol{\theta}_s \right \|_2 + \|\boldsymbol{\theta}_s-\boldsymbol{\theta}\|_2.
\end{align*}
Using $\|\boldsymbol{\theta}_s-\boldsymbol{\theta}\|_2 \leq \sigma_s(\boldsymbol{\theta}) = \mathcal{O} \left ( s^{1- \frac{1}{\xi}}\right)$ completes the proof. 

\QEDB

\subsection{{Proofs of Corollaries \ref{cor:1} and \ref{cor:2}}}
\begin{proof}[{Proof of Corollary \ref{cor:1}}]
The claim is a direct consequence of the boundedness of covariates and can be treated by replacing $\boldsymbol{\theta}$ with the augmented parameter vector $[\mu,\boldsymbol{\theta}']'$ and augmenting the covariate vectors with an initial component of 1. The reader can easily verify that all the proof steps can be repeated in the same fashion.
\end{proof}

\begin{proof}[{Proof of Corollary \ref{cor:2}}]
The claim is a direct consequence of the boundedness of covariates which results in $\phi(\cdot)$ being Lipschitz and hence the stationarity of the underlying process. Moreover, for twice-differentiable $\phi(\cdot)$, the proof of Lemma \ref{lemma1} in Appendix \ref{appprf} can be generalized in a straightforward fashion. The reader can easily verify that all the remaining portions of the proofs of the main theorems can be repeated for such $\phi(\cdot)$ in a similar fashion to that of the {canonical self-exciting process}.
\end{proof}

\section{Proofs of Proposition \ref{prop:hawkes_properties} and Lemma \ref{lemma1}}\label{app:hawkes_psd}

\subsection{Proof of Proposition \ref{prop:hawkes_properties}}

The {canonical self-exciting process} can be viewed as a Markov chain with states $X_i= {\mathbf{x}_{i-p}^{i-1}}$. Since each $x_i$ has two possible values, there are $2^p$ possible states. This Markov chain is irreducible since transition from any state to any other state is possible in at most $p$ steps. Also, transition from an all-zero state to itself is possible. Hence the chain is aperiodic as well. This implies that there exists a stationary distribution for the Markov chain. We also know that if $\{X_i\}_{i=1}^\infty$ is a stationary Markov Chain, then for any functional $f(.)$, $\{f(X_i)\}_{i=1}^\infty$ is a strictly stationary stochastic process (SSS). Therefore the {canonical self-exciting process} and the spiking probability sequence $\lambda_1^n$ are both SSS. In particular, we have
\begin{align*}
\pi_\star:= \mathbb{E}[x_i]= \mathbb{E}\left[\mathbb{E}\left[x_i|\lambda_i\right]\right] = \mathbb{E}[\lambda_i]=\mu + \pi_\star\mathbf{1}'\boldsymbol{\theta}.
\end{align*}
Hence, the stationary probability $\pi_\star$ satisfies:
\begin{equation*}
\pi_\star = \frac{\mu}{1-\mathbf{1}'\boldsymbol{\theta}}.
\end{equation*}

In order to prove the first two inequalities, we make the necessary assumption that the baseline rate $\mu$ is positive, due to the non-degeneracy assumption. In order to highlight the necessity of this condition, consider a sample path which contains $p$ successive zeros starting from index $i+1$ to $i+p$, corresponding to an all-zero covariate vector ${\mathbf{x}_{i+1}^{i+p}}$ (note that this sample path will almost surely occur). We then have $\lambda_{i+p+1} = \mu +  \boldsymbol{\theta}' {\mathbf{x}_{i+1}^{i+p}} = \mu$. Therefore, if $\mu$ is not positive, the process becomes degenerate.

The third inequality follows from the fact that for a covariate vector ${\mathbf{x}_{i+1}^{i+p}}$ with a support matching that of $\boldsymbol{\theta}^+$ we have $\lambda_{i+p+1} = \mu +  \boldsymbol{\theta}' {\mathbf{x}_{i+1}^{i+p}} = \mu + \mathbf{1}'\boldsymbol{\theta}^{+}$, which should be a valid probability. Moreover, the inequality is strict since the stationary probability $\pi_\star = \frac{\mu}{1-\mathbf{1}'\boldsymbol{\theta}}$ must be well-defined.

We will next calculate the power spectral density of the process. {Let $\boldsymbol{r}_{-\infty}^{\infty}$ and $\boldsymbol{c}_{-\infty}^\infty$ denote the autocorrelation and autocovariance values of the process, respectively. By the stationarity of the process we have:
\begin{align*}
r_k & = \mathbb{E}\left[x_{\cdot+k} x_\cdot \right]= \mathbb{E}\left[x_k x_{0}\right] = \mathbb{E} \left[ \mathbb{E}\left[x_kx_0 | {\mathbf{x}_{-\infty}^{k-1}} \right]\right] = \mathbb{E}\left[\mu x_0 + \boldsymbol{\theta}' {\mathbf{x}_{k-p}^{k-1}}x_0  \right] = \mu \pi_\star +  \boldsymbol{\theta}'\boldsymbol{r}_{k-p}^{k-1}.
\end{align*}
for $k > 0$. Similarly, by subtracting the means we have the following identity for the autocovariance:
\begin{equation}
\label{eq:yw_hawkes}
c_k = \boldsymbol{\theta}'\boldsymbol{c}_{k-p}^{k-1}.
\end{equation}
A straightforward calculation gives $c_0 = \pi_\star - \pi_\star^2$. Eq. (\ref{eq:yw_hawkes}) resembles the Yule-Walker equations for an AR process of order $p$ with parameter $\boldsymbol{\theta}$ and the innovations variance given by $\sigma^2 = \frac{\pi_\star-\pi_\star^2}{\left(1-\mathbf{1}'\boldsymbol{\theta}\right)^2}$. Thus, the power spectral density of the \textcolor{black}{canonical self-exciting process} can be expressed as:}
\begin{equation}
S(\omega) = \frac{1}{2\pi} \left( \pi_\star^2 \delta(\omega) + \frac{\pi_\star-\pi_\star^2}{\left(1-\mathbf{1}'\boldsymbol{\theta}\right)^2 \left|1 - \Theta(\omega)\right|^2} \right).
\end{equation}
We have $1-\mathbf{1}'\boldsymbol{\theta} \le 1 + \| \boldsymbol{\theta} \|_1$. Moreover,
\begin{align*}
|1 - \Theta(\omega)| &= \left| 1- \sum_k \theta_k e^{-j \omega k} \right| \\
&\leq 1 + \|\boldsymbol{\theta}\|_1 =  1 + \|\boldsymbol{\theta}^+\|_1 +  \|\boldsymbol{\theta}^-\|_1\\
& \le 1 + 2 (\pi_\max - \mu)  \le 1 + 2 \pi_\max,
\end{align*}
which implies the lower bound on $S(\omega)$. \QEDB

\subsection{Proof of Lemma \ref{lemma1}}

The proof is inspired by the elegant treatment of Negahban et al. \cite{Negahban}. The major difficulty in the proof lies in the high inter-dependence of the covariates and observations. 

{Noticing that the negative log-likelihood (\ref{L_def}) is twice differentiable}, a second order Taylor expansion of the negative log-likelihood (\ref{L_def}) around $\boldsymbol{\theta}$ yields:

\begin{align*}
\mathfrak{D_L}({\boldsymbol{\psi}},\boldsymbol{\theta})
&= \mathfrak{L}(\boldsymbol{\theta}+{\boldsymbol{\psi}})-\mathfrak{L}(\boldsymbol{\theta})-{\boldsymbol{\psi}}'\nabla \mathfrak{L}(\boldsymbol{\theta})\\
       &= \frac{1}{n}\sum_{i=1}^n x_i \frac{\left({\boldsymbol{\psi}}' {\mathbf{x}_{i-p}^{i-1}}\right)^2}{\left(\mu+\boldsymbol{\theta}'{\mathbf{x}_{i-p}^{i-1}} + \nu({\boldsymbol{\psi}}'{\mathbf{x}_{i-p}^{i-1}})\right)^2} + \frac{1}{n}\sum_{i=1}^n (1-x_i) \frac{\left({\boldsymbol{\psi}}' {\mathbf{x}_{i-p}^{i-1}}\right)^2}{\left(1-\mu-\boldsymbol{\theta}'{\mathbf{x}_{i-p}^{i-1}} - \nu({\boldsymbol{\psi}}'{\mathbf{x}_{i-p}^{i-1}})\right)^2}\\
       & \geq  \frac{1}{n}\sum_{i=1}^n \left({\boldsymbol{\psi}}' {\mathbf{x}_{i-p}^{i-1}}\right)^2,
\end{align*}

for some $\nu \in [0,1]$. The inequality follows from the fact that both $\boldsymbol{\theta}$ and {$\boldsymbol{\theta}+\nu {\boldsymbol{\psi}}$} satisfy (\ref{eq:star}), and hence:

\begin{align}
\nonumber &\mu+\boldsymbol{\theta}'{\mathbf{x}_{i-p}^{i-1}}+ \nu{\boldsymbol{\psi}}'{\mathbf{x}_{i-p}^{i-1}}\leq \pi_\max <1,\\
\nonumber &1 - \mu - \boldsymbol{\theta}'{\mathbf{x}_{i-p}^{i-1}} -  \nu{\boldsymbol{\psi}}'{\mathbf{x}_{i-p}^{i-1}}\leq 1- \pi_{\min} <1.
\end{align}

The result of the Lemma \ref{lemma1} is equivalent to proving that

\begin{equation}
\label{rsceq}
\hat{\mathbb{E}}_n\left[\left({\boldsymbol{\psi}}'x_{\cdot-p}^{\cdot-1}\right)^2\right] \geq \kappa \|{\boldsymbol{\psi}}\|_2^2
\end{equation}

holds with probability greater than $1-2 \exp\left(-\frac{c\kappa^2n^3}{s^2p^2}\right)$. Since both sides of (\ref{rsceq}) are quadratic in ${\boldsymbol{\psi}}$, the statement is equivalent to proving

\begin{equation*}
\hat{\mathbb{E}}_n\left[({\boldsymbol{\psi}}' {\mathbf{x}_{\cdot-p}^{\cdot-1}})^2\right] \geq \kappa,
\end{equation*}

for all $ \|{\boldsymbol{\psi}}\|_2 \in \mathbb{S}_2(1)$. We establish this in two steps:
\medskip

\noindent {\bf Step 1.} First, we show that the statement holds for the true expectation:

\begin{equation}
\label{kappa}
\mathbb{E}\Big[ \left({\boldsymbol{\psi}}' {\mathbf{x}_{\cdot-p}^{\cdot-1}}\right)^2\Big] \geq \kappa_l >0
\end{equation}

for some $\kappa_l$ which will be specified below, for all $ \|{\boldsymbol{\psi}}\|_2 \in \mathbb{S}_2(1)$. To \textcolor{black}{establish} the inequality (\ref{kappa}), we use the following result:
\begin{prop}
\label{eig_conv}
Let ${\mathbf{R}} \in \mathbb{R}^{p \times p}$ be the $p \times p$ covariance matrix of a stationary process with power spectral density $S(\omega)$, and denote its maximum and minimum eigenvalues by $\lambda_{\max}(p)$ and $\lambda_{\min}(p)$ respectively then $\lambda_{\max}(p)$ is increasing in $p$, $\lambda_{\min}(p)$ is decreasing in $p$ and we have
\begin{equation}
\lambda_{\sf min}(p) \downarrow \inf_{\omega}S(\omega), \ \ \text{and} \ \ \lambda_{\sf max}(p) \uparrow \sup_{\omega}S(\omega).
\end{equation}
\end{prop}
\begin{proof}
This is a well-known result in stochastic processes. See \cite{grenander2001toeplitz} for a proof and detailed discussions.
\end{proof}

\noindent Using Proposition \ref{eig_conv}, we can lower-bound $\mathbb{E}\left[\left({\boldsymbol{\psi}}' {\mathbf{x}_{\cdot-p}^{\cdot-1}}\right)^2\right]$ by:
\begin{align*}
&\mathbb{E}\left[ \left({\boldsymbol{\psi}}'{\mathbf{x}_{\cdot-p}^{\cdot-1}}\right)^2\right]=  {\boldsymbol{\psi}}' \mathbf{R} {\boldsymbol{\psi}} \geq \lambda_\min(p) \geq \inf_{\omega}S(\omega).
\end{align*} 
Next, using Proposition \ref{prop:hawkes_properties} the bound of Eq. (\ref{kappa}) follows for
\[
\kappa_l :=  \frac{\pi_\star (1 - \pi_\star)}{2 \pi (1+2\pi_\max)^4}.
\]

\medskip

\noindent {\bf Step 2.} We now show that the empirical and the true expectations of $ \left({\boldsymbol{\psi}}'{\mathbf{x}_{\cdot-p}^{\cdot-1}}\right)^2$ are close enough to each other. Let 
\begin{equation*}
\mathfrak{D}_{{\boldsymbol{\psi}},n} := \hat{\mathbb{E}}_n\left[ \left({\boldsymbol{\psi}}'{\mathbf{x}_{\cdot-p}^{\cdot-1}}\right)^2\right]-\mathbb{E}\left[ \left({\boldsymbol{\psi}}'{\mathbf{x}_{\cdot-p}^{\cdot-1}}\right)^2\right].
\end{equation*}

and
\begin{equation*}
\mathfrak{D}_{n} := \sup \limits_{{\boldsymbol{\psi}} \in \mathbb{S}_2(1)} \left| \mathfrak{D}_{{\boldsymbol{\psi}},n} \right |.
\end{equation*}
The final step in proving Lemma \ref{lemma1} is given by the following proposition:
\begin{prop}\label{prop:hawkes_concentration} We have
\begin{equation}
\label{zbounded}
\mathbb{P}\left[\mathfrak{D}_{n} \geq \frac{\kappa_l}{4}\right] \leq 2 \exp\left(-\frac{c\kappa_l^2 n^3}{s^2p^2}\right),
\end{equation}
for some constant $c$.
\end{prop}

\begin{proof}
{The proof is given in Appendix \ref{app:c}.}
\end{proof}

{Finally, the statement of Lemma \ref{lemma1} follows from Proposition \ref{prop:hawkes_concentration} by taking $\kappa = \kappa_l / 4$. \QEDB}

\section{Proof of Proposition \ref{prop:hawkes_concentration}}\label{app:c}

{In order to establish the concentration inequality of Eq. (\ref{zbounded}), we need to invoke a result from concentration of dependent random variables. We proceed in two steps:}
\medskip

\noindent {{\bf Step 1.}} {We first establish a geometric property of $\mathfrak{D}_n$, namely its $\mathcal{O}(\frac{sp}{n})$-Lipschitz property with respect to the normalized Hamming metric. Recall that the normalized Hamming metric between two sequences ${\mathbf{x}_1^n}$ and ${\mathbf{y}_1^n}$ is defined as $d({\mathbf{x}_1^n,\mathbf{y}_1^n)} = \frac{1}{n}\sum_{i=1}^n\mathbf{1}(x_i \neq y_i)$.}

First, by evaluating the first order optimality conditions of the solution $\widehat{\boldsymbol{\theta}}_{{\sf sp}}$, it can be shown that the error vector ${\boldsymbol{\psi}} = \widehat{\boldsymbol{\theta}}_{\sf sp}-\boldsymbol{\theta}$ satisfies the inequality:
\begin{equation*}
{\|\boldsymbol{\psi}_{S^{c}} \|_1 \leq 3 \|{\boldsymbol{\psi}}_S\|_1 + 4 \| \boldsymbol{\theta}_{S^{c}}\|_1,}
\end{equation*}
{with $S$ denoting the support of the best $s$-term approximation to $\boldsymbol{\theta}$} (see for example \cite{Negahban}). By the assumption of $\sigma_S(\boldsymbol{\theta}) = \mathcal{O}(\sqrt{s})$, we can choose a constant $c_0$ such that $\sigma_S(\boldsymbol{\theta}) \le c_0 \sqrt{s}$. Hence,
\begin{align}
\label{cone}
\|{\boldsymbol{\psi}}\|_1 & \leq 4 \|{\boldsymbol{\psi}}_S\|_1  + \sigma_s(\boldsymbol{\theta})  \leq (4 + c_0)\sqrt{s}\|{\boldsymbol{\psi}}_S\|_2 \leq (4 + c_0) \sqrt{s}
\end{align}
where we have used the fact that $\| {\boldsymbol{\psi}}_S \|_1 \le \sqrt{s} \| {\boldsymbol{\psi}}_S \|_2 \le \sqrt{s}$ for all ${\boldsymbol{\psi}} \in \mathbb{S}_2(1)$. Therefore for all $i \in \{ 1,2,\cdots, n\}$, we have:

\begin{equation}
\label{fbound}
0 \leq \left({\boldsymbol{\psi}}'{\mathbf{x}_{i-p}^{i-1}}\right)^2 \leq \left\|{\boldsymbol{\psi}}\right\|_1^2 \leq (4 + c_0)^2 s.
\end{equation}

We first prove the claim for $\mathfrak{D}_{{\boldsymbol{\psi}},n}$. To establish the latter, we need to prove

\[
\frac{1}{n}\left|\sum_{i=1}^n \left({\boldsymbol{\psi}}'{\mathbf{x}_{i-p}^{i-1}}\right)^2-  \left({\boldsymbol{\psi}}'{\mathbf{y}_{i-p}^{i-1}}\right)^2\right| \leq C d({\mathbf{x}_{-p+1}^{n}},{\mathbf{y}_{-p+1}^{n}}),
\]
for some $C = \mathcal{O}(\frac{sp}{n})$, or equivalently
\begin{equation}
\label{lip1}
\left|\sum_{i=1}^n  \left({\boldsymbol{\psi}}'{\mathbf{x}_{i-p}^{i-1}}\right)^2- \left({\boldsymbol{\psi}}'{\mathbf{y}_{i-p}^{i-1}}\right)^2\right| \leq C' \sum_{i=-p+1}^n\mathbf{1}(x_i \neq y_i),
\end{equation}

for some $C' = \mathcal{O}(s)$.
Let us start by setting the values of ${\mathbf{x}_{-p+1}^{n}}$ equal to those of ${\mathbf{y}_{-p+1}^{n}}$ and iteratively change $x_j$ to $1-x_j$ for all indices $j$ where $x_j \neq y_j$ to obtain the configuration given by ${\mathbf{x}_{-p+1}^{n}}$.  For each such change (say $x_j$ to $1-x_j$), the left hand side changes by at most

\begin{align*}
&\left| \sum_{i=1}^n \left({\boldsymbol{\psi}}'{\mathbf{x}_{i-p}^{i-1}}\right)^2_{|x_j =1} -  \left({\boldsymbol{\psi}}'{\mathbf{x}_{i-p}^{i-1}}\right)^2_{|x_j =0} \right| \leq   \left({\boldsymbol{\psi}}'{\mathbf{x}_{j-p}^{j-1}}\right)^2 + 2\sum_{i\neq j}  |\psi_{i-j}| \|{\boldsymbol{\psi}}\|_1 \leq 3 \|{\boldsymbol{\psi}}\|_1^2 \leq 3(4+c_0)^2 s,
\end{align*}

where we have used the inequality given by Eq. (\ref{fbound}). Hence, the $C$ can be taken as $3 (4 + c_0)^2 s p / n$ and the claim of the proposition for $\mathfrak{D}_{{\boldsymbol{\psi}},n}$ follows.  A very similar argument can be used to extend the claim to $\mathfrak{D}_n$. Let ${\boldsymbol{\psi}}^\star := {\boldsymbol{\psi}}^\star({\mathbf{x}_{-p+1}^{n}})$ be the ${\boldsymbol{\psi}}$ for which the supremum in the definition of $\mathfrak{D}_n$ is achieved (such a choice of ${\boldsymbol{\psi}}$ exists by the Weierstrass extreme value theorem). Since ${\boldsymbol{\psi}}^{\star}$ also satisfies (\ref{cone}), a similar argument shows that $\mathfrak{D}_n$ is $\mathcal{O}(\frac{sp}{n})$-Lipschitz (with possibly different constants). 
\medskip

\noindent {{\bf Step 2.} Next, we establish the concentration of $\mathfrak{D}_n$ around zero.} Let $H = [{\mathbf{x}_{i-p}^{i-2}},1]$ and $\widehat{H} = [{\mathbf{x}_{i-p}^{i-2}},0]$ be two vectors (history components) of length $p$ which only differ in their last component, and let the mixing coefficient $\bar{\eta}_{ij}$ for $j \ge i$ be defined as:
\begin{equation}
\label{eta}
\bar{\eta}_{ij} = \|p({\mathbf{x}_j^n}|H)-p({\mathbf{x}_j^n}|\widehat{H})\|_{TV},
\end{equation}
with $\| \cdot \|_{TV}$ denoting the total variation difference of the probability measures {induced} on $\{0,1\}^{n-j+1}$. Also, let
\[
\eta_{ij} = \sup_{H,\widehat{H}} \bar{\eta}_{ij},
\]
and
\[
Q_{n,i} := 1+ \eta_{i,i+1}+ \cdots + \eta_{i,n}.
\]
We now invoke Theorem 1.1 of \cite{kontorovich2008concentration} in the form of the following proposition:
\begin{prop}
\label{conc_dep}
If $\mathfrak{D}_n$ is $C$-Lipschitz and $q := \max_{1 \leq i \leq n} Q_{n,i}$, then
\begin{equation*}
\mathbb{P}\left [| \mathfrak{D}_n - \mathbb{E}[\mathfrak{D}_n] | \geq t \right] \leq 2 \exp \left(\ \frac{-2nt^2}{qC^2} \right).
\end{equation*}
\end{prop}
\begin{proof}
The proof is identical to the beautiful treatment of \cite{kontorovich2008concentration} when specializing the underlying function of the variables ${\mathbf{x}_{-p+1}^i}$ to be $\mathfrak{D}_n$.
\end{proof}

As we showed in Step 1, $C = C' s p / n$, for some constant $C'$. Now, we have
\begin{align*}
\eta_{ij} \leq 2^{n-j+1} |\pi_{\max}^{n-j+1} -\pi_{\min}^{n-j+1}| \leq \left(2\pi_{\max}\right)^{n-j+1},
\end{align*}
where we have used the fact that each element of the measures $p({\mathbf{x}_j^n} | H)$ and $p({\mathbf{x}_j^n} | \widehat{H})$ satisfies the assumption (\ref{eq:star}) and that the size of the state space $\{0,1\}^{n-j+1}$ is given by $2^{n-j+1}$. By the assumption (\ref{eq:star}), we have $\eta_{ij} \le \rho^{n-j+1}$ for $\rho := 2 \pi_\max < 1$. Hence, $Q_{n,i} \le \frac{1}{1-\rho}$ for all $i$, and $q \le \frac{1}{1-\rho}$ by definition. Using the result of Proposition \ref{conc_dep}, we get:
\begin{equation}
\label{debound}
\mathbb{P}\left[\mathfrak{D}_{n} \ge \mathbb{E}[\mathfrak{D}_n] + \frac{\kappa_l}{2} \right ] \le 2 \exp \left(\ \frac{-n^3 \kappa_l^2 ( 1 - \rho)}{2 C' s^2 p^2} \right). 
\end{equation}
It only remains to show that the expectation in (\ref{debound}) can be suitably bounded. {Note that by a similar concentration argument for} $\mathfrak{D}_{ {\boldsymbol{\psi}}^\star,n}$, we have:

\begin{align*}
\mathbb{E}[\mathfrak{D}_n]&= \mathbb{E}[|\mathfrak{D}_{ {\boldsymbol{\psi}}^\star,n}|] = \int_0^\infty \left(1-F_{\left|\mathfrak{D}_{{\boldsymbol{\psi}}^\star,n}\right|}(t)\right)dt \leq \int_0^\infty 2 \exp\left(-\frac{2 (1-\rho) n^3t^2}{C' s^2 p^2}\right)dt = 2\sqrt{\frac{C'\pi}{(1-\rho)}}\frac{ps}{n^{3/2}}.
\end{align*}
Thus choosing $n \ge d_1 s^{2/3}p^{2/3} \log p$, for some positive constant $d_1$, $\mathbb{E}[\mathfrak{D}_n]$ drops as $1/\log^{3/2} p$, and will be smaller than $\kappa_l / 4$ for large enough $p$. Hence, combined with (\ref{debound}) and by defining $c := \frac{1-\rho}{2 C'}$ we have:
\begin{equation*}
\mathbb{P}\left[\mathfrak{D}_{n} \ge \frac{\kappa_l}{4} \right ] \le 2 \exp \left(\ \frac{-c n^3 \kappa_l^2}{s^2 p^2} \right),
\end{equation*}
which establishes the claim of Proposition \ref{prop:hawkes_concentration}. \QEDB

%\section{Extensions of the Main Results}\label{appnotes}

\section{Proof of Main Theorem on Compressible State-Spaces} \label{app:tv_main_proof}
\subsection{Proof of Theorem \ref{thm:tv_main}}
\begin{proof}
\label{proof:tv_main}
The main idea behind the proof is establishing appropriate cone and tube constraints \cite{candes2008introduction}. In order to avoid unnecessary complications we assume $s_1 \gg s_2 = \cdots=s_T$ and $n_1 \gg n_2 = \cdots=n_T$. Let $\widehat{\mathbf{x}}_t = \mathbf{x}_t + \mathbf{g}_t$ , $t \in [T]$ be an arbitrary solution to the primal form (\ref{eq:tv_prob_def_primal}). We define $\mathbf{z}_t = \mathbf{x}_t-\theta \mathbf{x}_{t-1}$ and $\mathbf{h}_t =  \mathbf{z}_t - \widehat{\mathbf{z}}_t  $ for $t \in [T]$. For a positive integer $p$, let $[p]:=\{1,2,\cdots,p\}$. For an arbitrary set $V \subset [p]$, $\mathbf{x}_{V}$ denotes the vector $\mathbf{x}$ restricted to the indices in $V$, i.e. all the components outside of $V$ set to zero. We can decompose 
$\left(\mathbf{h}_{t}\right)_{S_t^c} = \left(\mathbf{h}_{t}\right)_{I_1}+\left(\mathbf{h}_{t}\right)_{I_2}+\cdots+\left(\mathbf{h}_{t}\right)_{I_{r_t}},$ where $r_t=\lfloor p/4s_t \rfloor$, and $\left(\mathbf{h}_{t}\right)_{I_1}$ is the $4s_t$-sparse vector corresponding to $4s_t$ largest-magnitude entries remaining in $\left(\mathbf{h}_{t}\right)_{S_t^c}$, $\left(\mathbf{h}_{t}\right)_{I_2}$ is the $4s_t$-sparse vector corresponding to $4s$ largest-magnitude entries remaining in $\left(\mathbf{h}_{t}\right)_{S_t^c}-\left(\mathbf{h}_{t}\right)_{I_1}$ and so on. By the optimality of $ \left(\widehat{\mathbf{z}}_t\right)_{t=1}^T$ we have
\begin{equation*}
\label{eq:tv_cone1}
\sum_{t=1}^T \frac{\|\mathbf{z}_t + \mathbf{h}_t\|_1}{\sqrt{s_t}} \leq \sum_{t=1}^T \frac{\|\mathbf{z}_t\|_1}{\sqrt{s_t}}.
\end{equation*}
Using several instances of triangle inequality we have
\begin{align*}
\nonumber & \sum_{t=1}^T \frac{-\sigma_{s_t}(\mathbf{z}_t) + \|\left(\mathbf{h}_{t}\right)_{S_t^c}\|_1-\|\left(\mathbf{h}_{t}\right)_{S_t}\|_1+\|\left(\mathbf{z}_{t}\right)_{S_t}\|_1}{\sqrt{s_t}} & \leq \sum_{t=1}^T \frac{\|\mathbf{z}_t + \mathbf{h}_t\|_1}{\sqrt{s_t}} & \leq \sum_{t=1}^T \frac{\|\mathbf{z}_t\|_1}{\sqrt{s_t}} = \sum_{t=1}^T \frac{\|\left(\mathbf{z}_{t}\right)_{S_t}\|_1+\sigma_{s_t}(\mathbf{z}_t)}{\sqrt{s_t}},
\end{align*}
which after re-arrangement yields the cone condition given by
\begin{equation}
\label{eq:tv_cone_main}
\sum_{t=1}^T \frac{\|\left(\mathbf{h}_{t}\right)_{S_t^c}\|_1}{\sqrt{s_t}} \leq \sum_{t=1}^T \frac{ \|\left(\mathbf{h}_{t}\right)_{S_t}\|_1+\|\left(\mathbf{z}_{t}\right)_{S_t^c}\|_1}{\sqrt{s_t}}.
\end{equation}
Also, by the definition of partitions $(I_j)_{j=1}^{r_t}$ we have
\begin{align}
\label{eq:tv_s_largest}
\notag \sum_{t=1}^T \sum_{j=2}^{r_t} \|\left(\mathbf{h}_{t}\right)_{I_j}\|_2 &\leq \sum_{t=1}^T \sum_{j=2}^{r_t}2\sqrt{s_t} \|\left(\mathbf{h}_{t}\right)_{I_j}\|_\infty \\
\notag & \leq \sum_{t=1}^T \sum_{j=2}^{r_t} \frac{\|\left(\mathbf{h}_{t}\right)_{I_{j-1}}\|_1}{2\sqrt{s_t}} = \sum_{t=1}^T \frac{\|\left(\mathbf{h}_{t}\right)_{S_t^c}\|_1}{2\sqrt{s_t}} \\
& \leq \displaystyle \sum_{t=1}^T \frac{ \|\left(\mathbf{h}_{t}\right)_{S_t}\|_1+\sigma_{s_t}(\mathbf{z}_t)}{2\sqrt{s_t}} \leq \displaystyle\sum_{t=1}^T \frac{\|\left(\mathbf{h}_{t}\right)_{S_t}\|_2}{2}+\frac{ \sigma_{s_t}(\mathbf{z}_t)}{2\sqrt{s_t}}.
\end{align}
Moreover, using the feasibility of both $\mathbf{x}_t$ and $\widehat{\mathbf{x}}_t$ we have the tube constraints
\begin{equation*}
\begin{array}{l} 
\|\mathbf{y}_1-\mathbf{A}_1\mathbf{x}_1\|_2 \leq \varepsilon \Rightarrow \|\theta \left( \mathbf{y}_{1}\right)_{[n_2]}-\theta\mathbf{A}_2\mathbf{x}_1\|_2 \leq \theta  \varepsilon,\\
\|\mathbf{y}_2-\mathbf{A}_2\mathbf{x}_2\|_2 \leq \sqrt{\frac{n_2}{n_1}}\varepsilon ,
\end{array}
\end{equation*}
from which we conclude $\|\mathbf{y}_2-\theta \left( \mathbf{y}_{1}\right)_{[n_2]}-\mathbf{A}_2\mathbf{z}_2\|_2 \leq (1+\theta)\varepsilon$. Similarly $\|\mathbf{y}_2-\theta \left(\mathbf{y}_1\right)_{[n_2]}-\mathbf{A}_2\widehat{\mathbf{z}}_2\|_2 \leq (1+\theta)\varepsilon$. Therefore the triangle inequality yields $\|\mathbf{A}_2{\mathbf{h}}_2\|_2 \leq 2 (1+\theta)\varepsilon$. Similarly for all $t \in [T] \backslash \{2\}$, we have the tighter bound  
\begin{equation}
\label{eq:tv_tube}
\|\mathbf{A}_t{\mathbf{h}}_t\|_2 \leq 2 (1+\theta)\sqrt{\frac{n_t}{n_1}}\varepsilon,
\end{equation}
which is a consequence of having fewer measurements for $t \in [T] \backslash \{2\}$. In conjunction, (\ref{eq:tv_cone_main}), (\ref{eq:tv_s_largest}), and (\ref{eq:tv_tube}) yield
\begin{align*}
 \displaystyle 2(1+\theta)\left(T+\sqrt{\frac{n_1}{n_2}}-1\right) \varepsilon  \  & \geq \|\mathbf{A}_1\mathbf{h}_1\|_2+ \sum_{t=2}^T \sqrt{\frac{n_1}{n_t}}\|\mathbf{A}_t\mathbf{h}_t\|_2  \\
& \geq \sum_{t=1}^T \|\widetilde{\mathbf{A}}_t\left(\mathbf{h}_{t}\right)_{S_t \cup I_1}\|_2 -\sum_{t=1}^T \sum_{j=2}^{r_t}\|\widetilde{\mathbf{A}}_t
\left(\mathbf{h}_{t}\right)_{I_j}\|_2\\
& \geq \sqrt{1-\delta_{4s}} \sum_{t=1}^T \|\left(\mathbf{h}_{t}\right)_{S_t \cup I_1}\|_2 - \frac{\sqrt{1+\delta_{4s}}}{2} \sum_{t=1}^T \sum_{j=2}^{r_t}\|
\left(\mathbf{h}_{I_j ,t}\right)\|_2\\
& \geq \displaystyle \sqrt{1-\delta_{4s}} \sum_{t=1}^T \|\left(\mathbf{h}_{t}\right)_{S_t \cup I_1}\|_2 - \frac{\sqrt{1+\delta_{4s}}}{2} \sum_{t=1}^T \|\left(\mathbf{h}_{t}\right)_{S_t \cup I_1}\|_2+\frac{\sigma_{s_t}(\mathbf{z}_t)}{\sqrt{s_t}}\\
& \geq \displaystyle \left( \sqrt{1-\delta_{4s}}-\frac{\sqrt{1+\delta_{4s}}}{2}\right) \sum_{t=1}^T \|\left(\mathbf{h}_{t}\right)_{S_t \cup I_1}\|_2 -\frac{\sqrt{1+\delta_{4s}}}{2}\sum_{t=1}^T \frac{ \sigma_{s_t}(\mathbf{z}_t)}{\sqrt{s_t}}.
\end{align*}
Therefore after rearrangement for $\delta_{4s} <1/3$
\begin{align*}
\displaystyle \sum_{t=1}^T \|\left(\mathbf{h}_{t}\right)_{S_t \cup I_1}\|_2 \leq 8.37(1+\theta) \left(T+\sqrt{\frac{n_1}{n_2}}-1\right)\varepsilon + \frac{5}{2}\sum_{t=1}^T \frac{ \sigma_{s_t}(\mathbf{z}_t)}{\sqrt{s_t}}. 
\end{align*}
Next, using (\ref{eq:tv_s_largest}) yields
\begin{align}
\label{eq:tv_bound_h}
\notag \sum_{t=1}^T \|\mathbf{h}_{t}\|_2 & \leq \sum_{t=1}^T \sum_{j=2}^{r_t} \|\left(\mathbf{h}_{t} \right)_{I_j}\|_2 + \|\left(\mathbf{h}_{t}\right)_{S_t \cup I_1}\|_2 \\
& \leq \displaystyle 12.55 \left(T+\sqrt{\frac{n_1}{n_2}}-1\right) \varepsilon + 3\sum_{t=1}^T \frac{ \sigma_{s_t}(\mathbf{z}_t)}{\sqrt{s_t}}.
\end{align}
By definition we have $\mathbf{h}_t = \mathbf{g}_t - \theta \mathbf{g}_{t-1}$ for $t \in [T]$ with $\mathbf{g}_0 = \mathbf{0}$. Therefore by induction we have $\mathbf{g}_t = \sum_{j=1}^t \theta^{t-j} \mathbf{h}_j$ or in matrix form
\begin{equation}
\nonumber
\mathcal{G} \!:=\! \left[ \begin{array}{l}
\mathbf{g}_1 \\
\mathbf{g}_2 \\
\vdots\\
\mathbf{g}_t
\end{array} \right] = {\underbrace{\left[ \begin{array}{llll}
\mathbf{I} & 0& \cdots & 0\\
\theta \mathbf{I} & \mathbf{I} &  \cdots & 0\\
\theta^2 \mathbf{I} & \theta \mathbf{I} &  \cdots & 0\\
\vdots & \vdots & \ddots &\vdots\\
\theta^{T-1} \mathbf{I} &\theta^{T-2} \mathbf{I} & \cdots& \mathbf{I}
\end{array} \right]}_{\mathcal{A}}}  {\left[ \begin{array}{l}
\mathbf{h}_1 \\
\mathbf{h}_2 \\
\vdots\\
\mathbf{h}_t
\end{array} \right]} \!=:\! \mathcal{A}\mathcal{H}.
\end{equation}
Using several instances of the triangle inequality we get:
\begin{align*}
\sum_{t=1}^T \|\mathbf{g}_t\|_2 &= \sum_{t=1}^T \left\|\sum_{j=1}^t \theta^{t-j} \mathbf{h}_j \right\|_2 \leq \sum_{t=1}^T \sum_{j=1}^t \theta^{t-j} \left\|\mathbf{h}_j\right\|_2 \\
&\leq \sum_{j=1}^{T-1} \theta^j \sum_{t=1}^T \left\|\mathbf{h}_t\right\|_2 = \frac{1-\theta^T}{1-\theta} \sum_{t=1}^T \|\mathbf{h}_t\|_2,
\end{align*}
which in conjunction with (\ref{eq:tv_bound_h}) completes the proof.
\end{proof}

\subsection{The Expectation Maximization Algorithm} \label{app:tv_EM}
In this section we give a short overview of the EM algorithm and its connection to iteratively re-weighted least squares (IRLS) algorithms. More details can be found in \cite{babadi_IRLS} and the references therein. Given the observations $\mathbf{y}$, the goal of the EM algorithm is to find the ML estimates of a set of parameters $\mathbf{\Theta}$ by maximizing the likelihood $\mathfrak{L}({\mathbf{\Theta}}):= p(\mathbf{y}|{\mathbf{\Theta}})$. Such maximization problems are typically intractable, but often become significantly simpler by introducing a latent  variable $\mathbf{u}$. The EM algorithm connects solving the ML problem to maximizing $\widetilde{\mathfrak{L}}(\mathbf{\Theta}):=p(\mathbf{y},\mathbf{u}|\mathbf{\Theta})$, if one knew $\mathbf{u}$. 

Consider the state-space model:
\begin{equation}
\label{eq:NI_state_space}
\begin{array}{l}
\mathbf{x}_t = \mathbf{\Theta} \mathbf{x}_{t-1}+ \frac{{\boldsymbol{\omega}}_t}{\sqrt{\mathbf{u}_t}},\\
\mathbf{y}_t = \mathbf{A}_t \mathbf{x}_t + \mathbf{v}_t, \quad \mathbf{v}_t\sim \mathcal{N}(\mathbf{0},\sigma^2 \mathbf{I}),
\end{array}
\end{equation}
where ${\boldsymbol{\omega}} \sim \mathcal{N}(0, \mathbf{I})$, $\mathbf{u}_t$ is a positive i.i.d. random vector, and the square root operator and division of the two vectors are understood as element-wise operations. Let $\delta_{t,j}^2: = (\mathbf{x}_t-\mathbf{\Theta}\mathbf{x}_{t-1})_j^2$ for $j=1,2,\cdots,p$. For an appropriate choice of the density of $(\mathbf{u}_t)_j$ denoted by $p_U(\cdot)$, we have \cite{babadi_IRLS}:
\begin{equation}
\nonumber p\big({\textstyle \frac{\boldsymbol{\omega}_t}{\sqrt{\mathbf{u}_t}}}|\mathbf{\Theta}\big)= p(\mathbf{x}_t| \mathbf{x}_{t-1},\mathbf{\Theta})  \propto \exp \Big({\textstyle -\lambda \sum_{j=1}^p \kappa(\delta_{t,j}^2)}\Big),
\end{equation}
where
\begin{equation}
\kappa(z) := -2 \ln \left( \int_{0}^{\infty} u^{n/2} e^{-uz/2}p_U(u) du \right), \forall z \ge 0,
\end{equation}
and $\kappa'(z)$ is a completely monotone function \cite{lange1993normal}. Random vectors of the form $\mathbf{w}_t = \frac{{\boldsymbol{\omega}}_t}{\sqrt{\mathbf{u}_t}}$ are known as Normal/Independent \cite{lange1993normal}. Note that a choice of $\kappa(z) = \sqrt{z^2 + \epsilon^2}$ results in the $\epsilon$-perturbed Laplace distributions used in our model \cite{babadi_IRLS}. Given $T$ observations $(\mathbf{y}_t)_{t=1}^T \in \mathbb{R}^{n_t}$ and conditionally independent samples $(\mathbf{x}_t)_{t=1}^T \in \mathbb{R}^p$, we denote the objective function of the MAP estimator by $\mathfrak{L}(\left(\mathbf{x}_t\right)_{t=1}^T,\mathbf{\Theta})$, that is $\log \mathfrak{L}(\left(\mathbf{x}_t\right)_{t=1}^T,\mathbf{\Theta}) = \sum_{t=1}^T\log p(\mathbf{y}_t|\mathbf{x}_t, \mathbf{\Theta})+\log p\left(\mathbf{x}_t|\mathbf{x}_{t-1},\mathbf{\Theta}\right)$. Consider the current estimates $\left\{ \big(\widehat{\mathbf{x}}_t^{(l)}\big)_{t=1}^T,\widehat{\mathbf{\Theta}}^{(l)}\right\}$ at iteration $l$. Then:
\begin{align}
\label{eq:tv_EM_Q}
\notag &\log \mathfrak{L}(\left(\mathbf{x}_t\right)_{t=1}^T,\mathbf{\Theta}) - \sum_{t=1}^T\log p(\mathbf{y}_t|\mathbf{x}_t, \mathbf{\Theta})\\
\notag &=\sum_{t,j=1}^{T,p}\log \left(\int_{\left(\mathbf{u}_t\right)_j}p\left(\left({\boldsymbol{\omega}}_t\right)_j,\left(\mathbf{u}_t\right)_j|\mathbf{\Theta} \right) d\left(\mathbf{u}_t\right)_j\right) \\
\notag &= \displaystyle \sum_{t,j=1}^{T,p}\log \left(\int_{\left(\mathbf{u}_t\right)_j} \frac{p\left(\left(\mathbf{u}_t\right)_j\Big|\Big (\big(\widehat{\mathbf{x}}_{t}^{(l)}\big)_j\Big)_{t=1}^T,\widehat{\mathbf{\Theta}}^{(l)}\right)}{p\left(\left(\mathbf{u}_t\right)_j\Big|\Big (\big(\widehat{\mathbf{x}}_{t}^{(l)}\big)_j\Big)_{t=1}^T,\widehat{\mathbf{\Theta}}^{(l)}\right)}p \left(\left({\boldsymbol{\omega}}_t\right)_j,\left(\mathbf{u}_t\right)_j|\mathbf{\Theta}\right)d\left(\mathbf{u}_t\right)_j \right)\\
\notag & \geq \displaystyle \sum_{t,j=1}^{T,p}\int_{\left(\mathbf{u}_t\right)_j} p\left(\left(\mathbf{u}_t\right)_j\Big|\Big (\big(\widehat{\mathbf{x}}_{t}^{(l)}\big)_j\Big)_{t=1}^T,\widehat{\mathbf{\Theta}}^{(l)} \right) \log \left( \frac{p\left(\left({\boldsymbol{\omega}}_t\right)_j,\left(\mathbf{u}_t\right)_j|\mathbf{\Theta}\right)}{p \left(\left(\mathbf{u}_t\right)_j\Big|\Big (\big(\widehat{\mathbf{x}}_{t}^{(l)}\big)_j\Big)_{t=1}^T,\widehat{\mathbf{\Theta}}^{(l)}\right)}\right)d\left(\mathbf{u}_t\right)_j\\
 & =\displaystyle \sum_{t,j=1}^{T,p}\mathbb{E}_{\left(\mathbf{u}_t\right)_j \Big|\Big (\big(\widehat{\mathbf{x}}_{t}^{(l)}\big)_j\Big)_{t=1}^T,\widehat{\mathbf{\Theta}}^{(l)}}\Big\{\log p(\left({\boldsymbol{\omega}}_t\right)_j,\left(\mathbf{u}_t\right)_j|\mathbf{\Theta}) \Big\} + C,
\end{align}
where the inequality follows from Jensen's inequality and the constant $C$ accounts for terms which do not depend on $\mathbf{\Theta}$. The so called Q-function is defined as:
\begin{align}\label{eq:Qf}
\nonumber & Q\Big(\left(\mathbf{x}_t\right)_{t=1}^T,\mathbf{\Theta}\Big|\left(\widehat{\mathbf{x}}_t^{(l)}\right)_{t=1}^T,\widehat{\mathbf{\Theta}}^{(l)}\Big):= \sum_{t=1}^T\log p(\mathbf{y}_t|\mathbf{x}_t, \mathbf{\Theta}) \\
& +\sum_{t,j=1}^{T,p}\mathbb{E}_{\left(\mathbf{u}_t\right)_j \Big|\Big (\big(\widehat{\mathbf{x}}_{t}^{(l)}\big)_j\Big)_{t=1}^T,\widehat{\mathbf{\Theta}}^{(l)}}\Big \{ \log p(\left({\boldsymbol{\omega}}_t\right)_j,\left(\mathbf{u}_t\right)_j|\mathbf{\Theta}) \Big \}.
\end{align}
The EM algorithm maximizes the lower bound given by the Q-function of (\ref{eq:Qf}) instead of the log-likelihood itself. Moreover for all $t \in [T], j \in [p]$ and $\kappa(z) = \sqrt{z^2 + \epsilon^2}$ we have \cite{lange1993normal}:
\begin{align*}
\displaystyle \mathbb{E}_{\left(\mathbf{u}_t\right)_j \Big|\Big (\big(\widehat{\mathbf{x}}_{t}^{(l)}\big)_j\Big)_{t=1}^T,\widehat{\mathbf{\Theta}}^{(l)}}\Big\{\log p(\left({\boldsymbol{\omega}}_t\right)_j,\left(\mathbf{u}_t\right)_j|\mathbf{\Theta}) \Big\} = -\frac{\lambda}{2} \frac{(\mathbf{x}_t-\mathbf{\Theta}\mathbf{x}_{t-1})_j^2 + \epsilon^2}{\sqrt{\big(\widehat{\mathbf{x}}^{(l)}_t-\widehat{\mathbf{\Theta}}^{(l)}\widehat{\mathbf{x}}^{(l)}_{t-1}\big)_j^2 + \epsilon^2}},
\end{align*}
which after replacement results in the state-space model given by (\ref{eq:tv_dynamic}). This expectation gets updated in the outer EM loop using the \textit{final} outputs of the inner loop. The outer EM algorithm can thus be summarized as forming the Q-function (E-step) and maximizing over $\boldsymbol{\Theta}$ (M-step), which is known to converge to a stationary point due to its ascent property \cite{babadi_IRLS}. As discussed in Section \ref{sec:tv_algorithm}, the outer M-step is implemented by another instance of the EM algorithm by alternating between Fixed Interval Smoothing (E-step) and updating $\boldsymbol{\Theta}$ (M-step).

%{EM penalty: effect of going from $\ell_1$ to $\ell_{1,\epsilon}$}.

%\begin{equation}
%p_{\mathbf{Y}}(\mathbf{y}) = \frac{1}{4^n}e^{-\|\mathbf{y}\|_1/2}.
%\end{equation}

\chapter{Statistical Tests of Goodness of Fit}

In this appendix, we will give an overview of the statistical goodness-of-fit tests for assessing the accuracy of the AR model estimates. A detailed treatment can be found in  \cite{lehmann1986testing}. 
\section{Goodness-Of-Fit Tests forAutoregressive Models} \label{app:ar_tests}
%{These tests assume knowledge of the reference distribution $F_0$ for the null hypothesis. For simulation purposes we have used a two-fold cross-validation estimating $F_0$ using half of the data as training.}
\subsection{Residue-based tests}
Let $\widehat{\boldsymbol{\theta}}$ be an estimate of the parameters of the process. The residues (estimated innovations) of the process based on $\widehat{\boldsymbol{\theta}}$ are given by
\vspace{-.2cm}
\begin{equation*}
e_k = x_k - \widehat{\boldsymbol{\theta}}^T\mathbf{x}_{k-p}^{k-1}, \quad\quad i=1,2,\cdots,n. 
\end{equation*}
The main idea behind most of the available statistical tests is to quantify how close the sequence $\{e_i\}_{i=1}^{n}$ is to an i.i.d. realization of a known distribution $F_0$ which is most likely absolutely continuous . Let us denote the empirical distribution of the $n$-samples by $\widehat{F}_n$. If the samples are generated from $F_0$ the Glivenko-Cantelli theorem suggests that:
\begin{equation*}
\sup_t \;|\widehat{F}_n(t)-F_0(t)| \stackrel{\sf a.s.}\longrightarrow 0.
\vspace{-.2cm}
\end{equation*}
\noindent That is, for large $n$ the empirical distribution $\widehat{F}_n$ is uniformly close to $F_0$. The Kolmogorov-Smirnov (KS) test, Cram\'{e}r-von Mises (CvM) criterion and the Anderson-Darling (AD) test are three measures of discrepancy between $\widehat{F}_n$ and $F_0$ which are easy to compute and are sufficiently discriminant against alternative distributions. More specifically, the limiting distribution of the following three random variables are known:
\noindent The KS test statistic
\begin{equation*}
K_n:=\sup_t \; |\widehat{F}_n(t)-F_0(t)|,
\end{equation*}
the CvM statistic
\begin{equation*}
C_n:= \int \big(\widehat{F}_n(t)-F_0(t)\big)^2 dF_0(t),
\end{equation*}
and the AD statistic
\begin{equation*}
A_n:= \int \frac{\big(\widehat{F}_n(t)-F_0(t)\big)^2}{F_0(t)\left(1-F_0(t)\right)}dF_0(t).
\end{equation*}
For large values of $n$, the Glivenko-Cantelli theorem also suggests that these statistics should be small. A simple calculation leads to the following equivalent for the statistics:
\[
K_n = \max_{1\leq i \leq n} \max \left\{\left|\frac{i}{n}-F_0(e_i) \right|, \left|\frac{i-1}{n}-F_0(e_i) \right| \right\},
\]
\[
nC_n = \frac{1}{12n}+\sum_{i=1}^n \left(F_0(e_i)-\frac{2i-1}{2n} \right)^2,
\]
and
\vspace{-.2cm}
\[
nA_n = -n-\frac{1}{n}\sum_{i=1}^n\left(2i-1\right)\Big( \log F_0(e_i)+ \log\left(1-F_0(e_i)\Big)\right).
\]
\vspace{-.7cm}

\subsection{\textcolor{black}{Spectral domain} tests for Gaussian AR processes}
The aforementioned KS, CvM and AD tests all depend on the distribution of the innovations. For Gaussian AR processes, the spectral  versions of these tests are introduced in \cite{anderson1997goodness}. These tests are based on the similarities of the periodogram of the data and the estimated power-spectral density of the process. The key idea is summarized in the following lemma:
\begin{lemma}
\label{lemma:ar_spec_CvM}
Let $S(\omega)$ be the (normalized) power-spectral density of stationary process with bounded spectral spread, and $\widehat{S}_n(\omega)$ be the periodogram of the $n$ samples of a realization of such a process, then for all $\omega$ we have:
\begin{equation}
\label{eq:ar_spec_gaussian}
\sqrt{n}\left(2\int_{0}^{\omega}\left(\widehat{S}_n(\lambda)- S(\lambda)\right) d\lambda \right) \stackrel{\sf d.}\longrightarrow \mathcal{Z}(\omega),
\end{equation}
where $\mathcal{Z}(\omega)$ is a zero-mean Gaussian process.
%\[
%\mathbb{E}[\mathcal{Z}(\omega)\mathcal{Z}(\lambda) ] = 4\pi 
%\]
\end{lemma}
The explicit formula for the covariance function of $\mathcal{Z}(.)$ is calculated in \cite{anderson1997goodness}. Lemma \ref{lemma:ar_spec_CvM} suggests that for a good estimate $\widehat{\boldsymbol{\theta}}$ which admits a power spectral density $S(\omega;\widehat{\boldsymbol{\theta}})$,  one should get a (\textit{close} to) Gaussian process replacing $S(\omega)$ with $S(\omega;\widehat{\boldsymbol{\theta}})$ in (\ref{eq:ar_spec_gaussian}). The spectral form of the CvM, KS and AD statistics can thus be characterized given an estimate $\widehat{\boldsymbol{\theta}}$.

\section{Goodness-Of-Fit Tests for Point Process Models}\label{appks}

In this appendix, we will give an overview of the statistical tools used to assess the goodness-of-fit of point process models. A detailed treatment can be found in \cite{Brown_pp}.
\medskip

\subsection{The Time-Rescaling Theorem}
Let $0<t_1<t_2<\cdots$ be a realization of a continuous point process with conditional intensity $\lambda(t)>0$, i.e. $t_k$ is the first instance at which $N(t_k)=k$. Define the transformation

\begin{equation}
\label{trt}
z_k := Z(t_k)= \int_{t_{k-1}}^{t_k} \lambda(t) dt.
\end{equation}
Then, the transformed point process with events occurring at $t'_k = \sum_{i=1}^k z_k$ corresponds to a {homogeneous} Poisson process with rate 1. Equivalently,  $z_1,z_2,\cdots$ are \textit{i.i.d} \textit{exponential} random variables. The latter can be used to construct statistical tests for the goodness-of-fit.

\medskip

\subsection{The Komlogorov-Smirnov Test for Homogeneity}

Suppose that we have obtained the rescaled process through (\ref{trt}) with the \textit{estimated} conditional intensity. When applying the time-rescaling theorem to the discretized process, if the estimated conditional intensity is close to its true value, the rescaled process is expected to behave as a {homogeneous} Poisson process with rate $1$. The Kolmogorov-Smirnov (KS) test can be used to check for the homogeneity of the process. Let $z_k$'s be the rescaled times and define the transformed rescaled times by the inverse exponential CDF:
\begin{equation*}
u_k:= 1-e^{-z_k}.
\end{equation*}
If the true conditional intensity was used to rescale the process, the random variables $u_k$ must be i.i.d. ${\sf Uniform}(0,1]$ distributed. The KS test plots the empirical qualities of $u_k$'s versus the true quantiles of the uniform density given by $b_k = \frac{k-1/2}{J}$, where $J$ is the total number of observed spikes. If the conditional intensity is well estimated, the resulting curve  must lie near the $45^\circ$ line. The asymptotic statistics of the KS distribution can be used to construct confidence intervals for the test. For instance, the $95\%$ and $99 \%$ confidence intervals are approximately given by $\pm \frac{1.36}{\sqrt{J}}$ and $\pm \frac{1.63}{\sqrt{J}}$ hulls around the $45^\circ$ line, respectively.

\medskip
\subsection{The Autocorrelation Function Test for Independence}
 In order to check for the independence of the resulting rescaled intervals $z_k$, the following transformation is used:
\begin{equation*}
v_k = \Phi^{-1}(u_k)
\end{equation*}
where $\Phi$ is the standard Normal CDF. If the true conditional intensity was used to rescale the process, then $v_k$'s would be i.i.d. Gaussian and their uncorrelatedness would imply independence. The Autocorrelation Function (ACF) of the variables $v_k$ must then be close to the discrete delta function. The $95\%$ and $99 \%$ confidence intervals can be considered using the asymptotic statistics of the sample ACF, approximately given by $\pm \frac{1.96}{\sqrt{J}}$ and $\pm \frac{2.575}{\sqrt{J}}$, respectively.
\medskip

\noindent \textit{\textbf{Remark.}} The binning size used for discretizing the data can potentially affect the ISI distribution of the time-rescaled process. In order to avoid these issues, we have used the empirical ISI distribution estimated from a large realization of the process (estimated from the training data) as the null hypothesis for both tests (performed on the test data).

{
\small
\bibliographystyle{plain}
\bibliography{proposal}
}

\end{document}